\documentclass[11pt]{article}
\usepackage{pdfsync}
\usepackage{amsfonts}
\usepackage{mathrsfs}
\usepackage{color}
\usepackage{amssymb,amsmath,amsthm}
\usepackage{ulem}

\usepackage{appendix}
\usepackage{cite}
\usepackage{hyperref}
\hypersetup{
	citecolor=red,
	urlcolor=cyan,
}

\usepackage[titles]{tocloft}

 \allowdisplaybreaks

\catcode`!=11
\let\!int\int \def\int{\!int}
\let\!lim\lim \def\lim{\!lim}
\let\!sum\sum \def\sum{\!sum}
\let\!sup\sup \def\sup{\!sup}
\let\!inf\inf \def\inf{\!inf}
\let\!max\max \def\max{\!max}
\let\!min\min \def\min{\!min}

\catcode`!=12

\let\oldsection\section
\renewcommand\section{\setcounter{equation}{0}\oldsection}

\newcommand{\ba}{\begin{aligned}}
\newcommand{\ea}{\end{aligned}}
\newcommand{\be}{\begin{equation}}
\newcommand{\ee}{\end{equation}}
\def\p{\partial}
\def\pf{\it{Proof.}\rm\quad}

\newcommand{\na}{\nabla}

\def\R{\mathbb{R}}

\newtheorem{thm}{Theorem}[section]

\newtheorem{pro}[thm]{Proposition}
\newtheorem{lem}[thm]{Lemma}
\newtheorem{cor}[thm]{\textbf Corollary}
\newtheorem{re}[thm]{Remark}

\everymath{\displaystyle}

\begin{document}
\title{\bf Stability threshold for 3D Boussinesq equations with rotation near the Couette flow and stratified temperature
 		\thanks{hwting702@163.com (Wenting Huang), luozekailzk@163.com (Zekai Luo), sunying@bnu.edu.cn (Ying Sun), xjxu@bnu.edu.cn (Xiaojing Xu)}
\author{
{Wenting Huang$^1$, \quad Zekai Luo$^1$, \quad Ying Sun$^2$,  \quad    Xiaojing Xu$^1$ } \\
\small   $^1$ School of Mathematical Sciences, Laboratory of Mathematics and Complex Systems, MOE,   \\  \small Beijing Normal University, Beijing, 100875, China.\\  \small   $^2$ School of Mathematical Sciences, Beijing University of Posts and Telecommunications, \\ \small Beijing, 100876, China.}
}
\date{}
\author{xxx}
\maketitle

\noindent{\bf Abstract.} 
The most important features that distinguish fluid flow in the
atmosphere and ocean are the effects of rotation and stratification, which are described by the Boussinesq equations. This paper examines the stability threshold at high Reynolds numbers $\textbf{Re}$ for the three-dimensional   Boussinesq equations with rotation on the domain $\Omega=\{(x,\,y,\,z)\in \mathbb{T} \times \mathbb{R} \times  \mathbb{T}\}$ around the Couette flow $(y,0,0)$ and the  vertically  stratified temperature $\Theta_s=1+\alpha^2 z$. For the linear system without rotation, stratification not only suppresses the lift-up effect but also exhibits certain dispersion effects, except for some points where degradation occurs, which will bring essential difficulties to nonlinear estimates.  In contrast, when rotation is taken into account, we observe that this degeneracy in dispersion effects disappears; furthermore, we can derive dispersive estimates for the second and third components of the simple-zero mode within the velocity field. However, neither the first component of velocity nor temperature exhibits any dispersion characteristics. Fortunately, we find a combined quantity from these components to obtain its dispersive estimates. Additionally, we develop three good unknowns to minimize linear coupling terms as much as possible while mitigating growth induced by linear stretching terms; through constructing a series of multipliers, we achieve enhanced dissipation and inviscid damping effects.
In our analysis of the nonlinear system aimed at establishing an improved stability threshold, we utilize quasi-linearization methods to rectify deficiencies in dispersive estimates related to both the first component of velocity and temperature, as well as address regularity issues along vertical directions caused by buoyancy forces and stratification.
Consequently, we demonstrate that if initial perturbations in velocity and temperature satisfy
$$\left\|u_{\mathrm{in}}\right\|_{H^{N+2}\cap W^{N+3,1}}+\left\|\theta_{\mathrm{in}}\right\|_{H^{N+1}\cap W^{N+3,1}}<\delta \mathbf{Re}^{-\frac{14}{15}},$$
for any $N\geq 11$ and some $\delta>0$ independent of $\mathbf{Re}$, then the solution to the 3D Boussinesq equations with rotation is nonlinearly stable  without transitioning away from the steady state.



\tableofcontents

\section{Introduction}\label{HSXSEC1}

\qquad In this paper, we consider the following three-dimensional (3D) Boussinesq equations with rotation, which model the motion of incompressible ﬂuids under both the Coriolis force and buoyancy:
\begin{equation}\label{1.1}
 \begin{cases}
\partial_{t} V+V \cdot \nabla V-\nu \Delta V+\nabla P+\beta \mathrm{e}_3 \times V=\Theta \mathbf{g} \mathrm{e}_3, \\
\partial_{t} \Theta +V \cdot \nabla \Theta-\mu \Delta \Theta=0, \\
\nabla \cdot V=0.
\end{cases}
\end{equation}
where $(x_1, y, z) \in \mathbb{T} \times \mathbb{R} \times \mathbb{T}$ (the torus $\mathbb{T}$ is the periodized interval $[0,2\pi]$), $V=\left(V^1, V^2, V^3\right)^{\mathrm{T}}$ denotes the velocity field, $P$ is the pressure, $\Theta$ is the temperature, $\mathrm{e}_3=(0, 0, 1)^{\mathrm{T}}$ is the unit vector in the vertical direction. The term $\beta \mathrm{e}_3 \times V$ represents the Coriolis force, where $ \beta \in \mathbb{R}\backslash\{0\}$ characterizes the rotation speed $\beta/2$. $\Theta \mathbf{g} \mathrm{e}_3$ symbolizes the buoyancy with the  gravity $\mathbf{g}$. $\nu$ represents the viscosity coefficient, i.e., $\nu=\mathbf{Re}^{-1}$, where $\mathbf{Re}$ is the Reynolds number. And $\mu$ is the thermal diffusivity. For simplicity, we take $\mu=\nu,\,\mathbf{g}=1$. The derivation of  \eqref{1.1} can be referred to
 \cite{Majda2003, Pedlosky1987}. The system \eqref{1.1} is supplemented with initial data:
\begin{align}\label{invalue}
(V, \Theta)(0, x_1, y, z)=(V_{\mathrm{in}}, \Theta_{\mathrm{in}})(x_1, y, z).
\end{align}

We  study the dynamic stability behavior of solutions to \eqref{1.1}
around the nontrivial stationary solution, which is known as Couette flow with the linearly stratified temperature profile 
\begin{equation}\label{1.2}
V_s(y)=(y, 0, 0)^{\mathrm{T}},\,\,\,\,\Theta_s(z)=1+\alpha^2 z,
\end{equation}
where $\alpha > 0$ represents the Brunt-Väisälä frequency, which indicates the fluid's response strength to displacements in the direction of gravity. This scenario corresponds to a stable stratification where warmer fluid overlays colder fluid, with gravity acting as a restoring force for vertical perturbations and thereby supporting internal gravity waves.
The associated pressure  of this steady state is
\begin{equation}\label{1.4}
P_s(y, z)=z+\frac{1}{2} (\alpha^2z^2-\beta y^2).
\end{equation}

We  introduce the perturbations
   $u=V-(y, 0, 0)^{\mathrm{T}},  \theta=-\frac{1}{\alpha}\left(\Theta-\Theta_s \right),  p=P-P_s$,
the perturbed $(u, p, \theta)$ satisfies the following system
\begin{equation}\label{1.6}
\begin{cases}
\partial_{t} u+y \partial_{x_1} u-\nu \Delta u+\nabla p^L+\begin{pmatrix}
(1-\beta) u^2 \\
\beta u^1 \\
\alpha \theta
\end{pmatrix}=-u \cdot \nabla u-\nabla p^{NL}, \\
\partial_{t} \theta+y \partial_{x_1} \theta-\nu \Delta \theta-\alpha u^3=-u \cdot \nabla \theta, \\
\nabla \cdot u=0, \\
(u, \theta)(0, x_1, y, z)=(u_{\mathrm{in}}, \theta_{\mathrm{in}})(x_1, y, z),
\end{cases}
\end{equation}
where the initial data $u_{\mathrm{in}}=V_{\mathrm{in}}-(y, 0,0)^{\mathrm{T}}$, $\theta_{\mathrm{in}}=-\frac{1}{\alpha}\left(\Theta_{\mathrm{in}}-(1+\alpha^2 z)\right)$. 
The pressure $p^{L}$ and $p^{NL}$ are determined by 
\begin{align*}
\Delta p^{L}=-2\partial_{x_1} u^2+\beta\left(\partial_{x_1} u^2 -\partial_{y} u^1\right)-\alpha \partial_{z} \theta, \quad 
\Delta p^{NL}=-\partial_{i}u^j \partial_{j}u^i. 
\end{align*}

The Coriolis force, in its role concerning fluid stability, is influenced by the variation of the parameter $\beta$. The relevant physical quantity for quantifying the effects of rotation is known as the Bradshaw-Richardson number (see \cite{Bradshaw1969,Huang2018}), $$B_{\beta}:= \beta(\beta-1)\in [-\frac{1}{4},\infty).$$ This number effectively captures and quantifies the stable or unstable nature of dynamics described in equation \eqref{1.6}.

\subsection{Background and Motivation}
\qquad Before elaborating on the background and motivation,   some relevant notations are introduced. The Fourier transform of a function $f= f(x_1, y, z)$  will be denoted as follows.   Given $(k,\eta,l)\in \mathbb{Z}\times \mathbb{R}\times \mathbb{Z}$, define 
\begin{equation}\label{2.1}
	\mathcal{F}(f)=\widehat{f}\left(k, \eta, l\right)=\frac{1}{(2 \pi)^{2}}\int_{x_1 \in \mathbb{T}}\int_{y \in \mathbb{R}} \int_{z \in \mathbb{T}}f\left(x_1, y, z\right)e^{- i  \left(k x_1+ \eta y+ l z\right)}  dx_1dydz.
\end{equation}
Denoting the inverse operation to the Fourier transform by $\mathcal{F}^{-1}$ or $\check{f}$
\begin{equation}\label{2.2}
	\mathcal{F}^{-1}(f)=\check{f} \left(x_1, y, z\right)=\sum_{k \in \mathbb{Z}}\int_{\eta \in \mathbb{R}} \sum_{l \in \mathbb{Z}} f \left(k, \eta, l\right)e^{ i \left(k x_1+ \eta y+ l z\right)}  d\eta.
\end{equation}
In addition, we also define the Fourier transform in the $z$-direction of a function $h= h(y, z)$ by
\begin{equation*}
	h^l (y) = \frac{1}{2 \pi}\int_{z \in \mathbb{T}}h\left(y, z\right)e^{- i  l z} dz.
\end{equation*}
For a general Fourier multiplier with symbol  $m(k,\eta,l)$, we write $mf$ to denote $\mathcal{F}^{-1}(m(k,\eta,l)\hat{f})$. Thus, it holds
\begin{equation*}
	\widehat{m(D)f}=m(k,\eta,l)\widehat{f}.
\end{equation*}

 We define the projections on the zero modes and the non-zero modes in $x_1$ of a function $f$  as follows:
\begin{equation}\label{HSX333}
	P_{k=0}f=f_{0}=\frac{1}{2 \pi}\int_{\mathbb{T}} f(x_1,y,z)d{x_1}, \quad	P_{\neq}f=f_{\neq}=f-f_{0}. 
\end{equation}
Furthermore, the projection onto the zero frequency in $x_1$ and $z$ of a function $f(x_1,y,z)$ is denoted by
\begin{equation*}
	\overline{f}_0=\frac{1}{(2 \pi)^2}\int_{\mathbb{T}} \int_{\mathbb{T}} f(x_1,y,z)d{x_1} dz, \quad \quad \widetilde{f}_0=f_0-\overline{f}_0,
\end{equation*}
where $\overline{f}_{0}$ represents the double-zero modes and $\widetilde{f}_{0}$ stands for the simple-zero modes of $f$.

\subsubsection{Known results for the Navier-Stokes equation with/without rotation}
\qquad The stability of plane Couette flow has been extensively studied since the pioneering works of Rayleigh \cite{R1879} and Kelvin \cite{K1887}. It is well established that plane Couette flow remains linearly stable for all Reynolds numbers, as noted in \cite{DR, R1973}. However, at high Reynolds numbers, the flow can exhibit nonlinear instability and transition to turbulence under small but finite perturbations—a phenomenon known as the Sommerfeld paradox. One approach to resolving this paradox involves examining the transition threshold problem, a concept rigorously formulated by Trefethen et al. \cite{T1993}. Bedrossian, Germain, and Masmoudi \cite{MR3612004, MR3974608} developed a comprehensive mathematical framework for this issue that can be articulated as follows:

Given a norm $\|\cdot\|_{X}$, determine an exponent $\gamma = \gamma(X)$ such that
\begin{align*}
	&\left\|  u_{\mathrm{in}}  \right\|_{X} \lesssim \; \mathbf{Re}^{-\gamma } \,\, \Rightarrow  \text{ stability},\\
	&
	\left\|  u_{\mathrm{in}}  \right\|_{X}  \gg  \;\mathbf{Re}^{-\gamma } \,\, \Rightarrow  \text{ instability},
\end{align*}
where $\gamma$ is called the \textit{transition threshold}. 

In the absence of the rotation and temperature, the system \eqref{1.1} reduces to the standard Navier-Stokes equations. Significant progress has been made on the stability threshold problem for  Navier-Stokes equations near the Couette flow. For 2D case on the domain
$\mathbb{T}\times \mathbb{R}$,    Bedrossian, Masmoudi and Vicol \cite{MR3448924} demonstrated that the threshold  $\gamma\leq 0$ for the perturbations in Gevrey space.  Bedrossian, Vicol and Wang \cite{MR3867637} showed that  the threshold $\gamma\leq 1/2$ for the perturbation  in Sobolev space $H^N$ with  $N>1$. Wei and Zhang \cite{WZ2023} made the energy estimates in short and long time scale, and constructed a new multiplier to obtain that the threshold  $\gamma\leq 1/3$ in $H^3$. On the bounded domain $\mathbb{T}\times [-1,1]$, the results  depend on the boundary conditions. Under Navier-slip boundary conditions,  Wei and Zhang \cite{WZ2026}  proved that $\gamma\leq {1}/{3}$ for the perturbations in Sobolev class, and Bedrossian, He, Iyer and Wang \cite{BHIW2-2024} showed that  $\gamma\leq 1/2 $ for the perturbations in Gevrey space; Under no-slip boundary conditions, $\gamma\leq {1}/{2}$ for the perturbations in Sobolev space \cite{MR4121130}.  There are still many important works in this area, one can refer to the references \cite{AB2025, MR3974608, BHIW1-2024, BHIW2-2024,   DWZ2021, LLZ2025, LMZ2025, MR4176913} and their related literature. The 3D case is considerably more challenging due to the lift-up effect.  Bedrossian, Germain and Masmoudi \cite{MR3612004} designed the  Fourier multipliers that precisely encode the interplay between the dissipation and possible growth  and   obtained the stability threshold index $\gamma\leq 3/2$ for Sobolev space $H^{\sigma}$ $(\sigma>9/2)$. For a related work also see  \cite{MR4126259, BGM2022} and so on. Later on, Wei and Zhang \cite{MR4373161} showed that the transition threshold $\gamma \leq1$ of 3D Navier-Stokes equations in $H^2$. They further provided a mathematically rigorous proof that the regularity of initial perturbations does not affect this transition threshold.  Chen, Wei and Zhang \cite{CWZ2024}  demonstrated that the transition threshold index is also $\gamma\leq 1$ in the finite channels with the non-slip boundary condition. 

Without thermal stratification (i.e., $\Theta=0$), the system \eqref{1.1}  corresponds to the Navier-Stokes equations with rotation. In contrast to the standard Navier-Stokes equations, the Navier-Stokes equation with rotation  not only
represents a fundamental physical model, but also exhibits stabilizing effects induced by the Coriolis force.
Our prior studies \cite{HSX2024,HSX2024-2} addressed the stability of the Navier-Stokes equation with rotation near Couette flow on the domain $\mathbb{T} \times \mathbb{R} \times \mathbb{T}$. The analysis considered two distinct parameter regimes: one with significant dispersive effects for $B_{\beta}>0$, and another exhibiting a strong lift-up effect at $B_{\beta}=0$. In the dispersive regime, we established the stability threshold  $\gamma\leq 1$. Subsequently, under the additional
condition that 
\begin{align}\label{0in}
	\overline{u_0}(t=0,y):=\int_{\mathbb{T}^2} u_{\mathrm{in}}(x_1,y,z) \mathrm{d}{x_1} \mathrm{d}z=0,
\end{align}
Coti Zelati, Del Zotto and Widmayer \cite{CZDZW2025} established nonlinear dispersive estimates, thereby obtaining the threshold $\gamma \leq 8/9$ for $B_{\beta}>0$ and $\gamma \leq 5/6$ for $B_{\beta}\gtrsim \nu^{-1}$. Recently, Li–Sun–Wang–Wei–Zhang \cite{LSWZ2025} further removed the restricted condition \eqref{0in}. By deriving suitable Strichartz estimates
\[
\left\| e^{\mathcal{A} t} h^l \right\|_{L_t^2 (L_y^{\infty}(\mathbb{R}))} 
+ \nu^{1/2} |l| \left\| e^{\mathcal{A} t} h^l \right\|_{L_t^1 (L_y^{\infty}(\mathbb{R}))} 
\lesssim \nu^{-1/3} |l|^{-\frac{1}{6}} \| h^l \|_{L^2(\mathbb{R})},
\]
where the dispersive semigroup $\mathcal{A}:=\sqrt{B_{\beta}} \mathcal{R}_3^l+\nu(\partial_y^2-l^2)$ with $\mathcal{R}_3^l:=-il(l^2-\partial_y^2)^{-\frac{1}{2}}$ and $h^l=\sum_{j \in \mathbb{Z}} P_j h^l$,
they obtained an improved threshold $\gamma \leq 5/6$ in anisotropic Sobolev spaces on the domain $\mathbb{T} \times \mathbb{R}\times\mathbb{T}$. Moreover, they also provided a  stability threshold of the form $2/3+\delta$ with $\delta>0$ on the domain $\mathbb{T} \times \mathbb{R}^2$.

\subsubsection{Motivation behind \eqref{1.1} with vertical stratification}\label{s1.1.2}
\qquad The interplay between rotation and vertical thermal stratification establishes a fundamental geophysical framework that underpins numerous atmospheric, oceanic, and astrophysical flows. Consequently, examining their combined effects on the stability of Couette flow is directly pertinent to geophysical applications. Recent experimental research by Oxley and Kerswell~\cite{OK2024} has initiated an investigation into the linear stability of the 3D rotating Boussinesq equations in relation to Couette–Poiseuille flow, taking into account the synergistic influences of stratification, rotation, and viscosity. However, the corresponding mathematical theory—particularly regarding the stability of the steady-state solution \eqref{1.2}-\eqref{1.4}—remains incomplete.

The stability of Couette flow under \textit{vertical} stratification $\Theta_s = 1 + \alpha^2 z$ (buoyancy aligned with $\theta\mathbf{e}_3$) remains an unsolved problem, particularly when rotation is present. This contrasts with the more well-developed theory for horizontally stratified configurations. We first recall the linear dispersion relations for two canonical non-rotating ($\beta=0$) setups:
	\begin{itemize}
		\item \textbf{Vertical stratification} ($\Theta_s = 1 + \alpha^2 z$ and $\theta\mathbf{e}_3$):\\
		Notice that the simple zero modes satisfy
		\begin{equation}
		\partial_{t} \begin{bmatrix}
		{\widetilde{U}_{0}^{2}} \\
		{\widetilde{\Theta}_{0}}
		\end{bmatrix}=\begin{bmatrix}
		\nu \Delta_{L,0} & \alpha \partial_{y} \Delta_{L,0}^{-1} \partial_{z} \\
		-\alpha \partial_{z}^{-1} \partial_{y} & \nu \Delta_{L,0}
		\end{bmatrix}\begin{bmatrix}
		{\widetilde{U}_{0}^{2}} \\
		{\widetilde{\Theta}_{0}}
		\end{bmatrix},\nonumber
		\end{equation}
		and the eigenvalues are
		\[
		\lambda_{\pm} = -\nu(\eta^2+l^2) \pm i\frac{\alpha|\eta|}{\sqrt{\eta^2+l^2}}.
		\]
		The dispersion relation is $\frac{h}{|\eta,l|} = \frac{\alpha|\eta|}{|\eta,l|}$. And it can be seen that $h$ will serve as the denominator of the single-zero modes. We encounter a singularity at wavenumber $\eta = 0$, which poses intrinsic difficulties in establishing the dispersion estimate.

		\quad When the buoyancy term $\Theta \mathrm{e}_3$ is replaced by $\Theta \mathrm{e}_2$ with $\mathrm{e}_2=(0,1,0)^{\mathrm{T}}$, Coti Zelati and Del Zotto \cite{Zelati2023, ZZW2024} first established the linear and nonlinear stability results for the 3D Boussinesq equations near the stably stratified Couette flow, i.e., $V_s = (y, 0, 0)$, $\Theta_s = 1 + \alpha^2 y$ with $\alpha > 0$. In addition, they found that the dispersive
		structure exhibited by the zero-frequency components as follows:
		\item \textbf{Horizontal stratification} ($\Theta_s = 1 + \alpha^2 y$ and  $\theta\mathbf{e}_2$):\\
		Coti Zelati, Del Zotto and Widmayer investigated the 3D Boussinesq system with horizontal stratification in  \cite{ZZW2024}. For the simple-zero modes satisfy
		\begin{equation}
		\partial_{t} \begin{bmatrix}
		{\widetilde{U}_{0}^{2}} \\
		{\widetilde{\Theta}_{0}}
		\end{bmatrix}=\begin{bmatrix}
		\nu \Delta_{L,0} & -\alpha \partial_{z} \Delta_{L,0}^{-1} \partial_{z} \\
		\alpha & \nu \Delta_{L,0}
		\end{bmatrix}\begin{bmatrix}
		{\widetilde{U}_{0}^{2}} \\
		{\widetilde{\Theta}_{0}}
		\end{bmatrix},\nonumber
		\end{equation}
		and the eigenvalues become
		\[
		\lambda_{\pm}=-\nu (\eta^2+l^2)\pm i \frac{\alpha |l|}{\sqrt{\eta^2+l^2}}.
		\]
		Here, dispersion persists for all simple-zero modes, namely $l\neq 0$, and no degeneracy occurs. Based on this,  under the extra condition in  \cite{ZZW2024} that 
		\[
		\int_{\mathbb{T}^2} u_{\mathrm{in}} \mathrm{d}{x_1} \mathrm{d}z=\int_{\mathbb{T}^2} \theta_{\mathrm{in}} \mathrm{d}{x_1} \mathrm{d}z=0,
		\]
		 Coti Zelati, Del Zotto and Widmayer further proved the stability thresholds: $\gamma\leq 11/12$ for $\alpha>1/4$ and $\gamma\leq 8/9$ for $\alpha\gtrsim \nu^{-1}$. More recently, Cui, Wang and Wang \cite{CWW2025} demonstrated that when the initial velocity and temperature satisfy $\left\| V_{\mathrm{in}}-(y,0,0)  \right\|_{H^2} \leqslant \delta \mathbf{Re}^{-1}$ and $\left\| \Theta_{\mathrm{in}}  \right\|_{H^2} \leqslant \delta \mathbf{Re}^{-2}$ for some $\delta>0$, the 3D Boussinesq equations exhibit nonlinear stability. In their work,  the steady state  of the temperature is $\Theta_s=0$, while the steady state of the velocity  remains the Couette flow. More specifically, for the relevant results of the 2D Boussinesq equation, one can refer to the  references \cite{Arbon2025, BBCD2023, DWZ2021, MSZ2022, NZ2024, ZhaiZ2023, ZZ2023}.
	\end{itemize}
			
  When the system incorporates the rotational effect, one finds that rotation can suppress the lift-up effect and have the dispersion effect. It is great surprised that when both the rotational and stratification effects are taken into account,  we find the dispersion relation is reshaped, and the dispersive waves is restored at all streamwise frequencies $\eta \in \mathbb{R}$, thereby  eliminating the singularity at $\eta=0$ present in the non-rotating, vertically stratified case.
	\begin{itemize}
		\item \textbf {Vertical stratification with rotation} ($\Theta_s = 1 + \alpha^2 z$,  and  $\theta\mathbf{e}_3$)\\
		The linearized equations for $(\widetilde{U}_0^1, \widetilde{U}_0^2, \widetilde{\Theta}_0)$ read:
		\begin{equation*}
		\partial_{t} 
		\begin{bmatrix}
		\widetilde{U}_{0}^{1} \\
		\widetilde{U}_{0}^{2} \\
		\widetilde{\Theta}_{0}
		\end{bmatrix}
		=
		\begin{bmatrix}
		\nu \Delta_{L,0}  & \beta-1 & 0 \\[4pt]
		-\beta \partial_{z} \Delta_{L,0}^{-1} \partial_{z} & \nu \Delta_{L,0} & \alpha \partial_{y} \Delta_{L,0}^{-1} \partial_{z} \\[4pt]
		0 & -\alpha \partial_{z}^{-1} \partial_{y} & \nu \Delta_{L,0}
		\end{bmatrix}
		\begin{bmatrix}
		\widetilde{U}_{0}^{1} \\
		\widetilde{U}_{0}^{2} \\
		\widetilde{\Theta}_{0}
		\end{bmatrix},
		\end{equation*}
		
	and the eigenvalues are
		\[
		\lambda_1 = -\nu (\eta^2+l^2), \qquad 
		\lambda_{\pm} = -\nu (\eta^2+l^2) \pm i \, \frac{h(\alpha, \beta, \eta,l)}{\sqrt{\eta^2+l^2}},
		\]
		where the dispersion relation $h(\alpha, \beta, \eta,l)$ is given by
\[
	\frac{h(\alpha, \beta, \eta, l)}{|\eta,l|} = \frac{\sqrt{\alpha^2 \eta^2 + B_{\beta} l^2}}{|\eta, l|}, \quad (\eta, l) \in \mathbb{R} \times (\mathbb{Z}\setminus\{0\}),
\]
where $\alpha$ and $\beta$  represent the strengths of  stratification and the speed of rotation, respectively. The key observation is that $h(\alpha, \beta, \eta,l) > 0$ for all $l \neq 0$ and $\eta\in \R$ under the condition $B_{\beta}>0$ due to the participation of the rotation effect, which can simplify the situation. 
	\end{itemize}
\qquad The dynamics of fluids are significantly influenced by rotation and temperature stratification. Rapid rotation excites inertial waves, while perturbations in a stably stratified state generate internal gravity waves. It is also important to note that all these waves exhibit dispersive characteristics. As highlighted by Babin et al. \cite{BMN1999}, one of the primary challenges in understanding geophysical flow dynamics lies in accounting for the oscillatory effects produced by such dispersive waves. This phenomenon has already been applied to investigate the global well-posedness of equation \eqref{1.1}, as referenced in works such as \cite{HZ2023, IMT2017, KMY2012, MS2022, SLF2025}.


	


\subsection{Summary of main result}
\qquad We first introduce a coordinate transform  described in detail in \cite{MR4126259, MR3612004, Zelati2023},  which  is a central tool for our analysis. Due to the characteristics of the transport structure in (\ref{1.6}), we define a linear change of variable 
\begin{equation}\label{2.5}
x=x_{1}- t y, \quad \quad y=y, \quad \quad z=z.
\end{equation}
In the new coordinates, the corresponding differential operators become
\begin{equation}\label{2.6}
\nabla =(\partial_{x_1},\partial_{y},\partial_{z}) \mapsto \nabla_{L}:=(\partial_{x},\partial_{y}^{L},\partial_{z}), \quad\quad \Delta\mapsto \Delta_{L}
\end{equation}
where 
\begin{equation}\label{2.77}
\partial_{y}^{L}:=\partial_{y}-t\partial_{x},\quad  \Delta_{L}:=\partial_{x}^{2}+(\partial_{y}^{L})+\partial_{z}^{2}.
\end{equation}
Furthermore, we define 
\[  
\nabla_{L,h}:=(\partial_{x},\partial_{y}-t\partial_{x}),\, \quad
\Delta_{L,h}:=(\p_x)^2+(\p_y-t\p_x)^2,
\]
which implies the Fourier symbol  denoted by (\ref{2.5}) as
\begin{equation}\label{HSX444}
	|\widehat{\nabla_{L}}|=|k,\eta-kt,l|,\quad		|\widehat{\nabla_{L,h}}|=|k,\eta-kt|,
\end{equation}
and 
\begin{equation}\label{2.8}
	p(t, k, \eta, l):=\widehat{-\Delta_{L}}=k^2+(\eta- kt)^2+l^2,\quad
	p_{h}(t, k, \eta):= \widehat{-\Delta_{L,h}}=k^2+(\eta- kt)^2.
\end{equation}
Note that $p$ and $p_{h}$ are time dependent, thus their time derivative is 
\begin{equation}\label{2.9}
	\partial_{t}	{p}= \partial_{t}{p_{h}} =-2 k (\eta- kt).
\end{equation}

Combining the definition of the new coordinates as in (\ref{2.5}) and denoting $U^i(t, x, y, z)=u^i(t, x_1, y, z)$ $(i=1, 2, 3)$, $  \Theta(t, x, y, z)=\theta(t, x_1, y, z)$. \eqref{1.6} then can be  rewritten as follows

\begin{align}\label{2.11}
\begin{cases}
\partial_{t} U-\nu \Delta_{L} U+\begin{pmatrix}
\left(1-\beta \right) U^2 \\
\beta U^1 \\
\alpha\Theta\end{pmatrix}+\left( \beta-2 \right) \nabla_{L} \Delta_{L}^{-1} \partial_{x} U^2-\beta \nabla_{L} \Delta_{L}^{-1} \partial_{y}^{L} U^1\\
\quad-\alpha\nabla_{L}\Delta_{L}^{-1}\partial_{z}\Theta=-U\cdot \nabla_{L} U+\nabla_{L} \Delta_{L}^{-1}(\partial_{i}^{L} U^j \partial_{j}^{L} U^i),\\
\partial_{t}\Theta-\nu\Delta_{L}\Theta-\alpha U^{3}=-U\cdot \nabla_{L}\Theta,\\
U(t=0)=U_{\text{in}},\quad \Theta(t=0)= \Theta_{\text{in}}.
\end{cases}
\end{align}

It is important to note that under the coordinate transformation \eqref{2.5}, the initial value conditions presented in \eqref{2.11} are fundamentally equivalent to the initial values stated in \eqref{1.6}. Specifically, we have $U_{\text{in}}=u_{\text{in}}$ and $\Theta_{\text{in}}=\theta_{\text{in}}$. In the following section, we will not make a special distinction regarding this matter. We will examine the behavior of solutions derived from the linearization of \eqref{1.2}.
Our first main result gives the linear stability of \eqref{1.1} near the Couette flow and vertical stratification as follows.

%
\begin{thm}[\textbf{Linear stability}]\label{1.1.}
	Assume   that $\nu>0$, $\alpha>0$, $\beta \in \mathbb{R}$ with $B_{\beta}>0$. Let $u_{\mathrm{in}}$ with a divergence-free, smooth vector field, then the solutions $u=u_{\neq}+\widetilde{u}_0+\overline{u}_0$ and $\theta=\theta_{\neq}+\widetilde{\theta}_0+\overline{\theta}_0$ to the linearized  equation for \eqref{1.6}  satisfy the following  linear stability estimates:
	
	$(1)$  The enhanced dissipation  and inviscid damping of the non-zero  modes $(U_{\neq}, \Theta_{\neq})$:
\begin{align}
 \left\| U^1_{\neq} (t) \right\|_{H^{s}}  &\lesssim\max \left\{\sqrt{\frac{\beta-1}{\beta}}, \sqrt{\frac{\beta}{\beta-1}} \right\} e^{-\frac{1}{16}  \nu^{\frac{1}{3}} t} \left(\left\|  u_{\mathrm{in}}\right\|_{H^{s+3}}+\left\|  \theta_{\mathrm{in}}\right\|_{H^{s+2}} \right) , \label{A.11x}  \\
  {\left\langle {t} \right\rangle} \left\|  U^2_{\neq} (t) \right\|_{H^{s}} &\lesssim \max \left\{\sqrt{\frac{\beta-1}{\beta}}, \sqrt{\frac{\beta}{\beta-1}} \right\} e^{-\frac{1}{16} \nu^{\frac{1}{3}} t} \left(\left\|  u_{\mathrm{in}}\right\|_{H^{s+5}}+\left\|  \theta_{\mathrm{in}}\right\|_{H^{s+4}} \right) , \label{A.22x}  \\
  \left\| (U^3_{\neq}, \Theta_{\neq}) (t) \right\|_{H^s} &\lesssim   e^{-\frac{1}{16} \nu^{\frac{1}{3}} t} \left(\left\|  u^3_{\mathrm{in}}\right\|_{H^{s+3}} + \sqrt{\frac{\beta}{\beta-1}} \left\|  w_{\mathrm{in}}^3\right\|_{H^{s+2}}+\left\|  \theta_{\mathrm{in}}\right\|_{H^{s+2}} \right).\label{A.33x}  
\end{align}

	$(2)$ Cancellation of lift-up effect in the Bradshow-Richardson stable regime: 
	
	For $k=0$,  if $l \neq 0$, then the simple-zero modes $(\widetilde{u}_{0}, \widetilde{\theta}_{0})$ on the Fourier space are written as
    		\begin{align*}
    	\widehat{\widetilde{({u}^{1})_0}}(t, \eta, l)=&\left[ \frac{\alpha^2 \eta^2}{h^2} e^{\lambda_1 t}+\frac{\beta (\beta-1) l^2}{2h^2} (e^{\lambda_2 t}+e^{\lambda_3 t})\right] \widehat{\widetilde{({u^{1}_{\mathrm{in}}})_0}} \\
    	&+\frac{i (\beta-1) |\eta, l|}{2h} (e^{\lambda_3 t}-e^{\lambda_2 t})\widehat{\widetilde{({u^{2}_{\mathrm{in}}})_0}} \\
    	&+\left[\frac{\alpha (\beta-1) \eta l}{h^2} e^{\lambda_1 t}-\frac{\alpha (\beta-1) \eta l}{2h^2}(e^{\lambda_2 t}+e^{\lambda_3 t})\right]\widehat{\widetilde{({\theta_{\mathrm{in}}})_0}}, \\
    	\widehat{\widetilde{({u}^{2})_0}}(t, \eta, l)=&\left[ \frac{i \beta l^2 }{2h|\eta, l|} (e^{\lambda_2 t}-e^{\lambda_3 t})\right] \widehat{\widetilde{({u^{1}_{\mathrm{in}}})_0}}+\frac{1}{2} (e^{\lambda_2 t}+e^{\lambda_3 t})\widehat{\widetilde{({u^{2}_{\mathrm{in}}})_0}}   \\
    	&+\frac{i\alpha \eta l}{2h |\eta, l|}(e^{\lambda_3 t}-e^{\lambda_2 t})\widehat{\widetilde{({\theta_{\mathrm{in}}})_0}},  \\
    	\widehat{\widetilde{({u}^{3})_0}}(t, \eta, l)=&-\frac{\eta}{l}\widehat{\widetilde{({u}^{2})_0}},\\
    	\widehat{\widetilde{(\theta)_0}}(t, \eta, l)=&\left[ \frac{\alpha \beta \eta l}{h^2} e^{\lambda_1 t}-\frac{\alpha \beta \eta l}{2h^2} (e^{\lambda_2 t}+e^{\lambda_3 t})\right] \widehat{\widetilde{({u^{1}_{\mathrm{in}}})_0}} \nonumber\\
    	&+\frac{i \alpha \eta |\eta, l|}{2h l} (e^{\lambda_2 t}-e^{\lambda_3 t})\widehat{\widetilde{({u^{2}_{\mathrm{in}}})_0}} \\
    	&+\left[\frac{\beta (\beta-1) l^2}{h^2} e^{\lambda_1 t}+\frac{\alpha^2 \eta^2}{2h^2}(e^{\lambda_2 t}+e^{\lambda_3 t})\right]\widehat{\widetilde{({\theta_{\mathrm{in}}})_0}}, 
    \end{align*}
	where $h=h(\alpha, \beta, \eta, l)=\sqrt{\alpha^2 \eta^2+ \beta (\beta-1) l^2}$, $\lambda_1=-\nu (\eta^2+l^2)$, $\lambda_2=-\nu (\eta^2+l^2)+i \frac{h}{|\eta, l|}$ and $\lambda_3=-\nu (\eta^2+l^2)-i \frac{h}{|\eta, l|}$. 
	
	If  $l=0$, then the double-zero modes $(\overline{u}_0, \overline{\theta}_0)$ satisfy
	\begin{align}
	\label{thm1.1.1.1}&\overline{u}_0^1(t, y)=e^{\nu \partial_{yy} t} \overline{(u_{\mathrm{in}}^1)_0}, \qquad  
	\overline{u}_0^2(t, y)=0, \\ \label{thm1.1.2}
	&\overline{u}_0^3(t, y)=e^{\nu \partial_{yy} t} \left[ \cos(\alpha t)\overline{(u_{\mathrm{in}}^3)_0}-\sin(\alpha t)\overline{(\theta_{\mathrm{in}})_0}\right],  \\
	&\overline{\theta}_0^3(t, y)=e^{\nu \partial_{yy} t} \left[ \sin(\alpha t)\overline{(u_{\mathrm{in}}^3)_0}+\cos(\alpha t)\overline{(\theta_{\mathrm{in}})_0}\right].  \label{thm1.1.3}
	\end{align}

	$(3)$ The dispersive estimates on the simple-zero modes $\widetilde{u}^2_0$ and $\widetilde{u}^3_0$   satisfy
	\begin{align}
	\left\| (\widetilde{u}_{0}^{2}, \widetilde{u}_{0}^{3})(t)   \right\|_{L^{\infty}(\mathbb{R} \times \mathbb{T})} \lesssim & \,  e^{-\nu t} \left(q t\right)^{-\frac{1}{3}}  \left\|  \widetilde{(u_{\mathrm{in}})_0} \right\|_{W^{3, 1}(\mathbb{R} \times \mathbb{T})} \label{thm1.1.4}
	\end{align}
	where $q=\frac{|B_{\beta}-\alpha^2|}{\alpha}$.
\end{thm}

\begin{re} 
	$(1)$  
The condition on \( B_{\beta} \) in Theorem \ref{1.1} only necessitates that \( B_{\beta} > 0 \). The proof of the enhanced dissipation for non-zero frequencies can be found in Appendix A. In Section \ref{sec3.1}, we impose the condition \( B_{\beta} > 1/4 \) as a hypocoercivity requirement derived from the cross-estimation presented in equation \eqref{00}. This restriction will facilitate obtaining improved thresholds in the proof of  Theorem \ref{1..1}.

	$(2)$ According to the expression of $q$ in Theorem \ref{1.1.}, the stratification and rotation effect play opposite roles. And  when $B_{\beta}=\alpha^2$, the linear dispersion effect will fail.
\end{re}
Now, we provide the nonlinear stability result of \eqref{1.1}--\eqref{invalue} in Sobolev space.
\begin{thm}[\textbf{Nonlinear stability}]\label{1..1}
	Let $\nu \in (0,1)$ and $\beta \in \mathbb{R}$ with $B_{\beta}>\frac{1}{4}$ and $B_{\beta}\neq \alpha^{2}$. For any $s \geqslant 7$ and $N\geqslant s+4$, there exists a constant $\delta=\delta(N,s,\alpha,\beta)>0$ such that if the initial data $(u_{\mathrm{in}}, \theta_{\mathrm{in}})$ satisfy
	\begin{equation}\label{1.23}
		\int_{\mathbb{T}^2} u_{\mathrm{in}}^3 dx_1 dz=\int_{\mathbb{T}^2} \theta_{\mathrm{in}} dx_1 dz=0,
	\end{equation}
	and
	\begin{align}
	\left\|u_{\mathrm{in}}\right\|_{H^{N+2}\cap W^{N+3,1}}+\left\|\theta_{\mathrm{in}}\right\|_{H^{N+1}\cap W^{N+3,1}}=\epsilon<\delta \nu^{\frac{14}{15}},
	\end{align}
	then there exists a unique global solution $(u, \theta)$ to \eqref{1.6}  and satisfying
	
	$(1)$ Global stability estimate in $H^{s}$:
	\begin{equation}
	\left\| (U, \Theta) \right\|_{L^{\infty} H^s}+\nu^{\frac{1}{6}} \left\| (U_{\neq}, \Theta_{\neq}) \right\|_{L^2 H^s}+ \nu^{\frac{1}{2}} \left\| \nabla_L (U, \Theta) \right\|_{L^2 H^s}\lesssim  \epsilon.
	\end{equation}
	
	$(2)$ Inviscid damping and enhanced dissipation:
	\begin{align}
	\left\| \left(U^1_{\neq}, U^3_{\neq}, \Theta_{\neq} \right)(t) \right\|_{L^{2}}+{\left\langle {t} \right\rangle}\left\|  U^2_{\neq} (t) \right\|_{L^{2}} \lesssim  \, b_{\alpha, \beta} \, e^{-\lambda\nu^{\frac{1}{3} }t} \epsilon.,
	\end{align}
	where $b_{\alpha, \beta}=\frac{1}{c_{\alpha}}\left(\frac{2\sqrt{B_{\beta}}+1}{2\sqrt{B_{\beta}}-1}\right)^{\frac{1}{2}}\max \left\{\sqrt{\frac{\beta-1}{\beta}}, \sqrt{\frac{\beta}{\beta-1}} \right\}$ and $\lambda=\frac{1}{16}\frac{2\sqrt{B_{\beta}}-1}{2\sqrt{B_{\beta}}+1} $.

	$(3)$ Nonlinear dispersion:
	\begin{align*}
	\|(u^2_{0},\tilde{u}^3_{0})(t)\|_{W^{2, \infty}} \lesssim& (qt)^{-\frac13}e^{-\nu t}\epsilon_0 + q^{-\frac13}\nu^{-\frac23}\epsilon^2,
	\end{align*} 
	where $\epsilon_0$ represents the initial data of the simple-zero modes and $q=\frac{|B_{\beta}-\alpha^2|}{\alpha}$.
\end{thm}

\begin{re}
	Without the additional assumption that $\int_{\mathbb{T}^2} u_{\mathrm{in}}^3 d{x_1}dz=\int_{\mathbb{T}^2} \theta_{\mathrm{in}} d{x_1}dz=0$ as in Theorem \ref{1..1}, we can obtain that the stability threshold $\gamma\leq 1$.
\end{re}


\subsection{Obstacles and strategies}
\subsubsection{Linear stability}
\qquad We analysis two main difficulties on the  linear system.
\paragraph{Difficulty 1:} The interactions of stratification, rotation, and Couette flow lead to complex linear coupling phenomena, which prevent us from  directly closing the energy estimate.

More precisely, it was observed in \eqref{2.11} that four equations for $(U^1,U^2,U^3,\Theta)$ are coupled with each other and cannot be decoupled directly as in the case of two equations. Moreover, although $U^2$ plays a significant role in the stability study of the Navier-Stokes system in \cite{MR3612004} and \cite{HSX2024}, we find that it is not suitable for the linear system \eqref{1.66}. To address this difficulty, we find several good unkowns to simplify the system, which are constructed in Strategy 1.


\paragraph{Strategy 1:  Constructing good unknowns.}
 Based on the characteristics of stratification and rotation effects, we find three good unknowns about $U^3,\ \Theta$ and the third component of vorticity $W^3$. Thanks to the incompressibility condition, we can use $U^3$ and $W^3$ to recover $U^1$ and $U^2$, thus those three variables are sufficient to describe the original system. 
 
 We start the process with the formal analysis in detail. Taking the Fourier transform of \eqref{1.66} in the moving frame, we have
\[  
	(\p_t+\nu p)\begin{pmatrix} \widehat{U^3_{\neq}}\\ \widehat{W^3_{\neq}} \\ \widehat{\Theta_{\neq}} \end{pmatrix}+ \begin{pmatrix} 2\frac{k(\eta-kt)l^2}{pp_h} & - \beta\frac{il}{p} & \alpha\frac{p_h}{p} \\ (1-\beta)il & 0 & 0 \\ -\alpha & 0 & 0 \end{pmatrix}\begin{pmatrix} \widehat{U^3_{\neq}}\\ \widehat{W^3_{\neq}} \\ \widehat{\Theta_{\neq}} \end{pmatrix} + \begin{pmatrix}2i\frac{k^2l}{pp_h} \widehat{W^3_{\neq}} \\ 0 \\ 0 \end{pmatrix} = 0.
\]
The third term has so many time decay that can be controlled by the Fourier multiplier $M$. To make full use of the paired linear coupling terms, we design new variables $Q_{\neq},\ K_{\neq}$ and $H_{\neq}$ to respectively represent $U^3_{\neq},\ W^3_{\neq}$ and $\Theta_{\neq}$. Let us assume that $\widehat{Q_{\neq}}=C_{Q_{\neq}}\widehat{U^3_{\neq}},\ \widehat{K_{\neq}}=C_{K_{\neq}}\widehat{W^3_{\neq}}$ and $\widehat{H_{\neq}}=C_{H_{\neq}}\widehat{\Theta_{\neq}}$. The ODE then becomes
\begin{align*}
	(\p_t+\nu p-\begin{pmatrix} \frac{\p_t C_{Q_{\neq}}}{C_{Q_{\neq}}}\\  \frac{\p_t C_{K_{\neq}}}{C_{K_{\neq}}} \\ \frac{\p_t C_{H_{\neq}}}{C_{H_{\neq}}} \end{pmatrix}+ \begin{pmatrix} 2\frac{k(\eta-kt)l^2}{pp_h} & - \beta\frac{il}{p}\frac{C_{Q_{\neq}}}{C_{K_{\neq}}} & \alpha\frac{p_h}{p}\frac{C_{Q_{\neq}}}{C_{H_{\neq}}} \\ (1-\beta)il\frac{C_{K_{\neq}}}{C_{Q_{\neq}}} & 0 & 0 \\ -\alpha\frac{C_{H_{\neq}}}{C_{Q_{\neq}}} & 0 & 0 \end{pmatrix})\begin{pmatrix} \widehat{Q_{\neq}} \\ \widehat{K_{\neq}} \\ \widehat{H_{\neq}} \end{pmatrix} + \begin{pmatrix}2i\frac{k^2l}{pp_h}\frac{C_{Q_{\neq}}}{C_{K_{\neq}}} \widehat{K_{\neq}} \\ 0 \\ 0 \end{pmatrix} = 0.
\end{align*}
To antisymmetrize the coefficient matrix, set $\frac{C_{H_{\neq}}^2}{C_{Q_{\neq}}^2}=\frac{p_h}{p},\ \frac{C_{Q_{\neq}}^2}{C_{K_{\neq}}^2}=-p\frac{\beta-1}{\beta}$, and choose $C_{H_{\neq}}=\sqrt{\frac{p_h}{p}}C_{Q_{\neq}},\ C_{K_{\neq}}=i\sqrt{\frac{\beta}{\beta-1}}\frac{1}{\sqrt{p}}C_{Q_{\neq}}$. We also note that $\left|\frac{l^2}{p}\frac{\p_t p_h}{p_h}\right|\left/\left|\frac{l}{\sqrt{p}}\right|\right.  \leq 1, \left|\frac{\p_t p}{p}\right|\left/\left|\sqrt{\frac{p_h}{p}}\right|\right. \leq 1$, which means that by using the linear coupling terms that appear in pairs, specific stretching terms can be eliminated. Finally, we choose $C_{Q_{\neq}} = -p^{\frac34}p_h^{-\frac12}$ and construct good unknowns
\begin{equation}\label{hhll} 
	Q_{\neq}\triangleq-|\na_L|^{\frac32}|\na_{L,h}|^{-1}U^3_{\neq},\ K_{\neq}\triangleq-i\sqrt{\frac{\beta}{\beta-1}}|\na_L|^{\frac12}|\na_{L,h}|^{-1}W^3_{\neq},
	\ H_{\neq}\triangleq-|\na_L|^{\frac12}\Theta_{\neq}.
\end{equation}
This approach not only counteracts the mutually coupled linear terms, thereby clearly extracting the enhanced dissipation and inviscid damping mechanisms, but also effectively manages the growth of the linear stretching term $\partial_{xy} \Delta^{-1}$ that arises from pressure. Proper handling of this growth is typically achieved through a meticulously designed Fourier multiplier $m$, as demonstrated in \eqref{4.....4}, which accurately encodes its dependence on frequency and viscosity. The design of this multiplier $m$ aims to ensure it possesses a more optimized bound \eqref{4.11-1}, with three well-chosen variables \eqref{hhll} facilitating this objective. Furthermore, we introduce a cross-term estimate \eqref{sss} that ensures $B_{\beta}$ satisfies the condition $B_{\beta} > 1/4$. For further details, please refer to section \ref{sec3.1}. Of course, if our goal is solely to achieve enhanced dissipation at non-zero frequencies, it suffices for us to have $B_{\beta} > 0$, which is thoroughly proven in Appendix \ref{secA}.
	
\paragraph{Difficulty 2:} The single-zero modes of temperature $\tilde{\theta}_0$ and the first component $\tilde{u}_0^1$ within the velocity field does not exhibit linear dispersion estimates.

Recall from section \ref{sec3.2} that the ODE with respect to time variables $t$ can be obtained and solved by taking the Fourier transform of the single zero frequency system \eqref{1..7} with respect to $(x,y,z)$. It can be seen that $u_0^2$ and $\tilde{u}_0^3$ have linear dispersion effects. Unfortunately, $\tilde{u}_0^1$ and $\tilde{\theta}_0$ only exhibit pure thermal effects and have no dispersion estimates. To address this diﬀiculty, we find a combined quantify from $\tilde{u}_0^1$ and $\tilde{\theta}_0$, and obtain its dispersive estimates, which are constructed in the following { Strategy 2}.
\paragraph{Strategy 2:  Constructing the dispersion combined quantities} Based on the structure of single-zero modes in \eqref{3.2.4}-\eqref{3.2.7}, we choose two combined quantities concerning $\tilde{u}^1_0$ and $\tilde{\theta}_0$. 
\begin{equation*}
	\begin{aligned}
		V_0^2=&-\beta\mathcal{R}^{-1}\p_{zz}\Delta_{L,0}^{-1}\tilde{U}_0^1 + \alpha\mathcal{R}^{-1}\p_{yz}\Delta_{L,0}^{-1}\tilde{\Theta}_0=:G_1\tilde{U}_0^1 + G_2\tilde{\Theta}_0,\\
		\Lambda_0=& \alpha\sqrt{\frac{\beta}{\beta-1}}\p_{yz}\Delta_{L,0}^{-1} \tilde{U}_0^1 + \sqrt{B_{\beta}}\p_{zz}\Delta_{L,0}^{-1}\tilde{\Theta}_0=:G_3\tilde{U}_0^1 + G_4\tilde{\Theta}_0.
	\end{aligned}
\end{equation*}  
By utilizing $V_0^2$ and $\Lambda_0$, we can successfully extract out of the dispersion structure from the pure thermal effect structure in $\tilde{U}_0^1$ and $\tilde{\Theta}_0$. Furthermore, these two combined quantities allow for the recovery of both $\tilde{U}_0^1$ and $\tilde{\Theta}_0$. Additionally, we find a corresponding combined quantity, denoted as $\tilde{V}_0^3$, which aligns with $\tilde{U}_0^3$, thus circumventing the need for direct estimation via the incompressible condition given by $\tilde{U}^3_0 = -\p_z^{-1}\p_y U_0^2$.  We then employ the method of oscillatory integrals to analyze the dispersive semigroup, leading to linear dispersive estimates, further details can be found in sections \ref{sec3.2} and \ref{sub4.1}.

\subsubsection{Nonlinear stability}
\qquad We deal with two main diﬀiculties on the analysis of nonlinear system.
\paragraph{Difficulty 1: Structural issue for zero modes}  
We revisit the analysis results of the linear part. Among the zero modes, only $U_0^2$ and $\tilde{U}_0^3$ exhibit linear dispersion estimates, while $\tilde {U}_0^1$ and $\tilde {\Theta}_0$ are characterized solely by thermal decay estimates. In examining nonlinear systems of zero modes, a significant challenge lies in managing the nonlinear interactions between simple-zero modes. Here, we continue to adopt the approach of linear stability and construct the combined quantity \( V_0^2 \) to derive additional nonlinear dispersion estimates that satisfy the following equation,
\begin{equation}
\begin{aligned}\label{com11}
	\partial_t V_0^2 - \nu \Delta_{L,0} V_0^2 + \mathcal{R} U_0^2 &= - G_1\widetilde{(U_0 \cdot\nabla U_0^1)} -G_2\widetilde{(U_0 \cdot\nabla \Theta_0)}\\&\quad+ {\text{non-zero modes contributions}}, 
\end{aligned}
\end{equation}
where $\mathcal{R}$ is the dispersion operator as   in \eqref{ss11}. We note that the operators $(G_1, G_2)$ and the convection operator $U_0 \cdot \nabla$ are non-commutative. Therefore, we can not rewrite $G_1{(U_0 \cdot \nabla \tilde{U}_0^1)} + G_2{(U_0 \cdot \nabla \tilde{\Theta}_0)}$ as a product of two terms that both have the dispersion property, that is
\[
G_1{(U_0 \cdot\nabla \tilde{U}_0^1)} +G_2{(U_0\cdot\nabla\tilde{\Theta}_0)} \neq U_0 \cdot\nabla (G_1\tilde{U}_0^1 + G_2\tilde{\Theta}_0)=\underbrace{U_0\cdot\nabla V_0^2}_{\text{ dispersion-dispersion}}.
\]
It is evident that the right side of the aforementioned equation comprises two types of product terms that are challenging to handle. The first type includes {\it dispersion-thermal type terms}, such as $\tilde{U}_0^{2, 3} \cdot \nabla_{y, z} \{U^1_{0}, \Theta_{0}\}$. The second type consists of {\it thermal-thermal type terms}, exemplified by $\overline{U}_{0}^3\p_zU^1_0$. Furthermore, similar to the right-hand side terms of the $V_0^2$ equation, the nonlinear terms of the $\Lambda_0$ equation also lack the dispersive structure. Consequently, it proves difficult to demonstrate the propagation of zero modes within high Sobolev regularity norms and to derive nonlinear dispersion estimates.
\paragraph{Strategy 1:  Quasilinear method for zero modes.}
To capture more nonlinear dispersive effects as accurately as possible, we adopt a quasi-linear approach decomposing the zero-frequency solution into  the main part and the perturbation part:
\begin{equation*}
	\begin{cases}
		U_0 = U_{0,1} + U_{0,2}, \\ \Theta_0 = \Theta_{0,1} + \Theta_{0,2}, \\
		V_0^2 = V_{0,1}^2 + V_{0,2}^2, \\ \tilde{V}_0^3 = \tilde{V}_{0,1}^3 + \tilde{V}_{0,2}^3, \\
		\Lambda_0 = \Lambda_{0,1} + \Lambda_{0,2}. 
	\end{cases}
\end{equation*}

The first subsystem (the main component) encompasses all initial value information at zero frequency. The nonlinear terms account for contributions at non-zero frequencies, while the zero-frequency component is meticulously designed to filter out {\it dispersion-dispersion type terms} in the equations governing dispersive variables and {\it dispersion-thermal type terms} in the equations pertaining to non-dispersive variables, as illustrated in \eqref{1stQuasi1}:
 \begin{align*}
 	\mathcal{N}_{\neq+0}(U_{0,1}^2)=&2\partial_y\Delta_{L,0}^{-1}\big(\partial_yU_{0,1}^2\partial_yU^2_{0,1}+\partial_y\tilde{U}^3_{0,1}\partial_zU^2_{0,1}\big)-\tilde{U}_{0,1}\cdot\nabla U_{0,1}^2 \\&+ {\text{non-zero modes contributions}}.
 \end{align*}

In relation to the second subsystem (perturbation part), it solely encompasses zero initial values along with the remaining nonlinear terms, as illustrated in \eqref{2ndQuasi1}--\eqref{2ndQuasi3}. Following decomposition, the first subsystem is employed to propagate high Sobolev regularity norms by means of Lemma \ref{lowDisper}, leading to corresponding nonlinear dispersion estimates and an improved threshold. Subsequently, the energy estimates derived from the first subsystem are introduced as external forces into the second subsystem; this external force dictates the threshold size of the perturbation system. In addressing zero modes within the entire nonlinear system, we note that non-zero frequencies exhibit enhanced dissipation and inviscid damping effects, making their interactions relatively manageable. However, when considering interactions among zero modes, we aim to investigate additional dispersion mechanisms.
 
Specifically, by applying the Duhamel principle, we derive low-regularity $L^\infty$ dispersion estimates for the dispersion variables $U_{0,1}^2$, $\tilde{U}_{0,1}^3$, $V_{0,1}^2$, and $\tilde{V}_{0,1}^3$ pertaining to the first subsystem. Utilizing Lemma \eqref{lowDisper}, these dispersion variables achieve higher regularity within Sobolev spaces. Through the incorporation of {\it dispersion-dispersion type terms}, we obtain nonlinear dispersion estimates for these variables, thereby enhancing our energy estimate results. Under condition \eqref{1.23}, we leverage the {\it dispersion-dispersion type terms} associated with $\overline{U_{0,1}^3}$ and $\overline{\Theta_{0,1}}$ to refine our energy estimates further. Additionally, employing precise evaluations of the "dispersion-thermal type terms" related to $\overline{U_{0,1}^1}$ and $\Lambda_{0,1}$ allows us to improve our energy estimates even more significantly—ultimately closing off the energy estimates for this subsystem.
 
The nonlinear terms of the second subsystem can be classified into four types of interactions: external force-external force, external force-disturbance, disturbance-external force, and disturbance-disturbance. The lower Sobolev regularity propagating on the system enables us to handle these structures by Lemma \ref{lowDisper} as follows:\\
(1) Contribution of two external force terms, since $s+4\leq N$, for $G\in\{U,\Theta\}$ we have
\[\langle U_{0,1}\cdot\na G_{0,1},G_{0,2}\rangle_{H^{s,s+\frac12}}\lesssim (\|\langle\p_z\rangle^{\frac12}\tilde{U}_{0,1}^{2,3}\|_{W^{s,\infty}}+\|\overline{U_{0,1}^3}\|_{H^s})\|\na G_{0,1}\|_{H^{s}}^{\frac12}F_{0,1}^{\frac12}(t)\nu^{-\frac14}F_{0,2}^{\frac12}(t)\nu^{-\frac12}.\]
(2) Contribution of only one external force term, for $G\in\{U,\Theta\}$,
\begin{align*}
	\langle U_{0,1}\cdot\na G_{0,2},\tilde{G}_{0,2}\rangle_{H^{s,s+\frac12}}\lesssim& (\|\langle\p_z\rangle^{\frac12}\tilde{U}_{0,1}^{2,3}\|_{W^{s,\infty}}+\|\overline{U_{0,1}^3}\|_{H^s})F_{0,2}^{\frac34}(t)\nu^{-\frac34}\|\tilde{G}_{0,2}\|_{H^{s}}^{\frac12},\\
	\langle U_{0,2}\cdot\na G_{0,1},\tilde{G}_{0,2}\rangle_{H^{s,s+\frac12}}\lesssim& \|\langle\p_z\rangle^{\frac12}\left(U_{0,2}\cdot\na G_{0,1}\right)\|_{H^{s-1}}F_{0,2}^{\frac12}(t)\nu^{-\frac12}\\
	\lesssim& (\|\langle\p_z\rangle^{\frac12}\tilde{U}_{0,2}^{2,3}\|_{W^{s-3,\infty}}^{\frac12}+\|\overline{U_{0,2}^3}\|_{H^{s-3}}^{\frac12})\|\na U_{0,2}\|_{H^{s,s+\frac12}}^{\frac12}\|\na G_{0,1}\|_{H^{s+\frac32}}F_{0,2}^{\frac12}(t)\nu^{-\frac12}\\
	\lesssim& (\|\langle\p_z\rangle^{\frac12}\tilde{U}_{0,2}^{2,3}\|_{W^{s-3,\infty}}^{\frac12}+\|\overline{U_{0,2}^3}\|_{H^{s-3}}^{\frac12})F_{0,2}^{\frac34}(t)\nu^{-\frac34}F_{0,1}^{\frac14}(t)\nu^{-\frac14}E_{0,1}^{\frac14}(t).
\end{align*}
	(3) No external force term case, for $G\in\{U,\Theta\}$
	$$\langle U_{0,2}\cdot\na G_{0,2},G_{0,2}\rangle_{H^{s,s+\frac12}}\lesssim E_{0,2}^{\frac12}(t)F_{0,2}(t)\nu^{-1}.$$
\paragraph{Difficulty 2: Regularity issue for non-zero modes}
 Recalling the three new variables $Q_{\neq}^*$, $K_{\neq}^*$ and $H_{\neq}^*$ constructed in Appendix \ref{secA}, they satisfy the linear equation \eqref{Non0L1}. In terms of nonlinear stability, these three variables can eliminate some linear coupling terms, but the equations they satisfy still contain linear stretching terms, which cause the variables to have a growth effect. To effectively control this growth, we introduce new variables $Q_{\neq}$, $K_{\neq}$ and $H_{\neq}$ as in \eqref{yyy} as well as the cross-estimation of $Q_{\neq}$ and $K_{\neq}$, which can eliminate the growth caused by $-(l^2 \p_t p_h)/(2 p p_h)\{\widehat{Q_{\neq}}, \widehat{K_{\neq}}\}$. The remaining stretching term $ (\p_t p)/(4p)\{\widehat{Q_{\neq}}, \widehat{K_{\neq}}, \widehat{H_{\neq}}\}$ is handled by using the multiplier $m$ as shown in \eqref{4.....4}. Unfortunately, incorporating rotation and stratification results in coupling among the equation's variables, leading to regularity issues within the vertical direction for the nonlinear system described by equation \eqref{Non0NL}. Specifically, we will deal with the analogues of the following nonlinear term:
   \begin{align*}
 	&\left|\int_0^t\int_{\mathbb{T}\times\R\times\mathbb{T}}\na^N \mathcal{A}|\na_L|^{\frac32}|\na_{L,h}|^{-1}(U_{\neq}^2\p_y^LU_{\neq}^3)\na^N (\mathcal{A}Q_{\neq}){\rm d}x{\rm d}y{\rm d}z{\rm d}\tau \right|\\
 	& \quad\lesssim \left|\int_0^t\int_{\mathbb{T}\times\R\times \mathbb{T}}\mathcal{A}|\na_L|^{\frac32}|\na_{L,h}|^{-1}(U_{\neq}^2\na^N \p_y^LU_{\neq}^3)\na^N (\mathcal{A}Q_{\neq}){\rm d}x{\rm d}y{\rm d}z{\rm d}\tau \right| \\&\qquad+ \left|\int_0^t\int_{ \mathbb{T}\times\R\times \mathbb{T}} \mathcal{A}|\na_L|^{\frac32}|\na_{L,h}|^{-1}(\na^N U_{\neq}^2\p_y^LU_{\neq}^3)\na^N (\mathcal{A}Q_{\neq}){\rm d}x{\rm d}y{\rm d}z{\rm d}\tau \right|.
 \end{align*}
 The treatment of the first term does not bring about essential difficulties. Since $U_{\neq}^2$ has a relatively low regularity, the inviscid damping effect can be fully exploited to handle it, thereby obtaining an improved threshold. However, for the treatment of the second term,  $U_{\neq}^2$ has high regularity, and the potential good factor $|\na_{L,h}|^{-1}$ and the inviscid damping effect of $U_{\neq}^2$ cannot be utilized. Specifically,
  \begin{align*}
 	&\left|\int_0^t\int_{\mathbb{T}\times\R\times\mathbb{T}} \mathcal{A}|\na_L|^{\frac32}|\na_{L,h}|^{-1}(\na^N U_{\neq}^2\p_y^LU_{\neq}^3)\na^N (\mathcal{A}Q_{\neq}){\rm d}x{\rm d}y{\rm d}z{\rm d}\tau \right|
 	\\
 	&\quad \lesssim 	\|\mathcal{A}\na_LQ_{\neq}\|_{L^2_tH^N}(\|\mathcal{A}|\na_L|^{\frac12}(Q_{\neq},K_{\neq})\|_{L^4_tH^N}^2+\|\mathcal{A}(Q_{\neq,K_{\neq}})\|_{L^\infty_tH^N} 	\|\mathcal{A}\na_LQ_{\neq}\|_{L^2_t H^N})\\
 	&\quad \lesssim E_{\neq}^{\frac12}(t)F_{\neq}(t)\nu^{-1}\\
 	&\quad \lesssim \delta^3\nu^{3c-1}.
 \end{align*}
 When the frequency in the vertical direction is much greater than that in the horizontal direction, that is, $|l| \gg \sqrt{k^2 + (\eta - kt)^2}$, a rough estimation yields a threshold of 1. To address the  difficulty about regularity issue, we employ the quasi-linearization method.
 
\paragraph{Strategy 2: Quasilinear method for non-zero modes.} We decompose the non-zero frequency solution into a main part and a perturbation part:
\begin{equation*} 
Q_{\neq} = Q_{\neq,1} + Q_{\neq,2}, \quad K_{\neq} = K_{\neq,1} + K_{\neq,2}, \quad H_{\neq} = H_{\neq,1} + H_{\neq,2}.
\end{equation*}

The first subsystem (the main part): It carries all the non-zero frequency initial value information, carefully screens the nonlinear terms to avoid regularity issue, as shown in $\eqref{Non0NL1}$. By using enhanced dissipation and inviscid damping effects, the corresponding energy estimates can be closed under the high regularity norms in Sobolev space.

The second subsystem (perturbation) is characterized by an initial value of zero and the presence of remaining nonlinear terms. By utilizing the enhanced energy estimates derived from the first subsystem as external force inputs to the second subsystem, this external force effectively determines the threshold size of the perturbation system. The nonlinear terms within the second subsystem can be categorized into four types of interactions: force-force, force-perturbation, perturbation-force, and perturbation-perturbation. This framework allows us to address these nonlinear terms individually under lower regularity norms in Sobolev space. With our designed quasi-linearization method, we are able to resolve these previously challenging terms as follows:
\begin{align*}
&\left|\int_0^t\int_{\mathbb{T}\times\R\times\mathbb{T}} \mathcal{B}|\na_L|^{\frac32}|\na_{L,h}|^{-1}(\na^s U_{\neq,1}^2\p_y^LU_{\neq,1}^3)\na^s (\mathcal{B}Q_{\neq,2}){\rm d}x{\rm d}y{\rm d}z{\rm d}\tau \right|
\\
&\quad \lesssim 	\|\mathcal{B}Q_{\neq,2}\|_{L^\infty_tH^s}\nu^{-\frac16}\nu^{-\frac16}\|\mathcal{A}(Q_{\neq,1},K_{\neq,1})\|_{L^2_tH^{s+4}}\|\mathcal{A}Q_{\neq,1}\|_{L^2_tH^{s+\frac72}}\\
&\quad \lesssim \nu^{-\frac23}E_{\neq,2}^{\frac12}(t)\int_0^t F_{\neq,1}(\tau){\rm d}\tau\\&\quad\lesssim \delta^3\nu^{b+2c-\frac23}.
\end{align*}
Here, we employ the difference in high and low regularity $N \geq 4 + s$. It suffices to have $b + 2c - {2}/{3} \geq 2b$ to close the bootstrap argument. Clearly, $c = {14}/{15}$ and $b = 1$ meet the requirement.

The rest of the paper is organized as follows. In section \ref{Preliminaries} we give some preliminaries for later use. In section \ref{section3}, we discuss the linear stability effects of solutions and give a rigorous mathematical proof of Theorem \ref{1.1.}. Section \ref{section4}  is devoted to the brief idea of proof  of Theorem \ref{1..1}, including the bootstrap hypotheses and the weighted energy estimates of non-zero modes and zero modes. The  proof details of nonlinear stability are given in sections \ref{sub4.2}--\ref{sub4.6}.

\section{Preliminaries}\label{Preliminaries} \label{sec2}

\subsection{Notation}
\qquad Given two quantities $A$ and $B$, we denote $A \lesssim B$, if there exists a constant $C$ such that $A\leqslant CB$ where  $C>0$ depends only on $N$ and $\alpha$, but not on $\delta$, $\nu$  and $\beta$. We similarly denote $A \ll B$ if $A \leqslant cB$ for a small constant $c \in(0, 1)$ to emphasize the small size of the implicit constant.  $\sqrt{1+t^2}$ is represented by ${\left\langle {t} \right\rangle}$ and  $\sqrt{k^2+\eta^2+l^2}$ is  denoted by $|k, \eta, l|$. For two functions $f$ and $g$ and a norm $\|\cdot\|_{X}$, we write 
$$\|(f,g)\|_{X}=\sqrt{\|f\|_{X}^{2}+\|g\|_{X}^{2}}.$$
Unless specified otherwise, in the rest of this section $f$ and $g$ denote functions from $\mathbb{T} \times \mathbb{R} \times \mathbb{T}$ to $\mathbb{R}^{n}$ for some $n\in \mathbb{N}$.

\subsection{Functional spaces}
\qquad The Schwartz space $\mathscr{S}(\mathbb{R}^{n})$ with $n$ being the spatial dimension is the set of smooth functions $f$ 
on $\mathbb{R}^{n}$, we have
$$\|f\|_{\mathscr{S}}\triangleq \sup_{x\in \mathbb{R}^{n} }\left(1+|x|\right)^{\alpha} |\partial^{\beta}f(x)|<\infty,\quad \forall \, \alpha,\beta\in \mathbb{Z}_{+}^{n}.$$

The Sobolev space $H^{s} (s\geqslant0) $ is given  by  the norm 
\begin{equation}\label{HSX-A}
	\left\| f \right\|_{H^{s}}=\left\| {\left\langle {D} \right\rangle}^{s} f \right\|_{L^2}=\left\| {\left\langle {k, \eta, l} \right\rangle}^{s} \widehat{f} \right\|_{L^2}.
\end{equation}
To better estimate the zero modulus below, here we define the space
\begin{equation}\label{HSX-B}
	\left\| f_{0} \right\|_{H_{y,z}^{s,s+\frac{1}{2}}}\triangleq
	\left\| \left\langle {\partial_{z}} \right\rangle^{\frac{1}{2}} f_{0} \right\|_{H_{y,z}^{s}}=\left\| {\left\langle {\eta, l} \right\rangle}^{s}\langle l\rangle^{\frac{1}{2}} \widehat{f_{0}} \right\|_{L^2}.
\end{equation}

Recall that, for $s>\frac{3}{2}$, $H^{s}$ is an algebra. Hence, for any $f,g\in H^{s}$, one has
\begin{align}\label{2.3}
	\left\|f g\right\|_{H^{s}} \lesssim \left\|f \right\|_{H^{s}}\left\|g\right\|_{H^{s}}.
\end{align}
We write the associated inner product as
$$ \langle f,g \rangle_{H^{s}}=\int \langle\nabla\rangle^{s}f\cdot \langle \nabla \rangle^{s} g dV. $$
For function $f(t,x,y,z)$ of space and time defined on the time interval $(a,b)$, we define the Banach space $L_{t}^{p}(a,b;H^{s})$ for $1\leqslant p\leqslant \infty$ by the norm
$$\|f\|_{L_{t}^{p}(a,b; H^{s})}=\|\|f\|_{H^{s}}\|_{L_{t}^{p}(a,b)}.$$
For simplicity of notation, we usually simply write $\|f\|_{L_{t}^{p}H^{s}}$ since the time-interval of integrating in this work will be the same basically everywhere.

In the field of harmonic analysis, the Van der Corput lemma is an estimate for oscillatory integrals. In this paper, we use this lemma to obtain the dispersive estimates on the zero frequency of  velocity. 
\begin{lem}[\textbf{Van der Corput lemma} \cite{SM1993}]\label{lem2.1}
	Suppose  $\phi(\xi)$ is real-valued and smooth in $(a, b)$, and that $\left|\phi^{(k)}(\xi)\right| \geqslant 1$  for all  $\xi \in(a, b)$. For any  $\lambda \in \mathbb{R}$, there is a positive constant  $c_{k}$, which does not depend on  $\phi$ and  $\lambda$  such that
	\begin{equation}\label{vdcl}
		\left|\int_{a}^{b} e^{i \lambda \phi(\xi)} d\xi \right| \leqslant c_{k} \lambda^{-\frac{1}{k}}
	\end{equation}
	holds when:
	
	$(1)$  $k \geqslant 2$, or
	
	$(2)$  $k=1$ and $\phi(\xi)$ is monotone.
\end{lem}

\section{The linear stability}\label{section3}

\qquad In this section, we elaborate on the linear stability results and corresponding proof of Bouessinesq equation with rotation near the Couette flow. Now we neglect the effect of nonlinear terms on the right-hand side of \eqref{1.6}, that is, we only need to consider the stability of the linearized equation
\begin{equation}\label{1.66}
	\begin{cases}
		\partial_{t} u+y \partial_{x_1} u-\nu \Delta u+\begin{pmatrix}
			(1-\beta) u^2 \\
			\beta u^1 \\
			\alpha \theta
	    	\end{pmatrix}+ (\beta-2)\nabla \Delta^{-1}\partial_{x_1}u^{2}\\ \hskip5cm -\beta\nabla\Delta^{-1}\partial_{y}u^{1}-\alpha\nabla\Delta^{-1}\partial_{z}\theta  =0, \\
		\partial_{t} \theta+y \partial_{x_1} \theta-\nu \Delta \theta-\alpha u^3=0, \\
		u(0, x_1, y, z)=u_{\mathrm{in}}, \quad \theta(0, x_1, y, z)=\theta_{\mathrm{in}}.
	\end{cases}
\end{equation}
 
By the moving frame \eqref{2.5}, we recall the definition of $ U(t,x,y,z)$ and $ \Theta(t,x,y,z) $, and denote $W^3\triangleq\p_xU^2-\p_y^LU^1$. Here, we introduce the new unknowns
\begin{equation} \label{yyy} 
Q_{\neq}:=-|\na_L|^{\frac32}|\na_{L,h}|^{-1}U^3_{\neq},\quad K_{\neq}:=-i\sqrt{\frac{\beta}{\beta-1}}|\na_L|^{\frac12}|\na_{L,h}|^{-1}W^3_{\neq},
\quad H_{\neq}:=-|\na_L|^{\frac12}\Theta_{\neq}.
\end{equation}   
One can represent  the linearized system of \eqref{1.66} as
\begin{align}
	\label{Non0NL-1}
	\begin{cases}
		\p_t Q_{\neq}  -\nu\Delta_L Q_{\neq}  - \frac12 \p_x\p_y^L|\na_L|^{-2} Q_{\neq}  - \p_x\p_y^L(|\na_L|^{-2}-|\na_{L,h}|^{-2}) Q_{\neq}  + i\sqrt{B_{\beta}}\p_z|\na_L|^{-1} K_{\neq}  \\
		\qquad + 2i\sqrt{\frac{\beta-1}{\beta}}\p_z\p_{xx}|\na_L|^{-1}|\na_{L,h}|^{-2}K_{\neq}  + \alpha|\na_{L,h}||\na_L|^{-1}H_{\neq}  = 0, \\
		\p_t K_{\neq} - \nu\Delta_L K_{\neq}  + \frac12 \p_x\p_y^L|\na_L|^{-2} K_{\neq}  + \p_x\p_y^L(|\na_L|^{-2}-|\na_{L,h}|^{-2}) K_{\neq}\\ \qquad - i\sqrt{B_{\beta}}\p_z|\na_L|^{-1}Q_{\neq}  = 0,\\
		\p_t H_{\neq}  -\nu\Delta_L H_{\neq}  - \frac12 \p_x\p_y^L|\na_L|^{-2} H_{\neq}  - \alpha|\na_{L,h}||\na_L|^{-1}Q_{\neq} =  0.
	\end{cases}
\end{align}
Thus, by taking the Fourier transform of the spatial variable, when $k\neq 0$,  \eqref{Non0NL-1} can be written
\begin{align}\label{Non0L}
	\begin{cases}
		\p_t \widehat{Q_{\neq}} + \nu p \widehat{Q_{\neq}} - \frac14 \frac{\p_t p}{p}\widehat{Q_{\neq}} - \frac12 \frac{l^2}{p}\frac{\p_t p_h}{p_h} \widehat{Q_{\neq}} \\
		\qquad\ \ - \sqrt{B_{\beta}}\frac{l}{p^{1/2}}\widehat{K_{\neq}} + \alpha \frac{p_h^{1/2}}{p^{1/2}}\widehat{H_{\neq}} + 2\sqrt{\frac{\beta-1}{\beta}}\frac{k^2}{p_h}\frac{l}{p^{1/2}}\widehat{K_{\neq}}  = 0, \\
		\p_t\widehat{K_{\neq}} + \nu p\widehat{K_{\neq}} + \frac{1}{4}\frac{\p_t p}{p}\widehat{K_{\neq}} +\frac12 \frac{l^2}{p}\frac{\p_t p_h}{p_h} \widehat{K_{\neq}}+ \sqrt{B_{\beta}}\frac{l}{p^{1/2}}\widehat{Q_{\neq}} = 0,\\
		\p_t\widehat{H_{\neq}} + \nu p\widehat{H_{\neq}} - \frac{1}{4}\frac{\p_t p}{p} \widehat{H_{\neq}} - \alpha \frac{p_h^{1/2}}{p^{1/2}}\widehat{Q_{\neq}} = 0.
	\end{cases}
\end{align}


\subsection{Dynamics of nonzero modes}\label{sec3.1}

\qquad As observed in \cite{MR3612004}, to capture the associated effects of inviscid damping and enhanced dissipation, it is convenient to use a suitably constructed Fourier multiplier $m$ with symbol $m=m(t,k,\eta,l)$, defined by 
\begin{align}\label{4.....4}	
	-\frac{\dot{m}}{m}=&\left\{\begin{array}{ll}
		\frac{1}{2}\frac{| k \left( \eta- kt \right)|}{k^2+\left(\eta- kt\right)^2+l^2}, & \text{ if } \ |t-\frac{\eta}{k}|\leq 1000\nu^{-\frac13}, \\
		0,  & \text{ if } \ |t-\frac{\eta}{k}|> 1000\nu^{-\frac13},
	\end{array}\right. 
\end{align}
with the initial data
\begin{align*}
m(t=0,k,\eta,l)=&\begin{cases}
1, & \text{ if } \  \frac{\eta}{k}\leq 0,\\
\left(\frac{k^2+\eta^2+l^2}{k^2+l^2}\right)^{\frac14}, & \text{ if } \ 0<\frac{\eta}{k}<1000\nu^{-\frac13},\\
\left(\frac{k^2+(1000\nu^{-\frac13}k)^2+l^2}{k^2+l^2}\right)^{\frac14},& \text{ if } \ 1000\nu^{-\frac13}\leq \frac{\eta}{k}.
\end{cases} 
\end{align*}
The central role of $m$ is to deal with the slowly decaying term $\partial_{x}\partial_{y}^{L}|\nabla_{L}|^{-2}\{Q,K,H\}  $ in \eqref{Non0NL-1}, balancing the growth that solutions experience in the transition between the inviscid damping regime $t<\frac{\eta}{k}$, and the dissipative regime $t\gtrsim \frac{\eta}{k}+\nu^{-\frac{1}{3}}$.

Define the additional multipliers $M_{i}(t, k, \eta, l), i=1,2,...,7$,  by $M_{i}(t=0, k, \eta, l)=1$ and
\begin{itemize}
	\item if  $k=0$, $M_{i}(t, k=0, \eta, l)=1$ for all $t$.
	\item if  $k \neq 0$,
	\begin{align}\label{4.5}
		\frac{\dot{M_{1}}}{M_{1}}&= \frac{- \nu^{\frac{1}{3}} }{\left[\nu^{\frac{1}{3}} \big( t-\frac{\eta}{k} \big)\right]^{2}+1}.
	\end{align}
	\item if  $k \neq 0$,
\begin{align}
  -\frac{\dot{M}_2}{M_2}&=\frac{1}{4\sqrt{B_{\beta}}}\left|\frac{\p_t p}{p}\frac{l}{\sqrt{p}}\frac{\p_t p_h}{p_h}\right|, \label{11}\\  
	-\frac{\dot{M}_3}{M_3}&=\left|\p_t\left(-\frac{1}{2\sqrt{B_{\beta}}}\frac{l}{\sqrt{p}}\frac{\p_t p_h}{p_h}\right)\right|, \label{22} \\
 -\frac{\dot{M}_4}{M_4}&=\left|\frac{2}{\beta}\frac{l^2}{p}\frac{k^2}{p_h}\frac{\p_t p_h}{p_h}\right|, \label{33} \\ -\frac{\dot{M}_5}{M_5}&=2\sqrt{\frac{\beta-1}{\beta}}\frac{k^2}{p_h}\left|\frac{l}{\sqrt{p}}\right|,\label{44}\\  
  -\frac{\dot{M}_6}{M_6}&=\frac{\alpha}{2\sqrt{B_{\beta}}}\left|\frac{l}{p}\frac{\p_tp_h}{\sqrt{p_h}}\right|, \label{55}  \\
   -\frac{\dot{M}_7}{M_7}&=\frac{1}{(k^2+|\eta-kt|^2)^{\frac{1+\kappa}{2}}},\label{66}
     \end{align}
\end{itemize}
for any $\kappa>0$. We then define
\begin{equation}\label{666}  
M=\prod_{j=1}^{7} M_j. 
\end{equation}

For the expressions and properties of the multiplier $m$, there is a similar situation as in \cite{MR3612004}, so we use it as known conclusion and ignore the proof process.

\begin{lem}
	$(1)$ The multiplier $m(t, k, \eta, l)$ can be given by the following exact formula:
\begin{itemize}
	\item $m(t, k=0, \eta, l)=1.$
	\item if  $k \neq 0$, $ \frac{\eta}{k}<-1000\nu^{-\frac13}: \ m(t,k,\eta,l)\equiv 1.$
	\item if $k\neq 0$, $-1000\nu^{-\frac13}\leq\frac{\eta}{k}\leq 0$:
	\begin{equation}  
		 m=\begin{cases}
		\left(\frac{k^2+\eta^2+l^2}{p}\right)^{\frac14},& \text{ if } \ 0\leq t\leq\frac{\eta}{k}+1000\nu^{-\frac13},\\
		\left(\frac{k^2+\eta^2+l^2}{k^2+(1000\nu^{-\frac13}k)^2+l^2}\right)^{\frac14},& \text{ if } \ \frac{\eta}{k}+1000\nu^{-\frac13}\leq t.
	     \end{cases} \nonumber
     \end{equation}
	\item if $k\neq 0, 0<\frac{\eta}{k}<1000\nu^{-\frac13}$:    \begin{equation}m=\begin{cases}
		\left(\frac{p}{k^2+l^2}\right)^{\frac14},& \text{ if } \ 0\leq t\leq\frac{\eta}{k},\\
		\left(\frac{k^2+l^2}{p}\right)^{\frac14},& \text{ if } \ \frac{\eta}{k}\leq t\leq\frac{\eta}{k} + 1000\nu^{-\frac13},\\
		\left(\frac{k^2+l^2}{k^2+(1000\nu^{-\frac13}k)^2+l^2}\right)^{\frac14},& \text{ if } \ \frac{\eta}{k} + 1000\nu^{-\frac13}\leq t.
	\end{cases}\nonumber
	\end{equation}
 
	\item if $k\neq 0, 1000\nu^{-\frac13}\leq\frac{\eta}{k}$: 
	\begin{equation} 
		m=\begin{cases}
		\left(\frac{k^2+(1000\nu^{-\frac13}k)^2+l^2}{k^2+l^2}\right)^{\frac14},& \text{ if } \ 0\leq t\leq \frac{\eta}{k}-1000\nu^{-\frac13},\\
		\left(\frac{p}{k^2+l^2}\right)^{\frac14},& \text{ if } \ \frac{\eta}{k}-1000\nu^{-\frac13}\leq t\leq\frac{\eta}{k},\\
		\left(\frac{k^2+l^2}{p}\right)^{\frac14},& \text{ if } \ \frac{\eta}{k}\leq t \leq \frac{\eta}{k}+1000\nu^{-\frac13},\\
		\left(\frac{k^2+l^2}{k^2+(1000\nu^{-\frac13}k)^2+l^2}\right)^{\frac14},& \text{ if } \ \frac{\eta}{k}+1000\nu^{-\frac13}\leq t.
	\end{cases}\nonumber
	\end{equation}
\end{itemize}	
$(2)$ In particular,  $m(t, k, \eta, l)$  are bounded above and below, but its bound depends on $\nu$
\begin{equation}\label{4.11-1}
	\nu^{\frac{1}{6}}\lesssim m\left(t, k, \eta, l \right) \leqslant \nu^{-\frac{1}{6}}.
\end{equation}
$(3)$  $m(t, k, \eta, l)$  and the frequency have the following relationship
\begin{equation}\label{4.11-2}
	\frac{|k,l|^{\frac{1}{2}}}{p^{\frac{1}{4}}} \lesssim m\left(t, k, \eta, l \right) \lesssim\frac{p^{\frac{1}{4}}}{|k,l|^{\frac{1}{2}}}.
\end{equation}
$(4)$  Moreover,  for all $t, k\neq0, \eta, \eta', l, l'$,  the following product estimates hold that
\begin{equation}\label{4.11-3}
 m(t, k, \eta, l)\lesssim \langle \eta-\eta', l-l' \rangle^{\frac{1}{2}}m(t,k, \eta', l').
\end{equation}

\end{lem}
The following lemma embodies some properties of the ghost multipliers $M_{j}(t, k, \eta, l)$.
\begin{lem}\label{lem4.2}
	$(1)$ The multiplier $M_{j}$ have positive upper and lower bounds:
	\begin{align}
	&0<c<M_{j}(t, k, \eta, l) \leqslant 1, \quad  j \in \{1,2,...5,7\}, \label{4.13}\\
	&0<c_{\alpha}<M_{6}(t, k, \eta, l) \leqslant 1, \label{new111}
	\end{align}
	for a universal constant $c$ that does not depend on $\nu$, $\alpha$, $\beta$ and the  frequency, and $c_{\alpha}=\exp(-\frac{\alpha \pi}{\sqrt{B_{\beta}}})$. 
	
	$(2)$ For $k \neq 0$, we have
	\begin{equation}\label{4.15}
		1 \leq 2\left(\nu^{-\frac{1}{6}} \sqrt{-\frac{\dot{M_{1}}}{ M_{1}}}\left(k, \eta, l\right)+\nu^{\frac{1}{3}}\big|k, \eta- kt, l \big|\right).
	\end{equation}
\end{lem}


\pf The proof process of \eqref{4.13} and \eqref{4.15} is similar to that of \cite{MR3612004, Liss2020}   and we omit it. Here we only prove \eqref{new111}.
According to \eqref{55}, we have $ \dot{M_6}/{M_6} \leqslant 0$ for any $k$, $l \in \mathbb{Z}$, $\eta \in \mathbb{R}$. Hence, for any $t \geqslant 0$, we obtain $M_6(t, k, \eta, l) \leqslant 1$.

If $k=0$, then $M_6(t, k=0, \eta, l)=1$. If $k \neq 0$, by \eqref{55}, we have
\begin{align*}
-\ln M_6=&\frac{\alpha}{\sqrt{B_{\beta}}}\int_{0}^{t} \frac{|k||l||\eta-kt|}{(k^2+(\eta-kt)^2+l^2)\sqrt{k^2+(\eta-kt)^2}} d \tau \\
\leqslant& \frac{\alpha|l|}{\sqrt{B_{\beta}}}\int_{0}^{t} \frac{|k|}{k^2+(\eta-kt)^2+l^2} d \tau \\
\leqslant& \frac{\alpha|l|}{\sqrt{B_{\beta}}}\int_{-\infty}^{+\infty} \frac{1}{k^2+l^2+s^2} d s \\
\leqslant& \frac{\alpha \pi}{\sqrt{B_{\beta}}},
\end{align*}
where we have used the fact $\int_{-\infty}^{+\infty} \frac{1}{k^2+l^2+s^2} d s=\frac{\pi}{|k, l|}$. Hence,
\begin{equation*}
0< \exp(-\frac{\alpha \pi}{\sqrt{B_{\beta}}}) \leqslant M_6 \leqslant 1. 
\end{equation*} 
Hence, we finish the proof of Lemma \ref{lem4.2}\hfill $\square$

\begin{re}
	In  Theorem \ref{1..1}, we assume that $B_{\beta} > \frac{1}{4}$. This means that the lower bounds $c$ and $c_{\alpha}$ of the multiplier $M_{i}$ $(i=1,2,...,7)$ in Lemma \ref{lem4.2} are independent of $\beta$.
\end{re}
Using Lemma \ref{lem4.2}, we directly give the following inference.
\begin{cor}\label{cor4.1} (\cite{MR3612004, Liss2020})
 	For any $g$ and $s \geqslant 0$, the following inequality holds
	\begin{equation}\label{4.17}
		\left\|  g_{\neq}  \right\|_{L_t^2 H^{s}} \leq 2 \nu^{-\frac{1}{6}} \left( \left\|  \sqrt{-\frac{\dot{M_{1}}}{M_{1}}} g_{\neq} \right\|_{L_{t}^2 H^{s}}+ \nu^{\frac{1}{2}} \left\|  \nabla_{L} g_{\neq} \right\|_{L_t^2 H^{s}}\right),
	\end{equation}
	which will be used frequently below.
\end{cor}

The proof of Theorem \ref{1.1.}  is directly obtained from the following Lemma \ref{lem3.5} and Corollary \ref{dispercor} and the Lemma \ref{lemA.4} in Appendix A. Here, We first give the linear enhanced dissipation and inviscid damping on $U_{\neq}$ when $B_\beta>{1}/{4}$.  
\begin{lem}\label{lem3.4}
	Assume that $\nu > 0$ and $\beta \in \mathbb{R}$ with $B_\beta>\frac{1}{4}$, then there holds
	\begin{align}
		&\left\| U^1_{\neq} (t) \right\|_{L^{2}} \lesssim  \, b_{\alpha, \beta} \, e^{-\lambda  \nu^{\frac{1}{3}} t} \left(\left\|  u_{\mathrm{in}}\right\|_{H^{2}}+\left\|  \theta_{\mathrm{in}}\right\|_{H^{1}} \right) , \label{3.1x}  \\
		&{\left\langle {t} \right\rangle} \left\|  U^2_{\neq} (t) \right\|_{L^{2}} \lesssim  \, b_{\alpha, \beta} \, e^{-\lambda \nu^{\frac{1}{3}} t} \left(\left\|  u_{\mathrm{in}}\right\|_{H^{\frac{9}{2}}}+\left\|  \theta_{\mathrm{in}}\right\|_{H^{\frac{7}{2}}} \right) , \label{3.2x}  \\
		&\left\| (U^3_{\neq}, \Theta_{\neq}) (t) \right\|_{L^2} \lesssim  \, d_{\alpha, \beta} \, e^{-\lambda \nu^{\frac{1}{3}} t} \left(\left\|  u^3_{\mathrm{in}}\right\|_{H^{2}} + \sqrt{\frac{\beta}{\beta-1}} \left\|  \p_xu^2_{\mathrm{in}}-\p_yu^1_{\mathrm{in}}\right\|_{H^{1}}+\left\|  \theta_{\mathrm{in}}\right\|_{H^{1}} \right),\label{3.3x}  
	\end{align}
	where  $\lambda=\frac{1}{16}\frac{2\sqrt{B_{\beta}}-1}{2\sqrt{B_{\beta}}+1}$, $b_{\alpha, \beta}=\frac{1}{c_{\alpha}}\left(\frac{2\sqrt{B_{\beta}}+1}{2\sqrt{B_{\beta}}-1}\right)^{\frac{1}{2}}\max \left\{\sqrt{\frac{\beta-1}{\beta}}, \sqrt{\frac{\beta}{\beta-1}} \right\}$, $c_{\alpha}=\exp(-\frac{\alpha \pi}{\sqrt{B_{\beta}}})$ and $d_{\alpha, \beta}=\frac{1}{c_{\alpha}}\left(\frac{2\sqrt{B_{\beta}}+1}{2\sqrt{B_{\beta}}-1}\right)^{\frac{1}{2}}$.
\end{lem}

\pf To control the evolution of the nonzero modes, we introduce the Fourier multiplies
\begin{equation}\label{888}
	\mathcal{A}:=mMe^{\lambda \nu^{\frac{1}{3}}t},
\end{equation} 
which incorporates the multiplier $m$ already used in the linear analysis and additional ghost weights ${M}$ as in \eqref{666} as well as the time weight $e^{\lambda\nu^{\frac{1}{3}}t}$, which $\lambda$ can be confirmed later. Recall the system  \eqref{Non0L}, energy estimates yield that 
\begin{equation}\label{999}
\begin{aligned}
    \frac12&\frac{d}{dt}|\mathcal{A}(\widehat{Q_{\neq}},\widehat{K_{\neq}},\widehat{H_{\neq}})|^2  + \nu p |\mathcal{A}(\widehat{Q_{\neq}},\widehat{K_{\neq}},\widehat{H_{\neq}})|^2 +\left[(-\frac{\dot{M}}{M})+(-\frac{\dot{m}}{m})-\lambda\nu^{\frac13}\right]|\mathcal{A}(\widehat{Q_{\neq}},\widehat{K_{\neq}},\widehat{H_{\neq}})|^2\\
    &-\left(\frac12\frac{\p_t p_h}{p_h}-\frac14\frac{\p_t p}{p}\right)\left(|\widehat{\mathcal{A}Q_{\neq}}|^2-|\widehat{\mathcal{A}K_{\neq}}|^2\right)-\frac14\frac{\p_t p}{p}|\widehat{\mathcal{A}H_{\neq}}|^2+2\sqrt{\frac{\beta-1}{\beta}}\frac{k^2}{p_h}\frac{l}{\sqrt{p}}\widehat{\mathcal{A}K_{\neq}}\overline{\widehat{\mathcal{A}Q_{\neq}}}=0.
\end{aligned}
\end{equation}
Consider the cross terms, we introduce operator $G$ whose fourier symbol  satisfies $$\widehat{G}:=-\frac{1}{2\sqrt{B_{\beta}}}\frac{l}{\sqrt{p}}\frac{\p_t p_h}{p_h},$$
thus one has 
\begin{align}\label{sss}
	&\p_t\left(\widehat{G}\widehat{\mathcal{A}Q_{\neq}}\overline{\widehat{\mathcal{A}K_{\neq}}}\right) \nonumber\\
	&\quad=(\p_t \widehat{G})\widehat{\mathcal{A}Q_{\neq}}\overline{\widehat{\mathcal{A}K_{\neq}}} + \left[(\frac{\dot{M}}{M})+(\frac{\dot{m}}{m})+\lambda\nu^{\frac13}\right]2\widehat{G}\widehat{\mathcal{A}Q_{\neq}}\overline{\widehat{\mathcal{A}K_{\neq}}} \nonumber\\
	&\qquad -\widehat{G}\left[\nu p\widehat{\mathcal{A}Q_{\neq}}-\left(\frac12\frac{\p_t p_h}{p_h}-\frac14\frac{\p_t p}{p}\right)\widehat{\mathcal{A}Q_{\neq}}-\sqrt{B_{\beta}}\frac{l}{\sqrt{p}}\widehat{\mathcal{A}K_{\neq}}+\alpha\sqrt{\frac{p_h}{p}}\widehat{\mathcal{A}H_{\neq}}+2\sqrt{\frac{\beta-1}{\beta}}\frac{k^2}{p_h}\frac{l}{\sqrt{p}}\widehat{\mathcal{A}K_{\neq}}\right]\overline{\widehat{\mathcal{A}K_{\neq}}} \nonumber\\
	&\qquad-\widehat{G}\widehat{\mathcal{A}Q_{\neq}}\left[\nu p\overline{\widehat{\mathcal{A}K_{\neq}}}+\left(\frac12\frac{\p_t p_h}{p_h}-\frac14\frac{\p_t p}{p}\right)\overline{\widehat{\mathcal{A}K_{\neq}}}+\sqrt{B_{\beta}}\frac{l}{\sqrt{p}}\overline{\widehat{\mathcal{A}Q_{\neq}}}\right].
\end{align}
This combined with \eqref{999} implies that
\begin{align*}
    \frac12&\frac{d}{dt}\left(|\mathcal{A}(\widehat{Q_{\neq}},\widehat{K_{\neq}},\widehat{H_{\neq}})|^2+2\widehat{G}\widehat{\mathcal{A}Q_{\neq}}\overline{\widehat{\mathcal{A}K_{\neq}}}\right) +\left[ \nu p+\left(-\frac{\dot{M}}{M}\right)\right]\left(|\mathcal{A}(\widehat{Q_{\neq}},\widehat{K_{\neq}},\widehat{H_{\neq}})|^2+2\widehat{G}\widehat{\mathcal{A}Q_{\neq}}\overline{\widehat{\mathcal{A}K_{\neq}}}\right)\\
    &=\lambda\nu^{\frac13}\left(|\mathcal{A}(\widehat{Q_{\neq}},\widehat{K_{\neq}},\widehat{H_{\neq}})|^2+2\widehat{G}\widehat{\mathcal{A}Q_{\neq}}\overline{\widehat{\mathcal{A}K_{\neq}}}\right)\\
    &\quad-\left[(-\frac{\dot{m}}{m})-\frac14\frac{\p_t p}{p}\right]|\widehat{\mathcal{A}Q_{\neq}}|^2-\left[(-\frac{\dot{m}}{m})+\frac14\frac{\p_t p}{p}\right]|\widehat{\mathcal{A}K_{\neq}}|^2-\left[(-\frac{\dot{m}}{m})-\frac14\frac{\p_t p}{p}\right]\|\widehat{\mathcal{A}H_{\neq}}|^2-(-\frac{\dot{m}}{m})2\widehat{G}\widehat{\mathcal{A}Q_{\neq}}\overline{\widehat{\mathcal{A}K_{\neq}}}\\
    &\quad+(\p_t \widehat{G})\widehat{\mathcal{A}Q_{\neq}}\overline{\widehat{\mathcal{A}K_{\neq}}}+\frac{1}{\beta}\frac{l^2}{p}\frac{k^2}{p_h}\frac{\p_t p_h}{p_h}|\widehat{\mathcal{A}K_{\neq}}|^2-2\sqrt{\frac{\beta-1}{\beta}}\frac{k^2}{p_h}\frac{l}{\sqrt{p}}\widehat{\mathcal{A}K_{\neq}}\overline{\widehat{\mathcal{A}Q_{\neq}}}+\frac{\alpha}{2\sqrt{B_{\beta}}}\frac{l}{p}\frac{\p_tp_h}{\sqrt{p_h}}\widehat{\mathcal{A}H_{\neq}}\overline{\widehat{\mathcal{A}K_{\neq}}}.
\end{align*}

Donate that $e_{\neq}(t):=|\mathcal{A}(\widehat{Q_{\neq}},\widehat{K_{\neq}},\widehat{H_{\neq}})|^2+2\widehat{G}\widehat{\mathcal{A}Q_{\neq}}\overline{\widehat{\mathcal{A}K_{\neq}}}$, then using the   expression of $\hat{G}$  yields that 
\begin{equation}
\begin{aligned}\label{00}
    \left(1-\frac{1}{2\sqrt{B_{\beta}}}\right)|\mathcal{A}(\widehat{Q_{\neq}},\widehat{K_{\neq}},\widehat{H_{\neq}})|^2\leq e_{\neq}(t)\leq \left(1+\frac{1}{2\sqrt{B_{\beta}}}\right)|\mathcal{A}(\widehat{Q_{\neq}},\widehat{K_{\neq}},\widehat{H_{\neq}})|^2,
\end{aligned}
\end{equation}
Notice that the definition of $m$ as in \eqref{4.....4}, we have
$ 
 -\frac{\dot{m}}{m}+\frac{\nu p}{2}-\frac14 \big|\frac{\p_t p}{p}\big|\gtrsim 0$.
A direct calculation yields
\begin{equation}
\begin{aligned}\label{1010}
    &\frac{d}{dt}e_{\neq}(t) + \left[\nu p+(-\frac{\dot{M}}{M})\right]\left(1-\frac{1}{2\sqrt{B_{\beta}}}\right)|\mathcal{A}(\widehat{Q_{\neq}},\widehat{K_{\neq}},\widehat{H_{\neq}})|^2\\
    &\quad\leq 2\lambda\nu^{\frac13}\left(1+\frac{1}{2\sqrt{B_{\beta}}}\right)|\mathcal{A}(\widehat{Q_{\neq}},\widehat{K_{\neq}},\widehat{H_{\neq}})|^2-\left[(-\frac{\dot{M_2}}{M_2})+(-2\frac{\dot{m}}{m})\widehat{G}\right]2\widehat{\mathcal{A}Q_{\neq}}\overline{\widehat{\mathcal{A}K_{\neq}}}\\
    &\qquad-\left[(-\frac{\dot{M_3}}{M_3})-\p_t \widehat{G}\right]2\widehat{\mathcal{A}Q_{\neq}}\overline{\widehat{\mathcal{A}K_{\neq}}}-\left[(-\frac{\dot{M_4}}{M_4})-\frac{2}{\beta}\frac{l^2}{p}\frac{k^2}{p_h}\frac{\p_t p_h}{p_h}\right]|\widehat{\mathcal{A}K_{\neq}}|^2\\
    &\qquad-\left[(-\frac{\dot{M_5}}{M_5})+2\sqrt{\frac{\beta-1}{\beta}}\frac{k^2}{p_h}\frac{l}{\sqrt{p}}\right]2\widehat{\mathcal{A}K_{\neq}}\overline{\widehat{\mathcal{A}Q_{\neq}}}-\left[(-\frac{\dot{M_6}}{M_6})-\frac{\alpha}{2\sqrt{B_{\beta}}}\frac{l}{p}\frac{\p_tp_h}{\sqrt{p_h}}\right]2\widehat{\mathcal{A}H_{\neq}}\overline{\widehat{\mathcal{A}K_{\neq}}},
\end{aligned}
\end{equation}
where we have used \eqref{00} and  $M=\prod_{j=1}^{7} M_j$ and $M_j(t,k,\eta, l),\ 1\leq j\leq 7$  are defined by \eqref{4.5}-\eqref{66}.
In the following, we introduce two new notation
\begin{align*}
    E_{\neq}(t):=&\|\mathcal{A}(Q_{\neq},K_{\neq},H_{\neq})\|_{H^r}^2 +2\langle G\mathcal{A}Q_{\neq},\mathcal{A}K_{\neq}\rangle_{H^r},\\
    F_{\neq}(t):=&\nu\|\mathcal{A}\na_L (Q_{\neq},K_{\neq},H_{\neq})\|_{H^r}^2 +2\nu\langle G\mathcal{A}\na_L Q_{\neq},\mathcal{A}\na_L K_{\neq}\rangle_{H^r}\\
    &+\left\|\mathcal{A}\sqrt{-\frac{\dot{M}}{M}} \left(Q_{\neq},K_{\neq},H_{\neq}\right)\right\|_{H^r}^2 +2\left\langle G\mathcal{A}\sqrt{-\frac{\dot{M}}{M}} Q_{\neq},\mathcal{A}\sqrt{-\frac{\dot{M}}{M}} K_{\neq}\right\rangle_{H^r},
\end{align*}
for any $r \geq 0$. Using \eqref{00}, $  E_{\neq}(t) $ and $  F_{\neq}(t) $ satisfy
\begin{align*}
    & \left(1-\frac{1}{2\sqrt{B_{\beta}}}\right)\|\mathcal{A}(Q_{\neq},K_{\neq},H_{\neq})\|_{H^r}^2\leq E_{\neq}(t)\leq  \left(1+\frac{1}{2\sqrt{B_{\beta}}}\right)\|\mathcal{A}(Q_{\neq},K_{\neq},H_{\neq})\|_{H^r}^2,\\
    & \left(1-\frac{1}{2\sqrt{B_{\beta}}}\right)\left(\nu\|\mathcal{A}\na_L(Q_{\neq},K_{\neq},H_{\neq})\|_{H^r}^2 + \left\|\mathcal{A}\sqrt{-\frac{\dot{M}}{M}}(Q_{\neq},K_{\neq},H_{\neq})\right\|_{H^r}^2\right)\\
	&\quad \leq F_{\neq}(t)\leq \left(1+\frac{1}{2\sqrt{B_{\beta}}}\right)\left(\nu\|\mathcal{A}\na_L(Q_{\neq},K_{\neq},H_{\neq})\|_{H^r}^2 + \left\|\mathcal{A}\sqrt{-\frac{\dot{M}}{M}}(Q_{\neq},K_{\neq},H_{\neq})\right\|_{H^r}^2\right),
\end{align*}
thus, recalling \eqref{1010} and going back to the physical space, we obtain that
\begin{align*}
    &\frac{\rm d}{{\rm d}t}E_{\neq}(t) + \left(1-\frac{1}{2\sqrt{B_{\beta}}}\right)\nu\|\mathcal{A}\na_L(Q_{\neq},K_{\neq},H_{\neq})\|_{H^r}^2 + \left(1-\frac{1}{2\sqrt{B_{\beta}}}\right)\|\mathcal{A}\sqrt{-\frac{\dot{M}}{M}}(Q_{\neq},K_{\neq},H_{\neq})\|_{H^r}^2\\
    &\quad\leq2\lambda\nu^{\frac13}\left(1+\frac{1}{2\sqrt{B_{\beta}}}\right)\|\mathcal{A}(Q_{\neq},K_{\neq},H_{\neq})\|_{H^r}^2.
\end{align*}
Taking $\lambda=\frac{1}{16}\frac{2\sqrt{B_{\beta}}-1}{2\sqrt{B_{\beta}}+1}$   and integrating the above inequality in time yields that
\begin{align*}
    &\|\mathcal{A}(Q_{\neq},K_{\neq},H_{\neq})\|_{H^r}^2 + \nu\|\mathcal{A}\na_L(Q_{\neq},K_{\neq},H_{\neq})\|_{L^2_tH^r}^2 + \left\|\mathcal{A}\sqrt{-\frac{\dot{M}}{M}}(Q_{\neq},K_{\neq},H_{\neq})\right\|_{L^2_tH^r}^2\\
    & \quad\leq \frac{2\sqrt{B_{\beta}}+1}{2\sqrt{B_{\beta}}-1}\|\mathcal{A}(Q_{\neq},K_{\neq},H_{\neq})(t=0)\|_{H^r}^2 +\frac{\nu^{\frac13}}{8}\|\mathcal{A}(Q_{\neq},K_{\neq},H_{\neq})\|_{L^2_tH^r}^2.
\end{align*}
Due to $\mathcal{A}=e^{\lambda \nu^{\frac{1}{3}} t }mM$, \eqref{4.17} and 
 \eqref{4.11-2},  we have
\begin{equation}\label{3x.0}
\begin{aligned}
	&\|m M(Q_{\neq},K_{\neq},H_{\neq})\|_{H^r}^2 + \frac{ \nu}{2}\|mM\na_L(Q_{\neq},K_{\neq},H_{\neq})\|_{L^2 H^r}^2
	\\&\qquad + \frac{1}{2} \left\|\sqrt{-\dot{M} M}m(Q_{\neq},K_{\neq},H_{\neq})\right\|_{L^2 H^r}^2  \\
& 	\quad\leq \frac{2\sqrt{B_{\beta}}+1}{2\sqrt{B_{\beta}}-1} e^{-2\lambda \nu^{\frac{1}{3}} t} \| (Q_{\mathrm{in}},K_{\mathrm{in}},H_{\mathrm{in}})\|_{H^{r+{\frac{1}{2}}}}^2.
\end{aligned}
\end{equation}

Next, recalling the definition of $Q$, $K$ and  $H$ and combining with \eqref{4.11-2}, we have the following fact
\begin{align}\label{3x.1}
	|\widehat{U^3_{\neq}}|= \frac{p_h^{\frac{1}{2}}}{p^{\frac{3}{4}}}|\widehat{Q_{\neq}}| \leq \frac{1}{c_{\alpha}}\frac{p_h^{\frac12}}{p^{\frac12}(k^2+l^2)^{\frac14}}|mM\widehat{Q_{\neq}}|\leq \frac{1}{c_{\alpha}(k^2+l^2)^{\frac14}}|mM\widehat{Q_{\neq}}|,
\end{align}
and
\begin{align}\label{3x.2}
	|\widehat{\Theta_{\neq}}|=\frac{1}{p^{\frac{1}{4}}}|H_{\neq}|\leq \frac{1}{c_{\alpha}(k^2+l^2)^{\frac14}}|mM\widehat{H_{\neq}}|.
\end{align}
Hence, together  \eqref{3x.0} with \eqref{3x.1}-\eqref{3x.2} gives 
\begin{equation}
\begin{aligned}\label{3x.3}
	\left\|{U^3_{\neq}}\right\|_{L^2}^2\leq& \frac{1}{c_{\alpha}^{2}}\left\| mM{Q_{\neq}}   \right\|_{L^2}^2 \\
	\leq& \frac{1}{c_{\alpha}^{2}}\frac{2\sqrt{B_{\beta}}+1}{2\sqrt{B_{\beta}}-1} e^{-2\lambda \nu^{\frac{1}{3}} t} \| (Q_{\mathrm{in}},K_{\mathrm{in}},H_{\mathrm{in}})\|_{H^{\frac{1}{2}}}^2 \\
	\leq& \frac{1}{c_{\alpha}^{2}}\frac{2\sqrt{B_{\beta}}+1}{2\sqrt{B_{\beta}}-1} e^{-2\lambda \nu^\frac{1}{3} t} \left(\| u^3_{\mathrm{in}}\|_{H^2}^2+{\frac{\beta}{\beta-1}} \left\|  \p_xu^2_{\mathrm{in}}-\p_yu^1_{\mathrm{in}}\right\|_{H^1}^2+\| \theta_{\mathrm{in}}\|_{H^1}^2 \right),
\end{aligned}
\end{equation}
similarly, we have
\begin{equation}
\begin{aligned}\label{3x.4}
	\left\|{\Theta_{\neq}}\right\|_{L^2}^2	\leq \frac{1}{c_{\alpha}^{2}}\frac{2\sqrt{B_{\beta}}+1}{2\sqrt{B_{\beta}}-1} e^{-2\lambda \nu^{\frac{1}{3}} t} \left(\| u^3_{\mathrm{in}}\|_{H^2}^2+{\frac{\beta}{\beta-1}} \left\|  \p_xu^2_{\mathrm{in}}-\p_yu^1_{\mathrm{in}}\right\|_{H^1}^2+\| \theta_{\mathrm{in}}\|_{H^1}^2 \right).
\end{aligned}
\end{equation}

Using the incompressible condition $\partial_{x} U^1+\partial_{y}^L U^2=-\partial_{z} U^3$ and the definition of $W^3$, one gets
\begin{equation}\label{3.20}
	\begin{cases}
		\left(\partial_{x x}+\partial_{y y}^L\right) U^{1}=-\partial_{y}^{L} W^{3}-\partial_{x z} U^{3},  \\
		\left(\partial_{x x}+\partial_{y y}^L\right) U^{2}=\partial_{x} W^{3}-\partial_{y z}^{L} U^{3},
	\end{cases}
\end{equation}
By \eqref{3.20} and $\eqref{4.11-2}$, we get $\widehat{U_{\neq}^1}$ and $\widehat{U_{\neq}^2}$ which satisfy
\begin{equation}\label{ll6}
\begin{aligned}
	|\widehat{U_{\neq}^1}|\leq &\frac{1}{p_h}\left( \sqrt{\frac{\beta-1}{\beta}}  \frac{|\eta-kt| p_h^{\frac{1}{2}}}{p^{\frac{1}{4}}} |\widehat{K_{\neq}}|+ \frac{|k||l| p_h^{\frac{1}{2}}}{p^{\frac{3}{4}}} |\widehat{Q_{\neq}}|\right)\\
	\leq & \frac{1}{c_{\alpha}}\frac{1}{\left(k^2+l^2\right)^{\frac{1}{4}}} \left(\sqrt{\frac{\beta-1}{\beta}}|mM\widehat{K_{\neq}}| + |mM\widehat{Q_{\neq}}|\right),
\end{aligned}
\end{equation}
and 
\begin{equation}\label{ll7}
	\begin{aligned}
	|\widehat{U_{\neq}^2}|\leq &\frac{1}{p_h}\left( \sqrt{\frac{\beta-1}{\beta}}  \frac{|k| p_h^{\frac{1}{2}}}{p^{\frac{1}{4}}} |\widehat{K_{\neq}}|+ \frac{|\eta-kt||l| p_h^{\frac{1}{2}}}{p^{\frac{3}{4}}} |\widehat{Q_{\neq}}|\right)\\
	\leq &  \frac{1}{c_{\alpha}}\left(\sqrt{\frac{\beta-1}{\beta}} \frac{|k|^{\frac{1}{2}}}{p_h^{\frac12}}|mM\widehat{K_{\neq}}| +\frac{|l|^{\frac{1}{2}}}{p^{\frac12}} |mM\widehat{Q_{\neq}}|\right).
\end{aligned}
\end{equation}
Thus,  this together \eqref{3x.0}  leads to $L^2$-norm of $U_{\neq}^1(t)$ and $U_{\neq}^2(t)$
\begin{equation}\label{4x.5}
\begin{aligned}
	\left\|{U^1_{\neq}}\right\|_{L^2}^2\leq& \frac{1}{c_{\alpha}^{2}}\left({\frac{\beta-1}{\beta}}\left\|mM\widehat{K_{\neq}}\right\|_{L^2}^2 + \left\|mM\widehat{Q_{\neq}}\right\|_{L^2}^2\right) \\
	\lesssim& \frac{1}{c_{\alpha}^{2}}\frac{2\sqrt{B_{\beta}}+1}{2\sqrt{B_{\beta}}-1} \max \left\{ {\frac{\beta-1}{\beta}}, 1 \right\} e^{-2\lambda \nu^{\frac{1}{3}} t} \| (Q_{\mathrm{in}},K_{\mathrm{in}},H_{\mathrm{in}})\|_{H^{\frac{1}{2}}}^2 \\
	\lesssim& \frac{1}{c_{\alpha}^{2}}\frac{2\sqrt{B_{\beta}}+1}{2\sqrt{B_{\beta}}-1} \max \left\{ {\frac{\beta-1}{\beta}},  {\frac{\beta}{\beta-1}} \right\} e^{-2\lambda \nu^{\frac{1}{3}} t} \left(\| u_{\mathrm{in}}\|_{H^2}^2+\| \theta_{\mathrm{in}}\|_{H^1}^2 \right),
\end{aligned}
\end{equation}
and
\begin{equation}\label{4x.6}
\begin{aligned}
	\left\|{U^2_{\neq}}\right\|_{L^2}^2\leq& {\left\langle {t} \right\rangle}^{-2} \frac{1}{c_{\alpha}^{2}}\left({\frac{\beta-1}{\beta}}\left\|mM\widehat{K_{\neq}}\right\|_{H^{\frac{5}{2}}}^2 + \left\|mM\widehat{Q_{\neq}}\right\|_{H^{\frac{5}{2}}}^2\right) \\
	\lesssim& {\left\langle {t} \right\rangle}^{-2} \, \frac{1}{c_{\alpha}^{2}}\frac{2\sqrt{B_{\beta}}+1}{2\sqrt{B_{\beta}}-1} \max \left\{ {\frac{\beta-1}{\beta}}, 1 \right\} e^{-2\lambda \nu^{\frac{1}{3}} t} \| (Q_{\mathrm{in}},K_{\mathrm{in}},H_{\mathrm{in}})\|_{H^3}^2 \\
	\lesssim& {\left\langle {t} \right\rangle}^{-2} \,\frac{1}{c_{\alpha}^{2}}\frac{2\sqrt{B_{\beta}}+1}{2\sqrt{B_{\beta}}-1} \max \left\{ {\frac{\beta-1}{\beta}},  {\frac{\beta}{\beta-1}} \right\} e^{-2\lambda \nu^{\frac{1}{3}} t} \left(\| u_{\mathrm{in}}\|_{H^{\frac{9}{2}}}^2+\| \theta_{\mathrm{in}}\|_{H^{\frac{7}{2}}}^2 \right),
\end{aligned}
\end{equation}
where we have used  the fact $p_h^{-1} \lesssim {\left\langle {t} \right\rangle}^{-2} |k, \eta, l|^2$. Hence, together \eqref{3x.3}, \eqref{3x.4}, \eqref{4x.5} with \eqref{4x.6} gives \eqref{3.1x}-\eqref{3.3x} and  the proof of Lemma \ref{lem3.4} is finished. It is worth noting that in Lemma \ref{lem3.4} where we require $B_{\beta} > 1/4$, but in fact, obtaining the enhanced dissipation of $U_{\neq}$ and $ \Theta_{\neq}$ does not necessitate this restriction. We have provided a detailed explanation of this point in Appendix \ref{secA} later on. 
 \hfill $\square$

\subsection{Analysis of the zero-modes $(u_0, \theta_0)$}\label{sec3.2}
%

\qquad This section is devoted to the dynamics of the zero modes in the $x$ frequency, which satisfy
\begin{equation}\label{1..7}
	\begin{cases}
		\partial_{t} u_0^1-\nu \Delta_{y, z} u_0^1+
		(1-\beta)u_0^2=0, \\
		\partial_{t} u_0^2-\nu \Delta_{y, z} u_0^2+
		\beta \partial_{z} \Delta_{y, z}^{-1} \partial_{z} u_0^1-\alpha \partial_{y} \Delta_{y, z}^{-1} \partial_{z} \theta_0=0, \\
		\partial_{t} u_0^3-\nu \Delta_{y, z} u_0^3-\beta \partial_{z} \Delta_{y, z}^{-1} \partial_{y} u_0^1+\alpha \partial_{y} \Delta_{y, z}^{-1} \partial_{y} \theta_{0}=0, \\
		\partial_{t} \theta_0-\nu \Delta_{y, z} \theta_0-
		\alpha u_0^3=0, \\
		u(t=0)=u_{\mathrm{in}}, \quad \theta(t=0)=\theta_{\mathrm{in}}.
	\end{cases}
\end{equation}
The zero modes  can be decomposed into  simple-zero modes and double-zero modes, that is,   $u_0(t, y, z)=\overline{u}_0(t, y)+\widetilde{u}_{0}(t, y, z)$ and $\theta_0(t, y, z)=\overline{\theta}_0(t, y)+\widetilde{\theta}_{0}(t, y, z)$.  For the linearized system \eqref{1..7}, we express the solution in an explicit form via Fourier transform and Duhamel’s principle. One gives the following lemma.
\begin{lem}\label{lem3.5}
	$(1)$ Assume that $l=0$, then \eqref{1..7} degenerates into the following heat equations 
	\begin{equation}\label{3.2.2}
		\begin{cases}
			\partial_{t} \overline{u}_0^1-\nu \partial_{y}^2 \overline{u}_0^1=0, \\
			\overline{u}_{0}^{2}=0,\\
			\partial_{t} \overline{u}_0^3-\nu \partial_{y}^2 \overline{u}_0^3+\alpha \overline{\theta}_0=0, \\
			\partial_{t} \overline{\theta}_0-\nu \partial_{y}^2 \overline{\theta}_0-
			\alpha \overline{u}_0^3=0,\\
			\overline{u_0}(t=0)=\overline{({u}_{\mathrm{in}})_0}, \overline{\theta_0}(t=0)=\overline{({\theta}_{\mathrm{in}})_0}.
		\end{cases}
	\end{equation}
	In this case, the double zero frequency $(\overline{u}_0, \overline{\theta}_0)$  of \eqref{3.2.2} is easily obtained, which is \eqref{thm1.1.1.1}--\eqref{thm1.1.3} as in Theorem \ref{1.1.}.
	
	$(2)$ Assume that $l \neq 0$, the simple zero frequency $\widetilde{u}_{0}$ of system \eqref{1..7} can be expressed as the following form
	\begin{align}
		\widehat{\widetilde{({u}^{1})_0}}(t, \eta, l)=&\left[ \frac{\alpha^2 \eta^2}{h^2} e^{\lambda_1 t}+\frac{\beta (\beta-1) l^2}{2h^2} (e^{\lambda_2 t}+e^{\lambda_3 t})\right] \widehat{\widetilde{({u^{1}_{\mathrm{in}}})_0}} \nonumber\\
		&+\frac{i (\beta-1) |\eta, l|}{2h} (e^{\lambda_3 t}-e^{\lambda_2 t})\widehat{\widetilde{({u^{2}_{\mathrm{in}}})_0}} \nonumber\\
		&+\left[\frac{\alpha (\beta-1) \eta l}{h^2} e^{\lambda_1 t}-\frac{\alpha (\beta-1) \eta l}{2h^2}(e^{\lambda_2 t}+e^{\lambda_3 t})\right]\widehat{\widetilde{({\theta_{\mathrm{in}}})_0}},\label{3.2.4}\\
		\widehat{\widetilde{({u}^{2})_0}}(t, \eta, l)=&\left[ \frac{i \beta l^2 }{2h|\eta, l|} (e^{\lambda_2 t}-e^{\lambda_3 t})\right] \widehat{\widetilde{({u^{1}_{\mathrm{in}}})_0}}+\frac{1}{2} (e^{\lambda_2 t}+e^{\lambda_3 t})\widehat{\widetilde{({u^{2}_{\mathrm{in}}})_0}} \nonumber \\
		&+\frac{i\alpha \eta l}{2h |\eta, l|}(e^{\lambda_3 t}-e^{\lambda_2 t})\widehat{\widetilde{({\theta_{\mathrm{in}}})_0}}, \label{3.2.5}\\
		\widehat{\widetilde{({u}^{3})_0}}(t, \eta, l)=&-\frac{\eta}{l}\widehat{\widetilde{({u}^{2})_0}},\label{3.2.6}\\
		\widehat{\widetilde{(\theta)_0}}(t, \eta, l)=&\left[ \frac{\alpha \beta \eta l}{h^2} e^{\lambda_1 t}-\frac{\alpha \beta \eta l}{2h^2} (e^{\lambda_2 t}+e^{\lambda_3 t})\right] \widehat{\widetilde{({u^{1}_{\mathrm{in}}})_0}} \nonumber\\
		&+\frac{i \alpha \eta |\eta, l|}{2h l} (e^{\lambda_2 t}-e^{\lambda_3 t})\widehat{\widetilde{({u^{2}_{\mathrm{in}}})_0}} \nonumber\\
		&+\left[\frac{\beta (\beta-1) l^2}{h^2} e^{\lambda_1 t}+\frac{\alpha^2 \eta^2}{2h^2}(e^{\lambda_2 t}+e^{\lambda_3 t})\right]\widehat{\widetilde{({\theta_{\mathrm{in}}})_0}}, \label{3.2.7}
	\end{align}
	where $h=h(\alpha, \beta, \eta, l)=\sqrt{\alpha^2 \eta^2+ \beta (\beta-1) l^2}$, $\lambda_1=-\nu (\eta^2+l^2)$, $\lambda_2=-\nu (\eta^2+l^2)+i \frac{h}{|\eta, l|}$ and $\lambda_3=-\nu (\eta^2+l^2)-i \frac{h}{|\eta, l|}$. 
\end{lem}

\pf First, we give the proof of $(1)$. When $l=0$, $\overline{u}_0^1$ satisfies the standard heat equation which gives the estimate \eqref{thm1.1.1.1}$_{1}$. The coupling of $\overline{u}_0^3$ and $\overline{\theta}_0$ in a zero-order manner indicates that this does not directly lead to energy dissipation or enhancement. Due to the symmetry of the coupling relationship, the Fourier transform is obtained
\begin{equation*} 
	\begin{cases}
		\partial_{t} \widehat{\overline{u}_0^3}+\nu \eta^2 \widehat{\overline{u}_0^3}+\alpha \widehat{\overline{\theta}_0}=0, \\
		\partial_{t} \widehat{\overline{\theta}_0}+\nu \eta^2 \widehat{\overline{\theta}_0}-
		\alpha \widehat{\overline{u}_0^3}=0,\\
		\widehat{\overline{u^3_0}}(t=0)=\widehat{\overline{({u}^3_{\mathrm{in}})_0}}, \widehat{\overline{\theta_0}}(t=0)=\widehat{\overline{({\theta}_{\mathrm{in}})_0}}.
	\end{cases}
\end{equation*}
Let	$\phi=\widehat{\overline{u}_0^3}+i \widehat{\overline{\theta}_0}$ and $\psi=\widehat{\overline{u}_0^3}-i \widehat{\overline{\theta}_0}$ with the initial data $\phi_{\mathrm{in}}=\widehat{\overline{({u}^3_{\mathrm{in}})_0}}+i \widehat{\overline{({\theta}_{\mathrm{in}})_0}}$ and $\psi_{\mathrm{in}}=\widehat{\overline{({u}^3_{\mathrm{in}})_0}}-i \widehat{\overline{({\theta}_{\mathrm{in}})_0}}$, then we have
\begin{equation*}
	\partial_{t} \phi=(-\nu \eta^2+i \alpha) \phi, \quad \partial_{t} \psi=(-\nu \eta^2-i \alpha) \psi.
\end{equation*}
Thus, a direct calculation yields
\begin{equation*}
	\phi(t, \eta)=e^{-\nu \eta^2 t+ i \alpha t}\phi_{\mathrm{in}}, \quad \psi(t, \eta)=e^{-\nu \eta^2 t- i \alpha t}\psi_{\mathrm{in}}.
\end{equation*}
Noticing that $\widehat{\overline{u}_0^3}=\frac{\phi+\psi}{2}$ and $\widehat{\overline{\theta}_0}=\frac{\phi-\psi}{2i}$, we immediately have
\begin{equation}
\begin{aligned}\label{5.x}
	\widehat{\overline{u}_0^3}(t, \eta)=&\frac{1}{2}e^{-\nu \eta^2 t}\left[ e^{i \alpha t}\left(\widehat{\overline{({u}^3_{\mathrm{in}})_0}}+i \widehat{\overline{({\theta}_{\mathrm{in}})_0}}\right)+e^{- i \alpha t}\left(\widehat{\overline{({u}^3_{\mathrm{in}})_0}}-i \widehat{\overline{({\theta}_{\mathrm{in}})_0}}\right)\right]\\
	=&e^{-\nu \eta^2 t} \left[ \cos(\alpha t) \widehat{\overline{({u}^3_{\mathrm{in}})_0}}-\sin(\alpha t) \widehat{\overline{({\theta}_{\mathrm{in}})_0}}\right].
\end{aligned}
\end{equation}
Similarly, it holds
\begin{equation}\label{6.x}
	\widehat{\overline{\theta}_0}(t, \eta)=e^{-\nu \eta^2 t} \left[ \sin(\alpha t) \widehat{\overline{({u}^3_{\mathrm{in}})_0}}+\cos(\alpha t) \widehat{\overline{({\theta}_{\mathrm{in}})_0}}\right].
\end{equation}
\eqref{5.x} and \eqref{6.x} implies  \eqref{thm1.1.2}  and \eqref{thm1.1.3}.

Next, we prove the case when $l\neq 0$. By using the incompressible condition, $\widetilde{u}_{0}^{3}$   can be expressed as $-\partial_{z}^{-1} \partial_{y} \widetilde{u}_{0}^{2}(t, y, z)$. To simplify the calculation, taking  the Fourier transform to the $\widetilde{u}_{0}^{1}$,  $\widetilde{u}_{0}^{2}$ and $\widetilde{\theta}_{0}$ equations in \eqref{1..7}, it yields
\begin{equation*}
	\partial_{t} \begin{bmatrix}
		\widehat{\widetilde{u}_{0}^{1}} \\
		\widehat{\widetilde{u}_{0}^{2}} \\
		\widehat{\widetilde{\theta}_{0}}
	\end{bmatrix}=\begin{bmatrix}
		-\nu(\eta^2+l^2)  & \beta-1 & 0 \\
		-\beta \frac{l^2}{\eta^2+l^2} & -\nu(\eta^2+l^2) & \alpha \frac{\eta l}{\eta^2+l^2} \\
		0 & -\alpha \frac{\eta}{l} & -\nu(\eta^2+l^2)
	\end{bmatrix}\begin{bmatrix}
		\widehat{\widetilde{u}_{0}^{1}} \\
		\widehat{\widetilde{u}_{0}^{2}} \\
		\widehat{\widetilde{\theta}_{0}}
	\end{bmatrix}=:\mathcal{A}\begin{bmatrix}
		\widehat{\widetilde{u}_{0}^{1}} \\
		\widehat{\widetilde{u}_{0}^{2}} \\
		\widehat{\widetilde{\theta}_{0}}
	\end{bmatrix}.
\end{equation*}
Then, the single zero frequency satisfies the following expression
\begin{equation}
	\begin{aligned}
\label{7.xx}
	\begin{bmatrix}
	\widetilde{u}_{0}^{j}(t, \eta, l) \\
	\widehat{\widetilde{\theta}_{0}}(t, \eta, l)  
\end{bmatrix}=e^{\mathcal{A}t}\begin{bmatrix}
\widehat{\widetilde{({u^j_{\mathrm{in}}})_0}}\\
\widehat{\widetilde{({\theta_{\mathrm{in}}})_0}} 
\end{bmatrix},
\quad \mathrm{for} \quad j=1,2.
\end{aligned}
\end{equation}

We compute the fundamental matrix $e^{\mathcal{A} t}$ explicitly. The characteristic polynomial associated
with $\mathcal{A}$ is given by
\begin{equation*}
	\left(\lambda+ \nu \left( \eta^2+l^2 \right)\right)\left[ \left(\lambda+ \nu \left( \eta^2+l^2 \right)\right)^2+\alpha^2 \frac{\eta^2}{\eta^2+l^2}+\beta \left( \beta-1\right) \frac{l^2}{\eta^2+l^2}\right]=0,
\end{equation*}
and thus the spectra of $\mathcal{A}$ as follows
\begin{equation*}
	\lambda_1=-\nu |\eta, l|^2, \quad \lambda_2=-\nu |\eta, l|^2+i \frac{h}{|\eta, l|}, \quad  \quad \lambda_3=-\nu |\eta, l|^2-i \frac{h}{|\eta, l|},
\end{equation*}
where $h=h(\alpha, \beta, \eta, l)=\sqrt{\alpha^2 \eta^2+ \beta (\beta-1) l^2}$ with $B_{\beta}>0$.
Note that the eigenvalues $\lambda_2$ and $\lambda_3$ are a pair of conjugate complex roots. The eigenvectors $\mathcal{V}_{i}$ $(i=1, 2, 3)$ associated with  $\lambda_{i}$  satisfy
\begin{align*}
	\left( \lambda_{i} \textbf{I} - \mathcal{A} \right) \mathcal{V}_{i}=\begin{bmatrix}
		\lambda_{i}+\nu(\eta^2+l^2)  & 1-\beta & 0 \\
		\beta \frac{l^2}{\eta^2+l^2} & \lambda_{i}+\nu(\eta^2+l^2) & -\alpha \frac{\eta l}{\eta^2+l^2} \\
		0 & \alpha \frac{\eta}{l} & \lambda_{i}+\nu(\eta^2+l^2)
	\end{bmatrix} \begin{bmatrix}
		\mathcal{V}_{i 1}\\
		\mathcal{V}_{i 2}\\
		\mathcal{V}_{i 3}
	\end{bmatrix}=0,
\end{align*}
which gives
\begin{equation*}
	\mathcal{V}_{1}=\begin{bmatrix}
		\alpha \frac{\eta}{l} \\
		0 \\
		\beta 
	\end{bmatrix}, \quad \quad
	\mathcal{V}_{2}=\begin{bmatrix}
		i(1-\beta) \\
		\frac{h}{|\eta, l|}\\
		i \alpha \frac{\eta}{l}
	\end{bmatrix}, \quad \quad
	\mathcal{V}_{3}=\begin{bmatrix}
		i(\beta-1) \\
		\frac{h}{|\eta, l|}\\
		-i \alpha \frac{\eta}{l}
	\end{bmatrix} .
\end{equation*}
Let 
\begin{equation}\label{3.2.9}
	\mathcal{P}:=[\mathcal{V}_{1}, \mathcal{V}_{2}, \mathcal{V}_{3}]=\begin{bmatrix}
		\alpha \frac{\eta}{l} & \quad i(1-\beta) & 	\quad i(\beta-1) \\
		0  &  \quad	\frac{h}{|\eta, l|} & \quad	\frac{h}{|\eta, l|}\\
		\beta & \quad i \alpha \frac{\eta}{l} &  \quad -i \alpha \frac{\eta}{l}
	\end{bmatrix},\quad \mathcal{J}:=\begin{bmatrix}
		e^{\lambda_1 t} &   & 	  \\
		&  e^{\lambda_2 t} &  \\
		&   &  e^{\lambda_3 t}
	\end{bmatrix},
\end{equation}
then we can calculate the adjoint matrix $\mathcal{P^*}$ and determinant $\det(\mathcal{P})$ of $\mathcal{P}$, respectively,
\begin{equation*}
	\mathcal{P^*}=\begin{bmatrix}
		-2 i\alpha \frac{h^2}{l (\eta^2+l^2)} & \quad 0 & 	\quad 2 i(1-\beta) \frac{h}{|\eta, l|} \\
		\beta \frac{h}{|\eta, l|}   &  \quad	-i \alpha^2 \frac{\eta^2}{l^2}-iB_{\beta} & \quad	-\alpha \frac{\eta h}{l |\eta, l|} \\
		-\beta \frac{h}{|\eta, l|} & \quad -i \alpha^2 \frac{\eta^2}{l^2}-iB_{\beta} &  \quad \alpha \frac{\eta h}{l |\eta, l|}
	\end{bmatrix},
\end{equation*}
and 
$\label{3.2.11}
det(\mathcal{P})=-2i \frac{h^3}{l^2 |\eta, l|}.$
Hence, we obtain that 
$e^{\mathcal{A} t}=\frac{1}{\det(\mathcal{P})} \mathcal{P} \mathcal{J} \mathcal{P^*}$, which combined with \eqref{7.xx} gives
\eqref{3.2.4}--\eqref{3.2.7}. \hfill $\square$

Contrast to the classical Navier-Stokes equations, by means of \eqref{3.2.5} and \eqref{3.2.6} we find that although the combined effect of rotation and stratification brings about fluid oscillation, it cancels the lift-up effect. In other words, in the zero frequency part, the oscillation effect caused by the rotation and stratification is significantly stronger than the lift-up effect caused by the Couette flow. 

\subsection{Linear dispersive estimates}\label{sub4.1}
\qquad Recall the analysis of the zero frequency in section \ref{sec3.2}, we know that the simple-zero modes of $u^{2}$ and $u^{3}$ exhibit a dispersion mechanism from the expression. However, the simple-zero modes of $u^{1}$ and $\theta$ has no dispersion attenuation estimation, which brings difficulties to our research on the nonlinear stability threshold problem.  Fortunately we introduce some new variables, which brought about a good structure.
\begin{equation}
	\begin{aligned}\label{000}
		V_0^2=&-\beta\mathcal{R}^{-1}\p_{zz}\Delta_{L,0}^{-1}\tilde{U}_0^1 + \alpha\mathcal{R}^{-1}\p_{yz}\Delta_{L,0}^{-1}\tilde{\Theta}_0=:G_1\tilde{U}_0^1 + G_2\tilde{\Theta}_0,\\
		{V}_0^3=&-\p_z^{-1}\p_yV_0^2=\beta\mathcal{R}^{-1}\p_{yz}\Delta_{L,0}^{-1}\tilde{U}_0^1 - \alpha\mathcal{R}^{-1}\p_{yy}\Delta_{L,0}^{-1}\tilde{\Theta}_0=:G_1'\tilde{U}_0^1 + G_2'\tilde{\Theta}_0,\\
		\Lambda_0=& \alpha\sqrt{\frac{\beta}{\beta-1}}\p_{yz}\Delta_{L,0}^{-1} \tilde{U}_0^1 + \sqrt{B_{\beta}}\p_{zz}\Delta_{L,0}^{-1}\tilde{\Theta}_0=:G_3\tilde{U}_0^1 + G_4\tilde{\Theta}_0,
	\end{aligned}
\end{equation}  
where $l\neq 0$,    $\Delta_{L,0}:=\p_y^2+\p_z^2$ and its symbol $p_0:=\eta^2 + l^2$ and
\begin{equation}\label{ss11}
	\mathcal{R}:=i\left(B_{\beta}\p_{zz}\Delta_{L,0}^{-1}+\alpha^2\p_{yy}\Delta_{L,0}^{-1}\right)^{1/2}.
\end{equation}   
Then using the linearized system \eqref{1..7}, we give 
\begin{align}\label{S0L}
	\begin{cases}
		\p_t\widehat{V_0^j}+\nu p_0\widehat{V_0^j} + \sqrt{B_{\beta}l^2p_0^{-1}+\alpha^2\eta^2p_0^{-1}}\widehat{U_0^j} =0,\\
		\p_t\widehat{U_0^j}+\nu p_0\widehat{U_0^j} - \sqrt{B_{\beta}l^2p_0^{-1}+\alpha^2\eta^2p_0^{-1}}\widehat{V_0^j} =0,\\
		\p_t\widehat{\Lambda_0}+\nu p_0\widehat{\Lambda_0} =0,
	\end{cases} 
\end{align}
or equivalently
\begin{align}\label{sol1}
	\begin{cases}
		\p_t(U_0^j+iV_0^j) = \mathcal{L}(U_0^j+iV_0^j),\,\, \mathcal{L}:=\nu\Delta_{L,0}-\mathcal{R},\\
		(\p_t - \nu\Delta_{L,0})\Lambda_0=0,
	\end{cases}
\end{align}
for $j\in\{2,3\}$.   

Recalling the definition of  operator $\mathcal{R}$ as in \eqref{ss11} is of dispersive nature. To prove this, we  employ the Fourier transform.

For $f\in\R\times\mathbb{T}\to\mathbb{C}$,  we have $$ \mathcal{R}f(y,z)=\sum_{l\neq 0}e^{ilz}\mathcal{R}^lf_l(y),$$
with  $f_l(y):=\frac{1}{2\pi}\int_{\mathbb{T}}e^{-ilz}f(y,z){\rm d}z$,
where for $ h\in\mathscr{S}(\R)$, 
we set 
\[  
\mathcal{R}^l h(y,z):=\mathcal{F}_{\eta,l}^{-1}\left(i\sqrt\frac{B_{\beta}l^2+\alpha^2\eta^2}{l^2+\eta^2}\hat{h}(l,\eta)\right)(y,z).
\]

Then we can calculate the extra decay from the semigroup generated by $\mathcal{R}^l$ as follows. 
\begin{pro}\label{disperprop}
	For any $\sigma>0$, there exists a constant $C>0$ such that for all $h\in\mathscr{S}(\R) $ and  $l\neq0,\ a\geq0$,  then there holds that 
	\begin{align*}
		 |l|^a\left\|(l^2-\p_y^2)^{-\frac{a}2}e^{-t\mathcal{R}^l}h^l\right\|_{L_y^\infty(\R)}\lesssim |l|(qt)^{-\frac13}\|h^{l}\|_{L^1_y(\R)} + |l|^{-1}(qt)^{-\frac12}\|h^l\|_{W_y^{2+\sigma,1}(\R)} 
	\end{align*}  
	with   $ q:=\frac{|\beta(\beta-1)-\alpha^2|}{\alpha}$.
\end{pro}

\pf Recall the Littlewood-Paley decomposition, there exists a smooth function  $ \phi:\R\to[0,1]$ with $ \operatorname{supp}\phi\subset[-2,2]  $ and $ \phi\big|_{[-\frac32 ,\frac32]}=1$, such that
\[  
h=\sum_{j\in\mathbb{Z}} P_j h,\ \mathcal{F}(P_j h)(\eta)=\varphi_j(\eta)\hat{h}(\eta),\ \varphi_j(\eta):=\varphi(2^{-j}\eta),\ \varphi(\eta) = \phi(\eta)-\phi(2\eta).
\] 
By Young's convolution inequality, we have
\begin{equation} 
	\begin{aligned}\label{ll1}
		\left\|e^{-t\mathcal{R}^l}h\right\|_{L^{\infty}} \leq& \sum_{j\in\mathbb{Z}}\|e^{-t\mathcal{R}^l}P_j h\|_{L^{\infty}} \\
		\leq& \sum_{j\in\mathbb{Z}}\left\|\mathcal{F}^{-1}\left(e^{-it\sqrt{\frac{\beta(\beta-1)l^2+\alpha^2\eta^2}{l^2+\eta^2}}}\varphi(2^{-j}\eta)\right)*\mathcal{F}^{-1}\left(\varphi^{\dagger}(2^{-j}\eta)\hat{h}(\eta)\right)\right\|_{L^{\infty}}\\
		\leq& \sum_{j\in\mathbb{Z}}\left\|e^{-t\mathcal{R}^l}\check{\varphi_j} \right\|_{L^\infty}\left\|P_j^{\dagger} h\right\|_{L^1},
	\end{aligned}
\end{equation}
where $\varphi^{\dagger}$ has similar  support properties  as $\varphi$  and $\varphi=\varphi\varphi^{\dagger} $, and $\ P_j^{\dagger}$ is the associated Littlewood-Paley projection.  We next calculate carefully and use stationary phase method to handle the behaviour of semigroup:
\begin{equation}\label{ll2}
	\begin{aligned}
		e^{-t\mathcal{R}^l}\check{\varphi_j}(y)&=\int_\R e^{- it\sqrt{\frac{\beta(\beta-1)l^2 + \alpha^2\eta^2}{l^2+\eta^2}}}e^{iy\eta}\varphi(2^{-j}\eta){\rm d}\eta\\
		&\stackrel{\eta=|l|\xi}{=}|l|\int_{\R} e^{iy|l|\xi-it\sqrt{\frac{\beta(\beta-1)l^2 + \alpha^2\eta^2}{l^2+\eta^2}}}\varphi(2^{-j}|l|\xi){\rm d}\xi\\
		&=|l|\int_{\R}e^{it\Phi(\xi)}\varphi(2^{-j}|l|\xi){\rm d}\xi=:|l|I(t,\alpha,\beta,l,j),
	\end{aligned}
\end{equation}
where the phase function
\begin{equation}\label{ll3} \Phi(\xi):=\frac{yl}{t}\xi-\sqrt{\frac{\beta(\beta-1)+\alpha^2\xi^2}{1+\xi^2}}.
\end{equation}
Direct calculation   shows that 
\begin{equation*}\label{l4}
	\Phi''(\xi)=[\beta(\beta-1)-\alpha^2]\frac{-3\alpha^2\xi^4-2\beta(\beta-1)\xi^2+\beta(\beta-1)}{(1+\xi^2)^{5/2}(\beta(\beta-1)+\alpha^2\xi^2)^{3/2}}.
\end{equation*}
When $ \beta(\beta-1)\neq\alpha^2,$ with $ \alpha>0 $ and $ \beta\notin[0,1]$, then one has
\[
\Phi''(\xi)=0\Rightarrow \xi=\pm\xi_0,\ \xi_0=\sqrt{\frac{\sqrt{\mu^2+3\mu}-\mu}{3}}\leq\frac{1}{\sqrt{2}},
\]
here  $ \mu=\frac{\beta(\beta-1)}{\alpha^2}$.
We decompose 
\begin{align*}
	I(t,\alpha,\beta,l,j)=& I_{-}(t,\alpha,\beta,l,j)+I_{+}(t,\alpha,\beta,l,j),\ I_{+}(t,\alpha,\beta,l,j):=\int_{\R^+}e^{it\Phi(\xi)}\varphi(2^{-j}|l|\xi){\rm d}\xi,  
\end{align*}
and we focus on $I_{+}$. Let $ 2^{j_0+2}l^{-1}=\xi_0$, then $ j_0=-2+\log_2 l + \frac12 \log_{2}\frac{\sqrt{\mu^2+3\mu}-\mu}{3}$. 

{\textit{ Case 1.} } $j\in[j_0,j_0+4]. $ Let $\delta>0$ be a small parameter to be chosen later. We decompose   
\[ 
I_{+}(t,\alpha,\beta,l,j)=\int_{[\xi_0-\delta,\xi_0+\delta]}e^{it\Phi(\xi)}\varphi(2^{-j}|l|\xi){\rm d}\xi+\int_{\mathbb{R}_{+}\backslash[\xi_0-\delta,\xi_0+\delta]}e^{it\Phi(\xi)}\varphi(2^{-j}|l|\xi){\rm d}\xi=: I_{+}^1+ I_{+}^2.
\] 
It is easy to see that $|I_{+}^1|\leq 2\delta$. To bound $I_{+}^{2}$, there exists a small constant $c_0>0$, independent of $l,j$, such that
\[   
|\Phi''(\xi)|=|\beta(\beta-1)-\alpha^2|\frac{3\alpha^2\left(\xi^2+\frac{\sqrt{\mu^2+3\mu}+\mu}{3}\right)}{(1+\xi^2)^{5/2}\left(\beta(\beta-1)+\alpha^2\xi^2\right)^{3/2}}|\xi^2-\xi_0^2|\geq c_0q\delta.
\] 
Then applying Van der Corput lemma, we have 
\begin{align*}
	|I_{+}^2|\leq& C(\delta qt)^{-\frac12}\left[\|\varphi\|_{L^\infty}+\int_{\R^+\slash[\xi_0-\delta,\xi_0+\delta]}|\frac{\rm d}{{\rm d}\xi}\left(\varphi(2^{-j}|l|\xi)\right)|{\rm d}\xi\right]\leq 2C(\delta qt)^{-\frac12}.
\end{align*}
Let $ \delta\sim (\delta qt)^{-\frac12}$, we choose $ \delta=(qt)^{-\frac13}$ and obtain that
\[  
|I_+(t,\alpha,\beta,l,j)|\leq C(qt)^{-\frac13}.
\]

{\textit{Case 2.}} $j\neq[j_0,j_0+4]$, then $|\Phi''(\xi)|\geq c>0$. When $j<j_0$, we can easily get
\[ 
|I_{+}(t,\alpha,\beta,l,j)|\leq \operatorname{supp}_{\xi} \varphi(2^{-j}l\xi)=C2^j|l|^{-1}.
\]
Notice that in this case $|\xi|< \frac{\xi_{0}}{4}$, there exists a constant $C>0$ independent of $j,l$ such that $\Phi''(\xi)>C$. Applying   Lemma \ref{lem2.1} again yields 
\[   
|I_{+}(t,\alpha,\beta,l,j)|\leq C(qt)^{-\frac12}.
\]
On the other hand, when $j>j_0+4$,   using the fact that $|\xi|\leq C2^j|l|^{-1}$, we obtain $|\Phi''(\xi)|\geq C|\xi|^{-4}\geq C2^{-4j}|l|^4$.   We have by Lemma \ref{lem2.1}
\[ 
|I_{+}(t,\alpha,\beta,l,j)|\leq C(qt)^{-\frac12}2^{2j}|l|^{-2}.
\]
Summarizing the estimates above, we arrive at 
\begin{align*}
	|I(t,\alpha,\beta,j,l)|\lesssim&\begin{cases}
		\min\{2^j|l|^{-1},(qt)^{-\frac12}\},\ j<j_0,\\
		(qt)^{-\frac13},\ j_0\leq j\leq j_0+4,\\
		(qt)^{-\frac12}2^{2j}|l|^{-2},\ j_0+4<j.
	\end{cases}
\end{align*}

Inserting this into \eqref{ll1}, we find that
\begin{align*}
	    \|e^{-t\mathcal{R}^l}h\|_{L^\infty}\lesssim& |l|\sum_{j<j_0}\min\{2^j|l|^{-1},(qt)^{-\frac12}\}\|P_j^{\dagger}h\|_{L^1} \\
	    &+ |l|(qt)^{-\frac13}\sum_{j_0}^{j_0+4}\|P_j^{\dagger}h\|_{L^1} + |l|^{-1}(qt)^{-\frac12}\sum_{j>j_0+4}2^{2j}\|P_j^{\dagger}h\|_{L^1} \\
	    \lesssim& |l|(qt)^{-\frac13}\|h\|_{L^1} + |l|^{-1}(qt)^{-\frac12}\|h\|_{W^{2+\sigma,1}}. 
\end{align*}

	The proposition is proved in the case when $a=0$. For $a>0$, using the fact that
	$
	\|(l^2-\p_y^2)^{-\frac{a}2}P_j f\|_{L^1}\lesssim|l|^{-a}\|P_j^{\dagger}f\|_{L^1}
	$ and we finish the proof.\hfill $\square$

\begin{re}
	Recalling the proof of the proposition and combining it with the characteristics of the phase function, we obtain the following classifications:
	
	If $\beta=0$ or $\beta=1$, we have $ \Phi''(\xi)=3\alpha\frac{\xi}{(1+\xi^2)^{5/2}}$, which enjoys dispersion.
	
	If $\alpha=0$, we have $ \Phi''(\xi)=\sqrt{\beta(\beta-1)}\frac{1-2\xi^2}{(1+\xi^2)^{5/2}}$ which also enjoys dispersion. This is consistent with the dispersion mechanism exhibited in the stability problem of the rotating Navier-Stokes equation, see \cite{ CZDZW2025, HSX2024,  LSWZ2025}.
	
	If $\beta(\beta-1)=\alpha^2$, the phase function does not provide oscillation, resulting in the degeneration of the dispersion effect.
	
\end{re}
\begin{cor}\label{dispercor}
	There exists a constant $C>0$ such that for all $f\in\mathscr{S}(\R\times\mathbb{T})$ with $ \int_{\mathbb{T}}f(y,z){\rm d}z=0$, then for any $a\geq0,\ \epsilon>0$, there holds that  
	\begin{align*}
		\||\p_z|^a(-\Delta_{L,0})^{-\frac{a}2}e^{\mathcal{L}t}f\|_{L^\infty(\R\times\mathbb{T})}\leq& C e^{-\nu t}(qt)^{-\frac13}\|f\|_{W^{3,1}(\R\times\mathbb{T})}.
	\end{align*}
\end{cor}
\pf According to Proposition \ref{disperprop}, for $0<\sigma<\frac12$ small enough, we have 
\begin{align*}
	\left||\p_z|^a(-\Delta_{L,0})^{-\frac{a}2}e^{t\mathcal{L}}f(y,z)\right|=&\left|\sum_{l\neq0}|l|^a e^{izl-\nu l^2t}\int_{\eta}e^{iy\eta-\nu \eta^2t - t\mathcal{R}^l}(\eta^2+l^2)^{-\frac{a}2}\hat{f_l}(\eta){\rm d}\eta\right| \\
	\leq&   \sum_{l\neq0}e^{-\nu l^2 t}|l|^a \left\|e^{-\mathcal{R}^lt}\left(e^{\nu t\p_y^2}(l^2-\p_y^2)^{-\frac{a}2}f_l\right)\right\|_{L^\infty(\R)}\\
	\leq&   (qt)^{-\frac13}\sum_{l\neq0}e^{-\nu l^2t} \left(|l|\|e^{\nu t\p_y^2}f_l\|_{L^1_y(\R)}+|l|^{-1}\|e^{\nu t\p_y^2}f_l\|_{W^{2+\sigma,1}_y(\R)}\right) \\
	\leq&   e^{-\nu t}(qt)^{-\frac13}\left(\sum_{l\neq 0}|l|^{-1-\sigma}\|f\|_{L^1_yW^{2+\sigma,1}_z(\R\times\mathbb{T})}+\sum_{l\neq 0}|l|^{-1-\sigma}\|f\|_{W^{2+\sigma,1}_yW^{\sigma,1}_z(\R\times\mathbb{T})}\right)\\ \leq& C e^{-\nu t}(qt)^{-\frac13}\|f\|_{W^{3,1}(\R\times\mathbb{T})},
\end{align*}
thus, we finish the proof of the corollary.\qed
 
Based on Corollary \ref{dispercor}, we directly obtain the linear dispersion estimate \eqref{thm1.1.4} of $u_0^2$ and $\tilde{u}_0^3$.

\section{The nonlineqar stability}\label{section4}
\subsection{Quasi-linearization method}\label{sub4.2} 
\qquad In this section, we discuss the quasi-linearization method, which has been used in many papers, see \cite{MR4121130, CWZ2024, NZ2024, WZ2026, ZhaiZ2023}.   The quasi-linearization method is used to extract the nonlinear dispersive effect of simple-zero modes and improve the stability threshold of non-zero modes, simple-zero modes and double-zero modes. The main idea of quasi-linearization method is to decompose the system into two or more systems: a good system that carries the main structure, regularity and the size of original initial data, and some quasi-linearization and nonlinear systems that contain the smaller nonlinear part, which start from zero initial data.  Let's recall the equation \eqref{1..7} and \eqref{sol1} satisfied by zero modes  and consider the nonlinear terms, which gives
\begin{align}\label{S0NL1}
    &\begin{cases}
        \p_t U_0^2 - \nu \Delta_{L,0} U_0^2 - \mathcal{R}V_0^2 ={\mathcal{N}_{0+\neq}(U_0^2)}:=-(U\cdot\nabla_L U^2)_0 + \p_y\Delta_{L,0}^{-1}(\p_i^LU^j\p_j^LU^i)_0,\\
        \p_t V_0^2 - \nu \Delta_{L,0} V_0^2 + \mathcal{R} U_0^2 ={\mathcal{N}_{0+\neq}(V_0^2)}:= - G_1\widetilde{(U\cdot\nabla_L U^1)}_0 -G_2\widetilde{(U\cdot\nabla_L\Theta)}_0 ,\\
    \end{cases}\\
    & \quad \p_t \Lambda_0 - \nu\Delta_{L,0} \Lambda_0 ={\mathcal{N}_{0+\neq}(\Lambda_0)}:= - G_3\widetilde{(U\cdot\nabla_L U^1)}_0 -G_4\widetilde{(U\cdot\nabla_L\Theta)}_0,\label{S0NL2}\\
    &\begin{cases}\label{S0NL3}
        \p_t \tilde{U}_0^3 - \nu \Delta_{L,0} \tilde{U}_0^3 - \mathcal{R}\tilde{V}_0^3 ={\mathcal{N}_{0+\neq}(\tilde{U}_0^3)}:=-(\widetilde{U\cdot\nabla_L U^3})_0 + \p_z\Delta_{L,0}^{-1}(\widetilde{\p_i^LU^j\p_j^LU^i})_0,\\
        \p_t \tilde{V}_0^3 - \nu \Delta_{L,0} \tilde{V}_0^3 + \mathcal{R} \tilde{U}_0^3 ={\mathcal{N}_{0+\neq}(\tilde{V}_0^3)}:= - \widetilde{G_1'(U\cdot\nabla_L U^1)}_0 -\widetilde{G_2'(U\cdot\nabla_L\Theta)}_0,
    \end{cases}\\
    &\begin{cases}\label{S0NL4}
        \p_t \overline{U_0^1} - \nu \p_{yy} \overline{U_0^1} = -\p_y(\overline{U^2 U^1})_0,\\
        \p_t \overline{U_0^3} - \nu \p_{yy} \overline{U_0^3} +\alpha\overline{\Theta_0}= -\p_y(\overline{U^2 U^3})_0,\\
        \p_t \overline{\Theta_0} - \nu \p_{yy} \overline{\Theta_0} -\alpha\overline{U_0^3}= -\p_y(\overline{U^2 \Theta})_0.
    \end{cases} 
\end{align}
Due to the complex structure of the system, we cannot directly obtain the nonlinear dispersion estimates of simple zero modes. We will use a quasi-linearization method to decompose the system into two parts, one of which retains the regularity and size of the initial data and has a good nonlinear dispersion effect. Another part contains smaller nonlinear terms, and their initial data starts from zero. More precisely, we decompose 
\[  
     U_0=U_{0,1}+U_{0,2},\ \Theta_0=\Theta_{0,1} + \Theta_{0,2},\ V_0^2=V_{0,1}^2+V_{0,2}^2,\ \tilde{V}_0^3=\tilde{V}_{0,1}^3+\tilde{V}_{0,2}^3,\ \Lambda_0=\Lambda_{0,1}+\Lambda_{0,2},
\] 
which satisfies the following systems: 
\begin{flalign}\label{1stQuasi1}
    &\ \left\{\begin{array}{ll}
        \p_t U^2_{0,1} - \nu\Delta_{L,0}U^2_{0,1} - \mathcal{R}V_{0,1}^2 ={\mathcal{N}_{0+\neq}(U_{0,1}^2)}\\
        := -\tilde{U}_{0,1}\cdot\na U_{0,1}^2 + 2\p_y\Delta_{L,0}^{-1}(\p_yU_{0,1}^2\p_yU^2_{0,1}+\p_y\tilde{U}^3_{0,1}\p_zU^2_{0,1})- (U_{\neq,1}\cdot\na_L U^2_{\neq,1})_0 \\
        \quad + \p_y\Delta_{L,0}^{-1}(\p_i^L U_{\neq,1}\cdot\na_L U^i_{\neq,1})_0 ,\\
        \p_t V^2_{0,1} - \nu\Delta_{L,0}V^2_{0,1} + \mathcal{R}U_{0,1}^2 ={\mathcal{N}_{0+\neq}(V_{0,1}^2)}\\
        := -\tilde{U}_{0,1}\cdot\na V_{0,1}^2 - G_1\widetilde{(U_{\neq,1}\cdot\na_L U^1_{\neq,1})}_0 - G_2\widetilde{(U_{\neq,1}\cdot\na_L \Theta_{\neq,1})}_0,\\
        \p_t\Lambda_{0,1} - \nu\Delta_{L,0}\Lambda_{0,1} ={\mathcal{N}_{0+\neq}(\Lambda_{0,1})}\\
		:= {- G_3(\tilde{U}_{0,1}\cdot\na_L U^1_{0,1}) - G_4(\tilde{U}_{0,1}\cdot\na_L \Theta_{0,1})} - G_3\widetilde{(U_{\neq,1}\cdot\na_L U^1_{\neq,1})}_0 - G_4\widetilde{(U_{\neq,1}\cdot\na_L \Theta_{\neq,1})}_0.
    \end{array}\right. &
\end{flalign}
\begin{flalign}\label{1stQuasi2}
    &\ \left\{\begin{array}{ll}
        \p_t \tilde{U}^3_{0,1} - \nu\Delta_{L,0}\tilde{U}^3_{0,1} - \mathcal{R} \tilde{V}_{0,1}^3={\mathcal{N}_{0+\neq}(\tilde{U}_{0,1}^3)} \\
        := -\widetilde{\tilde{U}_{0,1}\cdot\na \tilde{U}_{0,1}^3} + 2\p_z\Delta_{L,0}^{-1}(\widetilde{\p_yU_{0,1}^2\p_yU^2_{0,1}}+\widetilde{\p_y\tilde{U}^3_{0,1}\p_zU^2_{0,1}}) - \widetilde{(U_{\neq,1}\cdot\na_L U^3_{\neq,1})}_0 \\
        \quad+ \p_z\Delta_{L,0}^{-1}\widetilde{(\p_i^L U_{\neq,1}\cdot\na_L U^i_{\neq,1})}_0 ,\\
        \p_t \tilde{V}^3_{0,1} - \nu\Delta_{L,0} \tilde{V}^3_{0,1} + \mathcal{R}\tilde{U}_{0,1}^3 ={\mathcal{N}_{0+\neq}(\tilde{V}_{0,1}^3)} \\
        := {\p_z^{-1}\p_y(\tilde{U}_{0,1}\cdot\na V_{0,1}^2)} - G_1'\widetilde{(U_{\neq,1}\cdot\na_L U^1_{\neq,1})}_0 - G_2'\widetilde{(U_{\neq,1}\cdot\na_L \Theta_{\neq,1})}_0.
    \end{array}\right. &
\end{flalign}
\begin{flalign}\label{1stQuasi3}
    &\ \left\{\begin{array}{ll}
        \p_t\overline{U_{0,1}^1} - \nu\p_{yy}\overline{U_{0,1}^1} ={\mathcal{N}_{0+\neq}(\overline{U_{0,1}^1})} := -\p_y\overline{(U_{0,1}^2\tilde{U}_{0,1}^1)} - \p_y(\overline{U^2_{\neq,1}U^1_{\neq,1}})_0,\\
        \p_t\overline{U_{0,1}^3} - \nu\p_{yy}\overline{U_{0,1}^3} + \alpha\overline{\Theta_{0,1}} ={\mathcal{N}_{0+\neq}(\overline{U_{0,1}^3})} := -\p_y\overline{(U_{0,1}^2\tilde{U}_{0,1}^3)} - \p_y(\overline{U^2_{\neq,1}U^3_{\neq,1}})_0,\\
        \p_t\overline{\Theta_{0,1}} - \nu\p_{yy}\overline{\Theta_{0,1}} - \alpha\overline{U^3_{0,1}} = {\mathcal{N}_{0+\neq}(\overline{\Theta_{0,1}})} := - \p_y(\overline{U^2_{\neq,1}\Theta_{\neq,1}})_0.
    \end{array}\right.&
\end{flalign}
with the initial data
\begin{align*}
	U_{0,1}^2(t=0)=&U_0^2(t=0),\ V_{0,1}^2(t=0)=V_0^2(t=0),\ \Lambda_{0,1}(t=0)=\Lambda_0(t=0),\  \tilde{U}_{0,1}^3(t=0)=\tilde{U}_0^3(t=0),\\
	\tilde{V}_{0,1}^3(t=0)=&\tilde{V}_0^3(t=0),\ \overline{U_{0,1}^1}(t=0)=\overline{U_0^1}(t=0),\ \overline{U_{0,1}^3}(t=0)=\overline{U_0^3}(t=0),\ \overline{\Theta_{0,1}}(t=0)=\overline{\Theta_0}(t=0),
\end{align*}
and
\begin{flalign}\label{2ndQuasi1}
    &\ \left\{\begin{array}{ll}
        \p_t U^2_{0,2} - \nu\Delta_{L,0}U^2_{0,2} - \mathcal{R}V_{0,2}^2 \\
        ={\mathcal{N}_{0+\neq}(U_{0,2}^2)}:= -U_0\cdot\na U_0^2 + 2\p_y\Delta_{L,0}^{-1}(\p_yU_0^2\p_yU^2_0+\p_yU^3_0\p_zU^2_0)\\
        \quad -\left[-\tilde{U}_{0,1}\cdot\na U_{0,1}^2 + 2\p_y\Delta_{L,0}^{-1}(\p_yU_{0,1}^2\p_yU^2_{0,1}+\p_y\tilde{U}^3_{0,1}\p_zU^2_{0,1})\right]\\
        \quad -(U_{\neq}\cdot\na_L U^2_{\neq})_0+ (U_{\neq,1}\cdot\na_L U^2_{\neq,1})_0 + \p_y\Delta_{L,0}^{-1}(\p_i^L U_{\neq}\cdot\na_L U^i_{\neq})_0-\p_y\Delta_{L,0}^{-1}(\p_i^L U_{\neq,1}\cdot\na_L U^i_{\neq,1})_0,\\ 
        \p_t V^2_{0,2} - \nu\Delta_{L,0}V^2_{0,2} + \mathcal{R}U_{0,2}^2 \\
        ={\mathcal{N}_{0+\neq}(V_{0,2}^2)}:= - G_1\widetilde{(U_0\cdot\na U^1_0)} - G_2\widetilde{(U_0\cdot\na \Theta_0)} + \tilde{U}_{0,1}\cdot\na V_{0,1}^2 \\
        \quad - G_1\widetilde{(U_{\neq}\cdot\na_L U^1_{\neq})}_0 + G_1\widetilde{(U_{\neq,1}\cdot\na_L U^1_{\neq,1})}_0 - G_2\widetilde{(U_{\neq}\cdot\na_L \Theta_{\neq})}_0 + G_2\widetilde{(U_{\neq,1}\cdot\na_L \Theta_{\neq,1})}_0,\\
        \p_t\Lambda_{0,2} - \nu\Delta_{L,0}\Lambda_{0,2} \\
        ={\mathcal{N}_{0+\neq}({\Lambda_{0,2}})}:= - G_3\widetilde{(U_0\cdot\na U_0^1)} - G_4\widetilde{(U_0\cdot\na \Theta_0)} +{G_3(\tilde{U}_{0,1}\cdot\na U^1_{0,1}) + G_4(\tilde{U}_{0,1}\cdot\na \Theta_{0,1})} \\
        \quad - G_3\widetilde{(U_{\neq}\cdot\na_L U^1_{\neq})}_0 + G_3\widetilde{(U_{\neq,1}\cdot\na_L U^1_{\neq,1})}_0 - G_4\widetilde{(U_{\neq}\cdot\na_L \Theta_{\neq})}_0 + G_4\widetilde{(U_{\neq,1}\cdot\na_L \Theta_{\neq,1})}_0.
    \end{array}    
    \right.&
\end{flalign}
\begin{flalign}\label{2ndQuasi2}
    &\ \left\{\begin{array}{ll}
        \p_t \tilde{U}^3_{0,2} - \nu\Delta_{L,0}\tilde{U}^3_{0,2} - \mathcal{R}V_{0,2}^3 \\
        ={\mathcal{N}_{0+\neq}(\tilde{U}_{0,2}^3)}:= -\widetilde{U_{0}\cdot\na U_{0}^3} + 2\p_z\Delta_{L,0}^{-1}(\widetilde{\p_yU_{0}^2\p_yU^2_{0}}+\widetilde{\p_yU^3_{0}\p_zU^2_{0}})\\
        \quad -\left[-\widetilde{\tilde{U}_{0,1}\cdot\na \tilde{U}_{0,1}^3} + 2\p_z\Delta_{L,0}^{-1}(\widetilde{\p_yU_{0,1}^2\p_yU^2_{0,1}} + \widetilde{\p_y\tilde{U}^3_{0,1}\p_zU^2_{0,1}})\right]\\
        \quad - (\widetilde{U_{\neq}\cdot\na_L U^3_{\neq}})_0 + \widetilde{(U_{\neq,1}\cdot\na_L U^3_{\neq,1})}_0 + \p_z\Delta_{L,0}^{-1}(\widetilde{\p_i^L U_{\neq}\cdot\na_L U^i_{\neq}})_0 - \p_z\Delta_{L,0}^{-1}\widetilde{(\p_i^L U_{\neq,1}\cdot\na_L U^i_{\neq,1})}_0,\\
        \p_t \tilde{V}^3_{0,2} - \nu\Delta_{L,0}\tilde{V}^3_{0,2} + \mathcal{R}\tilde{U}_{0,2}^3 \\
        ={\mathcal{N}_{0+\neq}(\tilde{V}_{0,2}^3)}:= - \widetilde{G_1'(U_{0}\cdot\na U^1_0)}- \widetilde{G_2'(U_{0}\cdot\na \Theta_0)} - {\p_z^{-1}\p_y(\tilde{U}_{0,1}\cdot\na V_{0,1}^2)} \\
        \quad - \widetilde{G_1'(U_{\neq}\cdot\na_L U^1_{\neq})_0} + G_1'\widetilde{(U_{\neq,1}\cdot\na_L U^1_{\neq,1})}_0 - \widetilde{G_2'(U_{\neq}\cdot\na_L \Theta_{\neq})_0}- G_2'\widetilde{(U_{\neq,1}\cdot\na_L \Theta_{\neq,1})}_0.
    \end{array}\right.&
\end{flalign}
\begin{flalign}\label{2ndQuasi3}
    &\ \left\{\begin{array}{ll}
        \p_t\overline{U_{0,2}^1} - \nu\p_{yy}\overline{U_{0,2}^1} = -\p_y\overline{(U_{0}^2\tilde{U}_0^1)} + \p_y\overline{(U_{0,1}^2\tilde{U}_{0,1}^1)}- \p_y(\overline{U^2_{\neq}U^1_{\neq}})_0 + \p_y(\overline{U^2_{\neq,1}U^1_{\neq,1}})_0,\\
        \p_t\overline{U_{0,2}^3} - \nu\p_{yy}\overline{U_{0,2}^3} + \alpha\overline{\Theta_{0,2}} = -\p_y\overline{(U_{0}^2\tilde{U}_0^3)} + \p_y\overline{(U_{0,1}^2\tilde{U}_{0,1}^3)} - \p_y(\overline{U^2_{\neq}U^3_{\neq}})_0 + \p_y(\overline{U^2_{\neq,1}U^3_{\neq,1}})_0,\\
        \p_t\overline{\Theta_{0,2}} - \nu\p_{yy}\overline{\Theta_{0,2}} - \alpha\overline{U^3_{0,2}} = -\p_y\overline{(U_{0}^2\tilde{\Theta}_0)} - \p_y(\overline{U^2_{\neq}\Theta_{\neq}})_0 + \p_y(\overline{U^2_{\neq,1}\Theta_{\neq,1}})_0
    \end{array}    
    \right.&
\end{flalign}
with zero initial data
\begin{align*}
	&U_{0,2}^2(t=0)=\tilde{U}_{0,2}^3(t=0)=V_{0,2}^2(t=0)=\tilde{V}_{0,2}^3(t=0)=\Lambda_{0,2}(t=0)=0,\\
	&\overline{U_{0,2}^1}(t=0)=\overline{U_{0,2}^3}(t=0)=\overline{\Theta_{0,2}}(t=0)=0,
\end{align*}
where we have frequently used the fact that
\[  
(\p_i^L U^j\p_j^L U^i)_0=\p_i U^j_0\p_jU^i_0 + (\p_i^LU^j_{\neq}\p_j^LU^i_{\neq})_0,\,\, \p_iU^j_0\p_j U^i_0=2(\p_yU_0^2\p_yU^2_0+\p_y\tilde{U}^3_0\p_zU^2_0) + 2\p_y\overline{U^3_0}\p_zU^2_0.
\]

Similarly, for non-zero modes, we recall that \eqref{Non0NL-1} and consider the nonlinear terms, which yields 
\begin{align}\label{Non0NL}
	\begin{cases}
		\p_t Q_{\neq} -\nu\Delta_L Q_{\neq} - \frac12 \p_x\p_y^L|\na_L|^{-2} Q_{\neq} - \p_x\p_y^L(|\na_L|^{-2}-|\na_{L,h}|^{-2}) Q_{\neq} \\
		\qquad \ + i\sqrt{\beta(\beta-1)}\p_z|\na_L|^{-1} K_{\neq} + 2i\sqrt{\frac{\beta-1}{\beta}}\p_z\p_{xx}|\na_L|^{-1}|\na_{L,h}|^{-2}K_{\neq} + \alpha|\na_{L,h}||\na_L|^{-1}H_{\neq} \\
		= |\na_L|^{\frac32}|\na_{L,h}|^{-1}(U\cdot\nabla_L U^3)_{\neq} - \p_z|\na_{L,h}|^{-1}|\na_L|^{-\frac12}(\p_i^L U^j\p_j^L U^i)_{\neq}, \\
		\p_t K_{\neq} - \nu\Delta_L K_{\neq} + \frac12 \p_x\p_y^L|\na_L|^{-2} K_{\neq}  + \p_x\p_y^L(|\na_L|^{-2}-|\na_{L,h}|^{-2}) K_{\neq} - i\sqrt{\beta(\beta-1)}\p_z|\na_L|^{-1}Q_{\neq} \\
		= i\sqrt{\frac{\beta}{\beta-1}}|\na_L|^{\frac12}|\na_{L,h}|^{-1}\left[\p_x(U\cdot\nabla_L U^2)_{\neq}  - \p_y^L(U\cdot\nabla_L U^1)_{\neq} \right],\\
		\p_t H_{\neq} -\nu\Delta_L H_{\neq} - \frac12 \p_x\p_y^L|\na_L|^{-2} H_{\neq} - \alpha|\na_{L,h}||\na_L|^{-1}Q_{\neq} =  |\na_L|^{\frac12}(U\cdot\nabla_L\Theta)_{\neq}.
	\end{cases}
\end{align}
We do some decomposition via quasi-linearization method as follows 
\[
Q_{\neq}=Q_{\neq,1}+Q_{\neq,2},\ K_{\neq}=K_{\neq,1}+K_{\neq,2},\ H_{\neq}=H_{\neq,1}+H_{\neq,2}.
\]
Then the first part $(Q_{\neq,1}, K_{\neq,1}, H_{\neq,1})$ enjoys the fisrt subsystem 
\begin{align}\label{Non0NL1}
	\begin{cases}
		\p_t Q_{\neq,1} -\nu\Delta_L Q_{\neq,1} - \frac12 \p_x\p_y^L|\na_L|^{-2} Q_{\neq,1} - \p_x\p_y^L(|\na_L|^{-2}-|\na_{L,h}|^{-2}) Q_{\neq,1} \\
		\qquad \ + i\sqrt{\beta(\beta-1)}\p_z|\na_L|^{-1} K_{\neq,1} + 2i\sqrt{\frac{\beta-1}{\beta}}\p_z\p_{xx}|\na_L|^{-1}|\na_{L,h}|^{-2}K_{\neq,1} + \alpha|\na_{L,h}||\na_L|^{-1}H_{\neq,1} \\
		= |\na_L|^{\frac32}|\na_{L,h}|^{-1}(U^1_{\neq,1}\p_x U^3_{\neq,1}+U^1_{0,1}\p_x U^3_{\neq,1}) \\
		\quad - \p_z|\na_{L,h}|^{-1}|\na_L|^{-\frac12}(\p_{x,z} U^{1,3}_{\neq,1}\cdot\na_{x,z} U^{1,3}_{\neq,1}+2\na_{L,h}U_{\neq,1}^2\p_y^LU_{\neq,1}^{1,2}+2\p_z U^{1,3}_{0,1}\cdot\na_{x,z} U^3_{\neq,1}+2\p_yU_{0,1}^1\p_xU_{\neq,1}^2), \\
		\p_t K_{\neq,1} - \nu\Delta_L K_{\neq,1} + \frac12 \p_x\p_y^L|\na_L|^{-2} K_{\neq,1} + \p_x\p_y^L(|\na_L|^{-2}-|\na_{L,h}|^{-2}) K_{\neq,1} - i\sqrt{\beta(\beta-1)}\p_z|\na_L|^{-1}Q_{\neq,1} \\
		= i\sqrt{\frac{\beta}{\beta-1}}|\na_L|^{\frac12}|\na_{L,h}|^{-1}\big[\p_xU_{\neq,1}\cdot\na_L U^2_{\neq,1}-\p_y^LU_{\neq,1}\cdot\na_LU^1_{\neq,1} +U_{\neq,1}\cdot\na_LW^3_{\neq,1}\\
		\quad +\p_xU_{\neq,1}\cdot\na U_{0,1}^2-\p_y^LU_{\neq,1}^2\p_y U_{0,1}^1-\p_yU_{0,1}^{1,3}\cdot\na_{x,z}U^1_{\neq,1}+U_{0,1}^{1,3}\cdot\na_{x,z}W^3_{\neq,1} \big],\\
		\p_t H_{\neq,1} -\nu\Delta_L H_{\neq,1} - \frac12 \p_x\p_y^L|\na_L|^{-2} H_{\neq,1} - \alpha|\na_{L,h}||\na_L|^{-1}Q_{\neq,1}\\
		=  |\na_L|^{\frac12}(U_{\neq,1}\cdot\nabla_L\Theta_{\neq,1}+U_{0,1}^{1,3}\cdot\nabla_{x,z}\Theta_{\neq,1}+U_{\neq,1}^3\p_z\Theta_{0,1}),\\
		Q_{\neq,1}(t=0)=Q_{\neq}(t=0),\ K_{\neq,1}(t=0)=K_{\neq}(t=0),\ H_{\neq,1}(t=0)=H_{\neq}(t=0),
	\end{cases}
\end{align}
and the rest part $(Q_{\neq,2}, K_{\neq,2}, H_{\neq,2})$ satisfies the second subsystem
\begin{align}\label{Non0NL2}
	\begin{cases}
		\p_t Q_{\neq,2} -\nu\Delta_L Q_{\neq,2} - \frac12 \p_x\p_y^L|\na_L|^{-2} Q_{\neq,2} - \p_x\p_y^L(|\na_L|^{-2}-|\na_{L,h}|^{-2}) Q_{\neq,2} \\
		\quad \ + i\sqrt{\beta(\beta-1)}\p_z|\na_L|^{-1} K_{\neq,2} + 2i\sqrt{\frac{\beta-1}{\beta}}\p_z\p_{xx}|\na_L|^{-1}|\na_{L,h}|^{-2}K_{\neq,2} + \alpha|\na_{L,h}||\na_L|^{-1}H_{\neq,2} \\
		= |\na_L|^{\frac32}|\na_{L,h}|^{-1}\left(U^{2,3}_{\neq,1}\na_{y,z}^LU^3_{\neq,1}+U^3_{\neq,1}\p_zU^3_{\neq,2}+U^3_{\neq,2}\p_zU^3_{\neq,1}+U^h_{\neq,1}\cdot\nabla_{L,h} U^3_{\neq,2}+ U^h_{\neq,2}\cdot\nabla_{L,h} U^3_{\neq,1}\right.\\
		\quad +U_{\neq,2}\cdot\na_L U^3_{\neq,2}+U^2_{0,1}\p_y^L U^3_{\neq,1}+U^3_{0,1}\p_z U^3_{\neq,1}+U^2_{\neq,1}\p_y U^3_{0,1}+U^3_{\neq,1}\p_zU^3_{0,1}+U^2_{0,2}\p_y^L U^3_{\neq,1}+U^3_{0,2}\p_z U^3_{\neq,1}\\
		\quad +U_{0,1}\cdot\na_L U_{\neq,2}^3 + U^1_{0,2}\p_x U^3_{\neq,1} + U^2_{\neq,1}\p_y U^3_{0,2} +U^2_{\neq,2}\p_y U^3_{0,1}+U^3_{\neq,1}\p_zU^3_{0,2} +U^3_{\neq,2}\p_zU^3_{0,1}+U_{\neq,2}^{2,3}\cdot\na_{y,z} U^3_{0,2} \\
		\quad \left. +U_{0,2} \cdot \nabla_L U^3_{\neq,2} \right) - \p_z|\na_{L,h}|^{-1}|\na_L|^{-\frac12}\left(2\p_zU_{\neq,1}^2\p_y^LU_{\neq,1}^3  +\p_i^L U^j_{\neq,1}\p_j^L U^i_{\neq,2}+\p_i^L U^j_{\neq,2}\p_j^L U^i_{\neq,1}+\p_i^L U^j_{\neq,2}\p_j^L U^i_{\neq,2}\right. \\
		\quad +2\p_yU_{0,1}^3\p_zU_{\neq,1}^2+ \p_i U^2_{0,1}\p_y^L U^i_{\neq,1} + \p_y^L U^i_{\neq,1}\p_i U^2_{0,1}+\p_i U^j_{0,1}\p_j^L U^i_{\neq,2}+ \p_i^L U^j_{\neq,2}\p_j U^i_{0,1}+ \p_i U^j_{0,2}\p_j^L U^i_{\neq,1}\\
		\quad \left.+ \p_i^L U^j_{\neq,1}\p_j U^i_{0,2} + \p_i U^j_{0,2}\p_j^L U^i_{\neq,2} + \p_i^L U^j_{\neq,2}\p_j U^i_{0,2}\right), \\
		\p_t K_{\neq,2} - \nu\Delta_L K_{\neq,2} + \frac12 \p_x\p_y^L|\na_L|^{-2} K_{\neq,2} + \p_x\p_y^L(|\na_L|^{-2}-|\na_{L,h}|^{-2}) K_{\neq,2} - i\sqrt{\beta(\beta-1)}\p_z|\na_L|^{-1}Q_{\neq,2} \\
		= i\sqrt{\frac{\beta}{\beta-1}}|\na_L|^{\frac12}|\na_{L,h}|^{-1}\left(\p_xU_{\neq,1}\cdot\na_L U^2_{\neq,2}+\p_xU_{\neq,2}\cdot\na_L U^2_{\neq,1}-\p_y^LU_{\neq,1}\cdot\na_LU^1_{\neq,2}-\p_y^LU_{\neq,2}\cdot\na_LU^1_{\neq,1}\right. \\
		\quad +U_{\neq,1}\cdot\na_LW^3_{\neq,2}+U_{\neq,2}\cdot\na_LW^3_{\neq,1}+\p_xU_{\neq,2}\cdot\na_L U^2_{\neq,2}-\p_y^LU_{\neq,2}\cdot\na_LU^1_{\neq,2}+U_{\neq,2}\cdot\na_LW^3_{\neq,2}\\
		\quad +{U_{0,1}^2\p_y^LW_{\neq,1}^3}-U_{\neq,1}^2\p_{yy}U^1_{0,1}-\p_y U_{0,1}^2\p_y^L U^1_{\neq,1}-\p_y^LU_{\neq,1}^3\p_z U_{0,1}^1 - U_{\neq,1}^3\p_z\p_yU^1_{0,1}+ \p_xU_{\neq,2}\cdot\na U_{0,1}^2\\
		\quad -\p_y^LU_{\neq,2}\cdot\na U_{0,1}^1-\p_yU_{0,1}\cdot\na_LU^1_{\neq,2}+U_{0,2}^2\p_y^L W^3_{\neq,1}-U_{\neq,1}^2\p_{yy}U^1_{0,2}+U_{0,1}\cdot\na_LW^3_{\neq,2} - U_{\neq,2}\cdot\na \p_yU^1_{0,1} \\
		\quad + \p_xU_{\neq,1}\cdot\na U_{0,2}^2-\p_y^LU_{\neq,1}\cdot\na U_{0,2}^1-\p_yU_{0,2}\cdot\na_LU^1_{\neq,1}+U_{0,2}^{1,3}\cdot\na_{x,z}W^3_{\neq,1} - U_{\neq,1}^3\p_z\p_yU^1_{0,2} \\
		\quad  \left.+ \p_xU_{\neq,2}\cdot\na U_{0,2}^2-\p_y^LU_{\neq,2}\cdot\na U_{0,2}^1-\p_yU_{0,2}\cdot\na_LU^1_{\neq,2}+U_{0,2}\cdot\na_LW^3_{\neq,2}- U_{\neq,2}\cdot\na \p_yU^1_{0,2}\right),\\
		\p_t H_{\neq,2} -\nu\Delta_L H_{\neq,2} - \frac12 \p_x\p_y^L|\na_L|^{-2} H_{\neq,2} - \alpha|\na_{L,h}||\na_L|^{-1}Q_{\neq,2}\\
		=  |\na_L|^{\frac12}(U_{\neq,1}\cdot\nabla_L\Theta_{\neq,2}+U_{\neq,2}\cdot\nabla_L\Theta_{\neq,1}+U_{\neq,2}\cdot\nabla_L\Theta_{\neq,2} +U^2_{\neq,1}\p_y\Theta_{0,1}+ {U_{0,1}^2\p_y^L\Theta_{\neq,1}}+ U_{0,2}^2\p_y^L\Theta_{\neq,1}\\
		\quad +U^2_{\neq,1}\p_y\Theta_{0,2}+U_{0,1}\cdot\nabla_L\Theta_{\neq,2}+U_{\neq,2}\cdot\nabla\Theta_{0,1}+U_{0,2}^{1,3}\cdot\nabla_{x,z}\Theta_{\neq,1}+U_{\neq,1}^3\p_z\Theta_{0,2}+U_{0,2}\cdot\nabla_L\Theta_{\neq,2}+U_{\neq,2}\cdot\nabla\Theta_{0,2} ),\\
		Q_{\neq,2}(t=0)=0,\ K_{\neq,2}(t=0)=0,\ H_{\neq,2}(t=0)=0.
	\end{cases}
\end{align}

With the help of analysis above and the relationships between variables, such as \eqref{3x.1}, \eqref{3x.2}, \eqref{ll6} and \eqref{ll7}, theorem \ref{1..1} is mainly proved by using the following bootstrap argument. We set $N\geq s+4$ with $s>6$ and define $$\mathcal{B}:= mMe^{\frac{\lambda}{2} \nu^{\frac{1}{3}}t}.$$

\begin{pro}[Bootstrap step]\label{boots}
	Under the hypothesis of Theorem \ref{1..1}, assume that for some $T>0$, we have the following bounds for $t\in[0,T]$:
	
	$(1)$ the bounds on $Q_{\neq,1}$,  ${K}_{\neq,1}$ and  ${H}_{\neq,1}$:	
		\begin{align}\label{qkh1}
			&\|\mathcal{A}(Q_{\neq,1},K_{\neq,1},H_{\neq,1})\|_{L^\infty_t H^N}^2 + \nu\|\na_L\mathcal{A}(Q_{\neq,1},K_{\neq,1},H_{\neq,1})\|_{L^2_t H^N}^2\\ &\quad+ \left\|\sqrt{-\frac{\dot{M}}{M}}\mathcal{A}(Q_{\neq,1},K_{\neq,1},H_{\neq,1})\right\|_{L^2_t H^N}^2 \leq 100\delta^2\nu^{2c},\nonumber
		\end{align}
		
		$(2)$ the bounds on $Q_{\neq,2}$,  ${K}_{\neq,2}$ and  ${H}_{\neq,2}$:	
		\begin{align}\label{qkh2}
		&\|\mathcal{B}(Q_{\neq,2},K_{\neq,2},H_{\neq,2})\|_{L^\infty_t H^s}^2 + \nu\|\na_L\mathcal{B}(Q_{\neq,2},K_{\neq,2},H_{\neq,2})\|_{L^2_t H^s}^2\\&\quad + \left\|\sqrt{-\frac{\dot{M}}{M}}\mathcal{B}(Q_{\neq,2},K_{\neq,2},H_{\neq,2})\right\|_{L^2_t H^s}^2 \leq 100\delta^2\nu^{2b},\nonumber
		\end{align}
		
	$(3)$ the bounds on $U_{0}$ and $ \Theta_{0}$:
		\begin{equation}\label{u01}
		\begin{aligned}
		\|(U_{0,1}^2,V_{0,1}^2,\tilde{U}_{0,1}^3,\tilde{V}_{0,1}^3,\Lambda_{0,1})\|_{L^\infty_t H^{N,N+\frac12}}^2 + \nu\|\na(U_{0,1}^2,V_{0,1}^2,\tilde{U}_{0,1}^3,\tilde{V}_{0,1}^3,\Lambda_{0,1})\|_{L^2_t H^{N,N+\frac12}}^2 \leq 100\delta^2\nu^{2c},\\
		\|(\tilde{U}_{0,1}^1,\tilde{\Theta}_{0,1})\|_{L^\infty_t H^{N,N+\frac12}}^2 + \nu\|\na(\tilde{U}_{0,1}^1,\tilde{\Theta}_{0,1})\|_{L^2_t H^{N,N+\frac12}}^2 \leq 100\delta^2\nu^{2c},\\
		\|(\overline{U_{0,1}^1},\overline{U_{0,1}^3},\overline{\Theta_{0,1}})\|_{L^\infty_tH^N}^2 + \nu\|\p_y(\overline{U_{0,1}^1},\overline{U_{0,1}^3},\overline{\Theta_{0,1}})\|_{L^2_tH^N}^2 \leq 100\delta^2\nu^{2c}, 
				\end{aligned}
	\end{equation}
		\begin{equation}\label{u02}
		\begin{aligned}
		\|(U_{0,2}^2,V_{0,2}^2,\tilde{U}_{0,2}^3,\tilde{V}_{0,2}^3,\Lambda_{0,2})\|_{L^\infty_t H^{s,s+\frac12}}^2 + \nu\|\na(U_{0,2}^2,V_{0,2}^2,\tilde{U}_{0,2}^3,\tilde{V}_{0,2}^3,\Lambda_{0,2})\|_{L^2_t H^{s,s+\frac12}}^2 \leq 100\delta^2\nu^{2a},\\
		\|(\tilde{U}_{0,2}^1,\tilde{\Theta}_{0,2})\|_{L^\infty_t H^{s,s+\frac12}}^2 + \nu\|\na(\tilde{U}_{0,2}^1,\tilde{\Theta}_{0,2})\|_{L^2_t H^{s,s+\frac12}}^2 \leq 100\delta^2\nu^{2a},\\
		\|(\overline{U_{0,2}^1},\overline{U_{0,2}^3},\overline{\Theta_{0,2}})\|_{L^\infty_tH^s}^2 + \nu\|\p_y(\overline{U_{0,2}^1},\overline{U_{0,2}^3},\overline{\Theta_{0,2}})\|_{L^2_tH^{s}}^2 \leq 100\delta^2\nu^{2a},
		\end{aligned}
\end{equation}
where $a,b,c$ are some constants which will be chosen later.	Assume that $\epsilon_{0}=\left\|u_{\mathrm{in}}\right\|_{H^{N+2}\cap W^{N+3,1}}+\left\|\theta_{\mathrm{in}}\right\|_{H^{N+1}\cap W^{N+3,1}}\leq \delta \nu^{\frac{14}{15}}$  for any $\delta>0$ and $\nu \in (0, 1)$,  the  estimates \eqref{qkh1}--\eqref{u02} hold on $[0,T]$. Then for $\delta$ sufficiently small depending only on $N, s, \alpha, \beta$,  but not on $T$, these same estimates hold 
with all the occurrences of 100 on the right-hand side replaced by 50.
\end{pro}
  
  \pf For the sake of simplicity in calculation, we assume that $c\leq a$ and $c\leq b$. We will maintain this assumption throughout the following proof. Proposition \ref{boots}  is directly derived from the following propositions \ref{E01}, \ref{propEneq1},   \ref{propE02}  and  \ref{propEneq2}. For specific details, see subsection \ref{sub4.3}-\ref{sub4.6}.  \hfill$\square$
  
For convenience, let us define some energy functionals which will be used in the following sections.

$(1)$ the energy functionals for non-zero modes of the systems \eqref{Non0NL1} and \eqref{Non0NL2}:
	\begin{align}
E_{\neq,1}(t)=&\|\mathcal{A}(Q_{\neq,1},K_{\neq,1},H_{\neq,1})\|_{H^N}^2,\label{pre1} \\
F_{\neq,1}(t)=&\nu\|\mathcal{A}\na_L(Q_{\neq,1},K_{\neq,1},H_{\neq,1})\|_{H^N}^2 + \left\|\mathcal{A}\sqrt{-\frac{\dot{M}}{M}}(Q_{\neq,1},K_{\neq,1},H_{\neq,1})\right\|_{H^N}^2 \label{pre2}, 
\end{align}
\begin{align}
	E_{\neq,2}(t)=&\|\mathcal{B}(Q_{\neq,2},K_{\neq,2},H_{\neq,2})\|_{H^s}^2,\label{pre3}\\
	F_{\neq,2}(t)=&\nu\|\mathcal{B}\na_L(Q_{\neq,2},K_{\neq,2},H_{\neq,2})\|_{H^s}^2 + \left\|\mathcal{B}\sqrt{-\frac{\dot{M}}{M}}(Q_{\neq,2},K_{\neq,2},H_{\neq,2})\right\|_{H^s}^2 \label{pre4}.
\end{align} 

$(2)$ the energy functionals for zero modes of the systems \eqref{1stQuasi1}-\eqref{1stQuasi3} and \eqref{2ndQuasi1}-\eqref{2ndQuasi3}:
	\begin{align}
	E_{0,1}(t)=&\|(U^2_{0,1},\tilde{U}^3_{0,1}, V^2_{0,1},\tilde{V}^3_{0,1},\Lambda_{0,1} )\|_{H^{N,N+\frac12}}^2 + \|(\overline{U^1_{0,1}},\overline{U^3_{0,1}},\overline{\Theta_{0,1}})\|_{H^{N,N+\frac12}}^2,\label{pre5}  \\
	F_{0,1}(t)=&\nu\left(\|\na (U^2_{0,1},\tilde{U}^3_{0,1},V^2_{0,1},\tilde{V}^3_{0,1},\Lambda_{0,1})\|_{H^{N,N+\frac12}}^2 + \|\na (\overline{U^1_{0,1}},\overline{U^3_{0,1}},\overline{\Theta_{0,1}})\|_{H^{N,N+\frac12}}^2\right),\label{pre6}\\
	 E_{0,2}(t)=&\|(U^2_{0,2},\tilde{U}^3_{0,2}, V^2_{0,2},\tilde{V}^3_{0,2},\Lambda_{0,2} )\|_{H^{s,s+\frac12}}^2 + \|(\overline{U^1_{0,2}},\overline{U^3_{0,2}},\overline{\Theta_{0,2}})\|_{H^{s,s+\frac12}}^2, \label{pre7} \\
	F_{0,2}(t)=&\nu\left(\|\na (U^2_{0,2},\tilde{U}^3_{0,2},V^2_{0,2},\tilde{V}^3_{0,2},\Lambda_{0,2})\|_{H^{s,s+\frac12}}^2 + \|\na (\overline{U^1_{0,2}},\overline{U^3_{0,2}},\overline{\Theta_{0,2}})\|_{H^{s,s+\frac12}}^2\right).\label{pre8}
\end{align}

\subsection{ Energy estimates on $U_{0,1} $ and $ \Theta_{0,1}$   }\label{sub4.3}
\qquad In this section, we aim to prove that under the bootstrap hypotheses of Proposition \ref{boots}, the estimates on $U_{0,1} $, and $ \Theta_{0,1}$  hold (i.e., \eqref{u01}), with 100 replaced by 50 on the right-hand side.

\begin{pro}\label{E01}
 Suppose that the assumptions in Proposition \ref{boots} are true, then it holds that for any $t\in [0,T]$,
	\begin{align*}
		\sup_{t>0}E_{0,1}(t)+\int_0^t F_{0,1}(\tau){\rm d}\tau\leq& \epsilon_0^2 + 100C\delta^2\nu^{2c}\left(\delta^2\nu^{2c-\frac43} + \delta^4\nu^{4c-\frac{10}{3}}\right) \\
		&+ 100q^{-\frac23}\delta^2\nu^{2c}\left[\nu^{-\frac{4}{3}}\epsilon_0^2 + C\left(\delta^2\nu^{2c-\frac{4}{3}} + \delta^4\nu^{4c-\frac{10}{3}} + \delta^6\nu^{6c-\frac{16}{3}}\right)\right],
	\end{align*}
 where $ E_{0,1}(t)$ and $ F_{0,1}(t)$ are defined by \eqref{pre5}
and \eqref{pre6}, respectively. 
\end{pro}
 
 \pf  Proposition \ref{E01} is proved by the following Lemmas \ref{refine u023}, \ref{double0} and \ref{lambda}. \hfill$\square$

We first established the low regularity $L^\infty$ dispersion estimates of $U_0^2$and $\tilde{U}_0^3$, which provide improved index estimates.
\begin{lem}\label{lowDisper}
    Suppose that the assumptions in Proposition \ref{boots} are true, then it holds that for any $t\in [0,T]$,
    \begin{align*}
        \sum_{\gamma\in\mathbb{N}^2_0,|\gamma|\leq {s-3}} \|\nabla^{\gamma}_{y,z}\langle\p_z\rangle^{\frac12}(U^2_0,\tilde{U}^3_0,V_0^2,\tilde{V}_0^3)(t)\|_{L^{\infty}_{y,z}} \lesssim& (qt)^{-\frac13}e^{-\nu t}\epsilon_0 + q^{-\frac13}\nu^{-\frac23}\delta^2\nu^{2c},\\
        \sum_{\gamma\in\mathbb{N}^2_0,|\gamma|\leq {N-4}} \|\nabla^{\gamma}_{y,z}\langle\p_z\rangle^{\frac12}(U^2_{0,1},\tilde{U}^3_{0,1},V_{0,1}^2,\tilde{V}_{0,1}^3)(t)\|_{L^{\infty}_{y,z}} \lesssim& (qt)^{-\frac13}e^{-\nu t}\epsilon_0 + q^{-\frac13}\nu^{-\frac23}\delta^2\nu^{2c},\\
		\sum_{\gamma\in\mathbb{N}^2_0,|\gamma|\leq {s-3}} \|\nabla^{\gamma}_{y,z}\langle\p_z\rangle^{\frac12}(U^2_{0,2},\tilde{U}^3_{0,2},V_{0,2}^2,\tilde{V}_{0,2}^3)(t)\|_{L^{\infty}_{y,z}} \lesssim& q^{-\frac13}\nu^{-\frac23}\delta^2\nu^{2c},
    \end{align*} 
here $q$ is the same as in Proposition \ref{disperprop}.
\end{lem}
\pf Let $\Gamma_0^2:= U_0^2+iV_0^2,\ \Gamma_0^3:= \tilde{U}_0^3+i\tilde{V}_0^3$, and donate $g_1,\ g_2,\ g_1',\ g_2'$ as the symbol of $G_1,\ G_2,\ G_1',\ G_2'$, respectively. According to the Duhamel's principle, 
\begin{align*}
    \p_t \Gamma_0^{i} = \mathcal{L}\Gamma_0^i + \mathcal{N}(U,\Gamma_0^i) \Rightarrow \Gamma_0^i(t)=e^{t\mathcal{L}}\Gamma^i_0(0) + \int_0^t e^{(t-\tau)\mathcal{L}}\mathcal{N}_{0+\neq}(\Gamma_0^i)(\tau){\rm d}\tau,\ i\in\{2,3\},
\end{align*}
where 
\begin{align*} 
    \mathcal{N}_{0+\neq}(\Gamma_{0}^2)= \mathcal{N}_{0+\neq}(U_{0}^2) + i\mathcal{N}_{0+\neq}(V_{0}^2),\ \mathcal{N}_{0+\neq}(\Gamma_{0}^3)= \mathcal{N}_{0+\neq}(\tilde{U}_{0}^3) + i\mathcal{N}_{0+\neq}(\tilde{V}_{0}^2).
\end{align*}
It is easy to see that $\int_{\mathbb{T}}\Gamma_0^i(t,y,z){\rm d}z = \int_{\mathbb{T}}\mathcal{N}_{0+\neq}(\Gamma_0^i)(t,y,z){\rm d}z =0$ for $t\geq 0,\ i\in\{2,3\}$. Notice that $|g_1|+|g_2|+|g_1'|+|g_2'|\lesssim 1$, then we can apply Corollary \ref{dispercor} and obtain  
\begin{align*}
    &\|\nabla^{\gamma}_{y,z}\Gamma_0^{2,3}(t)\|_{L^\infty_{y,z}} \lesssim e^{-\nu t}(qt)^{-\frac13}\|(U^{in}_0,\Theta_0^{in})\|_{W^{s,1}(\R\times\mathbb{T})} +  \int_{0}^te^{-\nu(t-\tau)}(q(t-\tau))^{-\frac13}\|\mathcal{N}_{0+\neq}(\Gamma_0^{2,3})(\tau)\|_{W^{s,1}(\R\times\mathbb{T})}{\rm d}\tau.
\end{align*}
Notice that $ \|t^{-\frac13}e^{-\nu t}\|_{L^2_t}\lesssim \nu^{-\frac16}$. By bootstrap assumptions,    for $i\in\{2,3\}$ we have   
\begin{align*}
    \|\mathcal{N}_{0+\neq}(\Gamma_0^i)\|_{L^2_tW^{s,1}(\R\times\mathbb{T})} 
	\lesssim& \sum_{F\in\{U^1,U^2,U^3\}}\left(\|F_0\|_{L^{\infty}_tH^{s}} + \|\mathcal{A}F_{\neq}\|_{L^{\infty}_tH^s}\right)  \\
    &\times \left(\sum_{F\in\{U^1,U^2,U^3,\Theta\}}\left(\|\nabla F_0\|_{L^2_tH^{s}} + \|\nabla_L\mathcal{A}F_{\neq}\|_{L^2_tH^s}\right) + \|\p_y(\overline{\Theta_0},\overline{U^1_0},\overline{U^3_0})\|_{L^2_tH^{s}}\right) \\
    \lesssim& \nu^{-\frac12}\delta^2\nu^{2c},
\end{align*}
thus, one has
\[
    \int_{0}^te^{-\nu(t-\tau)}(q(t-\tau))^{-\frac13}\|\mathcal{N}_{0+\neq}(\Gamma_0^i)(\tau)\|_{W^{s,1}(\R\times\mathbb{T})}{\rm d}\tau\lesssim q^{-\frac13}\nu^{-\frac23}\delta^2\nu^{2c}.
\]
The estimates of $U^2_{0,1},\tilde{U}^3_{0,1}$ and $U^2_{0,2},\tilde{U}^3_{0,2}$ are similar to that of  $U^2_{0},\tilde{U}^3_{0}$, so we omit the proof. Note that we can propagation $\p_z^{\frac12}$ in our estimats, we finish the proof. \hfill$\square$

For the main system \eqref{1stQuasi1}-\eqref{1stQuasi3}, we establish the nonlinear dispersive estimate of $U_{0,1}^2$ and $\tilde{U}_{0,1}^3$. Let $\Gamma_{0,1}^2:= U_{0,1}^2+iV_{0,1}^2,\ \Gamma_{0,1}^3:= \tilde{U}_{0,1}^3+i\tilde{V}_{0,1}^3$, thus 
\begin{align*}
    \p_t \Gamma_{0,1}^2 - \nu\Delta_{L,0}\Gamma_{0,1}^2 + i\mathcal{R}\Gamma_{0,1}^2 =  \mathcal{N}_{0+\neq}(\Gamma_{0,1}^2),\ \p_t \Gamma_{0,1}^3 - \nu\Delta_{L,0}\Gamma_{0,1}^3 + i\mathcal{R}\Gamma_{0,1}^3 =  \mathcal{N}_{0+\neq}(\Gamma_{0,1}^3),
\end{align*}
where 
\begin{align*}   
    \mathcal{N}_{0+\neq}(\Gamma_{0,1}^2)= \mathcal{N}_{0+\neq}(U_{0,1}^2) + i\mathcal{N}_{0+\neq}(V_{0,1}^2),\ \mathcal{N}_{0+\neq}(\Gamma_{0,1}^3)= \mathcal{N}_{0+\neq}(\tilde{U}_{0,1}^3) + i\mathcal{N}_{0+\neq}(\tilde{V}_{0,1}^2)
\end{align*}
are the nonlinear terms. We then decompose 
\[
    U_{0,1}^2=U_{0,1}^{2,in}+U_{0,1}^{2,nl},\ \tilde{U}_{0,1}^3=\tilde{U}_{0,1}^{3,in}+\tilde{U}_{0,1}^{3,nl},\ V_{0,1}^2=V_{0,1}^{2,in}+V_{0,1}^{2,nl},\ \tilde{V}_{0,1}^3=\tilde{V}_{0,1}^{3,in}+\tilde{V}_{0,1}^{3,nl}.
\]
\begin{lem}\label{refine u023}  Suppose that the assumptions in Proposition \ref{boots} are true, then it holds that for any  $t\in [0,T]$,
    \begin{align*}
        \|\langle\p_z\rangle^{\frac12}(U_{0,1}^{2,in},\tilde{U}_{0,1}^{3,in})(t)\|_{W^{N,\infty}}+\|\langle \p_z\rangle^{\frac12}(V_{0,1}^{2,in},V_{0,1}^{3,in})(t)\|_{W^{N,\infty}} \lesssim& (qt)^{-\frac13}e^{-\nu t}\epsilon_0 ,\\
        \|(U_{0,1}^{2,nl},\tilde{U}_{0,1}^{3,nl})(t)\|_{H^{N,N+\frac12}} + \|(V_{0,1}^{2,nl},V_{0,1}^{3,nl})(t)\|_{H^{N,N+\frac12}} \lesssim& \delta\nu^c(\delta\nu^{c-\frac23 }) + q^{-\frac13}\delta\nu^c(\nu^{-\frac{2}{3}}\epsilon_0 + \delta^2\nu^{2c-\frac53}).
    \end{align*}
\end{lem}
\pf To give the  nonlinear dispersive estimate of $U_{0,1}^{2}$ and $U_{0,1}^{3}$, we only estimate $\Gamma_{0,1}^2$ and $\Gamma_{0,1}^3$,  since 
\[  
    U_{0,1}^2=\mathfrak{Re}\Gamma_{0,1}^2,\ \tilde{U}_{0,1}^3=\mathfrak{Re}\Gamma_{0,1}^3,\ V_{0,1}^2=\mathfrak{Im}\Gamma_{0,1}^2,\ \tilde{V}_{0,1}^3=\mathfrak{Im}\Gamma_{0,1}^3.
\]
Thanks to the Duhamel's formula, we have
\[ 
    \Gamma_{0,1}^{k,in}(t)=e^{t\mathcal{L}}\Gamma_{0,1}^k(0),\ \Gamma_{0,1}^{k,nl}(t)=\int_0^t e^{(t-\tau)\mathcal{L}}\mathcal{N}_{0+\neq}(\Gamma_{0,1}^k)(\tau){\rm d}\tau,\ k\in\{2,3\},
\]
with $\mathcal{L}$ being the same as in \eqref{sol1}. The estimate of $\Gamma_{0,1}^{k,in}$  follows directly from the Corollary \ref{dispercor}. 
For estimation of $\Gamma_{0,1}^{k,nl}$, we mainly divide it into the contribution from single zero frequency and   double zero frequency. Specifically.  Using Lemma \ref{lowDisper}, we derive the nonlinear dispersion estimates which is more precise than   one obtains by the direct energy method. Recall the system \eqref{1stQuasi1} and \eqref{1stQuasi2},  we calculate the term  $\Gamma_{0,1}^{k,nl}$ which is is equivalent to  requiring  the estimation of the following terms.
For the estimates of simple zero contributions, one has 
\begin{align*}
    &\|\tilde{U}_{0,1}\cdot\na \Gamma_{0,1}^2\|_{H^{N,N+\frac12}} + {\|\tilde{U}_{0,1}\cdot\na \tilde{U}_{0,1}^3\|_{H^{N,N+\frac12}} + \|\p_z^{-1}\p_y(\tilde{U}_{0,1}\cdot\na V_{0,1}^2)\|_{H^{N,N+\frac12}}} \\
	&\quad\leq \|\tilde{U}_{0,1}\cdot\na \Gamma_{0,1}^2\|_{H^{N,N+\frac12}} + \|\tilde{U}_{0,1}\cdot\na \tilde{U}_{0,1}^3\|_{H^{N,N+\frac12}} + \|\p_z^{-1}(\p_y\tilde{U}_{0,1}\cdot\na V_{0,1}^2)\|_{H^{N,N+\frac12}}\\
	&\qquad + \|\p_z^{-1}(\tilde{U}_{0,1}\cdot\na \p_z\tilde{V}_{0,1}^3)\|_{H^{N,N+\frac12}}\\&\quad
	\lesssim \|(U_{0,1}^2,\tilde{U}_{0,1}^3,V_{0,1}^2,\tilde{V}_{0,1}^3)\|_{W^{2,\infty}}\|\na (U_{0,1}^2,\tilde{U}_{0,1}^3,V_{0,1}^2,\tilde{V}_{0,1}^3)\|_{H^{N,N+\frac12}} \\&\quad
    \lesssim   q^{-\frac13}\left(t^{-\frac13}e^{-\nu t}\epsilon_0 + \delta^2\nu^{2c-\frac23}\right)\|\na (U_{0,1}^2,\tilde{U}_{0,1}^3,V_{0,1}^2,\tilde{V}_{0,1}^3)\|_{H^{N,N+\frac12}},
\end{align*}
and
\begin{align*}
    &\|\na_{y,z}\Delta_{L,0}^{-1}(\p_yU^2_{0,1}\p_yU^2_{0,1}+\p_y\tilde{U}^3_{0,1}\p_zU^2_{0,1})\|_{H^{N,N+\frac12}}\\
    &\quad\lesssim \|\p_yU^2_{0,1}\p_yU^2_{0,1}\|_{H^{N-1,N-\frac12}}+\|\p_y\tilde{U}^3_{0,1}\p_zU^2_{0,1}\|_{H^{N-1,N-\frac12}}\\&\quad \lesssim \|(U^2_{0,1}, \tilde{U}_{0,1}^3)\|_{W^{2,\infty}}(\| U^2_{0,1}\|_{H^{N,N+\frac12}} + \|\tilde{U}^3_{0,1}\|_{H^{N,N+\frac12}}) \\
   &\quad \lesssim q^{-\frac13}\left(t^{-\frac13}e^{-\nu t}\epsilon_0 + \delta^2\nu^{2c-\frac23}\right)\|\na (U_{0,1}^2,\tilde{U}_{0,1}^3,V_{0,1}^2,\tilde{V}_{0,1}^3)\|_{H^{N,N+\frac12}}.
\end{align*}
Setting,
 $I\in\{\tilde{U}_0\cdot\na \Gamma_0^k,\ \na_{y,z}\Delta_{L,0}^{-1}(\p_yU^2_{0,1}\p_yU^2_{0,1}+\p_y\tilde{U}^3_{0,1}\p_zU^2_{0,1},\p_z^{-1}\p_y(\tilde{U}_{0,1}\cdot\na V_{0,1}^2)\}$, then
\begin{align*}
    \int_0^{t} \|e^{(t-\tau)\mathcal{L}}I(\tau)\|_{H^{N,N+\frac12}}{\rm d}\tau \lesssim& \int_0^{t} e^{-\nu(t-\tau)}q^{-\frac13}\left(\tau^{-\frac13}e^{-\nu \tau}\epsilon_0 + \delta^2\nu^{2c-\frac23}\right)\|\na (U_{0,1}^2,\tilde{U}_{0,1}^3,V_{0,1}^2,\tilde{V}_{0,1}^3)\|_{H^{N,N+\frac12}}{\rm d}\tau\\
    \lesssim&   q^{-\frac13}\cdot \delta\nu^c\nu^{-\frac12}(\nu^{-\frac16}\epsilon_0+ \delta^2\nu^{-\frac23}\nu^{2c}\nu^{-\frac12})=q^{-\frac13}\delta\nu^c\left(\nu^{-\frac{2}{3}}\epsilon_0 + \delta^2\nu^{2c-\frac53}\right).
\end{align*}
For the nonzero contributions $\mathcal{N}_{\neq}(U_{0,1}^2,\tilde{U}_{0,1}^3,V_{0,1}^2,\tilde{V}_{0,1}^3)$, we have for $F\in\{U^1,U^2,U^3,\Theta\}$,
\begin{align*}
    &\|\mathcal{N}_{\neq}(U_{0,1}^2,\tilde{U}_{0,1}^3,V_{0,1}^2,\tilde{V}_{0,1}^3)\|_{H^{N,N+\frac12}}\\
	&\quad\lesssim \|U_{\neq,1}\|_{W^{1,\infty}}\|\na_L F_{\neq,1}\|_{H^{N,N+\frac12}} + \| U_{\neq,1}\|_{H^{N,N+\frac12}}\|\na_L F_{\neq,1}\|_{L^{\infty}}\\ &\quad\lesssim \|\mathcal{A}U_{\neq,1}\|_{H^{N}}\|\na_L \mathcal{A}F_{\neq,1}\|_{H^{N}},
\end{align*}
and thus we obtain
\begin{align*}
    \int_0^{t} \|e^{(t-\tau)\mathcal{L}}\mathcal{N}_{\neq}(U_{0,1}^2,\tilde{U}_{0,1}^3,V_{0,1}^2,\tilde{V}_{0,1}^3)(\tau)\|_{H^{N,N+\frac12}}{\rm d}\tau \lesssim \int_0^{t} \|\mathcal{A}U_{\neq,1}\|_{H^{N}}\|\na_L \mathcal{A}F_{\neq,1}\|_{H^{N}}{\rm d}\tau \lesssim \delta\nu^c\cdot \delta\nu^{c-\frac23},
\end{align*}
Thus we finish the proof of Lemma \ref{refine u023}.\hfill$\square$

 We then establish the nonlinear estimates of $\overline{U^1_{0,1}},\overline{U^3_{0,1}}$ and $\overline{\Theta_{0,1}}$. 
\begin{lem}\label{double0}
	Suppose that the assumptions in Proposition \ref{boots} are true, 
    and  in addition, $\overline{U^3_0}(t=0)= \overline{\Theta_0}(t=0)=0$,  then it holds that for any  $t\in [0,T]$,
    \begin{align*}
        &\|\overline{U^1_{0,1}}\|_{L^{\infty}_t H^N}^2 + \nu\|\p_y \overline{U^1_{0,1}}\|_{L^2_t H^N}^2 \lesssim \epsilon_0^2+\delta^2\nu^{2c}(\delta^2\nu^{2c-\frac43} + \delta^4\nu^{4c-\frac{10}{3}}) + q^{-\frac23}\delta^2\nu^{2c}(\delta^2\nu^{2c-\frac{4}{3}} +\delta^4 \nu^{4c-\frac{10}{3}} + \delta^6\nu^{6c-\frac{16}{3}}),\\
		&\|(\overline{U^3_{0,1}},\overline{\Theta_{0,1}})\|_{L^{\infty}_t H^N}^2 + \nu\|\p_y (\overline{U^3_{0,1}},\overline{\Theta_{0,1}})\|_{L^2_t H^N}^2 \lesssim \delta^2\nu^{2c}(\delta^2\nu^{2c-\frac43} + q^{-\frac23}\delta^4\nu^{4c-\frac{10}{3}}).
    \end{align*} 
\end{lem}

\pf The energy estimates tell us that 
\begin{align*}
    &\|\overline{U^1_{0,1}}\|_{H^{N}}^2 + \nu\|\p_y \overline{U^1_{0,1}}\|_{L^2_t H^{N}}^2\\&\quad \lesssim \|\overline{U^1_{0,1}}(0)\|_{H^{N}}^2 + \nu^{-1}\int_0^{t} \|\overline{U_{0,1}^2\tilde{U}^1_{0,1}}\|_{H^{N}}^2  {\rm d}\tau+  \nu^{-1}\int_0^{t} \|(\overline{U_{\neq,1}^2U^1_{\neq,1}})_0\|_{H^{N}}^2 {\rm d}\tau,
\end{align*}
and
\begin{align*}
	&\|(\overline{U^3_{0,1}},\overline{\Theta_{0,1}})\|_{H^{N}}^2 + \nu\|\p_y (\overline{U^3_{0,1}},\overline{\Theta_{0,1}})\|_{L^2_t H^{N}}^2 \\
   & \quad\lesssim \|(\overline{U^3_{0,1}},\overline{\Theta_{0,1}})(0)\|_{H^{N}}^2 + \nu^{-1}\int_0^{t} \|\overline{U_{0,1}^2\tilde{U}^3_{0,1}}\|_{H^{N}}^2 {\rm d}\tau\\&\qquad+  \nu^{-1}\int_0^{t} (\|(\overline{U_{\neq,1}^2U^3_{\neq,1}})_0\|_{H^{N}}^2 + \|(\overline{U_{\neq,1}^2\Theta_{\neq,1}})_0\|_{H^{N}}^2) {\rm d}\tau.
\end{align*}

On the one hand, the zero modes contributions can be expressed as
\begin{align*}
    \nu^{-1}\int_0^{t} \|\overline{U_{0,1}^2\tilde{U}_{0,1}^1}\|_{H^{N}}^2 {\rm d}\tau \lesssim& \nu^{-1}\int_0^{t} \left(\|\langle\p_z\rangle^{\frac12}U_{0,1}^{2,in}\|_{W^{N,\infty}}^2 + \|U_{0,1}^{2,nl}\|_{H^{N,N+\frac12}}^2\right)\|\tilde{U}_{0,1}^1\|_{H^{N,N+\frac12}}^2 {\rm d}\tau \\
    \lesssim&   q^{-\frac23}\delta^2\nu^{2c}\nu^{-1}\int_0^{t} \tau^{-\frac23}e^{-2\nu \tau}\|\tilde{U}_{0,1}^1\|_{H^{N,N+\frac12}}^2{\rm d}\tau \\
    &+\nu^{-1}\int_0^{t} \left(\delta\nu^c(\delta\nu^{c-\frac23 }) + q^{-\frac13} \delta\nu^c \left(\delta\nu^{c-\frac{2}{3}} + \delta^2\nu^{2c-\frac53}\right)\right)^2\|\tilde{U}_{0,1}^1\|_{H^{N,N+\frac12}}^2 {\rm d}\tau \\ 
    \lesssim&   \delta^2\nu^{2c}\delta^4\nu^{4c-\frac{10}{3}} + q^{-\frac23}\delta^2\nu^{2c}(\delta^2\nu^{2c-\frac{4}{3}} + \delta^4\nu^{4c-\frac{10}{3}} + \delta^6\nu^{6c-\frac{16}{3}}),\\
	\nu^{-1}\int_0^{t} \|\overline{U_{0,1}^2\tilde{U}^3_{0,1}}\|_{H^{N}}^2 {\rm d}\tau\lesssim& \nu^{-1}\int_0^{t} \|\tilde{U}_{0,1}^{2,3}\|_{L^\infty}^2\|\tilde{U}_{0,1}^{2,3}\|_{H^{N,N+\frac12}}^2 {\rm d}\tau\\
	\lesssim& \nu^{-1}\int_0^{t} q^{-\frac23}\delta^2\nu^{2c}\left(\tau^{-\frac13}e^{-\nu\tau}+\delta\nu^{c-\frac23}\right)^2\|\tilde{U}_{0,1}^{2,3}\|_{H^{N,N+\frac12}}^2 {\rm d}\tau\\
	\lesssim& q^{-\frac23}\delta^2\nu^{2c}(\delta^2\nu^{2c-\frac43} + \delta^4\nu^{4c-\frac{10}{3}}).
\end{align*}
On the other hand, the nonzero modes contributions can be expressed as $\mathcal{N}_{\neq}(\overline{U^1_{0,1}},\overline{U^3_{0,1}},\overline{\Theta_{0,1}})$, and one has
\[ 
    \nu^{-1}\int_0^{t}\mathcal{N}_{\neq}(\overline{U^1_{0,1}},\overline{U^3_{0,1}},\overline{\Theta_{0,1}}){\rm d}\tau\lesssim \nu^{-1}\int_0^{t}\|\mathcal{A}F_{\neq,1}\|_{H^N}^4{\rm d}\tau \lesssim \delta^2\nu^{2c}(\delta^2\nu^{2c-\frac43}).
\]
Thus we finish the proof of Lemma \ref{double0}.\hfill $\square$

 In the following, we then establish the nonlinear estimates of $\Lambda_{0,1}$. 
\begin{lem}\label{lambda}
    Suppose that the assumptions in Proposition \ref{boots} are true, then it holds that for any  $t\in [0,T]$,
    \begin{align*}
        &\|(\Lambda_{0,1},\tilde{U}^1_{0,1},\tilde{\Theta}_{0,1})\|_{L^\infty_t H^{N,N+\frac12}}^2 + \nu\|\na(\Lambda_{0,1},\tilde{U}^1_{0,1},\tilde{\Theta}_{0,1})\|_{L^2_t H^{N,N+\frac12}}^2 \\
		& \quad\lesssim{\epsilon_0^2}+\delta^2\nu^{2c}(\delta^2\nu^{2c-\frac43} + \delta^4\nu^{4c-\frac{10}{3}}) + q^{-\frac23}\delta^2\nu^{2c}(\delta^2\nu^{2c-\frac{4}{3}} + \delta^4\nu^{4c-\frac{10}{3}} + \delta^6\nu^{6c-\frac{16}{3}}).
    \end{align*}

\end{lem}
\pf  Recall \eqref{1stQuasi1}$_{1}$, the energy estimates tell us that for $F\in\{U^1,\Theta\}$,
\begin{align*}
    \|\Lambda_{0,1}\|_{H^{N,N+\frac12}}^2 + \nu\|\na\Lambda_{0,1}\|_{L^2_t H^{N,N+\frac12}}^2 \lesssim& \|\Lambda_{0,1}(t=0)\|_{H^{N,N+\frac12}}^2 + \nu^{-1}\int_0^{t} \sum_{F\in\{U^1,\Theta\}}\|\tilde{U}_{0,1}^{2,3}F_{0,1}\|_{H^{N,N+\frac12}}^2{\rm d}\tau \\
    & + \nu^{-1}\int_0^{t} \left( \|U_{\neq,1}U_{\neq,1}^1\|_{H^{N,N+\frac12}}^2 + \|U_{\neq,1}\Theta_{\neq,1}\|_{H^{N,N+\frac12}}^2\right){\rm d}\tau.
\end{align*}
For the zero modes contributions,  the estimates  are the same as  the processing of   double-zero modes  in Lemma \ref{double0} 
\begin{align*}
    \nu^{-1}\int_0^{t} \sum_{F\in\{U^1,\Theta\}}\|\tilde{U}_{0,1}^{2,3}F_{0,1}\|_{H^{N,N+\frac12}}^2{\rm d}\tau \lesssim \delta^2\nu^{2c}( \delta^4\nu^{4c-\frac{10}{3}}) + q^{-\frac23}\delta^2\nu^{2c}(\delta^2\nu^{2c-\frac{4}{3}} + \delta^4\nu^{4c-\frac{10}{3}} +\delta^6 \nu^{6c-\frac{16}{3}}).
\end{align*}
We donate the nonzero modes contributions as $\mathcal{N}_{\neq}(\Lambda_{0,1})$. By \eqref{qkh1}, we have 
\[ 
    \nu^{-1}\int_0^{t} \mathcal{N}_{\neq}(\Lambda_{0,1}){\rm d}\tau \lesssim \nu^{-1}\int_0^{t}\|\mathcal{A}F_{\neq,1}\|_{H^N}^4{\rm d}\tau \lesssim \delta^2\nu^{2c}(\delta^2\nu^{2c-\frac43}).
\]
Thus, we obtain the estimation of $\Lambda_{0,1}$. By using \eqref{000} and Lemma \ref{refine u023}, we can return to the estimates of $\tilde{U}^1_{0,1}$ and $\tilde{\Theta}_{0,1}$.
  We finish the proof of  Lemma \ref{lambda}. \hfill$\square$

\vspace{0.3cm}

\subsection{  Energy estimates on $Q_{\neq,1}, K_{\neq,1}$ and $H_{\neq,1}$}\label{sub4.4}
\qquad  In this section, we  aim to   prove that under the bootstrap hypotheses of Proposition \ref{boots}, the estimates on $U_{\neq,1} $, and $ \Theta_{\neq,1}$  hold (i.e., \eqref{qkh1}), with 100 replaced by 50 on the right-hand side. To this end, we present some inequalities, which are frequently employed in non-zero modes estimates. 

\begin{lem}\label{lem4.7}
	Under the assumptions in Proposition \ref{boots},   there holds that
	\begin{align}
		\label{tool1}
		\|\mathcal{A}|\na_L|^{\frac12}(Q_{\neq,1},K_{\neq,1},H_{\neq,1})\|_{L^2_tH^N}^2\lesssim& \nu^{-\frac23}\int_0^t F_{\neq,1}(\tau){\rm d}\tau,\\
		\|\mathcal{B}|\na_L|^{\frac12}(Q_{\neq,2},K_{\neq,2},H_{\neq,2})\|_{L^2_tH^s}^2\lesssim& \nu^{-\frac23}\int_0^t F_{\neq,2}(\tau){\rm d}\tau,\label{tool1.1}\\
		\label{tool5}
		\|\mathcal{A}|\na_{L,h}|^{-\frac12}(Q_{\neq,1},K_{\neq,1},H_{\neq,1})\|_{L^2_tH^N}^2\lesssim& \nu^{-\frac{\kappa}{3(1+\kappa)}}\int_0^t F_{\neq,1}(\tau){\rm d}\tau,\\
		\|\mathcal{B}|\na_{L,h}|^{-\frac12}(Q_{\neq,2},K_{\neq,2},H_{\neq,2})\|_{L^2_tH^s}^2\lesssim& \nu^{-\frac{\kappa}{3(1+\kappa)}}\int_0^t F_{\neq,2}(\tau){\rm d}\tau, \label{tool5.1}\\
		\label{tool7}
		\|p^{\frac14}\widehat{m(Q_{\neq,i},K_{\neq, i},H_{\neq,i})}\|_{L^\infty_tL^1}^2\lesssim&\nu^{-\frac13}\sup_{t>0}E_{\neq,i}(t), \\
		\label{tool8}
		\||k,l|^{\frac32}p_h^{\frac14}\widehat{K_{\neq,i}}\|_{L^2_tL^1}^2\lesssim&\nu^{-1}\int_0^t F_{\neq,i}(\tau){\rm d}\tau,\\
		\label{tool2}
		\|\widehat{\na_{x,z} (U_{\neq,i},\Theta_{\neq,i})}\|_{L^\infty_tL^1}^2\lesssim& \sup_{t>0}E_{\neq,i}(t), \\
		\label{tool3}
		\|\widehat{\na_L (U_{\neq,i},\Theta_{\neq,i})}\|_{L^\infty_tL^1}^2 \lesssim& \nu^{-\frac23}\sup_{t>0}E_{\neq,i}(t), \\
		\label{tool4}
		\|\na_L U_{\neq,i}^2\|_{L^\infty_tL^\infty}^2\lesssim& \sup_{t>0}E_{\neq,i}(t), \\
		\label{tool6}
		\|\widehat{U^2_{\neq,i}}\|_{L^2_tL^1}^2\lesssim& \sup_{t>0}E_{\neq,i}(t), 
	\end{align}
where $i=1,2$ and  $\kappa>0$ is  same as in \eqref{66}.
\end{lem}

\pf For any $3<r\leq N$, by Corollary \ref{cor4.1}, it holds 

\begin{align*}
\|\mathcal{A}|\na_L|^{\frac12}(Q_{\neq,1},K_{\neq,1},H_{\neq,1})\|_{L^2_tH^r}	&\lesssim \|\mathcal{A}(Q_{\neq,1},K_{\neq,1},H_{\neq,1})\|_{L^2_tH^r}^{\frac12}\|\mathcal{A}\na_L (Q_{\neq,1},K_{\neq,1},H_{\neq,1})\|_{L^2_tH^r}^{\frac12}\\& \lesssim \nu^{-\frac13}(\int_0^t F_{\neq,1}{\rm d}\tau)^{\frac12},\\
\|\mathcal{A}|\na_{L,h}|^{-\frac12}(Q_{\neq,1},K_{\neq,1},H_{\neq,1})\|_{L^2_tH^r}	&\lesssim \|\mathcal{A}(Q_{\neq,1},K_{\neq,1},H_{\neq,1})\|_{L^2_tH^r}^{\frac{\kappa}{1+\kappa}}\left\|\mathcal{A}\sqrt{-\frac{\dot{M_7}}{M_7}} (Q_{\neq,1},K_{\neq,1},H_{\neq,1})\right\|_{L^2_tH^r}^{\frac{1}{1+\kappa}}\\
&\lesssim \|\mathcal{A}(Q_{\neq,1},K_{\neq,1},H_{\neq,1})\|_{L^2_tH^r}^{\frac{\kappa}{1+\kappa}}\left\|\mathcal{A}\sqrt{-\frac{\dot{M}}{M}} (Q_{\neq,1},K_{\neq,1},H_{\neq,1})\right\|_{L^2_tH^r}^{\frac{1}{1+\kappa}}\\
&\lesssim \nu^{-\frac{\kappa}{6(1+\kappa)}}(\int_0^t F_{\neq,1}{\rm d}\tau)^{\frac12},
	\end{align*}
which gives \eqref{tool1} and \eqref{tool5.1}. The  proof  of \eqref{tool1.1} and \eqref{tool5.1} is similar. Direct calculation gives
	\begin{align*}
	\|\widehat{\na_{x,z} (U_{\neq},\Theta_{\neq})}\|_{L^\infty_tL^1}&\lesssim \|\langle k,\eta,l\rangle^r\widehat{\na_{x,z} (U_{\neq},\Theta_{\neq})}\|_{L^\infty_tL^2}\|\langle k,\eta,l\rangle^{-r}\|_{L^2}\\
	&\lesssim \|\mathcal{A}(Q_{\neq},K_{\neq},H_{\neq})\|_{L^\infty_tH^{r+1}},\\
	\|\widehat{\na_L (U_{\neq},\Theta_{\neq})}\|_{L^\infty_tL^1}&\lesssim \|e^{\lambda\nu^{\frac13}t}\widehat{\na (U_{\neq},\Theta_{\neq})}\|_{L^\infty_tL^1}\|\langle t\rangle e^{-\lambda\nu^{\frac13}t}\|_{L^\infty_t} \\
	&\lesssim \nu^{-\frac13}\|\mathcal{A}(Q_{\neq},K_{\neq},H_{\neq})\|_{L^\infty_tH^r}.
\end{align*}
Using the  relationships between variables, such as \eqref{3x.1}, \eqref{3x.2}, \eqref{ll6} and \eqref{ll7}, we have
\begin{align*}
	&\|\na_L U_{\neq,1}^2\|_{L^\infty_tL^\infty}\lesssim \|\mathcal{A}(Q_{\neq,1},K_{\neq,1})\|_{L^\infty_tH^r},\\
	&\|\widehat{U^2_{\neq}}\|_{L^2_tL^1}\lesssim \|\mathcal{A}(Q_{\neq},K_{\neq})\|_{L^\infty_tH^r}\|\langle t\rangle^{-1}\|_{L^2_t}\lesssim \|\mathcal{A}(Q_{\neq},K_{\neq})\|_{L^\infty_tH^r},\\
	&\|p^{\frac14}\widehat{mH_{\neq}}\|_{L^\infty_tL^1}\lesssim \||\na|^{\frac12}\widehat{\mathcal{A}H_{\neq}}\|_{L^\infty_tL^1}\|t^{\frac12}e^{-\lambda\nu^{\frac13}t}\|_{L^\infty_t}\lesssim\nu^{-\frac16}\|\mathcal{A} H_{\neq}\|_{L^\infty_tH^r},\\
	&\||k,l|^{\frac32}p_h^{\frac14}\widehat{K_{\neq}}\|_{L^2_tL^1}\lesssim \||k,l|p^{\frac12}\widehat{mK_{\neq}}\|_{L^2_tL^1}\lesssim \|\mathcal{A} \na_L K_{\neq,1}\|_{L^2_tH^r}.
\end{align*}  Thus, we finish the proof of Lemma \ref{lem4.7} \hfill$\square$
 
\begin{pro}\label{propEneq1}
Suppose that the assumptions in Proposition \ref{boots} are true, then it holds that for any $t\in [0,T]$,
	\begin{align}\label{pro1}
		\sup_{t>0}E_{\neq,1}(t)+\int_0^t F_{\neq,1}(\tau){\rm d}\tau\leq \epsilon_{0}^{2}+ C\delta^2 \nu^{2c}\delta\nu^{c-\frac56-\frac{\kappa}{6(1+\kappa)}}.
	\end{align}
 where $ E_{\neq,1}(t)$ and $ F_{\neq,1}(t)$ are defined by \eqref{pre1}
and \eqref{pre2}, respectively. 
\end{pro}
 
\pf According to the linear analysis, an energy estimate yields
\begin{equation}
\begin{aligned}\label{kkk}
	&\|\mathcal{A}(Q_{\neq,1},K_{\neq,1},H_{\neq,1})\|_{H^N}^2 + \nu\|\mathcal{A}\na_L(Q_{\neq,1},K_{\neq,1},H_{\neq,1})\|_{L^2_tH^N}^2 + \left\|\mathcal{A}\sqrt{-\frac{\dot{M}}{M}}(Q_{\neq,1},K_{\neq,1},H_{\neq,1})\right\|_{L^2_tH^N}^2\\
	&\quad\lesssim \frac{2\sqrt{\beta(\beta-1)}+1}{2\sqrt{\beta(\beta-1)}-1}\|\mathcal{A}(Q_{\neq,1},K_{\neq,1},H_{\neq,1})(t=0)\|_{H^N}^2 +\frac{\nu^{\frac13}}{16}\|\mathcal{A}(Q_{\neq,1},K_{\neq,1},H_{\neq,1})\|_{L^2_tH^N}^2\\
	&\qquad+\frac{4\sqrt{\beta(\beta-1)}}{2\sqrt{\beta(\beta-1)}-1}\int_0^t\mathcal{N}_{\mathcal{A}1}{\rm d}\tau+\frac{4\sqrt{\beta(\beta-1)}}{2\sqrt{\beta(\beta-1)}-1}\int_0^tG\mathcal{N}_{\mathcal{A}2}{\rm d}\tau,
\end{aligned}
\end{equation}
where 
\begin{align*}
	\int_0^t\mathcal{N}_{\mathcal{A}1}{\rm d}\tau 
	&=\int_0^t\big|\langle\mathcal{A}Q_{\neq,1},\mathcal{A}|\na_{L}|^{\frac32}|\na_{L,h}|^{-1}(U^1_{\neq,1}\p_xU^3_{\neq,1})\rangle_{H^N}\big|{\rm d}\tau\\
	&\quad + \int_0^t\big|\langle\mathcal{A}Q_{\neq,1},\mathcal{A}\p_z|\na_{L,h}|^{-1}|\na_L|^{-\frac12}(\p_{x,z} U^{1,3}_{\neq,1}\cdot\na_{x,z} U^{1,3}_{\neq,1}+2\na_{L,h}U_{\neq,1}^2\p_y^LU_{\neq,1}^{1,2})\rangle_{H^N}\big|{\rm d}\tau\\
	&\quad + \int_0^t\big|\langle\mathcal{A}K_{\neq,1},\mathcal{A}\sqrt{\frac{\beta}{\beta-1}}|\na_L|^{\frac12}|\na_{L,h}|^{-1}\big[\p_xU_{\neq,1}\cdot\na_L U^2_{\neq,1}-\p_y^LU_{\neq,1}\cdot\na_LU^1_{\neq,1} \\
	&\quad+U_{\neq,1}^{1,3}\cdot\na_{x,z}W^3_{\neq,1}+U_{\neq,1}^{2}\p_y^L W^3_{\neq,1}\big]\rangle_{H^N}\big|{\rm d}\tau\\&\quad +\int_0^t\big|\langle\mathcal{A}H_{\neq,1},\mathcal{A}|\na_L|^{\frac12}\big[U_{\neq,1}^{1,3}\cdot\na_{x,z}\Theta_{\neq,1}+U_{\neq,1}^2\p_y^L\Theta_{\neq,1}\big]\rangle_{H^N}\big|{\rm d}\tau\\&\quad + \int_0^t\big|\langle\mathcal{A}Q_{\neq,1},\mathcal{A}|\na_L|^{\frac32}|\na_{L,h}|^{-1}\big(U_{0,1}^1\p_xU_{\neq,1}^3\big)\rangle_{H^N}\big|{\rm d}\tau\\&\quad + \int_0^t\big|\langle\mathcal{A}Q_{\neq,1},\mathcal{A}\p_z|\na_L|^{-\frac12}|\na_{L,h}|^{-1}\big[2\p_z U^{1,3}_{0,1}\cdot\na_{x,z} U^3_{\neq,1}+2\p_yU_{0,1}^1\p_xU_{\neq,1}^2\big]\rangle_{H^N}\big|{\rm d}\tau\\
	&\quad + \sqrt{\frac{\beta}{\beta-1}}\int_0^t\big|\langle\mathcal{A}K_{\neq,1},\mathcal{A}|\na_L|^{\frac12}|\na_{L,h}|^{-1}\big[(\p_xU_{\neq,1}\cdot\na U^2_{0,1}-\p_y^LU_{\neq,1}^2\p_y U_{0,1}^1-\p_y U_{0,1}^{1,3}\cdot\na_{x,z} U_{\neq,1}^1)\\
	&\quad +(U_{0,1}^1\p_x W_{\neq,1}^3+U_{0,1}^3\p_z W_{\neq,1}^3)\big]\rangle_{H^N}\big|{\rm d}\tau\\&\quad + \int_0^t\big|\langle\mathcal{A}H_{\neq,1},\mathcal{A}|\na_L|^{\frac12}\big[U_{0,1}^{1,3}\na_{x,z}\Theta_{\neq,1}+U^3_{\neq,1}\p_z\Theta_{0,1}\big]\rangle_{H^N}\big|{\rm d}\tau\\&=:\sum_{i=1}^4\left[I_i(\neq,\neq)+J_{i}(0,\neq)\right],
\end{align*}
and $\int_0^tG\mathcal{N}_{\mathcal{A}2}{\rm d}\tau$ is the associated  nonlinear cross terms,
\begin{align*}
	\int_0^tG\mathcal{N}_{\mathcal{A}2}{\rm d}\tau=&\int_0^t\langle G\mathcal{A}K_{\neq,1},\text{Nonlinear terms in the equation of }Q_{\neq,1}\rangle_{H^N}{\rm d}\tau\\
	&+\int_0^t\langle G\mathcal{A}Q_{\neq,1},\text{Nonlinear terms in the equation of }K_{\neq,1}\rangle_{H^N}{\rm d}\tau.
\end{align*} 

Now we are ready to estimate each term on the right-hand side of \eqref{kkk}. We first focus on the  interaction between two non-zero frequencies. For $I_{1}$, using  the Plancherel formula, we have 
\begin{align*}
	I_1=&\int_0^t\big|\langle\mathcal{A}Q_{\neq,1},\mathcal{A}|\na_{L}|^{\frac32}|\na_{L,h}|^{-1}(U^1_{\neq,1}\p_xU^3_{\neq,1})\rangle_{H^N}\big|{\rm d}\tau\\
	\leq&\int_0^t\sum_{\substack{k,l\\k',l'}}\iint_{\eta,\eta'}\widehat{\langle\na\rangle^N}|\mathcal{A}\widehat{Q_{\neq,1}}(k,\eta,l)|\mathcal{A}p^{\frac34}p_h^{-\frac12}(k,\eta,l)\widehat{\langle\na\rangle^N}|\widehat{U_{\neq,1}^1}(k-k',\eta-\eta',l-l')||k'\widehat{U^3_{\neq,1}}(k',\eta',l')|{\rm d}\eta{\rm d}\eta'{\rm d}\tau\\
	\lesssim& \int_0^t\sum_{\substack{k,l\\k',l'}}\iint_{\eta,\eta'}\widehat{\langle\na\rangle^{N}}|\mathcal{A}\widehat{Q_{\neq,1}}(k,\eta,l)|e^{\lambda\nu^{\frac13}\tau}\frac{|l|^{\frac32}}{\sqrt{p_h(k,\eta)}}\widehat{\langle\na\rangle^N}|\widehat{U_{\neq,1}^1}(k-k',\eta-\eta',l-l')||k'\widehat{U^3_{\neq,1}}(k',\eta',l')|\\
	&+ \nu^{-\frac16}\widehat{\langle\na\rangle^N}|\mathcal{A}\widehat{Q_{\neq,1}}(k,\eta,l)|e^{\lambda\nu^{\frac13}\tau}p^{\frac14}(k,\eta)\widehat{\langle\na\rangle^N}|\widehat{U_{\neq,1}^1}(k-k',\eta-\eta',l-l')||\widehat{\p_xU^3_{\neq,1}}(k',\eta',l')|{\rm d}\eta{\rm d}\eta'{\rm d}\tau\\
	=:&I_{11}+I_{12}.
\end{align*}
where in the second inequality  we have used the following fact 
\begin{equation}\label{multi1} mp^{\frac34}p_h^{-\frac12}\lesssim \frac{|l|^{\frac32}}{\sqrt{p_h}}+\nu^{-\frac16}p^{\frac14},
\end{equation}
due to
\begin{equation}\label{6666}\begin{cases}
	mp^{\frac34}p_h^{-\frac12}\lesssim \sup(m)p_h^{\frac14}\lesssim\min\{\nu^{-\frac16},\frac{p_h^{\frac14}}{(k^2+l^2)^{\frac14}}\}p_h^{\frac14}\lesssim p_h^{\frac12}\lesssim p^{\frac12}\langle l\rangle^{\frac12},\ \text{if }l^2\leq p_h,\\
	mp^{\frac34}p_h^{-\frac12}\lesssim\frac{|l|^{\frac32}}{\sqrt{p_h}}\lesssim |l|^{\frac32}\lesssim p^{\frac12}\langle l\rangle^{\frac12},\ \text{if }l^2\geq p_h.
\end{cases}\end{equation}
For $I_{11}$, by the fact that $\sqrt{|k'|}\leq \sqrt{|k-k'|}+\sqrt{|k|}$, we have
\begin{align*}
	I_{11}\leq&\int_0^t\sum_{\substack{k,l\\k',l'}}\iint_{\eta,\eta'}\widehat{\langle\na\rangle^{N}}|l||\widehat{\mathcal{A}Q_{\neq,1}}(k,\eta,l)|e^{\lambda\nu^{\frac13}\tau}\widehat{\langle\na\rangle^{N}}\frac{|l|^{\frac12}}{\sqrt{p_h(k,\eta)}}|\widehat{U_{\neq,1}^1}(k-k',\eta-\eta',l-l')|\\
	&\times|k'||\widehat{U_{\neq,1}^3}(k',\eta',l')|{\rm d}\eta'{\rm d}\eta{\rm d}\tau \\
	\leq& \int_0^t\sum_{\substack{k,l\\k',l'}}\iint_{\eta,\eta'}\widehat{\langle\na\rangle^{N}}|l||\widehat{\mathcal{A}Q_{\neq,1}}(k,\eta,l)|e^{\lambda\nu^{\frac13}\tau}\widehat{\langle\na\rangle^{N}}\frac{|l|^{\frac12}}{\sqrt{p_h(k,\eta)}}\sqrt{|k-k'|}|\widehat{U_{\neq,1}^1}(k-k',\eta-\eta',l-l')|\\
	&  \times\sqrt{|k'|}|\widehat{U_{\neq,1}^3}(k',\eta',l')|{\rm d}\eta'{\rm d}\eta{\rm d}\tau + \int_0^t\sum_{\substack{k,l\\k',l'}}\iint_{\eta,\eta'}\widehat{\langle\na\rangle^{N}}|l||\widehat{\mathcal{A}Q_{\neq,1}}(k,\eta,l)|e^{\lambda\nu^{\frac13}\tau}\widehat{\langle\na\rangle^{N}}\frac{|l|^{\frac12}\sqrt{|k|}}{\sqrt{p_h(k,\eta)}}\\
	&  \times|\widehat{U_{\neq,1}^1}(k-k',\eta-\eta',l-l')|\sqrt{|k'|}|\widehat{U_{\neq,1}^3}(k',\eta',l')|{\rm d}\eta'{\rm d}\eta{\rm d}\tau\\
	\lesssim& \|\mathcal{A}\na_L Q_{\neq,1}\|_{L^2_tH^N}\left(\left\|e^{\lambda\nu^{\frac13}\tau}|l|^{\frac12}\big(\widehat{U_{\neq,1}^1}*\widehat{|\p_x|^{\frac12}U_{\neq,1}^3}\big)\right\|_{L^2_tH^N} +  \left\|e^{\lambda\nu^{\frac13}\tau}|l|^{\frac12}\big(\widehat{|\p_x|^{\frac12}U^1_{\neq,1}}*\widehat{|\p_x|^{\frac12}U^3_{\neq,1}}\big)\right\|_{L^2_tH^N}\right)\\
	\lesssim& \|\mathcal{A}\na_L Q_{\neq,1}\|_{L^2_tH^N}\|e^{\lambda\nu^{\frac13}\tau}|\p_z|^{\frac12}(|\p_x|^{\frac12}U_{\neq,1}^1|\p_x|^{\frac12}U_{\neq,1}^3)\|_{L^2_tH^N}\\
	\lesssim& \|\mathcal{A}\na_L Q_{\neq,1}\|_{L^2_tH^N}\|\mathcal{A}(Q_{\neq,1},K_{\neq,1})\|_{L^{\infty}_tH^N}\|\mathcal{A}|\na_L|^{\frac12}(Q_{\neq,1},K_{\neq,1})\|_{L^2_tH^N}\\
	\lesssim& \delta\nu^c\nu^{-\frac12}\delta\nu^c\delta\nu^c\nu^{-\frac13}=\delta^2\nu^{2c}\cdot\delta\nu^{c-\frac56},
\end{align*}
where in the last inequality we have used \eqref{tool1}. For $I_{12}$, we use \eqref{tool2} to obtain
\begin{align*}
	I_{12}\leq& \nu^{-\frac16}\int_0^t\sum_{\substack{k,l\\k',l'}}\iint_{\eta,\eta'}\widehat{\langle\na\rangle^N} p^{\frac14}(k,\eta,l)|\widehat{\mathcal{A}Q_{\neq,1}(k,\eta,l)}|\widehat{\langle\na\rangle^N}|\widehat{U_{\neq,1}^1}(k-k',\eta-\eta',l-l')||\widehat{\p_xU_{\neq,1}^3}(k',\eta',l')|{\rm d}\eta'{\rm d}\eta{\rm d}\tau\\
	\lesssim& \nu^{-\frac16}\|\mathcal{A}|\na_L|^{\frac12}Q_{\neq,1}\|_{L^2_tH^N}\big(\|\widehat{U^1_{\neq,1}}\|_{L^\infty_tL^1}\|\mathcal{A}|\na_L|^{\frac12}Q_{\neq,1}\|_{L^2_tH^N}+\|\mathcal{A}(Q_{\neq,1},K_{\neq,1})\|_{L^2_tH^N}\|\widehat{\p_xU^3_{\neq,1}}\|_{L^\infty_tL^1}\big)\\
	\lesssim&  \nu^{-\frac16}\delta\nu^c\nu^{-\frac13}(\delta^2\nu^{2c}\nu^{-\frac13}+\delta\nu^c\nu^{-\frac16}\delta\nu^c)\lesssim \delta^2\nu^{2c}\cdot\delta\nu^{c-\frac56}.
\end{align*}

Now we focus on $I_2$. Note that 
\begin{equation}
	\label{multi4}
		m|l|p^{-\frac14}p_h^{-\frac12}\lesssim  \frac{\sqrt{|l|}}{\sqrt
			{p_h}(k^2+l^2)^{\frac14}}\cdot\sqrt{|l|},
\end{equation}
by definition of $m$, using \eqref{tool3} and   inviscid damping effect gives
\begin{align*}
	I_2=&\int_0^t\big|\langle\mathcal{A}Q_{\neq,1},\mathcal{A}\p_z|\na_{L,h}|^{-1}|\na_L|^{-\frac12}(\p_{x,z} U^{1,3}_{\neq,1}\cdot\na_{x,z} U^{1,3}_{\neq,1}+2\na_{L,h}U_{\neq,1}^2\p_y^LU_{\neq,1}^{1,2})\rangle_{H^N}\big|{\rm d}\tau\\
	\leq& 2\int_0^t\sum_{\substack{k,l\\k',l'}}\iint_{\eta,\eta'}\widehat{\langle\na\rangle^N}\widehat{\mathcal{A}Q_{\neq,1}}(k,\eta,l)e^{\lambda\nu^{\frac13}\tau}m|l|p_h^{-\frac12}p^{-\frac14}(k,\eta,l)\widehat{\langle\na\rangle^N}\big[|\widehat{\na_{x,z}U_{\neq,1}^{1,3}}(k-k',\eta-\eta',l-l')|\\
	&  \times|\widehat{\na_{x,z}U_{\neq,1}^{1,3}}(k',\eta',l')|+|\widehat{\na_{L,h}U_{\neq,1}^2}(k-k',\eta-\eta',l-l')||\widehat{\p_y^LU_{\neq,1}^{1,2}}(k',\eta',l')|\big]{\rm d}\eta{\rm d}\eta'{\rm d}\tau\\
	\leq&  2\int_0^t\sum_{\substack{k,l\\k',l'}}\iint_{\eta,\eta'}\widehat{\langle\na\rangle^N}\frac{|l|^{\frac12}}{\sqrt{p_h}(k^2+l^2)^{\frac14}}\widehat{\mathcal{A}Q_{\neq,1}}(k,\eta,l)e^{\lambda\nu^{\frac13}\tau}\widehat{\langle\na\rangle^N}\sqrt{|l|}\big[|\widehat{\na_{x,z}U_{\neq,1}^{1,3}}(k-k',\eta-\eta',l-l')|\\
	&  \times|\widehat{\na_{x,z}U_{\neq,1}^{1,3}}(k',\eta',l')|+|\widehat{\na_{L,h}U_{\neq,1}^2}(k-k',\eta-\eta',l-l')||\widehat{\p_y^LU_{\neq,1}^{1,2}}(k',\eta',l')|\big]{\rm d}\eta{\rm d}\eta'{\rm d}\tau\\
	\lesssim& \left\|\mathcal{A}\sqrt{-\frac{\dot{M}}{M}}Q_{\neq,1}\right\|_{L^2_tH^N}\||l|^{\frac12}\widehat{\na_LU_{\neq,1}}\|_{L^\infty_tL^1}\|\mathcal{A}\na_L(Q_{\neq,1},K_{\neq,1})\|_{L^2_tH^N} \\
	\lesssim&\left\|\mathcal{A}\sqrt{-\frac{\dot{M}}{M}}Q_{\neq,1}\right\|_{L^2_tH^N}\|\mathcal{A}(Q_{\neq,1},K_{\neq,1})\|_{L^\infty_tL^N}\nu^{-\frac13}\|\mathcal{A}\na_L(Q_{\neq,1},K_{\neq,1})\|_{L^2_tH^N}\\
	\lesssim& \delta\nu^c\delta\nu^c\nu^{-\frac13}\delta\nu^c\nu^{-\frac12}=\delta^2\nu^{2c}\cdot\delta\nu^{c-\frac56}.
\end{align*}

For the terms in the equation of $K_{\neq,1}$, we decompose it into some parts as follows:

\begin{align*}
	I_3=&\int_0^t\big|\langle\mathcal{A}K_{\neq,1},\mathcal{A}\sqrt{\frac{\beta}{\beta-1}}|\na_L|^{\frac12}|\na_{L,h}|^{-1}\big[\p_xU_{\neq,1}\cdot\na_L U^2_{\neq,1}-\p_y^LU_{\neq,1}\cdot\na_LU^1_{\neq,1} \\
	&+U_{\neq,1}^{1,3}\cdot\na_{x,z}W^3_{\neq,1}+U_{\neq,1}^{2}\p_y^L W^3_{\neq,1}\big]\rangle_{H^N}\big|{\rm d}\tau =:I_{31}+I_{32}+I_{33}+I_{34}
\end{align*}
For $I_{31}$, we use \eqref{tool2}, \eqref{tool4} and the fact that $\sqrt{|k-k'|}\lesssim \sqrt{|k'|}\sqrt{|k|}$ to obtain 
\begin{align*}
	I_{31}=&\int_0^t\left|\langle\mathcal{A}K_{\neq,1},\mathcal{A}\sqrt{\frac{\beta}{\beta-1}}|\na_L|^{\frac12}|\na_{L,h}|^{-1}\big[\p_xU_{\neq,1}\cdot\na_{L} U^2_{\neq,1}\big]\rangle_{H^N}\right|{\rm d}\tau \\
	\leq& \sqrt{\frac{\beta}{\beta-1}}\int_0^t\sum_{\substack{k,l\\k',l'}}\iint_{\eta,\eta'}\widehat{\langle\na\rangle^N}\mathcal{A}\widehat{K_{\neq,1}}(k,\eta,l)\widehat{\langle\na\rangle^N}\mathcal{A}p^{\frac14}p_h^{-\frac12}(k,\eta,l,t)|k-k'||\widehat{U_{\neq,1}}(k-k',\eta-\eta',l-l')|\\
	&  \times |\widehat{\na_L U_{\neq,1}^2}(k',\eta',l')|{\rm d}\eta{\rm d}\eta'{\rm d}\tau\\
	\lesssim& \sqrt{\frac{\beta}{\beta-1}}\left[\int_0^t\sum_{\substack{k,l\\k',l'}}\iint_{\eta,\eta'}\widehat{\langle\na\rangle^N}\mathcal{A}\widehat{K_{\neq,1}}(k,\eta,l)\frac{e^{\lambda\nu^{\frac13}\tau}}{(k^2+l^2)^{\frac14}}\widehat{\langle\na\rangle^N}\left(\sqrt{|k|}\sqrt{|k'|}\sqrt{|k-k'|}\right.\right.\\
	&  \left.\times |\widehat{U_{\neq,1}}(k-k',\eta-\eta',l-l')||\widehat{\na_L U_{\neq,1}^2}(k',\eta',l')|\right){\rm d}\eta{\rm d}\eta'{\rm d}\tau + \int_0^t\sum_{\substack{k,l\\k',l'}}\iint_{\eta,\eta'}\widehat{\langle\na\rangle^N}\frac{\sqrt{|l|}}{\sqrt{p_h}(k^2+l^2)^{\frac14}}\mathcal{A}\widehat{K_{\neq,1}}(k,\eta,l)\\
	&  \left.\times e^{\lambda\nu^{\frac13}\tau}\widehat{\langle\na\rangle^N}\sqrt{|l|}\left(\sqrt{|k|}\sqrt{|k'|}\sqrt{|k-k'|}|\widehat{U_{\neq,1}}(k-k',\eta-\eta',l-l')| |\widehat{\na_L U_{\neq,1}^2}(k',\eta',l')|\right){\rm d}\eta{\rm d}\eta'{\rm d}\tau\right]\\
	\lesssim& \sqrt{\frac{\beta}{\beta-1}} \left(\|\mathcal{A}K_{\neq,1}\|_{L^2_tH^N}+\left\|\mathcal{A}\sqrt{-\frac{\dot{M}}{M}}K_{\neq,1}\right\|_{L^2_tH^N}^{\frac12}\|\mathcal{A}\na_L K_{\neq,1}\|_{L^2_tH^N}^{\frac12}\right)\\
	&\times\left( \|e^{\lambda\nu^{\frac13}\tau}\widehat{|\p_x|^{\frac12}U_{\neq,1}}\|_{L^\infty_tL^1}\|\na_L|\p_x|^{\frac12} U_{\neq,1}^2\|_{L^2_tH^N}+\|e^{\lambda\nu^{\frac13}\tau}\widehat{|\p_x|^{\frac12}U_{\neq,1}}\|_{L^\infty_tH^N}\|\na_L|\p_x|^{\frac12} U_{\neq,1}^2\|_{L^2_tL^1}\right)\\
		\lesssim& \sqrt{\frac{\beta}{\beta-1}} \delta\nu^{c}\nu^{-\frac14}\|\mathcal{A}(Q_{\neq,1},K_{\neq,1})\|_{L^\infty_tH^N}\|\mathcal{A}\na_L(Q_{\neq,1},K_{\neq,1})\|_{L^2_tH^N}\\
		\lesssim &\sqrt{\frac{\beta}{\beta-1}} \delta\nu^{c}\nu^{-\frac14}\delta\nu^c\delta\nu^c\nu^{-\frac12}=\delta^2\nu^{2c}\cdot \delta\nu^{c-\frac34}.
\end{align*}
where we have used the following inequality
\begin{equation}mp^{\frac14}p_h^{-\frac12}\lesssim (k^2+l^2)^{-\frac14}+\frac{\sqrt{|l|}\cdot\sqrt{|l|}}{\sqrt{p_h}(k^2+l^2)^{\frac14}},\label{multi2}\end{equation}
due to 
\begin{equation}\label{555}
\begin{cases}
	mp^{\frac14}p_h^{-\frac12}\lesssim \sup(m)p_h^{-\frac14}\lesssim\min\{\nu^{-\frac16},\frac{p_h^{\frac14}}{(k^2+l^2)^{\frac14}}\}p_h^{-\frac14}\lesssim 1,\ \text{if }l^2\leq p_h,\\
	mp^{\frac14}p_h^{-\frac12}\lesssim\sup(m)\frac{|l|^{\frac12}}{\sqrt{p_h}} \lesssim\min\{\nu^{-\frac16},\frac{|l|^{\frac12}}{(k^2+l^2)^{\frac14}}\}\frac{|l|^{\frac12}}{\sqrt{p_h}} \lesssim \langle l\rangle^{\frac12},\ \text{if }l^2\geq p_h.
\end{cases}
\end{equation}
For $I_{32}$, we use \eqref{tool2}, \eqref{tool3}, \eqref{tool4} and \eqref{multi2} to obtain
\begin{align*}
	I_{32}=&\int_0^t\left|\langle\mathcal{A}K_{\neq,1},\mathcal{A}\sqrt{\frac{\beta}{\beta-1}}|\na_L|^{\frac12}|\na_{L,h}|^{-1}\big[\p_y^LU_{\neq,1}\cdot\na_LU^1_{\neq,1}\big]\rangle_{H^N}\right|{\rm d}\tau\\
	\leq& \sqrt{\frac{\beta}{\beta-1}}\int_0^t\sum_{\substack{k,l\\k',l'}}\iint_{\eta,\eta'}\widehat{\langle\na\rangle^N}\mathcal{A}\widehat{K_{\neq,1}}(k,\eta,l)\widehat{\langle\na\rangle^N}\mathcal{A}p^{\frac14}p_h^{-\frac12}(k,\eta,l)\\
	&\times |\eta-\eta'-(k-k')t||\widehat{U_{\neq,1}}(k-k',\eta-\eta',l-l')|p^{\frac12}(k',\eta',l')|\widehat{U_{\neq,1}^1}(k',\eta',l')|{\rm d}\eta{\rm d}\eta'{\rm d}\tau\\
	\leq& \sqrt{\frac{\beta}{\beta-1}}\int_0^t\sum_{\substack{k,l\\k',l'}}\iint_{\eta,\eta'}\widehat{\langle\na\rangle^N}\mathcal{A}\widehat{K_{\neq,1}}(k,\eta,l)\widehat{\langle\na\rangle^N}\frac{e^{\lambda\nu^{\frac13}\tau}}{(k^2+l^2)^{\frac14}}\big[|\eta-\eta'-(k-k')t||\widehat{U_{\neq,1}}(k-k',\eta-\eta',l-l')|\\
	&\times p^{\frac12}(k',\eta',l')|\widehat{U_{\neq,1}^1}(k',\eta',l')|\big]{\rm d}\eta{\rm d}\eta'{\rm d}\tau + \sqrt{\frac{\beta}{\beta-1}}\int_0^t\sum_{\substack{k,l\\k',l'}}\iint_{\eta,\eta'}\widehat{\langle\na\rangle^N}\frac{\sqrt{|l|}}{\sqrt{p_h}(k^2+l^2)^{\frac14}}\mathcal{A}\widehat{K_{\neq,1}}(k,\eta,l)\\
	& \times e^{\lambda\nu^{\frac13}\tau}\widehat{\langle\na\rangle^N}\sqrt{|l|}\big[|\eta-\eta'-(k-k')t||\widehat{U_{\neq,1}}(k-k',\eta-\eta',l-l')|p^{\frac12}(k',\eta',l')|\widehat{U_{\neq,1}^1}(k',\eta',l')|\big]{\rm d}\eta{\rm d}\eta'{\rm d}\tau\\
	\lesssim& \sqrt{\frac{\beta}{\beta-1}}\|\mathcal{A}K_{\neq,1}\|_{L^2_tH^N}\left(\|e^{\lambda\nu^{\frac13}\tau}\p_y^LU_{\neq,1}^{1,3} \na_{x,z}U_{\neq,1}^1\|_{L^2_tH^N}+\|e^{\lambda\nu^{\frac13}\tau}\p_y^LU_{\neq,1}^2\p_y^LU_{\neq,1}^1\|_{L^2_tH^N}\right)\\
	&+\left\|\mathcal{A}\sqrt{-\frac{\dot{M}}{M}}K_{\neq,1}\right\|_{L^2_tH^N}\left(\|e^{\lambda\nu^{\frac13}\tau}|l|^{\frac12}\widehat{\p_y^LU_{\neq,1}}\|_{L^2_tH^N}\||l|^{\frac12}\widehat{\na_L U_{\neq,1}^1}\|_{L^\infty_tL^1}\right.\\
	& \left. +\|e^{\lambda\nu^{\frac13}\tau}|l|^{\frac12}\widehat{\p_y^LU_{\neq,1}}\|_{L^\infty_tL^1}\||l|^{\frac12}\widehat{\na_L U_{\neq,1}^1}\|_{L^2_tH^N}\right)\\
	\lesssim&\sqrt{\frac{\beta}{\beta-1}}\|\mathcal{A}K_{\neq,1}\|_{L^2_tH^N}\left(\|e^{\lambda\nu^{\frac13}\tau}\p_y^LU_{\neq,1}^{1,3} \na_{x,z}U_{\neq,1}^1\|_{L^2_tH^N} + \|e^{\lambda\nu^{\frac13}\tau}\p_y^LU_{\neq,1}^2\p_y^LU_{\neq,1}^1\|_{L^2_tH^N}\right)\\
	&+\left\|\mathcal{A}\sqrt{-\frac{\dot{M}}{M}}K_{\neq,1}\right\|_{L^2_tH^N}\|\mathcal{A}\na_L(Q_{\neq,1},K_{\neq,1})\|_{L^2_tH^N}\|\widehat{\na_L U_{\neq,1}}\|_{L^\infty_tW^{\frac12,1}}\\
	\lesssim&\sqrt{\frac{\beta}{\beta-1}}\|\mathcal{A}K_{\neq,1}\|_{L^2_tH^N}\left(\|e^{\lambda\nu^{\frac13}\tau} \p_y^LU_{\neq,1}^{1,3}\|_{L^2_tH^N}\|\na_{x,z}U_{\neq,1}^1\|_{L^\infty_tL^\infty}+\|\p_y^LU_{\neq,1}^{1,3}\|_{L^\infty_tL^\infty}\|e^{\lambda\nu^{\frac13}\tau}\na_{x,z}U_{\neq,1}^1\|_{L^2_tH^N}\right.\\
	& \left. +\|\p_y^LU_{\neq,1}^2\|_{L^\infty_tL^\infty}\|e^{\lambda\nu^{\frac13}\tau}\p_y^LU_{\neq,1}^1\|_{L^2_tH^N}+\|e^{\lambda\nu^{\frac13}\tau}\p_y^LU_{\neq,1}^2\|_{L^2_tH^N}\|\p_y^LU_{\neq,1}^1\|_{L^\infty_tL^\infty}\right) \\
	&+ \left\|\mathcal{A}\sqrt{-\frac{\dot{M}}{M}}K_{\neq,1}\right\|_{L^2_tH^N}\|\mathcal{A}\na_L(Q_{\neq,1},K_{\neq,1})\|_{L^2_tH^N}\|\widehat{\na_L U_{\neq,1}}\|_{L^\infty_tW^{\frac12,1}}\\
	\lesssim&\sqrt{\frac{\beta}{\beta-1}}\|\mathcal{A}K_{\neq,1}\|_{L^2_tH^N}\left(\|\mathcal{A}\na_L(Q_{\neq,1},K_{\neq,1})\|_{L^2_tH^N}\|\na_{x,z}U_{\neq,1}^1\|_{L^\infty_tL^\infty}\right.\\
	& \left.+\|\na_LU_{\neq,1}^{1,3}\|_{L^\infty_tL^\infty}\|\mathcal{A}|\na_L|^{\frac12}(Q_{\neq,1},K_{\neq,1})\|_{L^2_tH^N}+\|\p_y^LU_{\neq,1}^2\|_{L^\infty_tL^\infty}\|\mathcal{A}\na_L(Q_{\neq,1},K_{\neq,1})\|_{L^2_tH^N}\right) \\
	&+ \left\|\mathcal{A}\sqrt{-\frac{\dot{M}}{M}}K_{\neq,1}\right\|_{L^2_tH^N}\|\mathcal{A}\na_L(Q_{\neq,1},K_{\neq,1})\|_{L^2_tH^N}\|\widehat{\na_L U_{\neq,1}}\|_{L^\infty_tW^{\frac12,1}}\\
	\lesssim&\sqrt{\frac{\beta}{\beta-1}} \delta\nu^c\nu^{-\frac16}(\delta\nu^c\nu^{-\frac12}\delta\nu^c+\delta\nu^c\delta\nu^c\nu^{-\frac13}+\delta\nu^c\delta\nu^c\nu^{-\frac12})+\delta\nu^c(\delta\nu^c\nu^{-\frac12}\delta\nu^c\nu^{-\frac13})\\
	\lesssim& \sqrt{\frac{\beta}{\beta-1}}\delta^2\nu^{2c}\cdot\delta\nu^{c-\frac56}.
\end{align*}
For $I_{33}$, using \eqref{tool2}, \eqref{tool3}, \eqref{multi2} and $|(k',\eta'-k't)|\leq |(k,\eta-kt)|+|(k-k',\eta-\eta'-(k-k')t)|$ gives
\begin{align*}
	I_{33}=&\int_0^t\left|\langle\mathcal{A}K_{\neq,1},\mathcal{A}\sqrt{\frac{\beta}{\beta-1}}|\na_L|^{\frac12}|\na_{L,h}|^{-1}\big[U_{\neq,1}^{1,3}\cdot\na_{x,z}W^3_{\neq,1}\big]\rangle_{H^N}\right|{\rm d}\tau\\
	\leq& \int_0^t\left|\langle\mathcal{A}K_{\neq,1},\mathcal{A}|\na_L|^{\frac12}|\na_{L,h}|^{-1}\big[U_{\neq,1}^{1,3}|\na_{x,z}|^{\frac12}|\na_{L,h}|mK_{\neq,1}\big]\rangle_{H^N}\right|{\rm d}\tau\\
	\leq& \int_0^t\sum_{\substack{k,l\\k',l'}}\iint_{\eta,\eta'}\widehat{\langle\na\rangle^N}\mathcal{A}\widehat{K_{\neq,1}}(k,\eta,l)\widehat{\langle\na\rangle^N}\mathcal{A}p^{\frac14}p_h^{-\frac12}(k,\eta,l,t)|\widehat{U_{\neq,1}^{1,3}}(k-k',\eta-\eta',l-l')|\\
	& \times(k'^2+l'^2)^{\frac14}|(k',\eta'-k't)||\widehat{mK_{\neq,1}}(k',\eta',l')|{\rm d}\eta{\rm d}\eta'{\rm d}\tau\\
	\leq& \int_0^t\sum_{\substack{k,l\\k',l'}}\iint_{\eta,\eta'}\widehat{\langle\na\rangle^N}\mathcal{A}\widehat{K_{\neq,1}}(k,\eta,l)e^{\lambda\nu^{\frac13}\tau}\widehat{\langle\na\rangle^N}\left(\frac{1}{(k^2+l^2)^{\frac14}}+\frac{\sqrt{|l|}}{\sqrt{p_h(k,\eta)}(k^2+l^2)^{\frac14}}\sqrt{|l|}\right)\\
	& \times|\widehat{U_{\neq,1}^{1,3}}(k-k',\eta-\eta',l-l')|(k'^2+l'^2)^{\frac14}\big[|(k,\eta-kt)|\\
	& +|(k-k',\eta-\eta'-(k-k')t)|\big]|\widehat{mK_{\neq,1}}(k',\eta',l')|{\rm d}\eta{\rm d}\eta'{\rm d}\tau\\
	\lesssim&\|\mathcal{A}\na_LK_{\neq,1}\|_{L^2_tH^N}\|U_{\neq,1}^{1,3}\|_{L^\infty_tH^N}\|\mathcal{A}|\na_L|^{\frac12}K_{\neq}\|_{L^2_tH^N} \\
	&+ \|\mathcal{A}K_{\neq,1}\|_{L^2_tH^N}\left(\|e^{\lambda\nu^{\frac13}\tau}\na_L U_{\neq,1}^{1,3}\|_{L^2_tH^N}\|\widehat{|\na_{x,z}|^{\frac12}mK_{\neq,1}}\|_{L^\infty_tL^1}+\|e^{\lambda\nu^{\frac13}\tau}\widehat{\na_L U_{\neq,1}^{1,3}}\|_{L^\infty_tL^1}\||\na_{x,z}|^{\frac12}mK_{\neq,1}\|_{L^2_tH^N}\right)\\
	&+ \left\|\mathcal{A}\sqrt{-\frac{\dot{M}}{M}}K_{\neq,1}\right\|_{L^2_tH^N}\left(\|e^{\lambda\nu^{\frac13}\tau}|l|^{\frac12}\widehat{\na_L U_{\neq,1}^{1,3}}\|_{L^2_tH^N}\|\widehat{|\na_{x,z}|^{\frac12}mK_{\neq,1}}\|_{L^\infty_tL^1}\right.\\
	&\left. +\|\widehat{\na_L U_{\neq,1}^{1,3}}\|_{L^\infty_tL^1}\|e^{\lambda\nu^{\frac13}\tau}|\na_{x,z}|mK_{\neq,1}\|_{L^2_tH^N}\right)\\
	&+ \|\mathcal{A}K_{\neq,1}\|_{L^2_tH^N}\left(\|e^{\lambda\nu^{\frac13}\tau}|l|^{\frac12}\widehat{U_{\neq,1}^{1,3}}\|_{L^2_tH^N}\|\widehat{|\na_{x,z}|^{\frac12}mK_{\neq,1}}\|_{L^\infty_tL^1} + \|\widehat{U_{\neq,1}^{1,3}}\|_{L^\infty_tL^1}\|e^{\lambda\nu^{\frac13}\tau}|\na_{x,z}|mK_{\neq,1}\|_{L^2_tH^N}\right)\\
	\lesssim&\|\mathcal{A}\na_LK_{\neq,1}\|_{L^2_tH^N}\|\mathcal{A}(Q_{\neq,1},K_{\neq,1})\|_{L^\infty_tH^N}\|\mathcal{A}|\na_L|^{\frac12}K_{\neq}\|_{L^2_tH^N} \\
	&+ \left\|\mathcal{A}\sqrt{-\frac{\dot{M}}{M}}K_{\neq,1}\right\|_{L^2_tH^N}\left(\|\mathcal{A}\na_L(Q_{\neq,1},K_{\neq,1})\|_{L^2_tH^N}\|\widehat{\na_{x,z}mK_{\neq,1}}\|_{L^\infty_tL^1}\right.\\
	&\left. +\|p_h^{\frac12}\widehat{m(Q_{\neq,1},K_{\neq,1})}\|_{L^\infty_tL^1}\|\mathcal{A}\na_LK_{\neq,1}\|_{L^2_tH^N} \right)+\|\mathcal{A}K_{\neq,1}\|_{L^2_tH^N}\\
	&\times\left(\|\mathcal{A}\na_L(Q_{\neq,1},K_{\neq,1})\|_{L^2_tH^N}\|\widehat{|\na_{x,z}|^{\frac12}mK_{\neq,1}}\|_{L^\infty_tL^1} +\|p_h^{\frac12}\widehat{m(Q_{\neq,1},K_{\neq,1})}\|_{L^\infty_tL^1}\|\mathcal{A}|\na_L|^{\frac12}K_{\neq,1}\|_{L^2_tH^N} \right)\\
	\lesssim& \delta\nu^c\nu^{-\frac12}\delta\nu^c\delta\nu^{c-\frac13}+\delta\nu^c\nu^{-\frac16}(\delta\nu^c\nu^{-\frac12}\delta\nu^c+\delta\nu^c\nu^{-\frac13}\delta\nu^c\nu^{-\frac13})\\
	&\\
	\lesssim& \delta^2\nu^{2c}\cdot\delta\nu^{c-\frac56}.
\end{align*}
Similar to the proof of \eqref{multi2}, since \begin{equation}mp^{\frac14}p_h^{-\frac12}\lesssim\nu^{-\frac16}p_h^{-\frac14}+\frac{\sqrt{|l|}}{\sqrt{p_h}},\label{multi3}\end{equation} 
for the last term $I_{34}$, combining \eqref{multi3}  and $\sqrt{p_h(k',\eta')}\leq\sqrt{p_h(k,\eta)}+\sqrt{p_h(k-k',\eta-\eta')}$ yields
\begin{align*}
	I_{34}=&\int_0^t\left|\langle\mathcal{A}K_{\neq,1},\mathcal{A}\sqrt{\frac{\beta}{\beta-1}}|\na_L|^{\frac12}|\na_{L,h}|^{-1}\big[U_{\neq,1}^{2}\p_y^L W^3_{\neq,1}\big]\rangle_{H^N}\right|{\rm d}\tau\\
	\leq& \int_0^t\sum_{\substack{k,l\\k',l'}}\iint_{\eta,\eta'}\widehat{\langle\na\rangle^N}\mathcal{A}\widehat{K_{\neq,1}}(k,\eta,l)\widehat{\langle\na\rangle^N}\mathcal{A}p^{\frac14}p_h^{-\frac12}(k,\eta,l,t)\big[|\widehat{U_{\neq,1}^2}(k-k',\eta-\eta',l-l')|\\
	& \times |\eta'-k't|p^{-\frac14}(k',\eta',l',t)\sqrt{p_h(k',\eta')}|\widehat{K_{\neq,1}}(k',\eta',l')|\big]{\rm d}\eta{\rm d}\eta'{\rm d}\tau\\
	\leq& \int_0^t\sum_{\substack{k,l\\k',l'}}\iint_{\eta,\eta'}\widehat{\langle\na\rangle^N}\mathcal{A}\widehat{K_{\neq,1}}(k,\eta,l)e^{\lambda\nu^{\frac13}\tau}\widehat{\langle\na\rangle^N}\left(\nu^{-\frac16}p_h^{-\frac14}(k,\eta)+\frac{\sqrt{|l|}}{\sqrt{p_h(k,\eta)}}\right)\big[|\widehat{U_{\neq,1}^2}(k-k',\eta-\eta',l-l')|\\
	&\times |\eta'-k't|p^{-\frac14}(k',\eta',l')\big[\sqrt{p_h(k,\eta)}+\sqrt{p_h(k-k',\eta-\eta')}\big]|\widehat{K_{\neq,1}}(k',\eta',l')|\big]{\rm d}\eta{\rm d}\eta'{\rm d}\tau\\
	\leq& \nu^{-\frac16}\int_0^t\sum_{\substack{k,l\\k',l'}}\iint_{\eta,\eta'}\widehat{\langle\na\rangle^N}p_h^{\frac14}(k,\eta)\mathcal{A}\widehat{K_{\neq,1}}(k,\eta,l)e^{\lambda\nu^{\frac13}\tau}\widehat{\langle\na\rangle^N}\big[|\widehat{U_{\neq,1}^2}(k-k',\eta-\eta',l-l')|\\
	&\times p_h^{\frac14}(k',\eta',l')|\widehat{mK_{\neq,1}}(k',\eta',l')|\big]\nu^{-\frac16}{\rm d}\eta{\rm d}\eta'{\rm d}\tau + \nu^{-\frac16}\int_0^t\sum_{\substack{k,l\\k',l'}}\iint_{\eta,\eta'}\widehat{\langle\na\rangle^N}p_h^{-\frac14}(k,\eta)\mathcal{A}\widehat{K_{\neq,1}}(k,\eta,l)\\
	&\times e^{\lambda\nu^{\frac13}\tau}\widehat{\langle\na\rangle^N}\big[\sqrt{p_h(k-k',\eta-\eta')}|\widehat{U_{\neq,1}^2}(k-k',\eta-\eta',l-l')|p_h^{\frac14}(k',\eta',l')|\widehat{mK_{\neq,1}}(k',\eta',l')|\big]\nu^{-\frac16}{\rm d}\eta{\rm d}\eta'{\rm d}\tau\\
	&+\int_0^t\sum_{\substack{k,l\\k',l'}}\iint_{\eta,\eta'}\widehat{\langle\na\rangle^N}\sqrt{|l|}\mathcal{A}\widehat{K_{\neq,1}}(k,\eta,l)\widehat{\langle\na\rangle^N}\frac{e^{\lambda\nu^{\frac13}\tau}}{\sqrt{p_h(k,\eta)}}\sqrt{p_h(k,\eta)}\big[|\widehat{U_{\neq,1}^2}(k-k',\eta-\eta',l-l')|\\
	&\times p_h^{\frac14}(k',\eta',l')|\widehat{mK_{\neq,1}}(k',\eta',l')|\big]\nu^{-\frac16}{\rm d}\eta{\rm d}\eta'{\rm d}\tau + \int_0^t\sum_{\substack{k,l\\k',l'}}\iint_{\eta,\eta'}\widehat{\langle\na\rangle^N}p_h^{-\frac12}(k,\eta)\mathcal{A}\widehat{K_{\neq,1}}(k,\eta,l)e^{\lambda\nu^{\frac13}\tau}\\
	&\times \widehat{\langle\na\rangle^N}|l|^{\frac12}\big[\sqrt{p_h(k-k',\eta-\eta')}|\widehat{U_{\neq,1}^2}(k-k',\eta-\eta',l-l')|p_h^{\frac14}(k',\eta',l')|\widehat{mK_{\neq,1}}(k',\eta',l')|\big]\nu^{-\frac16}{\rm d}\eta{\rm d}\eta'{\rm d}\tau\\
	\lesssim& \nu^{-\frac16-\frac16}\big(\|\mathcal{A}|\na_L|^{\frac12}K_{\neq,1}\|_{L^4_tH^N}^2\|\widehat{U_{\neq,1}^2}\|_{L^2_tL^1}+\|\mathcal{A}|\na_L|^{\frac12}K_{\neq,1}\|_{L^2_tH^N}\|e^{\lambda\nu^{\frac13}\tau}U_{\neq,1}^2\|_{L^2_tH^N}\|\widehat{|\na_L|^{\frac12}mK_{\neq,1}}\|_{L^\infty_tL^1}\big)\\
	&+ \nu^{-\frac16-\frac16}\|\mathcal{A}|\na_L|^{-\frac12}K_{\neq,1}\|_{L^2_tH^N}\big(\|\widehat{\na_{L,h}U_{\neq,1}^2}\|_{L^\infty_tL^1}\|\mathcal{A}|\na_L|^{\frac12}K_{\neq,1}\|_{L^2_tH^N}\\
	& +\|\mathcal{A}|\na_L|^{\frac12}(Q_{\neq,1},K_{\neq,1})\|_{L^2_tH^N}\|\widehat{|\na_L|^{\frac12}mK_{\neq,1}}\|_{L^\infty_tL^1}\big)+\nu^{-\frac16}\left\|\mathcal{A}\sqrt{-\frac{\dot{M}}{M}}K_{\neq,1}\right\|_{L^2_tH^N}\\
	&\times \left(\|\widehat{|\p_z|^{\frac12}\na_{L,h}U_{\neq,1}^2}\|_{L^\infty_tL^1}\|\mathcal{A}\na_L (Q_{\neq,1},K_{\neq,1})\|_{L^2_tH^N}+\|\mathcal{A}\na_L (Q_{\neq,1},K_{\neq,1})\|_{L^2_tH^N}\|\widehat{|\p_z|^{\frac12}|\na_L|^{\frac12}mK_{\neq,1}}\|_{L^\infty_tL^1}\right)\\
	\lesssim& \nu^{-\frac13}(\delta^2\nu^{2c}\nu^{-\frac12}\delta\nu^c+\delta\nu^c\nu^{-\frac13}\delta\nu^c\nu^{-\frac{\kappa}{6(1+\kappa)}}\delta\nu^c\nu^{-\frac16})\\&+\nu^{-\frac13}\delta\nu^c\nu^{-\frac{\kappa}{6(1+\kappa)}}(\delta\nu^c\delta\nu^c\nu^{-\frac13}+\delta\nu^c\nu^{-\frac13}\delta\nu^{c-\frac{1}{6}})+\nu^{-\frac16}\delta\nu^c(\delta^2\nu^c\nu^c\nu^{-\frac12}+\delta^2\nu^c \nu^{-\frac{1}{2}}\nu^c\nu^{-\frac16}) \\
	\lesssim&\delta^2\nu^{2c}\cdot\delta\nu^{c-\frac56-\frac{\kappa}{6(1+\kappa)}},
\end{align*}
here we use \eqref{tool7}, \eqref{tool4}, \eqref{tool6} and \eqref{tool5}.

For the terms in the equation of $H_{\neq,1}$, using \eqref{tool5}, \eqref{tool7}, \eqref{tool6} and \eqref{4.11-1} yields that
\begin{align*}
	I_4=&\int_0^t\left|\langle\mathcal{A}H_{\neq,1},\mathcal{A}|\na_L|^{\frac12}\big[U_{\neq,1}^{1,3}\cdot\na_{x,z}\Theta_{\neq,1}+U_{\neq,1}^2\p_y^L\Theta_{\neq,1}\big]\rangle_{H^N}\right|{\rm d}\tau \\
	\leq& \int_0^t\sum_{\substack{k,l\\k',l'}}\iint_{\eta,\eta'}\widehat{\langle\na\rangle^N}\mathcal{A}\widehat{H_{\neq,1}}(k,\eta,l)\widehat{\langle\na\rangle^N}\mathcal{A}p^{\frac14}(k,\eta,l,t)\big[|\widehat{U_{\neq,1}^{1,3}}(k-k',\eta-\eta',l-l')|(k',l')||\widehat{\Theta_{\neq,1}}(k',\eta',l')|\\
	& +|\widehat{U_{\neq,1}^2}(k-k',\eta-\eta',l-l')|\eta'-k't||\widehat{\Theta_{\neq,1}}(k',\eta',l')|\big]{\rm d}\eta{\rm d}\eta'{\rm d}\tau\\
	\leq& \int_0^t\sum_{\substack{k,l\\k',l'}}\iint_{\eta,\eta'}\widehat{\langle\na\rangle^N}\mathcal{A}p^{\frac12}(k,\eta,l)\widehat{H_{\neq,1}}(k,\eta,l)e^{\lambda\nu^{\frac13}\tau}\widehat{\langle\na\rangle^N}\big[|\widehat{U_{\neq,1}^{1,3}}(k-k',\eta-\eta',l-l')(k'^2+l'^2)^{\frac12}p^{-\frac14}(k',\eta',l')\\
	& \times|\widehat{mH_{\neq,1}}(k',\eta',l')|\frac{p^{\frac14}(k',\eta',l')}{(k'^2+l'^2)^{\frac14}}\big]{\rm d}\eta{\rm d}\eta'{\rm d}\tau + \int_0^t\sum_{\substack{k,l\\k',l'}}\iint_{\eta,\eta'}\widehat{\langle\na\rangle^N}\mathcal{A}p^{\frac14}(k,\eta,l)\widehat{H_{\neq,1}}(k,\eta,l)e^{\lambda\nu^{\frac13}\tau}\\
	&\times\widehat{\langle\na\rangle^N}\nu^{-\frac16}\big[|\widehat{U_{\neq,1}^2}(k-k',\eta-\eta',l-l')|\eta'-k't|p^{-\frac14}(k',\eta',l')|\widehat{mH_{\neq,1}}(k',\eta',l')|\big]\nu^{-\frac16}{\rm d}\eta{\rm d}\eta'{\rm d}\tau\\
	\lesssim& \|\mathcal{A}|\na_L|^{\frac{1}{2}} H_{\neq,1}\|_{L^2_tH^N}\big(\|\mathcal{A}(Q_{\neq,1},K_{\neq,1})\|_{L^\infty_tH^N}\|\mathcal{A}\na_L H_{\neq,1}\|_{L^2_tH^N}\\
	& +\nu^{-\frac{1}{6}}\|\mathcal{A}|\na_L|^{-\frac12}(Q_{\neq,1},K_{\neq,1})\|_{L^2_tH^N}\|p^{\frac14}\widehat{mH_{\neq,1}}\|_{L^\infty_tL^1}\nu^{-\frac16}\big) + (\nu^{-\frac16}\|\mathcal{A}|\na_L|^{\frac12}H_{\neq,1}\|_{L^4_tH^N})^2\|\widehat{U^2_{\neq,1}}\|_{L^2_tL^1}\\
	\lesssim& \delta\nu^c\nu^{-\frac13}(\delta^2\nu^{2c}\nu^{-\frac12}+\nu^{-\frac{1}{6}}\delta\nu^c\nu^{-\frac{\kappa}{6(1+\kappa)}}\delta\nu^c\nu^{-\frac16}\delta\nu^c\nu^{-\frac16})+(\nu^{-\frac16}\delta\nu^c\nu^{-\frac14})^2\delta\nu^c\\\lesssim& \delta^2\nu^{2c}\cdot\delta\nu^{c-\frac56-\frac{\kappa}{6(1+\kappa)}}.
\end{align*}

We then handle the terms about nonzero mode interacts with simple zero modes in the equation of $Q_{\neq,1}$. Similar to \eqref{multi1}, we have $mp^{\frac34}p_h^{-\frac12}\lesssim \frac{|l|^{\frac32}}{\sqrt{p_h}}+\min\{p^{\frac12}(k^2+l^2)^{-\frac14},\nu^{-\frac16}p^{\frac14}\}$ which is used to 
\begin{align*}
	J_1=&\int_0^t\big|\langle\mathcal{A}Q_{\neq,1},\mathcal{A}|\na_L|^{\frac32}|\na_{L,h}|^{-1}\big(U_{0,1}^1\p_xU_{\neq,1}^3\big)\rangle_{H^N}\big|{\rm d}\tau\\
	\lesssim& \int_0^t\sum_{\substack{k,l\\l'}}\iint_{\eta,\eta'}\widehat{\langle\na\rangle^N}|\widehat{\mathcal{A}Q_{\neq,1}}(k,\eta,l)|\widehat{\langle\na\rangle^N}e^{\lambda\nu^{\frac13}\tau}mp_h^{-\frac12}p^{\frac34}(k,\eta,l) \big[|\widehat{U_{0,1}^1}(0,\eta-\eta',l-l')||k||\widehat{U_{\neq,1}^3}(k,\eta',l')|\big]{\rm d}\eta{\rm d}\eta'{\rm d}\tau\\
	\lesssim& \int_0^t\sum_{\substack{k,l\\l'}}\iint_{\eta,\eta'}\widehat{\langle\na\rangle^N}|\widehat{\mathcal{A}Q_{\neq,1}}(k,\eta,l)|e^{\lambda\nu^{\frac13}\tau}\widehat{\langle\na\rangle^N}\left(\frac{|l|^{\frac32}}{\sqrt{p_h(k,\eta)}}+\frac{\sqrt{p(k,\eta,l)}}{(k^2+l^2)^{\frac14}}\right)\\
	&\times \big[|\widehat{U_{0,1}^1}(0,\eta-\eta',l-l')||k||\widehat{U_{\neq,1}^3}(k,\eta',l')|\big]{\rm d}\eta{\rm d}\eta'{\rm d}\tau\\
	\lesssim& \int_0^t\sum_{\substack{k,l\\l'}}\iint_{\eta,\eta'}\Big\{\widehat{\langle\na\rangle^N}|l||\widehat{\mathcal{A}Q_{\neq,1}}(k,\eta,l)|e^{\lambda\nu^{\frac13}\tau}\widehat{\langle\na\rangle^N}|l|^{\frac12}\big[|\widehat{U_{0,1}}(0,\eta-\eta',l-l')|\frac{|k|}{\sqrt{p_h(k,\eta)}}|\widehat{U_{\neq,1}^3}(k,\eta',l')|\big]\\
	&+\widehat{\langle\na\rangle^N}p^{\frac12}(k,\eta,l)|\widehat{\mathcal{A}Q_{\neq,1}}(k,\eta,l)|e^{\lambda\nu^{\frac13}\tau}\widehat{\langle\na\rangle^N}\big[|\widehat{U_{0,1}^1}(0,\eta-\eta',l-l')|\frac{|k|}{(k^2+l^2)^{\frac14}}|\widehat{U_{\neq,1}^3}(k,\eta',l')|\big]\Big\}{\rm d}\eta{\rm d}\eta'{\rm d}\tau\\
	\lesssim& \|\mathcal{A}\na_LQ_{\neq,1}\|_{L^2_tH^N}\big(\|\widehat{U^1_{0,1}}\|_{L^\infty_tW^{\frac12,1}}\|\mathcal{A}Q_{\neq,1}\|_{L^2_tH^N}+\|U_{0,1}^1\|_{L^\infty_tH^{N,N+\frac12}}\|e^{\lambda\nu^{\frac13}\tau}\widehat{|\na_{x,z}|^{\frac12}U^3_{\neq,1}}\|_{L^2_tL^1}\big)\\
	\lesssim& \delta\nu^c\nu^{-\frac12}(\delta^2\nu^c\nu^c\nu^{-\frac16}+\delta^2\nu^c\nu^c\nu^{-\frac16})\\
	\lesssim&\delta^2\nu^{2c}\cdot\delta\nu^{c-\frac{2}{3}},
\end{align*}
where we have used the fact that $k'=k$, \eqref{tool1} and \eqref{tool2}. We also use \eqref{tool2} and $m|l|p_h^{-\frac12}p^{-\frac14}\lesssim \frac{\sqrt{|l|}}{\sqrt{p_h}(k^2+l^2)^{\frac14}}\cdot\sqrt{|l|}$ to obtain
\begin{align*}
	J_2=&\int_0^t\big|\langle\mathcal{A}Q_{\neq,1},\mathcal{A}\p_z|\na_L|^{-\frac12}|\na_{L,h}|^{-1}\big[2\p_z U^{1,3}_{0,1}\cdot\na_{x,z} U^3_{\neq,1}+2\p_yU_{0,1}^1\p_xU_{\neq,1}^2\big]\rangle_{H^N}\big|{\rm d}\tau\\
	\lesssim& \int_0^t\sum_{\substack{k,l\\l'}}\iint_{\eta,\eta'}\widehat{\langle\na\rangle^N}\widehat{\mathcal{A}Q_{\neq,1}}(k,\eta,l)e^{\lambda\nu^{\frac13}\tau}m|l|p_h^{-\frac12}p^{-\frac14}(k,\eta,l)\widehat{\langle\na\rangle^N}\big[|l-l'||\widehat{U_{0,1}^{1,3}}(0,\eta-\eta',l-l')|\\
	& \times|(k,l')||\widehat{U_{\neq,1}^{3}}(k,\eta',l')|+|k||\widehat{U_{\neq,1}^2}(k,\eta-\eta',l-l')||\eta'|\widehat{U_{0,1}^{1}}(0,\eta',l')\big]{\rm d}\eta{\rm d}\eta'{\rm d}\tau\\
	\lesssim& \int_0^t\sum_{\substack{k,l\\l'}}\iint_{\eta,\eta'}\widehat{\langle\na\rangle^N}\sqrt{-\frac{\dot{M}}{M}}\widehat{\mathcal{A}Q_{\neq,1}}(k,\eta,l)|l|^{\frac12}e^{\lambda\nu^{\frac13}\tau}\widehat{\langle\na\rangle^N}\big[|l-l'||\widehat{U_{0,1}^{1,3}}(0,\eta-\eta',l-l')|\\
	& \times|(k,l')||\widehat{U_{\neq,1}^{3}}(k,\eta',l')|+|k||\widehat{U_{\neq,1}^2}(k,\eta-\eta',l-l')||\eta'|\widehat{U_{0,1}^{1}}(0,\eta',l')\big]{\rm d}\eta{\rm d}\eta'{\rm d}\tau\\
	\lesssim& \left\|\mathcal{A}\sqrt{-\frac{\dot{M}}{M}}Q_{\neq,1}\right\|_{L^2_tH^N}\big(\|\widehat{|\p_z|^{\frac12}\na U_{0,1}}\|_{L^\infty_tL^1}\|\mathcal{A}\na_L(Q_{\neq,1},K_{\neq,1})\|_{L^2_tH^N} \\
	&+ \|\na U_{0,1}\|_{L^2_tH^{N,N+\frac12}}\|e^{\lambda\nu^{\frac13}\tau}\widehat{|\p_z|^{\frac12}\na_{x,z} U_{\neq,1}}\|_{L^\infty_tL^1}\big) \\
	\lesssim& \delta\nu^c(\delta^2\nu^c\nu^c\nu^{-\frac12}+\delta^2\nu^c\nu^{-\frac12}\nu^c)\\ \lesssim& \delta^2\nu^{2c}\cdot\delta\nu^{c-\frac{1}{2}}.
\end{align*}
For the terms about nonzero mode interacts with simple zero modes in the equation of $K_{\neq,1}$, one has
\begin{align*}
	J_3=&\sqrt{\frac{\beta}{\beta-1}}\int_0^t\big|\langle\mathcal{A}K_{\neq,1},\mathcal{A}|\na_L|^{\frac12}|\na_{L,h}|^{-1}\big[(\p_xU_{\neq,1}\cdot\na U^2_{0,1}-\p_y^LU_{\neq,1}^2\p_y U_{0,1}^1-\p_y U_{0,1}^{1,3}\cdot\na_{x,z} U_{\neq,1}^1)\\
	& +(U_{0,1}^1\p_x W_{\neq,1}^3+U_{0,1}^3\p_z W_{\neq,1}^3)\big]\rangle_{H^N}\big|{\rm d}\tau=:J_{31}+J_{32}
\end{align*}
For $J_{31}$, we use \eqref{tool2}, \eqref{tool4} and \eqref{multi2} to obtain 
\begin{align*}
	J_{31}=&\sqrt{\frac{\beta}{\beta-1}}\int_0^t\big|\langle\mathcal{A}K_{\neq,1},\mathcal{A}|\na_L|^{\frac12}|\na_{L,h}|^{-1}\big[(\p_xU_{\neq,1}\cdot\na U^2_{0,1}-\p_y^LU_{\neq,1}^2\p_y U_{0,1}^1-\p_y U_{0,1}^{1,3}\cdot\na_{x,z} U_{\neq,1}^1)\big]\rangle_{H^N}\big|{\rm d}\tau\\
	\leq& \sqrt{\frac{\beta}{\beta-1}}\int_0^t\sum_{\substack{k,l\\l'}}\iint_{\eta,\eta'}\widehat{\langle\na\rangle^N}\mathcal{A}\widehat{K_{\neq,1}}(k,\eta,l)\mathcal{A}p^{\frac14}p_h^{-\frac12}(k,\eta,l)\widehat{\langle\na\rangle^N}\\
	&\times\big[|k||\widehat{U_{\neq,1}}(k,\eta-\eta',l-l')||(0,\eta',l')||\widehat{U_{0,1}^2}(0,\eta',l')|+|\eta-\eta'-kt||\widehat{U_{\neq,1}^2}(k,\eta-\eta',l-l')||\eta'||\widehat{U_{0,1}^1}(0,\eta',l')| \\
	&+ |\eta-\eta'|\widehat{U_{0,1}^{1,3}}(0,\eta-\eta',l-l')||(k,l')||\widehat{U^1_{\neq,1}}(k,\eta',l')|\big]{\rm d}\eta{\rm d}\eta'{\rm d}\tau\\
	\leq& \sqrt{\frac{\beta}{\beta-1}}\int_0^t\sum_{\substack{k,l\\l'}}\iint_{\eta,\eta'}\widehat{\langle\na\rangle^N}\left(\frac{1}{(k^2+l^2)^{\frac14}}+\sqrt{-\frac{\dot{M}}{M}}\right)\mathcal{A}\widehat{K_{\neq,1}}(k,\eta,l)|l|^{\frac12}e^{\lambda\nu^{\frac13}\tau}\widehat{\langle\na\rangle^N}\\
	&\times\big[|k||\widehat{U_{\neq,1}}(k,\eta-\eta',l-l')||(0,\eta',l')||\widehat{U_{0,1}^2}(0,\eta',l')|+|\eta-\eta'-kt||\widehat{U_{\neq,1}^2}(k,\eta-\eta',l-l')||\eta'||\widehat{U_{0,1}^1}(0,\eta',l')| \\
	& + |\eta-\eta'|\widehat{U_{0,1}^{1,3}}(0,\eta-\eta',l-l')||(k,l')||\widehat{U^1_{\neq,1}}(k,\eta',l')|\big]{\rm d}\eta{\rm d}\eta'{\rm d}\tau\\
	\lesssim&\sqrt{\frac{\beta}{\beta-1}} (\|\mathcal{A}K_{\neq,1}\|_{L^2_tH^N}+\left\|\mathcal{A}\sqrt{-\frac{\dot{M}}{M}}K_{\neq,1}\right\|_{L^2_tH^N})\big(\|\widehat{U_{0,1}}\|_{L^\infty_tW^{2,1}}\|\mathcal{A}\na_L(Q_{\neq,1},K_{\neq,1})\|_{L^2_tH^N}\\
	&   +\|\na U_{0,1}\|_{L^2_tH^{N,N+\frac12}}\|e^{\lambda\nu^{\frac13}\tau}\widehat{|\p_z|^{\frac12}\na_{x,z}U^1_{\neq,1}}\|_{L^\infty_tL^1}+\|\na U_{0,1}\|_{L^2_tH^{N,N+\frac12}}\|e^{\lambda\nu^{\frac13}\tau}\widehat{|\p_z|^{\frac12}\na_LU^2_{\neq,1}}\|_{L^\infty_tL^1}\big)\\
	\lesssim&\sqrt{\frac{\beta}{\beta-1}} \delta(\nu^c\nu^{-\frac16}+\nu^c)\delta^2(\nu^c\nu^c\nu^{-\frac12}+\nu^c\nu^{-\frac12}\nu^c+\nu^c\nu^{-\frac12}\nu^c)\\
	\lesssim&\sqrt{\frac{\beta}{\beta-1}} \delta^2\nu^{2c}\cdot\delta\nu^{c-\frac23}.
\end{align*}
For $J_{32}$, we use \eqref{tool1} and \eqref{tool8} to obtain
\begin{align*}
	J_{32}=&\sqrt{\frac{\beta}{\beta-1}}\int_0^t\big|\langle\mathcal{A}K_{\neq,1},\mathcal{A}|\na_L|^{\frac12}|\na_{L,h}|^{-1}\big[U_{0,1}^1\p_x W_{\neq,1}^3+U_{0,1}^3\p_z W_{\neq,1}^3\big]\rangle_{H^N}\big|{\rm d}\tau\\
	\leq& \int_0^t\sum_{\substack{k,l\\l'}}\iint_{\eta,\eta'}\widehat{\langle\na\rangle^N}\widehat{\mathcal{A}K_{\neq,1}}(k,\eta,l)\mathcal{A}p^{\frac14}p_h^{-\frac12}(k,\eta,l,t)\widehat{\langle\na\rangle^N}\big[|\widehat{U_{0,1}}(0,\eta-\eta',l-l')||k|p^{-\frac14}(k,\eta',l')p_h^{\frac12}(k,\eta')\\
	&\times |\widehat{K_{\neq,1}}(k,\eta',l')|+|\widehat{U_{0,1}^3}(0,\eta-\eta',l-l')||l'|p^{-\frac14}(k,\eta',l')p_h^{\frac12}(k,\eta')|\widehat{K_{\neq,1}}(k,\eta',l')| \big]{\rm d}\eta{\rm d}\eta'{\rm d}\tau\\
	\leq& \int_0^t\sum_{\substack{k,l\\l'}}\iint_{\eta,\eta'}\widehat{\langle\na\rangle^N}\left(\frac{1}{(k^2+l^2)^{\frac14}}+\frac{\sqrt{|l|}\cdot\sqrt{|l|}}{\sqrt{p_h(k,\eta)}(k^2+l^2)^{\frac14}}\right)\widehat{\mathcal{A}K_{\neq,1}}(k,\eta,l)|e^{\lambda\nu^{\frac13}\tau}|\widehat{\langle\na\rangle^N}\big[|\widehat{U_{0,1}}(0,\eta-\eta',l-l')|\\
	&\times |k|p^{-\frac14}(k,\eta',l')p_h^{\frac12}(k,\eta')|\widehat{K_{\neq,1}}(k,\eta',l')|\\
	&+|\widehat{U_{0,1}^3}(0,\eta-\eta',l-l')||l'|p^{-\frac14}(k,\eta',l')p_h^{\frac12}(k,\eta')|\widehat{K_{\neq,1}}(k,\eta',l')| \big]{\rm d}\eta{\rm d}\eta'{\rm d}\tau\\
	\lesssim&\left(\|\mathcal{A}K_{\neq,1}\|_{L^2_tH^N}+\left\|\mathcal{A}\sqrt{-\frac{\dot{M}}{M}}K_{\neq,1}\right\|_{L^2_tH^N}\right)\big(\|U_{0,1}\|_{L^\infty_tH^{N,N+\frac12}}\|e^{\lambda\nu^{\frac13}\tau}|l|^{\frac12}p_h^{\frac12}\widehat{mK_{\neq,1}}\|_{L^2_tL^1}\\
	&+\|U_{0,1}\|_{L^\infty_tH^{N,N+\frac12}}\|e^{\lambda\nu^{\frac13}\tau}|l|^{\frac12}p_h^{\frac12}\widehat{|\na_{x,z}|mK_{\neq,1}}\|_{L^2_tL^1}\big)+\|\mathcal{A}\na_LK_{\neq,1}\|_{L^2_tH^N}\|\widehat{U_{0,1}}\|_{L^\infty_tW^{2,1}}\|\mathcal{A}K_{\neq,1}\|_{L^2_tH^N}\\
	\lesssim& \delta(\nu^c\nu^{-\frac16}+\nu^c)\delta^2(\nu^{2c}\nu^{-\frac12})+\delta\nu^c\nu^{-\frac12}\delta^2\nu^c\nu^c\nu^{-\frac16}\\
	\lesssim&\delta^2\nu^{2c}\cdot\delta\nu^{c-\frac{2}{3}}.
\end{align*}
For the terms about nonzero mode interacts with simple zero modes in the equation of $H_{\neq,1}$, we get
\begin{align*}
	J_4=&\int_0^t\big|\langle\mathcal{A}H_{\neq,1},\mathcal{A}|\na_L|^{\frac12}\big[U_{0,1}^{1,3}\na_{x,z}\Theta_{\neq,1}+U^3_{\neq,1}\p_z\Theta_{0,1}\big]\rangle_{H^N}\big|{\rm d}\tau \\
	\leq& \int_0^t\sum_{\substack{k,l\\l'}}\iint_{\eta,\eta'}\widehat{\langle\na\rangle^N}\widehat{\mathcal{A}H_{\neq,1}}(k,\eta,l)\mathcal{A}p^{\frac14}(k,\eta,l)\widehat{\langle\na\rangle^N}\big[|\widehat{U_{0,1}^{1,3}}(0,\eta-\eta',l-l')||(k,l')||\widehat{\Theta_{\neq,1}}(k,\eta',l')|\\
	& +|\widehat{U_{\neq,1}^3}(k,\eta-\eta',l-l')|l'||\widehat{\Theta_{0,1}}(0,\eta',l')|\big]{\rm d}\eta{\rm d}\eta'{\rm d}\tau\\
	\leq& \int_0^t\sum_{\substack{k,l\\l'}}\iint_{\eta,\eta'}\widehat{\langle\na\rangle^N}p^{\frac12}(k,\eta,l)|\widehat{\mathcal{A}H_{\neq,1}}(k,\eta,l)|(k^2+l^2)^{-\frac14}\widehat{\langle\na\rangle^N}e^{\lambda\nu^{\frac13}\tau}\\
	&\times\big[|\widehat{U_{0,1}^{1,3}}(0,\eta-\eta',l-l')||(k,l')||\widehat{\Theta_{\neq,1}}(k,\eta',l')|+|\widehat{U_{\neq,1}^3}(k,\eta-\eta',l-l')|l'||\widehat{\Theta_{0,1}}(0,\eta',l')|\big]{\rm d}\eta{\rm d}\eta'{\rm d}\tau\\
	\lesssim& \|\mathcal{A}\na_L H_{\neq,1}\|_{L^2_tH^N}\big(\|\widehat{U_{0,1}^{1,3}}\|_{L^\infty_tL^1}\|\mathcal{A}|\na_L|^{\frac12}H_{\neq,1}\|_{L^2_tH^N}+\|U_{0,1}^{1,3}\|_{L^\infty_tH^N}\|e^{\lambda\nu^{\frac13}\tau}\widehat{\na_{x,z}\Theta_{\neq,1}}\|_{L^2_tL^1}\\
	&+\|\widehat{\p_z\Theta_{0,1}}\|_{L^\infty_tL^1}\|\mathcal{A}Q_{\neq,1}\|_{L^2_tH^N}+\|e^{\lambda\nu^{\frac13}\tau}\widehat{U^3_{\neq,1}}\|_{L^4_tL^1}\||\p_z|^{\frac12}\Theta_{0,1}\|_{L^4_tH^{N,N+\frac12}}\big) \\
	\lesssim& \delta\nu^c\nu^{-\frac12}\delta^2(\nu^c\nu^c\nu^{-\frac13}+\nu^c\nu^c\nu^{-\frac16}+\nu^c\nu^c\nu^{-\frac16}+\nu^c\nu^{-\frac{1}{12}}\nu^c\nu^{-\frac14})\\
	\lesssim& \delta^2\nu^{2c}\cdot\delta\nu^{c-\frac{5}{6}},
\end{align*}
here we used the estimate   \eqref{tool2}. 
 
For  nonlinear cross terms, thanks to the fact   $|\hat{G}|\leq 1$,  the proof process for nonlinear terms mentioned above is no different from that for cross-terms, so we omit its proof process. Thus, we  finish the proof of proposition \ref{propEneq1}.  \hfill$\square$

\subsection{ Energy estimates on $U_{0,2} $ and $ \Theta_{0,2}$   }\label{sub4.5}
\qquad  In this section, we     prove that under the bootstrap hypotheses of Proposition \ref{boots}, the estimates on $U_{0,2} $, and $ \Theta_{0,2}$  hold (i.e., \eqref{u02}), with 100 replaced by 50 on the right-hand side. 



\begin{pro}\label{propE02}
	Suppose that the assumptions in Proposition \ref{boots} are true, then it holds that for any $t\in [0,T]$,
	\begin{equation}\label{L55}
	\begin{aligned}
	\sup_{t>0}	&E_{0,2}(t)+\int_0^t F_{0,2}(\tau){\rm d}\tau\\
		\leq&  C\left(\delta^4\nu^{4c-\frac{10}{3}} +{ \delta^6\nu^{6c-\frac{16}{3}} + \delta^8\nu^{2a+6c-\frac{22}{3}}}  + q^{-\frac23}(\delta^2\nu^{2a-\frac53} + {\delta^6\nu^{6c-\frac{16}{3}}} +{\delta^8\nu^{8c-\frac{22}{3}}})\right)\int_0^t F_{0,1}(\tau) {\rm d}\tau\\
		&+ C( q^{-\frac13}{\delta^3\nu^{3c-\frac{8}{3}}} + q^{-\frac13}\delta^2\nu^{2c-\frac53} + q^{-\frac43}\delta^4\nu^{4c-3} + \delta\nu^{a-1})\int_0^t F_{0,2}(\tau){\rm d}\tau \\
		&+ C\delta\nu^{a-\frac23}\int_0^t \left[F_{\neq,1}(\tau)+ F_{\neq,2}(\tau)\right]{\rm d}\tau + \delta^2\nu^{2a}\delta^2\nu^{2a-2}+\delta^4\nu^{4c-\frac43}+{q^{-\frac13}\delta^5\nu^{4c+a-\frac83}}\\
		&+q^{-\frac23}\delta^2\nu^{2c}(\delta^2\nu^{2c-\frac43} + \delta^4\nu^{4c-\frac{10}{3}}),
	\end{aligned}
	\end{equation}
	where $ E_{0,2}(t)$ and $ F_{0,2}(t)$ are defined by \eqref{pre7}
	and \eqref{pre8}, respectively. 
\end{pro}


Before proving this proposition, we  present the following lemma \ref{L11}. We  first establish an improved estimation of the double-zero modes.

\begin{lem}\label{L11}
Suppose that the assumptions in Proposition \ref{boots} are true, then it holds that for any $t\in [0,T]$,
\begin{equation}
    \begin{aligned}\label{L33}
		&\|(\overline{U^1_{0,2}},\overline{U^3_{0,2}},\overline{\Theta_{0,2}})\|_{L^{\infty}_t H^{s}}^2 + \nu\|\p_y (\overline{U^1_{0,2}},\overline{U^3_{0,2}},\overline{\Theta_{0,2}})\|_{L^2_t H^{s}}^2 \\
		&\quad\lesssim 	\delta^2\nu^{2a}\delta^2\nu^{2a-2}+\delta^4\nu^{4c-\frac43}+{q^{-\frac13}\delta^5\nu^{4c+a-\frac83}}+q^{-\frac23}\delta^2\nu^{2c}(\delta^2\nu^{2c-\frac43} + \delta^4\nu^{4c-\frac{10}{3}}),
	\end{aligned}
\end{equation}
\begin{equation}
	\begin{aligned}\label{L34}
		&\|(\overline{U^3_{0,2}},\overline{\Theta_{0,2}})\|_{L^{\infty}_t H^{s-3}}^2 + \nu\|\p_y (\overline{U^3_{0,2}},\overline{\Theta_{0,2}})\|_{L^2_t H^{s-3}}^2 \\
		&\quad\lesssim 	\delta^2\nu^{2a}\delta^4\nu^{4c-\frac{10}{3}} + \delta^4\nu^{4c-\frac43} + q^{-\frac23}\delta^2\nu^{2c}(\delta^2\nu^{2c-\frac43} + \delta^4\nu^{4c-\frac{10}{3}}).
	\end{aligned}
\end{equation}
\end{lem}

\pf  Now we consider the perturbation system \eqref{2ndQuasi1},\eqref{2ndQuasi2} and \eqref{2ndQuasi3}.   The energy estimates tell us that 
\begin{align*}
	&\|(\overline{U^1_{0,2}},\overline{U^3_{0,2}},\overline{\Theta_{0,2}})\|_{H^{s}}^2 + \nu\|\p_y (\overline{U^1_{0,2}},\overline{U^3_{0,2}},\overline{\Theta_{0,2}})\|_{L^2_t H^{s}}^2 \\
    & \quad\lesssim \nu^{-1}\int_0^{t} \|\overline{U_{0,1}^2\tilde{\Theta}_{0,1}}\|_{H^{s}}^2+\sum_{F\in U^1,U^3,\Theta}(\|\overline{U_{0,1}^2\tilde{F}_{0,2}}\|_{H^{s}}^2 + {\|\overline{\tilde{U}_{0,2}^{2,3}\cdot\na_{y,z}\tilde{F}_{0,1}}\|_{H^{s-1}}^2}+\|\overline{U_{0,2}^2\tilde{F}_{0,2}}\|_{H^{s}}^2) {\rm d}\tau\\
	&\qquad+  \nu^{-1}\int_0^{t} \sum_{F\in U^1,U^3,\Theta}(\|(\overline{U_{\neq,1}^2F_{\neq,2}})_0\|_{H^{s}}^2 + \|(\overline{U_{\neq,2}^2F_{\neq,1}})_0\|_{H^{s}}^2 + \|(\overline{U_{\neq,2}^2F_{\neq,2}})_0\|_{H^{s}}^2) {\rm d}\tau.
\end{align*}
For $F\in\{U^1,U^3,\Theta\}$, using estimates of Lemma \ref{lowDisper} to handle the zero modes contributions, which gives
{\begin{align*}
    \nu^{-1}\int_0^{t} \|\overline{U_{0,1}^2\tilde{F}_0}\|_{H^{s}}^2{\rm d}\tau\lesssim& \nu^{-1}\int_0^{t} \|U_{0,1}^2\|_{W^{s,\infty}}^2\|(\tilde{F}_{0,1},\tilde{F}_{0,2})\|_{H^{s}}^2{\rm d}\tau\\
	\lesssim& \nu^{-1}\int_0^tq^{-\frac23}\delta^2\nu^{2c}(\tau^{-\frac13}e^{-\nu \tau}+\delta\nu^{c-\frac23})^2\|(\tilde{F}_{0,1},\tilde{F}_{0,2})\|_{H^s}^2{\rm d}\tau\\
	\lesssim& q^{-\frac23}\delta^2(\nu^{2c}+\nu^{2a})(\delta^2\nu^{2c-\frac43}+\delta^4\nu^{4c-\frac{10}{3}}),\\
	\nu^{-1}\int_0^{t} {\|\overline{\tilde{U}_{0,2}^{2,3}\cdot\na_{y,z}\tilde{F}_{0,1}}\|_{H^{s-1}_y}^2}{\rm d}\tau\lesssim& \nu^{-1}\int_0^{t} \|\tilde{U}_{0,2}^{2,3}\|_{W^{s-3,\infty}}\|\na \tilde{U}_{0,2}\|_{H^{s,s+\frac12}}\|\na \tilde{F}_{0,1}\|_{H^{s+1}}^2{\rm d}\tau\\
	\lesssim& q^{-\frac13}\delta^5\nu^{4c+a-\frac83},\\
	\nu^{-1}\int_0^{t} \|\overline{U_{0,2}^2\tilde{F}_{0,2}}\|_{H^{s}}^2{\rm d}\tau\lesssim& \nu^{-1} \|U_{0,2}\|_{L^{\infty}_t H^s}^2\|F_{0,2}\|_{L^2_tH^s}^2\lesssim \delta^4\nu^{4a-2}.
\end{align*}}
For the nonzero modes contributions, we have 
\begin{align*}
	&\nu^{-1}\int_0^{t} \sum_{F\in U^1,U^3,\Theta}(\|(\overline{U_{\neq,1}^2F_{\neq,2}})_0\|_{H^{s}}^2 + \|(\overline{U_{\neq,2}^2F_{\neq,1}})_0\|_{H^{s}}^2 + \|(\overline{U_{\neq,2}^2F_{\neq,2}})_0\|_{H^{s}}^2) {\rm d}\tau\\
	& \quad\lesssim\nu^{-1}\int_0^t (\|\mathcal{A}(Q_{\neq,1},K_{\neq,1},H_{\neq,1})\|_{H^s}^2+\|\mathcal{B}(Q_{\neq,2},K_{\neq,2},H_{\neq,2})\|_{H^s}^2)\|\mathcal{B}(Q_{\neq,2},K_{\neq,2},H_{\neq,2})\|_{H^s}^2{\rm d}\tau\\ 
	&\quad\lesssim \delta^2(\nu^{2c}+\nu^{2b})\delta^2\nu^{2b-\frac43}.
\end{align*}
The proof of  \eqref{L34} is similar to that of \eqref{L33},
it suffice to note that
\begin{align*}
	\nu^{-1}\int_0^{t} {\|\overline{\tilde{U}_{0,2}^{2,3}\cdot\na_{y,z}\tilde{F}_{0,1}}\|_{H^{s-4}_y}^2}{\rm d}\tau\lesssim& \nu^{-1}\int_0^{t} \|\tilde{U}_{0,2}^{2,3}\|_{W^{s-4,\infty}}^2\|\na \tilde{F}_{0,1}\|_{H^s}^2{\rm d}\tau \lesssim q^{-\frac23}\delta^6\nu^{6c-\frac{10}3},\\
	\nu^{-1}\int_0^{t} {\|\overline{U_{0,2}^2\tilde{F}_{0,2}}\|_{H^{s-3}}^2}{\rm d}\tau\lesssim& \nu^{-1} \|U_{0,2}\|_{L^\infty_t W^{s-3,\infty}}^2\|\na \tilde{F}_{0,2}\|_{L^2_tH^s}^2\lesssim \delta^6\nu^{4c+2a-\frac{10}{3}}.
\end{align*}
where we need  $c\leq a$ and $c \leq b$. Thus we finish the proof.\hfill$\square$

{\textit  {Proof of Proposition \ref{propE02}.}} In the following, We will use the above lemma to prove Proposition \ref{propE02}. Energy estimates yield that
\begin{flalign*}
    &\  \frac12  \frac{{\rm d}}{{\rm d}t}(\|U^2_{0,2}\|_{H^{s,s+\frac12}}^2+\|V^2_{0,2}\|_{H^{s,s+\frac12}}^2) + \nu\|\na (U^2_{0,2},V^2_{0,2})\|_{H^{s,s+\frac12}}^2   \\&\quad
   = \langle\mathcal{N}_{\neq}(U_{0,2}^2),U_{0,2}^2\rangle_{H^{s,s+\frac12}}+\langle\mathcal{N}_{\neq}(V_{0,2}^2),V_{0,2}^2\rangle_{H^{s,s+\frac12}}-\langle U_0\cdot\na U_0^2,U^2_{0,2}\rangle_{H^{s,s+\frac12}} + \langle U_{0,1}\cdot\na U_{0,1}^2,U^2_{0,2}\rangle_{H^{s,s+\frac12}} \\
    &\qquad+ 2 \langle \p_{y}\Delta_{L,0}^{-1}(\p_yU_0^2\p_yU_0^2+\p_yU_0^3 \p_zU^2_0),U^2_{0,2}\rangle_{H^{s,s+\frac12}}-2 \langle \p_{y}\Delta_{L,0}^{-1}(\p_yU_{0,1}^2\p_yU_{0,1}^2+\p_y\tilde{U}_{0,1}^3 \p_zU^2_{0,1}),U^2_{0,2}\rangle_{H^{s,s+\frac12}} \\
    &\qquad-  \langle G_1\widetilde{(U_0\cdot\na U^1_0)},V^2_{0,2}\rangle_{H^{s,s+\frac12}} - \langle G_2\widetilde{(U_0\cdot\na \Theta_0)},V^2_{0,2}\rangle_{H^{s,s+\frac12}} +  \langle \tilde{U}_{0,1}\cdot\na V^2_{0,1},V^2_{0,2}\rangle_{H^{s,s+\frac12}},
\end{flalign*}
and 
\begin{flalign*}
    &  \frac12  \frac{{\rm d}}{{\rm d}t}(\|\tilde{U}^3_{0,2}\|_{H^{s,s+\frac12}}^2+\| \tilde{V}^3_{0,2}\|_{H^{s,s+\frac12}}^2) + \nu\|\na (\tilde{U}^3_{0,2}, \tilde{V}^3_{0,2})\|_{H^{s,s+\frac12}}^2\\&
   =\langle\mathcal{N}_{\neq}(U_{0,2}^2),U_{0,2}^2\rangle_{H^{s,s+\frac12}}+\langle\mathcal{N}_{\neq}(V_{0,2}^2),V_{0,2}^2\rangle_{H^{s,s+\frac12}}-\langle \widetilde{U_0\cdot\na U_0^3},\tilde{U}^3_{0,2}\rangle_{H^{s,s+\frac12}} + \langle \widetilde{\tilde{U}_{0,1}\cdot\na\tilde{U}_{0,1}^3},\tilde{U}^3_{0,2}\rangle_{H^{s,s+\frac12}}   \\
   &\quad+  2\langle \p_{z}\Delta_{L,0}^{-1}(\p_yU_0^2\p_yU^2_0+\p_{y}U_0^3 \p_zU^2_0),\tilde{U}^3_{0,2}\rangle_{H^{s,s+\frac12}}- 2\langle \p_{z}\Delta_{L,0}^{-1}(\p_yU_{0,1}^2\p_yU^2_{0,1} + \p_{y}\tilde{U}_{0,1}^3 \p_zU^2_{0,1}),\tilde{U}^3_{0,2}\rangle_{H^{s,s+\frac12}}\\
    &\quad -\langle  \widetilde{G_1'(U_0\cdot\na U^1_0)}, \tilde{V}^3_{0,2}\rangle_{H^{s,s+\frac12}} -  \langle \widetilde{G_2'(U_0\cdot\na \Theta_0)}, \tilde{V}^3_{0,2}\rangle_{H^{s,s+\frac12}}  -  {\langle \p_z^{-1}\p_y(\tilde{U}_{0,1}\cdot\na V_{0,1}^2), \tilde{V}^3_{0,2}\rangle_{H^{s,s+\frac12}}},
\end{flalign*}
and 
\begin{flalign*}
   &\frac12 \frac{{\rm d}}{{\rm d}t}\|\Lambda_{0,2}\|_{H^{s,s+\frac12}}^2+\nu\|\na \Lambda_{0,2}\|_{H^{s,s+\frac12}}^2  \\
   &\quad= \langle\mathcal{N}_{\neq}(\Lambda_{0,2}),\Lambda_{0,2}\rangle_{H^{s,s+\frac12}}+\langle {G_3(\tilde{U}_{0,1}\cdot\na U^1_{0,1})+G_4(\tilde{U}_{0,1}\cdot\na \Theta_{0,1})},\Lambda_{0,2}\rangle_{H^{s,s+\frac12}}\\
   &\qquad- \langle G_3\widetilde{(U_0\cdot\na U_0^1)},\Lambda_{0,2}\rangle_{H^{s,s+\frac12}} - \langle G_4\widetilde{(U_0\cdot\na \Theta_0)},\Lambda_{0,2}\rangle_{H^{s,s+\frac12}}.
\end{flalign*}
Then adding all the above terms together, we have 
\begin{align*}
    \frac{{\rm d}}{{\rm d}t}&E_{0,2}(t) + 2F_{0,2}(t)\\
	& \lesssim \left|\langle \na(U_{0,1}\cdot\na(\tilde{U}_{0,1},\tilde{\Theta}_{0,1})), (U_{0,2}^2,\tilde{U}_{0,2}^3,V_{0,2}^2,\tilde{V}_{0,2}^3)\rangle_{H^{s,s+\frac12}}\right|+ \left|\langle \overline{U_{0,1}^3}\p_z(U_{0,1},\Theta_{0,1}), (\tilde{U}_{0,2},\tilde{\Theta}_{0,2})\rangle_{H^{s,s+\frac12}}\right| \\
	&\quad+ \left|\langle U_{0,1}\cdot\na (U_{0,2},\Theta_{0,2}), (\tilde{U}_{0,2},\tilde{\Theta}_{0,2})\rangle_{H^{s,s+\frac12}}\right| + \left|\langle U_{0,2}\cdot\na (U_{0,1},\Theta_{0,1}), (\tilde{U}_{0,2},\tilde{\Theta}_{0,2})\rangle_{H^{s,s+\frac12}}\right| \\
	&\quad+ \left|\langle U_{0,2}\cdot\na (U_{0,2},\Theta_{0,2}), (\tilde{U}_{0,2},\tilde{\Theta}_{0,2})\rangle_{H^{s,s+\frac12}}\right| + \left|\langle U_{\neq,1}\cdot\na_L (U_{\neq,2},\Theta_{\neq,2})_0, (\tilde{U}_{0,2},\tilde{\Theta}_{0,2})\rangle_{H^{s,s+\frac12}}\right| \\
	&\quad+ \left|\langle U_{\neq,2}\cdot\na_L (U_{\neq,1},\Theta_{\neq,1})_0, (\tilde{U}_{0,2},\tilde{\Theta}_{0,2})\rangle_{H^{s,s+\frac12}}\right| + \left|\langle U_{\neq,2}\cdot\na_L (U_{\neq,2},\Theta_{\neq,2})_0, (\tilde{U}_{0,2},\tilde{\Theta}_{0,2})\rangle_{H^{s,s+\frac12}}\right|\\
	&\lesssim F^{\frac12}_{0,1}(t)\nu^{-\frac12}\left(F_{0,1}^{\frac12}(t)\nu^{-\frac12}\|(U_{0,2}^2,\tilde{U}_{0,2}^3,V_{0,2}^2,\tilde{V}_{0,2}^3)\|_{W^{s-3}} + F_{0,2}^{\frac12}(t)\nu^{-\frac12}\|\overline{U_{0,1}^3}\|_{H^{s,s+\frac12}}\right) \\
	&\quad+ F_{0,2}(t)\nu^{-1} \|(\overline{U_{0,1}^3},\overline{\Theta_{0,1}})\|_{H^{s+1}} + \|\langle\p_z\rangle^{\frac12}\tilde{U}_{0,1}^{2,3}\|_{W^{s,\infty}}F_{0,2}^{\frac12}\nu^{-\frac12}F_{0,2}^{\frac14}(t)\nu^{-\frac14}\|(\tilde{U}_{0,2},\tilde{\Theta}_{0,2})\|_{H^{s+1}}^{\frac12} \\
	&\quad+ { \|\overline{U_{0,2}^3}\|_{H^{s-3}}^{\frac12}\|\p_y\overline{U_{0,2}^3}\|_{H^s}^{\frac12}E_{0,1}^{\frac14}(t)F_{0,1}^{\frac14}(t)\nu^{-\frac14}F_{0,2}^{\frac12}(t)\nu^{-\frac12}} + {F_{0,2}^{\frac34}(t)\nu^{-\frac34}q^{-\frac16}\delta\nu^{c-\frac13}F_{0,1}^{\frac14}(t)\nu^{-\frac14}E_{0,1}^{\frac14}(t)}\\
	&\quad+ \delta\nu^a F_{0,2}(t)\nu^{-1} + \delta\nu^a\cdot F^{\frac12}_{\neq,1}(t)F^{\frac12}_{\neq,2}(t)\nu^{-\frac16-\frac12} + \delta\nu^a F_{\neq,2}(t)\nu^{-\frac16}\nu^{-\frac12}\\
	&\lesssim   \left({\delta^8\nu^{2a+6c-\frac{22}{3}} + \delta^6\nu^{6c-\frac{16}{3}}} + \delta^4\nu^{4c-\frac{10}{3}} + q^{-\frac23}(\delta^2\nu^{2a-\frac53} + {\delta^6\nu^{6c-\frac{16}{3}}} + {\delta^8\nu^{8c-\frac{22}{3}}} )\right)F_{0,1}(t) \\
	&\quad+F_{0,2}(t)+ (q^{-\frac13}{\delta^3\nu^{3c-\frac{8}{3}}} + q^{-\frac13}\delta^2\nu^{2c-\frac53} + q^{-\frac43}\delta^4\nu^{4c-3} + \delta\nu^{a-1})F_{0,2}(t) \\
	&\quad+ \delta\nu^{a-\frac23}F^{\frac12}_{\neq,1}(t)F^{\frac12}_{\neq,2}(t) + \delta\nu^{a-\frac23} F_{\neq,2}(t),
\end{align*}
where in the fifth line from the bottom, we have used the following fact for $G\in\{U,\Theta\}$ 
\begin{align*}
	\langle \tilde{U}_{0,2}^{2,3}\cdot\na_{y,z}G_{0,1},\tilde{G}_{0,2}\rangle_{H^{s,s+\frac12}}\lesssim& F_{0,2}^{\frac12}(t)\nu^{-\frac12}\|\langle\p_z\rangle^{\frac12}\left(\tilde{U}_{0,2}^{2,3}\cdot\na_{y,z}G_{0,1}\right)\|_{H^{s-1}}\\
	\lesssim& F_{0,2}^{\frac12}(t)\nu^{-\frac12}\|\langle\p_z\rangle^{\frac12}\tilde{U}_{0,2}^{2,3}\|_{W^{s-3,\infty}}^{\frac12}\|\na \tilde{U}_{0,2}\|_{H^{s,s+\frac12}}^{\frac12}\|\na G_{0,1}\|_{H^{s+\frac32}}\\
	\lesssim& F_{0,2}^{\frac34}(t)\nu^{-\frac34}q^{-\frac16}\delta\nu^{c-\frac13}F_{0,1}^{\frac14}(t)\nu^{-\frac14}E_{0,1}^{\frac14}(t).
\end{align*}
By using \eqref{000}, we can return to the estimates of $\tilde{U}^1_{0,2}$ and $\tilde{\Theta}_{0,2}$. Integrating the above inequality respect to $t$  immediately gives \eqref{L55}. Thus, the proof of Proposition \ref{propE02} is finished.\hfill$\square$

\subsection{  Energy estimates on $Q_{\neq,2}, K_{\neq,2}$ and $H_{\neq,2}$}\label{sub4.6} 
 \qquad  In this section, we  aim to   prove that under the bootstrap hypotheses of Proposition \ref{boots}, the estimates on $U_{\neq,2} $, and $ \Theta_{\neq,2}$  hold (i.e., \eqref{qkh2}), with 100 replaced by 50 on the right-hand side.  
\begin{pro}\label{propEneq2}
Suppose that the assumptions in Proposition \ref{boots} are true, then it holds that for any $t\in [0,T]$,
	\begin{equation}
	\begin{aligned}\label{L66}
		&\sup_{t>0}E_{\neq,2}(t)+  \int_0^tF_{\neq,2}(\tau){\rm d}\tau\\
		&\quad\lesssim  \max\left\{1,\frac{\beta}{\beta-1}\right\}\left[\left(\delta\nu^{c-\frac56} + \delta(\nu^{b-1}+\nu^{a-1})+\delta\nu^{\frac{\kappa}{6(1+\kappa)}-\frac{1}{12}}\right) \int_0^tF_{\neq,2}(\tau){\rm d}\tau\right. \\
		&\qquad+ \left(\delta^2\nu^{2c-\frac43}+\delta^2\nu^{2b-\frac53-\frac{\kappa}{3(1+\kappa)}}+\delta^2\nu^{2a-2+\frac{\kappa}{3(1+\kappa)}} +  \delta\nu^{b-\frac23} + \delta^2\nu^{2c-1-\frac{\kappa}{3(1+ \kappa)}}+\delta^2\nu^{2a-\frac{4}{3}}\right)\int_0^t F_{\neq,1}(\tau){\rm d}\tau\\
		&\qquad \left.+ \left(\delta\nu^{-\frac16+\frac{\kappa}{3(1+\kappa)}} + \delta^2\nu^{2b-2}+\delta^2 \nu^{2c-\frac{11}{6}}+\delta^4\nu^{2a+2b-4}\right)\int_0^t F_{0,2}(\tau){\rm d}\tau\right],
	\end{aligned}
	\end{equation}
 where $ E_{\neq,2}(t)$ and $ F_{\neq,2}(t)$ are defined by \eqref{pre3} and \eqref{pre4}, respectively.
\end{pro}

\pf According to the linear analysis, an energy estimate yields
\begin{equation}
	\begin{aligned}\label{LLL}
		&\|\mathcal{B}(Q_{\neq,2},K_{\neq,2},H_{\neq,2})\|_{H^s}^2 + \nu\|\mathcal{B}\na_L(Q_{\neq,2},K_{\neq,2},H_{\neq,2})\|_{L^2_tH^s}^2 + \left\|\mathcal{B}\sqrt{-\frac{\dot{M}}{M}}(Q_{\neq,2},K_{\neq,2},H_{\neq,2})\right\|_{L^2_tH^s}^2\\
		&\quad\lesssim  \frac{\nu^{\frac13}}{32}\|\mathcal{B}(Q_{\neq,2},K_{\neq,2},H_{\neq,2})\|_{L^2_tH^s}^2+\frac{4\sqrt{\beta(\beta-1)}}{2\sqrt{\beta(\beta-1)}-1}\int_0^t\mathcal{N}_{\mathcal{B}1}{\rm d}\tau+\frac{4\sqrt{\beta(\beta-1)}}{2\sqrt{\beta(\beta-1)}-1}\int_0^tG\mathcal{N}_{\mathcal{B}2}{\rm d}\tau,
	\end{aligned}
\end{equation}
where
\begin{align*}
	\int_0^t\mathcal{N}_{\mathcal{B}1}{\rm d}\tau=&\int_0^t\langle  \mathcal{B}Q_{\neq,2},\text{Nonlinear terms in the equation of }Q_{\neq,2}\rangle_{H^s}{\rm d}\tau\\
	&+\int_0^t\langle \mathcal{B}K_{\neq,2},\text{Nonlinear terms in the equation of }K_{\neq,2}\rangle_{H^s}{\rm d}\tau,\\
	&+\int_0^t\langle  \mathcal{B}H_{\neq,2},\text{Nonlinear terms in the equation of }H_{\neq,2}\rangle_{H^s}{\rm d}\tau,
\end{align*} 
and $\int_0^tG\mathcal{N}_{\mathcal{B}2}{\rm d}\tau$ is the associated  nonlinear cross terms
\begin{align*}
	\int_0^tG\mathcal{N}_{\mathcal{B}2}{\rm d}\tau=&\int_0^t\langle G\mathcal{B}K_{\neq,2},\text{Nonlinear terms in the equation of }Q_{\neq,2}\rangle_{H^s}{\rm d}\tau\\
	&+\int_0^t\langle G\mathcal{B}Q_{\neq,2},\text{Nonlinear terms in the equation of }K_{\neq,2}\rangle_{H^s}{\rm d}\tau.
\end{align*} 

We are ready to estimate each term on the right-hand side \eqref{LLL} as follows.

\textit{4.5.1 Contribution from two external force terms:} 
Using \eqref{3x.1}, \eqref{ll7} and \eqref{multi1}, we have
\begin{align*}
	&\big|\langle\mathcal{B}Q_{\neq,2},\mathcal{B}|\na_L|^{\frac32}|\na_{L,h}|^{-1}\big[U_{\neq,1}^2\p_y^L U_{\neq,1}^3+U_{\neq,1}^3\p_zU_{\neq,1}^3\big]\rangle_{H^s}\big|\\
    &\quad \lesssim\|\mathcal{B}Q_{\neq,2}\|_{H^{s}}\nu^{-\frac{1}{6}}\langle t\rangle^{\frac12}e^{-\frac{\lambda}{2}\nu^{\frac{1}{3}}t}\left\| |\nabla|^{\frac{5}{2}}\left(|\nabla|^{\frac{1}{2}}|\nabla|\langle t\rangle^{-1}e^{\lambda\nu^{\frac{1}{3}}t}m(K_{\neq,1},Q_{\neq,1})|\nabla|\langle t\rangle mQ_{\neq,1} \right) \right\|_{H^{s}}\\&\qquad+\|BQ_{\neq,2}\|_{H^{s}}\nu^{-\frac{1}{6}}\langle t\rangle^{\frac12}e^{-\frac{\lambda}{2}\nu^{\frac{1}{3}}t}\left\||\nabla|^{\frac{5}{2}}\left(mQ_{\neq,1}e^{\lambda\nu^{\frac{1}{3}}t}|\nabla|mQ_{\neq,1}\right)\right\|_{H^{s}}\\&\quad\lesssim
    \|BQ_{\neq,2}\|_{H^{s}}\nu^{-\frac{1}{6}}\nu^{-\frac{1}{6}}\|\mathcal{A}(K_{\neq,1},Q_{\neq,1})\|_{H^{s+4}}\|\mathcal{A}Q_{\neq,1}\|_{H^{s+\frac{7}{2}}} \\&\quad\lesssim (\delta \nu^{b})\nu^{-\frac{1}{6}}\nu^{-\frac{1}{6}}F_{\neq,1}\nu^{-\frac{1}{6}}F_{\neq,1}\nu^{-\frac{1}{6}}\\&\quad\lesssim\delta\nu^{b-\frac{2}{3}}F_{\neq,1},
\end{align*}
where we have used   $mp^{\frac34}p_h^{-\frac12}\lesssim\nu^{-\frac16}|k, \eta, l|^{\frac52}\langle t\rangle^{\frac12}$ due to the fact $p\lesssim \langle t\rangle^{2}|k,\eta,l|^{2}$ and $p^{-1}\lesssim \langle t\rangle^{-2}|k,\eta,l|^{2}$.
For the term { $U_{0,1}^2\p_y^LU_{\neq,1}^3$}, when $|l|\leq\sqrt{p_h}$, we use the fact that $p^{\frac34}p_h^{-\frac12}\lesssim p_h^{\frac14}$ and  \eqref{4.11-3} to obtain
{\begin{align*}
	&\big|\langle\mathcal{B}Q_{\neq,2},\mathcal{B}|\na_L|^{\frac32}|\na_{L,h}|^{-1}(U_{0,1}^2\p_y^LU_{\neq,1}^3)\rangle_{H^s}\big|\\
	&\quad\leq \sum_{\substack{k,l\\k',l'}}\iint_{\eta,\eta'}\widehat{\langle\na\rangle^s}\widehat{\mathcal{B}Q_{\neq,2}}(k,\eta,l)e^{\frac{\lambda}{2}\nu^{\frac13}\tau}mp^{\frac34}p_h^{-\frac12}(t,k,\eta,l)\widehat{\langle\na\rangle^s}\big[\widehat{U_{0,1}^2}(0,\eta-\eta',l-l') |\widehat{|\na_L|^{\frac12}Q_{\neq,1}}(k,\eta',l')|\big]{\rm d}\eta{\rm d}\eta'\\
	&\quad\leq \sum_{\substack{k,l\\k',l'}}\iint_{\eta,\eta'}\widehat{\langle\na\rangle^s}\widehat{\mathcal{B}|\na_L|^{\frac12}Q_{\neq,2}}(k,\eta,l)e^{\frac{\lambda}{2}\nu^{\frac13}\tau}\widehat{\langle\na\rangle^s}\big[\widehat{|\na|^{\frac12}U_{0,1}^2}(0,\eta-\eta',l-l') |\widehat{|\na_L|^{\frac12}mQ_{\neq,1}}(k,\eta',l')|\big]{\rm d}\eta{\rm d}\eta'\\
	&\quad\lesssim F_{\neq,2}^{\frac12}(t)\nu^{-\frac13}\|U_{0,1}^2\|_{H^{s+\frac12}}F_{\neq,1}^{\frac12}(t)\nu^{-\frac13} \\
	&\quad\lesssim \delta F_{\neq,2}(t) + \delta\nu^{2c-\frac43} F_{\neq,1}(t),
\end{align*}}
when $|l|>\sqrt{p_h}$, by using \eqref{4.11-3} and \eqref{3x.1}  we have 
{\begin{align*}
	&\big|\langle\mathcal{B}Q_{\neq,2},\mathcal{B}|\na_L|^{\frac32}|\na_{L,h}|^{-1}(U_{0,1}^2\p_y^LU_{\neq,1}^3)\rangle_{H^s}\big|\\
&\quad\lesssim \big|\langle\mathcal{B}|\na_{L,h}|^{-1}Q_{\neq,2},e^{\frac{\lambda}{2}\nu^{\frac13}t}\big[|\p_z|^{\frac32}U_{0,1}^2|\p_z|^{2}m|\na_L|^{\frac12}Q_{\neq,1}\big]\rangle_{H^s}\big| \\
&\quad\lesssim F_{\neq,2}^{\frac12}(t)\|U_{0,1}^2\|_{H^{s+2}}F_{\neq,1}^{\frac12}(t)\nu^{-\frac13}\\&\quad\leq \delta F_{\neq,2}(t)+ \delta\nu^{2c-\frac{2}{3}}F_{\neq,1}(t).
\end{align*}}

Due to \eqref{6666},   \eqref{3x.1} we have 
\begin{align*}
	&\big|\langle\mathcal{B}Q_{\neq,2},\mathcal{B}|\na_L|^{\frac32}|\na_{L,h}|^{-1}(U_{0,1}^3\p_zU_{\neq,1}^3)\rangle_{H^s}\big|\\
	&\quad\lesssim \big|\langle\mathcal{B}\na_{L}Q_{\neq,2},e^{\frac{\lambda}{2}\nu^{\frac13}t}\big[\langle\p_z\rangle^{\frac12}U_{0,1}^2\langle\p_z\rangle mQ_{\neq,1}\big]\rangle_{H^s}\big|\\
	&\quad\lesssim F_{\neq,2}^{\frac12}(t)\nu^{-\frac12}E_{0,1}^{\frac12}(t)F_{\neq,1}^{\frac12}(t)\nu^{-\frac16}\\&\quad\leq \delta F_{\neq,2}(t) + \delta\nu^{2c-\frac43}F_{\neq,1}(t).
\end{align*}
Direct calculation shows
\begin{align*}
	&\big|\langle\mathcal{B}Q_{\neq,2},\mathcal{B}|\na_L|^{\frac32}|\na_{L,h}|^{-1}\big[U_{\neq,1}^{2,3}\cdot\na_{y,z}U_{0,1}^3\big]\rangle_{H^s}\big|\\
	&\quad\lesssim F^{\frac12}_{\neq,2}(t)\nu^{-\frac16}\cdot\nu^{-\frac16}\cdot\langle t\rangle^{\frac12}e^{-\frac{\lambda}{2}\nu^{\frac13}t}F^{\frac12}_{\neq,1}(t)\nu^{-\frac16}(\delta\nu^c)\\
	&\quad\lesssim \delta F_{\neq,2}(t) + \delta\nu^{2c-\frac43}F_{\neq,1}(t),
\end{align*}
and
\begin{align*}
	&\big|\langle\mathcal{B}Q_{\neq,2}, \mathcal{B}\p_z|\na_{L,h}|^{-1}|\na_L|^{-\frac12}\big[\p_zU_{\neq,1}^2\p_y^LU_{\neq,1}^3+\p_yU_{0,1}^3\p_zU_{\neq,1}^2+\p_iU_{0,1}^2\p_y^LU_{\neq,1}^i\big]\rangle_{H^s}\big|\\
	&\quad\lesssim F^{\frac12}_{\neq,2}(t)(\delta\nu^c)F^{\frac12}_{\neq,1}(t)\nu^{-\frac12}\\
	&\quad\lesssim \delta F_{\neq,2}(t) + \delta\nu^{2c-1}F_{\neq,1}(t).
\end{align*}
where we have used the fact $mp^{\frac{3}{4}}p_{h}^{-\frac{1}{2}}\lesssim m|\hat{\nabla}|^{\frac{5}{2}}\langle t \rangle^{\frac{1}{2}}$ and $m|l|p_h^{-\frac{1}{2}}p^{-\frac{1}{4}}\lesssim \frac{\sqrt{|l|}}{\sqrt{p_{h}}}.$

We now focus on the equation of $K_{\neq,2}$. 
For the term { $U_{0,1}^2\p_y^LW_{\neq,1}^3$}, if $|l|\leq\sqrt{p_h}$, using \eqref{yyy}, \eqref{4.11-3} and $p^{\frac14}p_h^{-\frac12}\lesssim p_h^{-\frac14}$ and integrating by parts, one has
{\begin{align*}
	&\sqrt{\frac{\beta}{\beta-1}}\big|\langle\mathcal{B}K_{\neq,2},\mathcal{B}|\na_L|^{\frac12}|\na_{L,h}|^{-1}(U_{0,1}^2\p_y^LW_{\neq,1}^3)\rangle_{H^s}\big|\\
	&\leq \sqrt{\frac{\beta}{\beta-1}}(\big|\langle\mathcal{B}K_{\neq,2},\mathcal{B}|\na_L|^{\frac12}(U_{0,1}^2W_{\neq,1}^3)\rangle_{H^s}\big| + \big|\langle\mathcal{B}K_{\neq,2},\mathcal{B}|\na_L|^{\frac12}|\na_{L,h}|^{-1}(\p_yU_{0,1}^2W_{\neq,1}^3)\rangle_{H^s}\big|)\\
	&\leq \sum_{\substack{k,l\\k',l'}}\iint_{\eta,\eta'}\widehat{\langle\na\rangle^s}\widehat{\mathcal{B}K_{\neq,2}}(k,\eta,l)e^{\lambda\nu^{\frac13}\tau}mp^{\frac14}(t,k,\eta,l)\widehat{\langle\na\rangle^s}\big[\widehat{U_{0,1}^2}(0,\eta-\eta',l-l') |\widehat{|\na_L|^{\frac12}K_{\neq,1}}(k,\eta',l')|\big]{\rm d}\eta{\rm d}\eta'\\
	&\quad+ \sum_{\substack{k,l\\k',l'}}\iint_{\eta,\eta'}\widehat{\langle\na\rangle^s}\widehat{\mathcal{B}K_{\neq,2}}(k,\eta,l)e^{\lambda\nu^{\frac13}\tau}mp^{\frac14}p_h^{-\frac12}(t,k,\eta,l)\widehat{\langle\na\rangle^s}\big[\widehat{\p_yU_{0,1}^2}(0,\eta-\eta',l-l') |\widehat{|\na_L|^{\frac12}K_{\neq,1}}(k,\eta',l')|\big]{\rm d}\eta{\rm d}\eta'\\
	&\leq \sum_{\substack{k,l\\k',l'}}\iint_{\eta,\eta'}\widehat{\langle\na\rangle^s}\widehat{\mathcal{B}|\na_L|^{\frac12}K_{\neq,2}}(k,\eta,l)e^{\lambda\nu^{\frac13}\tau}\widehat{\langle\na\rangle^s}\big[\widehat{|\na_{y,z}|^{\frac12}U_{0,1}^2}(0,\eta-\eta',l-l') |\widehat{|\na_L|^{\frac12}mK_{\neq,1}}(k,\eta',l')|\big]{\rm d}\eta{\rm d}\eta'\\
	&\quad + \sum_{\substack{k,l\\k',l'}}\iint_{\eta,\eta'}\widehat{\langle\na\rangle^s}\widehat{\mathcal{B}|\na_{L,h}|^{-\frac12}K_{\neq,2}}(k,\eta,l)e^{\lambda\nu^{\frac13}\tau}\widehat{\langle\na\rangle^s}\big[\widehat{|\na_{y,z}|^{\frac12}\p_yU_{0,1}^2}(0,\eta-\eta',l-l') |\widehat{|\na_L|^{\frac12}mK_{\neq,1}}(k,\eta',l')|\big]{\rm d}\eta{\rm d}\eta'\\
	&\lesssim F_{\neq,2}^{\frac12}(t)\nu^{-\frac13}\|U_{0,1}^2\|_{H^{s+\frac12}}F_{\neq,1}^{\frac12}(t)\nu^{-\frac13} + F_{\neq,2}^{\frac12}(t)\nu^{-\frac{\kappa}{6(1+\kappa)}}\|U_{0,1}^2\|_{H^{s+\frac32}}F_{\neq,1}^{\frac12}(t)\nu^{-\frac13}\\
	&\lesssim \delta F_{\neq,2}(t) + \delta(\nu^{2c-\frac43}+\nu^{2c-\frac23-\frac{\kappa}{3(1+\kappa)}}) F_{\neq,1}(t),
\end{align*}}
if $|l|>\sqrt{p_h}$, using the fact that $mp^{\frac14}p_h^{-\frac12}\lesssim \frac{|l|^{\frac12}}{\sqrt{p_h}}$ and \eqref{yyy} and  integrating by parts yields that
{\begin{align*}
	&{\sqrt{\frac{\beta}{\beta-1}}\big|\langle \mathcal{B}K_{\neq,2},\mathcal{B}|\na_L|^{\frac12}|\na_{L,h}|^{-1}\big[U_{0,1}^2\p_y^LW_{\neq,1}^3\big]\rangle_{H^s}\big|}\\
	&\quad\leq\big|\langle \mathcal{B}|\na_{L,h}|^{-1}K_{\neq,2},e^{\frac{\lambda}{2}\nu^{\frac13}t}|\p_z|^{\frac12}\big[U_{0,1}^2\p_y^L|\na_L|^{-\frac12}|\na_{L,h}|mK_{\neq,1}\big]\rangle_{H^s}\big|\nu^{-\frac16}\\
	&\quad\leq e^{\frac{\lambda}{2}\nu^{\frac13}t}\nu^{-\frac16}\big(\big|\langle \mathcal{B}K_{\neq,2},\big[|\p_z|^{\frac12}U_{0,1}^2|\p_z|^{\frac12}m|\na_L|^{\frac12}K_{\neq,1}\big]\rangle_{H^s}\big| \\&\qquad  + \big|\langle \mathcal{B}|\na_{L,h}|^{-1}K_{\neq,2},\big[|\p_z|^{\frac12}\p_yU_{0,1}^2|\p_z|^{\frac12}m|\na_L|^{\frac12}K_{\neq,1}\big]\rangle_{H^s}\big|\big)\\&
	\quad\lesssim \nu^{-\frac16}F_{\neq,2}^{\frac12}(t)\nu^{-\frac16}\|U_{0,1}^2\|_{H^{s+\frac12}}F_{\neq,1}^{\frac12}(t)\nu^{-\frac13} + \nu^{-\frac16}F_{\neq,2}^{\frac12}(t)\|U_{0,1}^2\|_{H^{s+\frac32}}F_{\neq,1}^{\frac12}(t)\nu^{-\frac13}\\& \quad\lesssim \delta F_{\neq,2}(t) + \delta F_{\neq,1}(t)\nu^{2c-\frac43}.
\end{align*}}

By using $mp^{\frac14}p_h^{-\frac12}\lesssim p^\frac{1}{2}p_{h}^{-\frac{1}{2}} \lesssim|\hat{\na}|^2 $, one has
\begin{align*}
	&\sqrt{\frac{\beta}{\beta-1}}\big|\langle \mathcal{B}K_{\neq,2},\mathcal{B}|\na_L|^{\frac12}|\na_{L,h}|^{-1}\big[U_{\neq,1}^{2,3}\cdot\na_{y,z}\p_yU_{0,1}^1\big]\rangle_{H^s}\big|\\&\quad
	\lesssim \sqrt{\frac{\beta}{\beta-1}}F^{\frac12}_{\neq,2}(t)\nu^{-\frac16}F^{\frac12}_{\neq,1}(t)\nu^{-\frac16}(\delta\nu^c)\\&\quad \lesssim \delta F_{\neq,2}(t)+ \delta\frac{\beta}{\beta-1}\nu^{2c-\frac23}F_{\neq,1}(t),
\end{align*}
and
\begin{align*}
	&\sqrt{\frac{\beta}{\beta-1}}\big|\langle \mathcal{B}K_{\neq,2},\mathcal{B}|\na_L|^{\frac12}|\na_{L,h}|^{-1}\big[\na_{y,z}U_{0,1}\p_y^L U_{\neq,1}\big]\rangle_{H^s}\big|\\&\quad
	\lesssim  \sqrt{\frac{\beta}{\beta-1}}F^{\frac12}_{\neq,2}(t)\nu^{-\frac16}(\delta\nu^c)F^{\frac12}_{\neq,1}(t)\nu^{-\frac12}\\&\quad\lesssim \delta F_{\neq,2}(t) + \delta\frac{\beta}{\beta-1}\nu^{2c-\frac43}F_{\neq,1}(t).
\end{align*}

We now focus on the equation of $H_{\neq,2}$. For the term { $U_{0,1}^2\p_y^L\Theta_{\neq,1}$}, we use \eqref{yyy},  \eqref{4.11-3} to obtain
{\begin{align*}
	&\big|\langle\mathcal{B}H_{\neq,2},\mathcal{B}|\na_L|^{\frac12}(U_{0,1}^2\p_y^L\Theta_{\neq,1})\rangle_{H^s}\big|\\
	&\quad\leq \sum_{\substack{k,l\\k',l'}}\iint_{\eta,\eta'}\widehat{\langle\na\rangle^s}\widehat{\mathcal{B}H_{\neq,2}}(k,\eta,l)e^{\lambda\nu^{\frac13}\tau}mp^{\frac14}(t,k,\eta,l)\widehat{\langle\na\rangle^s}\big[\widehat{U_{0,1}^2}(0,\eta-\eta',l-l') |\widehat{|\na_L|^{\frac12}H_{\neq,1}}(k,\eta',l')|\big]{\rm d}\eta{\rm d}\eta'\\
	&\quad\leq \sum_{\substack{k,l\\k',l'}}\iint_{\eta,\eta'}\widehat{\langle\na\rangle^s}\widehat{\mathcal{B}|\na_L|^{\frac12}H_{\neq,2}}(k,\eta,l)e^{\lambda\nu^{\frac13}\tau}\widehat{\langle\na\rangle^s}\big[\widehat{|\na_{y,z}|^{\frac12}U_{0,1}^2}(0,\eta-\eta',l-l') |\widehat{|\na_L|^{\frac12}mH_{\neq,1}}(k,\eta',l')|\big]{\rm d}\eta{\rm d}\eta'\\
	&\quad\lesssim F_{\neq,2}^{\frac12}(t)\nu^{-\frac13}\|U_{0,1}^2\|_{H^{s+\frac12}}F_{\neq,1}^{\frac12}(t)\nu^{-\frac13}\\&\quad \lesssim \delta F_{\neq,2}(t) + \delta\nu^{2c-\frac43} F_{\neq,1}(t).
\end{align*}}

Due to  $mp^{\frac14}\lesssim p^{\frac12} $ and \eqref{ll7},  direct calculation yields
\begin{align*}
	\big|\langle \mathcal{B}H_{\neq,2},\mathcal{B}|\na_L|^{\frac12}\big[U_{\neq,1}^2\p_y\Theta_{0,1}\big]\rangle_{H^s}\big|\lesssim F^{\frac12}_{\neq,2}(t)\nu^{-\frac12}\|U_{\neq,1}^2\|_{H^s}(\delta\nu^c)
	\lesssim \delta F_{\neq,2}(t) + \delta F_{\neq,1}(t)\nu^{2c-1}.
\end{align*}

Recalling the system \eqref{Non0NL2}, we have already completed the estimation of the nonlinear terms for the two external force terms.

\textit{4.5.2 Contribution with only one external force term}
We consider the contribution of only one external force term for the second system. Focus on the equation of $Q_{\neq,2}$. Thanks to the incompressible condiction,  the fact    $p_h(k',\eta')^{\frac14}\lesssim p_h(k,\eta)^{\frac14}p_h(k-k',\eta-\eta')^{\frac14}$, \eqref{multi1} and \eqref{pro1}, we have
\begin{align*}
	&\left|\langle\mathcal{B}Q_{\neq,2},\mathcal{B}|\na_L|^{\frac32}|\na_{L,h}|^{-1}\big[U_{\neq,1}^3\p_zU_{\neq,2}^3\big]\rangle_{H^s}\right|\\
	&\quad\leq \sum_{k,k',l,l'}\int_{\eta,\eta'}\langle\hat{\na}\rangle^s|\widehat{\mathcal{B}Q_{\neq,2}}(k,\eta,l)|e^{\frac{\lambda}{2}\nu^{\frac13}t}\frac{|l|^{\frac32}}{\sqrt{p_h}}\langle\hat{\na}\rangle^s(|\widehat{U_{\neq,1}^3}(k-k',\eta-\eta',l-l')||\widehat{\na_{L,h}\cdot U_{\neq,2}^{1,2}}(k',\eta',l')|){\rm d}\eta{\rm d}\eta'\\
	&\qquad+ \sum_{k,k',l,l'}\int_{\eta,\eta'}\langle\hat{\na}\rangle^s|\widehat{\mathcal{B}Q_{\neq,2}}(k,\eta,l)|e^{\frac{\lambda}{2}\nu^{\frac13}t}\frac{\sqrt{p_h}}{\sqrt{|k,l|}}\langle\hat{\na}\rangle^s(|\widehat{U_{\neq,1}^3}(k-k',\eta-\eta',l-l')||l'||\widehat{U_{\neq,2}^3}(k',\eta',l')|){\rm d}\eta{\rm d}\eta'\\
	&\quad\lesssim \sum_{k,k',l,l'}\int_{\eta,\eta'}\langle\hat{\na}\rangle^s|l||\widehat{\mathcal{B}Q_{\neq,2}}(k,\eta,l)|\langle\hat{\na}\rangle^s\frac{|l|^{\frac12}e^{\frac{\lambda}{2}\nu^{\frac13}t}}{p_h^{\frac14}(k,\eta)}(|\widehat{|\na_{L,h}|^{\frac12}U_{\neq,1}^3}(k-k',\eta-\eta',l-l')|(|\widehat{|\p_x|^{\frac12}U_{\neq,2}^1}| \\
	&\qquad + |\widehat{|\p_y^L|^{\frac12}U_{\neq,2}^2}|)(k',\eta',l')){\rm d}\eta{\rm d}\eta' + \sum_{k,k',l,l'}\int_{\eta,\eta'}\langle\hat{\na}\rangle^s|\widehat{\mathcal{B}Q_{\neq,2}}(k,\eta,l)|e^{\frac{\lambda}{2}\nu^{\frac13}t}\frac{\sqrt{p_h}}{\sqrt{|k,l|}}\langle\hat{\na}\rangle^s\\
	&\qquad\times(|\widehat{U_{\neq,1}^3}(k-k',\eta-\eta',l-l')||l'||\widehat{U_{\neq,2}^3}(k',\eta',l')|){\rm d}\eta{\rm d}\eta'\\
	&\quad\lesssim F^{\frac12}_{\neq,2}(t)\nu^{-\frac12}(\delta\nu^{\frac56+\frac{\kappa}{6(1+\kappa)}})F^{\frac12}_{\neq,2}(t)\nu^{-\frac14} + F^{\frac12}_{\neq,2}(t)\nu^{-\frac12}\langle t\rangle^{\frac12} e^{-\frac{\lambda}{2}\nu^{\frac13}t}(\delta\nu^{\frac56+\frac{\kappa}{6(1+\kappa)}})F^{\frac12}_{\neq,2}(t)\nu^{-\frac14} \\
	&\qquad+ F^{\frac12}_{\neq,2}(t)\nu^{-\frac12}(\delta\nu^{\frac56+\frac{\kappa}{6(1+\kappa)}})F^{\frac12}_{\neq,2}(t)\nu^{-\frac13}\\
	&\quad\lesssim \delta F_{\neq,2}(t)[\nu^{\frac{\kappa}{6(1+\kappa)}-\frac{1}{12}} + \nu^{\frac{\kappa}{6(1+\kappa)}}],
	\end{align*}
	and
	\begin{align*}
	&\left|\langle\mathcal{B}Q_{\neq,2},\mathcal{B}|\na_L|^{\frac32}|\na_{L,h}|^{-1}\big[U_{\neq,1}^h\cdot\na_{L,h}U_{\neq,2}^3\big]\rangle_{H^s}\right|\\
	&\leq \sum_{k,k',l,l'}\int_{\eta,\eta'}\langle\hat{\na}\rangle^s|\widehat{\mathcal{B}Q_{\neq,2}}(k,\eta,l)|e^{\frac{\lambda}{2}\nu^{\frac13}t}\frac{l^{\frac32}}{\sqrt{p_h}}\langle\hat{\na}\rangle^s(|\widehat{U_{\neq,1}^h}(k-k',\eta-\eta',l-l')|p_h^{\frac12}(k',\eta')|\widehat{U_{\neq,2}^3}(k',\eta',l')|){\rm d}\eta{\rm d}\eta'\\
	&\quad+ \sum_{k,k',l,l'}\int_{\eta,\eta'}\langle\hat{\na}\rangle^s|\widehat{\mathcal{B}Q_{\neq,2}}(k,\eta,l)|e^{\frac{\lambda}{2}\nu^{\frac13}t}\nu^{-\frac16}p_h^{\frac14}(k,\eta)\langle\hat{\na}\rangle^s(|\widehat{U_{\neq,1}^1}(k-k',\eta-\eta',l-l')|k'|\widehat{U_{\neq,2}^3}(k',\eta',l')|){\rm d}\eta{\rm d}\eta'\\
	&\quad+ \sum_{k,k',l,l'}\int_{\eta,\eta'}\langle\hat{\na}\rangle^s|\widehat{\mathcal{B}Q_{\neq,2}}(k,\eta,l)|e^{\frac{\lambda}{2}\nu^{\frac13}t}\nu^{-\frac16}p_h^{\frac14}(k,\eta)\langle\hat{\na}\rangle^s(|\widehat{U_{\neq,1}^2}(k-k',\eta-\eta',l-l')|\eta'-k't|\widehat{U_{\neq,2}^3}(k',\eta',l')|){\rm d}\eta{\rm d}\eta'\\
	&\lesssim F^{\frac12}_{\neq,2}(t)\nu^{-\frac12}(\delta\nu^{\frac56+\frac{\kappa}{6(1+\kappa)}})F^{\frac12}_{\neq,2}(t)\nu^{-\frac16} + F^{\frac12}_{\neq,2}(t)\nu^{-\frac{\kappa}{6(1+\kappa)}}\langle t\rangle e^{-\frac{\lambda}{2}\nu^{\frac13}t}(\delta\nu^{\frac56+\frac{\kappa}{6(1+\kappa)}})F^{\frac12}_{\neq,2}(t)\nu^{-\frac12} \\
	&\quad+ F^{\frac12}_{\neq,2}(t)\nu^{-\frac12}(\delta\nu^{\frac56+\frac{\kappa}{6(1+\kappa)}})F^{\frac12}_{\neq,2}(t)\nu^{-\frac13} + F^{\frac14}_{\neq,2}(t)\nu^{-\frac14}\nu^{-\frac16}(\delta\nu^b)^{\frac12}F^{\frac12}_{\neq,1}(t)\nu^{-\frac{\kappa}{6(1+\kappa)}}(\delta\nu^b)^{\frac12}F^{\frac14}_{\neq,2}(t)\nu^{-\frac14}\cdot\nu^{-\frac16}\\& 
	\lesssim \delta F_{\neq,2}(t) + \delta\nu^{2b-\frac53-\frac{\kappa}{3(1+\kappa)}}F_{\neq,1}(t),
\end{align*}
where we have used the following inequality
\begin{align*}
|\widehat{\partial_{y}U_{\neq,2}^{3}}|\lesssim |\hat{|\partial_{y}|}\widehat{|\nabla_{L}|}^{-\frac{3}{2}}\widehat{|\nabla_{L,h}|}\widehat{|mQ_{\neq,2}|}|m^{-1}\lesssim |\widehat{m|\nabla_{L}|^{\frac{1}{2}}Q_{\neq,2}}|\nu^{-\frac{1}{6}}.
\end{align*}

Similarly, by \eqref{6666} we have
\begin{align*}
	&\big|\langle\mathcal{B}Q_{\neq,2},\mathcal{B}|\na_L|^{\frac32}|\na_{L,h}|^{-1}\big[U_{\neq,2}^3\p_zU_{\neq,1}^3\big]\rangle_{H^s}\big|\\
	&\quad\lesssim F^{\frac12}_{\neq,2}(t)\nu^{-\frac12}F^{\frac12}_{\neq,2}(t)\nu^{-\frac16}(\delta\nu^{\frac56+\frac{\kappa}{6(1+\kappa)}})\\&\quad=\delta F_{\neq,2}(t)\nu^{\frac16+\frac{\kappa}{6(1+\kappa)}}, 
	\end{align*}
	Similarly, since $\sqrt{|\eta'-k't|}\lesssim \sqrt{|\eta-\eta'-(k-k')t|}p_h^{\frac14}(k,\eta)$, we have
	\begin{align*}
	& \big|\langle\mathcal{B}Q_{\neq,2},\mathcal{B}|\na_L|^{\frac32}|\na_{L,h}|^{-1}\big[U_{\neq,2}^1\p_xU_{\neq,1}^3+U_{\neq,2}^2\p_y^LU_{\neq,1}^3\big]\rangle_{H^s}\big|\\
	&\quad\lesssim F^{\frac12}_{\neq,2}(t)\nu^{-\frac12}F^{\frac12}_{\neq,2}(t)\nu^{-\frac16}(\delta\nu^{\frac56+\frac{\kappa}{6(1+\kappa)}})\\
	&\qquad+ \sum_{k,k',l,l'}\int_{\eta,\eta'}\Big[ \langle\hat{\na}\rangle^s|\widehat{\mathcal{B}Q_{\neq,2}}(k,\eta,l)|\frac{|l|^{\frac32}e^{\frac{\lambda}{2}\nu^{\frac13}t}}{p_h^{\frac14}(k,\eta)}\langle\hat{\na}\rangle^s(|\widehat{|\p_y^L|^{\frac12}U_{\neq,2}^2}(k-k',\eta-\eta',l-l')|\eta'-k't|^{\frac12}\widehat{U_{\neq,1}^3}(k',\eta',l')|)\\
	&\qquad+ \langle\hat{\na}\rangle^s|\widehat{\mathcal{B}Q_{\neq,2}}(k,\eta,l)|e^{\frac{\lambda}{2}\nu^{\frac13}t}\frac{\sqrt{p_h}}{\sqrt{|k,l|}}(k,\eta)\langle\hat{\na}\rangle^s(|\widehat{U_{\neq,2}^2}(k-k',\eta-\eta',l-l')|\eta'-k't|\widehat{U_{\neq,1}^3}(k',\eta',l')|)\Big]{\rm d}\eta{\rm d}\eta'\\
	&\quad\lesssim \delta\nu^{\frac16+\frac{\kappa}{6(1+\kappa)}}F_{\neq,2}(t) + F^{\frac12}_{\neq,2}(t)\nu^{-\frac12}\left(F^{\frac12}_{\neq,2}(t)\nu^{-\frac14}\delta\nu^{\frac56+\frac{\kappa}{6(1+\kappa)}}\langle t\rangle^{\frac12}e^{-\frac{\lambda}{2}\nu^{\frac13}t} + F^{\frac12}_{\neq,2}(t)\nu^{-\frac{\kappa}{6(1+\kappa)}}\nu^{-\frac13}(\delta\nu^{\frac56+\frac{\kappa}{6(1+\kappa)}})\right)\\
	&\quad\lesssim F_{\neq,2}(t)\delta[\nu^{\frac{\kappa}{6(1+\kappa)}-\frac{1}{12}} + \nu^{\frac{\kappa}{6(1+\kappa)}} + 1],
	\end{align*}
	and 
	\begin{align*}
	\big|\langle\mathcal{B}Q_{\neq,2},\mathcal{B}|\na_L|^{\frac32}|\na_{L,h}|^{-1}\big[U_{0,2}^{1,3}\cdot\na_{x,z}U_{\neq,1}^3\big]\rangle_{H^s}&\lesssim F^{\frac12}_{\neq,2}(t)\nu^{-\frac12}\cdot(\delta\nu^a)F^{\frac12}_{\neq,1}(t)\nu^{-\frac16}\\&\lesssim \delta F_{\neq,2}(t) + \delta\nu^{2a-\frac43}F_{\neq,1}(t).
\end{align*}

For term {$U_{0,2}^2\p_y^LU_{\neq,1}^3$}, if $|l|\leq\sqrt{p_h}$, we have
{\begin{align*}
	\big|\langle\mathcal{B}Q_{\neq,2},\mathcal{B}|\na_L|^{\frac32}|\na_{L,h}|^{-1}\big[U_{0,2}^2\p_y^LU_{\neq,1}^3\big]\rangle_{H^s}\big|&\lesssim \big|\langle\mathcal{B}|\na_L|^{\frac12}Q_{\neq,2},\big[|\na_{y,z}|^{\frac12}U_{0,2}^2\p_y^LmU_{\neq,1}^3\big]\rangle_{H^s}\big|\\
	&\lesssim F_{\neq,2}^{\frac12}(t)\nu^{-\frac13}F_{0,2}^{\frac14}(t)\nu^{-\frac14}E_{0,2}^{\frac14}(t)F_{\neq,1}^{\frac14}(t)\nu^{-\frac14}E_{\neq,1}^{\frac14}(t)\\&\lesssim \delta F_{\neq,2}(t) + \delta\nu^{a-\frac56 + \frac{\kappa}{6(1+\kappa)}}F_{0,2}^{\frac12}(t)F_{\neq,1}^{\frac12}(t).
\end{align*}}
If $|l|>\sqrt{p_h}$, then $m\leq 1$.   One gets
{\begin{align*}
	&\big|\langle\mathcal{B}Q_{\neq,2},\mathcal{B}|\na_L|^{\frac32}|\na_{L,h}|^{-1}\big[U_{0,2}^2\p_y^LU_{\neq,1}^3\big]\rangle_{H^s}\big|\\
	&\lesssim \sum_{k,l,l'}\int_{\eta,\eta'}\langle\hat{\na}\rangle^s|\left(\frac{p}{p_h}\right)^{\frac14}\widehat{\mathcal{B}Q_{\neq,2}}(k,\eta,l)|e^{\frac{\lambda}{2}\nu^{\frac13}t}\frac{|l||(k,\eta)|^{\frac12}}{\langle\tau\rangle^{\frac12}}\\&\quad \times \langle\hat{\na}\rangle^s(|\widehat{U_{0,2}^2}(0,\eta-\eta',l-l')\nu^{-\frac16}|\eta'-kt|\widehat{mU_{\neq,1}^3}(k,\eta',l')|){\rm d}\eta{\rm d}\eta'\\
	& \lesssim F_{\neq,2}^{\frac12}\nu^{-\frac14}\langle t\rangle^{-\frac12}\cdot F_{0,2}^{\frac12}\nu^{-\frac12}\nu^{-\frac16}\|\mathcal{A}Q_{\neq,1}\|_{H^{s+2}}\langle t\rangle^{\frac12}\\ &\lesssim \delta F_{\neq,2}(t) +\delta \nu^{\frac{\kappa}{3(1+\kappa)}-\frac16}F_{0,2}(t),
\end{align*}}
where we have used \begin{equation}\begin{aligned}\label{LL88} m p^{\frac34}p_h^{-\frac12}\lesssim p^{\frac14}\cdot |l|\cdot p_h^{-\frac14}p_{h}^{-\frac14}\lesssim  p^{\frac14}p_h^{-\frac14}\frac{|l|\cdot |k,\eta|^{\frac12}}{\langle \tau\rangle^{\frac12}},\end{aligned}\end{equation}
and
$$
\left\|\left(\frac{p}{p_{h}}\right)^{\frac{1}{4}}\widehat{\mathcal{B}Q_{\neq,2}}\right\|_{L^{2}} \leq \|\widehat{\mathcal{B}\nabla_{L}Q_{\neq,2}}\|_{L^{2}}^{\frac{1}{2}}\|\widehat{\mathcal{B}\nabla_{L}^{-1}Q_{\neq,2}}\|_{L^{2}}^{\frac{1}{2}}.
$$

The estimate of   {$U_{0,1}\cdot\na_L U_{\neq,2}^3$} is similar to the term $U_{0,1}^2\p_y^L U_{\neq,1}^3$, we have 
{\begin{align*}
	\big|\langle\mathcal{B}Q_{\neq,2},\mathcal{B}|\na_L|^{\frac32}|\na_{L,h}|^{-1}\big[U_{0,1}\cdot\na_L U_{\neq,2}^3\big]\rangle_{H^s}\big|&\lesssim F_{\neq,2}^{\frac12}\nu^{-\frac13}\|U_{0,1}\|_{H^{s+\frac32}}F_{\neq,2}^{\frac12}\nu^{-\frac13}\\&\lesssim \delta \nu^{c-\frac23}F_{\neq,2}(t).
\end{align*}}
Integrating by part and using \eqref{6666} we have 
\begin{align*}
	&\big|\langle\mathcal{B}Q_{\neq,2},\mathcal{B}|\na_L|^{\frac32}|\na_{L,h}|^{-1}\big[U_{\neq,1}^2\p_yU_{0,2}^3\big]\rangle_{H^s}\big|\\
	&\quad\lesssim \big|\langle\mathcal{B}Q_{\neq,2},\mathcal{B}|\na_L|^{\frac32}\big[U_{\neq,1}^2U_{0,2}^3\big]\rangle_{H^s}\big| + \big|\langle\mathcal{B}Q_{\neq,2},\mathcal{B}|\na_L|^{\frac32}|\na_{L,h}|^{-1}\big[\p_y^LU_{\neq,1}^2U_{0,2}^3\big]\rangle_{H^s}\big|\\
	&\quad\lesssim \nu^{-\frac16}F^{\frac12}_{\neq,2}(t)\nu^{-\frac12}[F^{\frac14}_{\neq,1}(t)(\delta\nu^{\frac56+\frac{\kappa}{6(1+\kappa)}})^{\frac12}][F^{\frac14}_{0,2}(t)\nu^{-\frac14}(\delta\nu^a)^{\frac12}] + F^{\frac12}_{\neq,2}(t)\nu^{-\frac12}F^{\frac12}_{\neq,1}(t)\nu^{-\frac16}(\delta\nu^a)\\
	&\quad\lesssim 2\delta F_{\neq,2}(t) + \delta\nu^{a-1+\frac{\kappa}{6(1+\kappa)}}F^{\frac12}_{\neq,1}(t)F^{\frac12}_{0,2}(t) + \delta\nu^{2a-\frac43}F_{\neq,1}(t),
\end{align*}
and 
\begin{align*}
	&\big|\langle\mathcal{B}Q_{\neq,2},\mathcal{B}|\na_L|^{\frac32}|\na_{L,h}|^{-1}\big[U_{\neq,2}^2\p_yU_{0,1}^3\big]\rangle_{H^s}\big|+\big|\langle\mathcal{B}Q_{\neq,2},\mathcal{B}|\na_L|^{\frac32}|\na_{L,h}|^{-1}\big[U_{\neq,2}^3\p_zU_{0,1}^3\big]\rangle_{H^s}\big|\\
	&\quad\lesssim F^{\frac12}_{\neq,2}(t)\nu^{-\frac12}F^{\frac12}_{\neq,2}(t)\nu^{-\frac14}(\delta\nu^c)\\
	&\quad=\delta \nu^{c-\frac34}F_{\neq,2}(t).
\end{align*}
The estimate of term {$U_{\neq,1}^3\p_zU_{0,2}^3$} is similar to $U_{0,2}^2\p_y^LU_{\neq,1}^3$. Integrating by parts give
{\begin{align*}
	&\big|\langle\mathcal{B}Q_{\neq,2},\mathcal{B}|\na_L|^{\frac32}|\na_{L,h}|^{-1}\big[U_{\neq,1}^3\p_zU_{0,2}^3\big]\rangle_{H^s}\big|\\
	&\quad\lesssim \big|\langle\mathcal{B}Q_{\neq,2},\mathcal{B}|\na_L|^{\frac32}|\na_{L,h}|^{-1}\big[\p_y^LU_{\neq,1}^3U_{0,2}^2\big]\rangle_{H^s}\big| + \big|\langle\mathcal{B}Q_{\neq,2},\mathcal{B}|\na_L|^{\frac32}\big[U_{\neq,1}^3U_{0,2}^2\big]\rangle_{H^s}\big|\\
	& \quad\lesssim\delta F_{\neq,2}(t) + \delta\nu^{a-\frac56 + \frac{\kappa}{6(1+\kappa)}}F_{0,2}^{\frac12}(t)F_{\neq,1}^{\frac12}(t) +\delta \nu^{\frac{\kappa}{3(1+\kappa)}-\frac16}F_{0,2}(t)\\&\qquad + \big|\langle\mathcal{B}\na_L Q_{\neq,2},e^{\frac{\lambda}{2}\nu^{\frac13}t}\big[|\na_L|^{\frac12}mU_{\neq,1}^3|\p_z|^{\frac12}|\na|^{\frac12}U_{0,2}\big]\rangle_{H^s}\big|\\
	&\quad\lesssim \delta F_{\neq,2}(t) + \delta\nu^{a-\frac56 + \frac{\kappa}{6(1+\kappa)}}F_{0,2}^{\frac12}(t)F_{\neq,1}^{\frac12}(t) + \delta\nu^{\frac{\kappa}{3(1+\kappa)}-\frac16}F_{0,2}(t)\\&\qquad +F_{\neq,2}^{\frac12}(t)\nu^{-\frac12}F_{\neq,1}^{\frac14}\nu^{-\frac{1}{12}}E_{\neq,1}^{\frac14}(t)F_{0,2}^{\frac14}\nu^{-\frac14}E_{0,2}^{\frac14}(t)\\
	&\quad\lesssim \delta F_{\neq,2}(t) + \delta\nu^{a-\frac56 + \frac{\kappa}{6(1+\kappa)}}F_{0,2}^{\frac12}(t)F_{\neq,1}^{\frac12}(t) + \delta\nu^{\frac{\kappa}{3(1+\kappa)}-\frac16}F_{0,2}(t),
\end{align*}}
where we have used \eqref{4.11-3} and the incompressible condition and 
$$ |m|\nabla_{L}|^{\frac{1}{2}} U_{\neq,1}^{3}|\leq |mQ_{\neq,1}|\leq|mQ_{\neq,1}|^{\frac{1}{2}} |mQ_{\neq,1}|^{\frac{1}{2}} .$$

For the pressure terms, using the fact that $\sqrt{|\eta'-k't|}\lesssim \sqrt{|\eta-kt|}p_h^{\frac14}(k',\eta')$, \eqref{multi4} and the incompressible condition, we have
\begin{align*}
	&\big|\langle\mathcal{B}Q_{\neq,2},\mathcal{B}\p_z|\na_L|^{-\frac12}|\na_{L,h}|^{-1}\big[\p_{x,z}U_{\neq,1}\cdot\na_L U_{\neq,2}^{1,3}+\p_y^LU_{\neq,1}^{1,2}\cdot\na_{L,h}U_{\neq,2}^2+\p_y^LU_{\neq,1}^3\p_zU_{\neq,2}^2\big]\rangle_{H^s}\big|\\
	&\quad\lesssim F^{\frac12}_{\neq,2}(t)(1+\langle t\rangle e^{-\frac{\lambda}{2}\nu^{\frac13}t})(\delta\nu^{\frac56+\frac{\kappa}{6(1+\kappa)}})F^{\frac12}_{\neq,2}(t)\nu^{-\frac12} + \big|\langle \mathcal{B}Q_{\neq,2},\mathcal{B}\p_z|\na_L|^{-\frac12}|\na_{L,h}|^{-\frac12}\big[|\p_y^L|^{\frac12}U_{\neq,1}^3\p_z|\na_{L,h}|^{\frac12} U_{\neq,2}^2\big] \rangle_{H^s}\big|\\
	&\quad\lesssim \delta\nu^{\frac{\kappa}{6(1+\kappa)}}F_{\neq,2}(t)+F_{\neq,2}^{\frac12}(t)\nu^{-\frac14}\langle t\rangle^{\frac12}e^{-\frac{\lambda}{2}\nu^{\frac13}t}(\delta\nu^{\frac56+\frac{\kappa}{6(1+\kappa)}})F_{\neq,2}^{\frac12}(t)\nu^{-\frac12}\\
	&\quad \lesssim \delta F_{\neq,2}(t)\nu^{\frac{\kappa}{6(1+\kappa)}-\frac{1}{12}},
	\end{align*}
	\begin{align*}
	&\big|\langle\mathcal{B}Q_{\neq,2},\mathcal{B}\p_z|\na_L|^{-\frac12}|\na_{L,h}|^{-1}\big[\p_iU_{0,1}^j\p_j^LU_{\neq,2}^i\big]\rangle_{H^s}\big|\\&\quad\lesssim F^{\frac12}_{\neq,2}(t)\nu^{-\frac14}(\delta\nu^c)F^{\frac12}_{\neq,2}(t)\nu^{-\frac12}=\delta\nu^{c-\frac34}F_{\neq,2}(t),
	\end{align*}
	and 
	\begin{align*}
	&\big|\langle\mathcal{B}Q_{\neq,2},\mathcal{B}\p_z|\na_L|^{-\frac12}|\na_{L,h}|^{-1}\big[\p_iU_{0,2}^j\p_j^LU_{\neq,1}^i\big]\rangle_{H^s}\big|\\&\quad
	\lesssim F^{\frac12}_{\neq,2}(t)F^{\frac12}_{0,2}(t)\nu^{-\frac12}\nu^{-\frac13}(\delta\nu^{\frac56+\frac{\kappa}{6(1+\kappa)}})\\&\quad\lesssim\delta F_{\neq,2}(t) + \delta\nu^{\frac{\kappa}{3(1+\kappa)}}F_{0,2}(t).
\end{align*}

Now we focus on the equation of $K_{\neq,2}$. \eqref{555} and $|k-k',l-l'|^{\frac12}\lesssim |k',l'|^{\frac14}|k,l|^{\frac14}$ yields that
\begin{align*}
	&\sqrt{\frac{\beta}{\beta-1}}\big|\langle \mathcal{B}K_{\neq,2},\mathcal{B}|\na_L|^{\frac12}|\na_{L,h}|^{-1}\big[\p_xU_{\neq,1}^{1,3}\cdot\na_{x,z} U^2_{\neq,2}+\p_xU_{\neq,1}^2\p_y^L U^2_{\neq,2}+\p_xU_{\neq,2}\cdot\na_L U^2_{\neq,1}\big]\rangle_{H^s}\big|\\
	&\quad\lesssim \sqrt{\frac{\beta}{\beta-1}}\left(\big|\langle \mathcal{B}K_{\neq,2},e^{\frac{\lambda}{2}\nu^{\frac13}t}|\na_L||\na_{L,h}|^{-1}\big[\p_x|\na_{x,z}|^{\frac12}U_{\neq,1}^{1,3}\cdot|\na_{x,z}|^{\frac12} U^2_{\neq,2}\big]\rangle_{H^s}\big|\right.\\
	&\qquad \left.+F_{\neq,2}^{\frac12}(t)\nu^{-\frac12}(\delta\nu^{\frac56+\frac{\kappa}{6(1+\kappa)}})F_{\neq,2}^{\frac12}(t)\nu^{-\frac13}+\nu^{-\frac16}(F_{\neq,2}^{\frac14}(t)\nu^{-\frac14}(\delta\nu^b)^{\frac12})^2F_{\neq,1}^{\frac12}(t)\right)\\
	&\quad \lesssim \sqrt{\frac{\beta}{\beta-1}}\left((1+\delta\nu^{\frac{\kappa}{6(1+\kappa)}})F_{\neq,2}(t) + \delta^2\nu^{2b-\frac43}F_{\neq,1}(t)\right),\\
	&\sqrt{\frac{\beta}{\beta-1}}\big|\langle \mathcal{B}K_{\neq,2},\mathcal{B}|\na_L|^{\frac12}|\na_{L,h}|^{-1}\big[\p_xU_{\neq,2}\cdot\na U_{0,1}^2-\p_y^LU_{\neq,2}\cdot\na U_{0,1}^1-\p_yU_{0,1}\cdot\na_LU^1_{\neq,2}\big]\rangle_{H^s}\big|\\
	&\quad\lesssim \sqrt{\frac{\beta}{\beta-1}}F^{\frac12}_{\neq,2}(t)\nu^{-\frac13}F^{\frac12}_{\neq,2}(t)\nu^{-\frac12}(\delta\nu^c)\\&\quad=\sqrt{\frac{\beta}{\beta-1}}\delta\nu^{c-\frac56} F_{\neq,2}(t),
	\end{align*}
	and
	\begin{align*}
	&\sqrt{\frac{\beta}{\beta-1}}\big|\langle \mathcal{B}K_{\neq,2},\mathcal{B}|\na_L|^{\frac12}|\na_{L,h}|^{-1}\big[\p_xU_{\neq,1}\cdot\na U_{0,2}^2-U_{\neq,2}\cdot\na \p_yU^1_{0,1}\big]\rangle_{H^s}\big|\\
	&\quad\lesssim \sqrt{\frac{\beta}{\beta-1}}F^{\frac12}_{\neq,2}(t)\nu^{-\frac16}(\delta\nu^{\frac56+\frac{\kappa}{6(1+\kappa)}})F^{\frac12}_{0,2}(t)\nu^{-\frac12} + \sqrt{\frac{\beta}{\beta-1}}F^{\frac12}_{\neq,2}(t)\nu^{-\frac16}F^{\frac12}_{\neq,2}(t)\nu^{-\frac13}(\delta\nu^c)\\
	& \quad\lesssim\delta F_{\neq,2}(t) + \frac{\beta}{\beta-1}\delta\nu^{\frac13+\frac{\kappa}{3(1+\kappa)}}F_{0,2}(t) + \sqrt{\frac{\beta}{\beta-1}}\delta\nu^{c-\frac12} F_{\neq,2}(t).
\end{align*}

Now we divide the term {$\p_y^LU_{\neq,1}\cdot\na U_{0,2}^1+\p_yU_{0,2}\cdot\na_LU^1_{\neq,1}=(\p_yU_{0,2}^2\p_y^LU_{\neq,1}^1 + \p_y^LU_{\neq,1}^3\p_zU_{0,2}^1) + (\p_y^LU_{\neq,1}^2\p_yU_{0,2}^1+\p_yU_{0,2}^{1,3}\cdot\na_{x,z}U_{\neq,1}^1)$}. The estimate of first one is similar to $U_{0,2}^2\p_y^LU_{\neq,1}^3$, we discuss two case $|l|\leq\sqrt{p_h}$ and $|l|>\sqrt{p_h}$ using \eqref{LL88} gives
{\begin{align*}
	&\sqrt{\frac{\beta}{\beta-1}}\big|\langle \mathcal{B}K_{\neq,2},\mathcal{B}|\na_L|^{\frac12}|\na_{L,h}|^{-1}\big[\p_yU_{0,2}^2\p_y^LU_{\neq,1}^1 + \p_y^LU_{\neq,1}^3\p_zU_{0,2}^1\big]\rangle_{H^s}\big|\\
	&\quad\lesssim \sqrt{\frac{\beta}{\beta-1}}F_{\neq,2}^{\frac12}\nu^{-\frac{\kappa}{6(1+\kappa)}}E_{\neq,1}^{\frac12}\nu^{-\frac16}F_{0,2}^{\frac12}\nu^{-\frac12} \\
	&\qquad+ \sqrt{\frac{\beta}{\beta-1}}\big|\langle \mathcal{B}|\na_L|^{\frac12}|\na_{L,h}|^{-\frac12}K_{\neq,2},e^{\frac{\lambda}{2}\nu^{\frac13}t}\frac{|\na_{x,y}|^{\frac12}}{\langle t\rangle^{\frac12}}\big[\p_zU_{0,2}\nu^{-\frac16}|\na_L|^{\frac12}m(Q_{\neq,1},K_{\neq,1})\big]\rangle_{H^s}\big|\\
	&\quad\lesssim \sqrt{\frac{\beta}{\beta-1}}\left(F_{\neq,2}^{\frac12}\nu^{-\frac{\kappa}{6(1+\kappa)}}E_{\neq,1}^{\frac12}\nu^{-\frac16}F_{0,2}^{\frac12}\nu^{-\frac12}+F_{\neq,2}^{\frac12}(t)\nu^{-\frac14}\nu^{-\frac16}E_{\neq,1}^{\frac12}F_{0,2}^{\frac12}\nu^{-\frac12}\right)\\
	&\quad\lesssim \delta F_{\neq,2}(t) + \frac{\beta}{\beta-1}[\delta\nu^{2c-\frac{11}{6}}F_{0,2}(t)+\delta \nu^{\frac13}F_{0,2}(t)],
\end{align*}}
Similarly, we have
{\begin{align*}
	&\sqrt{\frac{\beta}{\beta-1}}\big|\langle \mathcal{B}K_{\neq,2},\mathcal{B}|\na_L|^{\frac12}|\na_{L,h}|^{-1}\big[\p_y^LU_{\neq,1}^2\p_yU_{0,2}^1+\p_yU_{0,2}^{1,3}\cdot\na_{x,z}U_{\neq,1}^1\big]\rangle_{H^s}\big|\\
	&\quad\lesssim \sqrt{\frac{\beta}{\beta-1}}\left(F_{\neq,2}^{\frac12}(t)\nu^{-\frac{\kappa}{6(1+\kappa)}}\nu^{-\frac{1}{6}}E_{\neq,1}^{\frac12}F_{0,2}^{\frac12}(t)\nu^{-\frac12} + F_{\neq,2}^{\frac12}(t)F_{0,2}^{\frac12}\nu^{-\frac12}E_{\neq,1}^{\frac12}\right)\\
	&\quad\lesssim\delta F_{\neq,2}(t) + \frac{\beta}{\beta-1}[\delta\nu^{2c-1}F_{0,2}(t) + \delta\nu^{\frac13}F_{0,2}(t)],
\end{align*}}
where we have used the fact $mp^{\frac{1}{4}}p_{h}^{-\frac{1}{2}}\lesssim \frac{\langle l \rangle^{\frac{1}{2}}}{p_{h}^{\frac{1}{2}}},$ if $|l|\geq \sqrt{p_{h}}$ and $mp^{\frac{1}{4}}p_{h}^{-\frac{1}{2}}\lesssim m|\nabla_{L}|^{-\frac{1}{2}},$ if $|l|\leq \sqrt{p_{h}}$.

For $\p_y^LU_{\neq,1}\cdot\na_LU^1_{\neq,2}+\p_y^LU_{\neq,2}\cdot\na_LU^1_{\neq,1}$, we use \eqref{multi2}, \eqref{pro1}, the facts $|\hat{\na}\lesssim |\hat{\na}|\langle t \rangle$ and $\sqrt{|k',l'|}\leq C\sqrt{|k-k',l-l'|}+C\sqrt{|k,l|}$ to obtain
\begin{align*}
	&\sqrt{\frac{\beta}{\beta-1}}\big|\langle \mathcal{B}K_{\neq,2},\mathcal{B}|\na_L|^{\frac12}|\na_{L,h}|^{-1}\big[\p_y^LU_{\neq,1}\cdot\na_LU^1_{\neq,2}+\p_y^LU_{\neq,2}\cdot\na_LU^1_{\neq,1}\big]\rangle_{H^s}\big|\\
	&\quad\lesssim \sqrt{\frac{\beta}{\beta-1}}\sum_{k,k',l,l'}\int_{\eta,\eta'}\langle\hat{\na}\rangle^s|\widehat{\mathcal{B}K_{\neq,2}}(k,\eta,l)|\frac{e^{\frac{\lambda}{2}\nu^{\frac13}t}\sqrt{l}}{\sqrt{p_h}(k,\eta)}\langle\hat{\na}\rangle^s\left[|\widehat{\p_y^LU_{\neq,1}}(k-k',\eta-\eta',l-l')||(k',\eta'-k't,l')|\right.\\
	&\qquad\left.\times |\widehat{U_{\neq,2}^1}(k',\eta',l')| + |\widehat{\p_y^LU_{\neq,2}}(k-k',\eta-\eta',l-l')||(k',\eta'-k't,l')||\widehat{U_{\neq,1}^1}(k',\eta',l')|\right]{\rm d}\eta{\rm d}\eta'\\
	&\qquad+ \sqrt{\frac{\beta}{\beta-1}}\sum_{k,k',l,l'}\int_{\eta,\eta'}\langle\hat{\na}\rangle^s|\widehat{\mathcal{B}K_{\neq,2}}(k,\eta,l)|\frac{e^{\frac{\lambda}{2}\nu^{\frac13}t}}{\sqrt{|k,l|}}\langle\hat{\na}\rangle^s\left[|\widehat{\p_y^LU_{\neq,1}^{1,3}}(k-k',\eta-\eta',l-l')|\right.\\
	&\qquad \times|\widehat{\na_{x,z}U_{\neq,2}^1}(k',\eta',l')|+\widehat{\p_y^LU_{\neq,2}^{1,3}}(k-k',\eta-\eta',l-l')||\widehat{\na_{x,z}U_{\neq,1}^1}(k',\eta',l')| + |\widehat{\p_y^LU_{\neq,1}^2}(k-k',\eta-\eta',l-l')|\\
	&\qquad \times|\widehat{\p_y^L U_{\neq,2}^1}(k',\eta',l')|\left.+ |\widehat{\p_y^LU_{\neq,2}^2}(k-k',\eta-\eta',l-l')||\widehat{\p_y^L U_{\neq,1}^1}(k',\eta',l')|\right]{\rm d}\eta{\rm d}\eta'\\
	&\quad\lesssim \sqrt{\frac{\beta}{\beta-1}}\left(F^{\frac12}_{\neq,2}(t)(\delta\nu^{\frac56+\frac{\kappa}{6(1+\kappa)}})\nu^{-\frac13}F^{\frac12}_{\neq,2}(t)\nu^{-\frac12} + F^{\frac12}_{\neq,2}(t)\nu^{-\frac16}(\delta\nu^{\frac56+\frac{\kappa}{6(1+\kappa)}})\nu^{-\frac13}F^{\frac12}_{\neq,2}(t)\nu^{-\frac16}\right.\\
	&\qquad\left.+ F^{\frac12}_{\neq,2}(t)\nu^{-\frac16}F^{\frac12}_{\neq,2}(t)\nu^{-\frac12}(\delta\nu^{\frac56+\frac{\kappa}{6(1+\kappa)}}) + F^{\frac12}_{\neq,2}(t)\nu^{-\frac16}F^{\frac12}_{\neq,2}(t)\nu^{-\frac13}(\delta\nu^{\frac56+\frac{\kappa}{6(1+\kappa)}})\nu^{-\frac13}\right)\\&\quad
	\lesssim \sqrt{\frac{\beta}{\beta-1}}\delta\nu^{\frac{\kappa}{6(1+\kappa)}}F_{\neq,2}(t).
\end{align*}

Using $mp^{\frac14}\lesssim \min\{\nu^{-\frac16}p^{\frac14},p^{\frac12}\}$, \eqref{ll7}, \eqref{multi2}  and recalling $W_{\neq}^3=i\sqrt{\frac{\beta-1}{\beta}}|\na_L|^{-\frac12}|\na_{L,h}|K_{\neq}$, and integration by parts, we have 
\begin{align*}
	&\sqrt{\frac{\beta}{\beta-1}}\big|\langle \mathcal{B}K_{\neq,2},\mathcal{B}|\na_L|^{\frac12}|\na_{L,h}|^{-1}\big[U_{\neq,1}^{1,3}\cdot\na_{x,z}W^3_{\neq,2}+U_{\neq,1}^2\p_y^L W^3_{\neq,2}\big]\rangle_{H^s}\big|\\
	& \quad\lesssim\big|\langle \mathcal{B}K_{\neq,2},\mathcal{B}|\na_L|^{\frac12}\big[U_{\neq,1}^{1,3}\cdot\na_{x,z}|\na_L|^{-\frac12}K_{\neq,2}\big]\rangle_{H^s}\big| + \big|\langle \mathcal{B}K_{\neq,2},\mathcal{B}|\na_L|^{\frac12}\big[U_{\neq,1}^2\p_y^L|\na_L|^{-\frac12}K_{\neq,2}\big]\rangle_{H^s}\big|\\
	&\qquad+ \big|\langle \mathcal{B}K_{\neq,2},\mathcal{B}|\na_L|^{\frac12}|\na_{L,h}|^{-1}\big[|\na_{L,h}|U_{\neq,1}^{1,3}\cdot\na_{x,z}|\na_L|^{-\frac12}K_{\neq,2}\big]\rangle_{H^s}\big|\\
	&\qquad+ \big|\langle \mathcal{B}K_{\neq,2},\mathcal{B}|\na_L|^{\frac12}|\na_{L,h}|^{-1}\big[|\na_{L,h}|U_{\neq,1}^2\p_y^L|\na_L|^{-\frac12}K_{\neq,2}\big]\rangle_{H^s}\big|\\
	&\quad\lesssim F^{\frac12}_{\neq,2}(t)\nu^{-\frac12}(\delta\nu^{\frac56+\frac{\kappa}{6(1+\kappa)}})F^{\frac12}_{\neq,2}(t)\nu^{-\frac13} + \nu^{-\frac16}F^{\frac14}_{\neq,2}(t)\nu^{-\frac14}(\delta\nu^b)^{\frac12}\|U_{\neq,1}^2\|_{H^s}F^{\frac14}_{\neq,2}(t)\nu^{-\frac14}(\delta\nu^b)^{\frac12}\nu^{-\frac16} \\
	&\qquad+ \sum_{k,k',l,l'}\int_{\eta,\eta'}\langle\hat{\na}\rangle^s|\widehat{\mathcal{B}K_{\neq,2}}(k,\eta,l)|\frac{e^{\frac{\lambda}{2}\nu^{\frac13}t}\sqrt{l}}{\sqrt{p_h}(k,\eta)}\langle\hat{\na}\rangle^s|\widehat{|\na_{L,h}|U_{\neq,1}^{1,3}}(k-k',\eta-\eta',l-l')\\
	&\qquad\times\widehat{\na_{x,z}|\na_{L}|^{-\frac12}K_{\neq,2}}(k',\eta',l')|{\rm d}\eta{\rm d}\eta'+ \sum_{k,k',l,l'}\int_{\eta,\eta'}\langle\hat{\na}\rangle^s|\widehat{\mathcal{B}K_{\neq,2}}(k,\eta,l)|\frac{e^{\frac{\lambda}{2}\nu^{\frac13}t}}{\sqrt{|k,l|}}\langle\hat{\na}\rangle^s\\
	&\qquad\times|\widehat{|\na_{L,h}|U_{\neq,1}^{1,3}}(k-k',\eta-\eta',l-l')\widehat{\na_{x,z}|\na_{L}|^{-\frac12}K_{\neq,2}}(k',\eta',l')|{\rm d}\eta{\rm d}\eta'+ (\delta\nu^b)F^{\frac12}_{\neq,1}(t)\nu^{-\frac16}F^{\frac12}_{\neq,2}(t)\nu^{-\frac12}\nu^{-\frac{1}{6}}\\
	&\quad\lesssim \delta\nu^{\frac{\kappa}{6(1+\kappa)}}F_{\neq,2}(t) + \delta\nu^{b-\frac56}F^{\frac12}_{\neq,2}(t)F_{\neq,1}^{\frac12}(t)+ F^{\frac12}_{\neq,2}(t)\nu^{-\frac{\kappa}{6(1+\kappa)}}\nu^{-\frac16}(\delta\nu^{\frac56+\frac{\kappa}{6(1+\kappa)}})F^{\frac12}_{\neq,2}(t)\nu^{-\frac12}\nu^{-\frac{1}{6}} \\
	&\qquad+ F^{\frac12}_{\neq,2}(t)\nu^{-\frac16}\nu^{-\frac13}(\delta\nu^{\frac56+\frac{\kappa}{6(1+\kappa)}})F^{\frac12}_{\neq,2}(t)\nu^{-\frac16}\nu^{-\frac16}+ \delta F_{\neq,2}(t) + \delta\nu^{2b-\frac53}F_{\neq,1}(t)\\
	&\quad \lesssim 5\delta F_{\neq,2}(t) + \delta\nu^{2b-\frac53}F_{\neq,1}(t).
\end{align*}
The estimate of {$U_{0,1}\cdot\na_LW^3_{\neq,2}$} is similar to $U_{0,1}^2\p_y^L W^3_{\neq,1}$, via the integration by part, we obtain
\begin{align*}
	&\sqrt{\frac{\beta}{\beta-1}}\big|\langle \mathcal{B}K_{\neq,2},\mathcal{B}|\na_L|^{\frac12}|\na_{L,h}|^{-1}\big[U_{0,1}\cdot\na_LW^3_{\neq,2}\big]\rangle_{H^s}\big|\\
	&\quad\leq \big|\langle \mathcal{B}K_{\neq,2},\mathcal{B}|\na_L|^{\frac12}|\na_{L,h}|^{-1}\big[U_{0,1}\cdot\na_L|\na_L|^{-\frac12}|\na_{L,h}|K_{\neq,2}\big]\rangle_{H^s}\big| \\
	&\quad\lesssim \big|\langle \mathcal{B}K_{\neq,2},\mathcal{B}|\na_L|^{\frac12}\big[U_{0,1}|\na_L|^{\frac12}K_{\neq,2}\big]\rangle_{H^s}\big| + \big|\langle \mathcal{B}K_{\neq,2},\mathcal{B}|\na_L|^{\frac12}|\na_{L,h}|^{-1}\big[|\p_y|U_{0,1}\|\na_L|^{\frac12}K_{\neq,2}\big]\rangle_{H^s}\big|\\
	&\quad\lesssim F_{\neq,2}^{\frac12}(t)\nu^{-\frac13}(\|U_{0,1}\|_{H^{s+\frac12}}+\|U_{0,1}\|_{H^{s+\frac32}})F_{\neq,2}^{\frac12}(t)\nu^{-\frac13}\\&\quad\lesssim \delta\nu^{c-\frac23}F_{\neq,2}(t).
\end{align*}

Using  \eqref{multi2}, we have
\begin{align*}
	&\sqrt{\frac{\beta}{\beta-1}}\big|\langle \mathcal{B}K_{\neq,2},\mathcal{B}|\na_L|^{\frac12}|\na_{L,h}|^{-1}\big[U_{\neq,2}\cdot\na_LW^3_{\neq,1}\big]\rangle_{H^s}\big|\\
	&\quad=\big|\langle \mathcal{B}K_{\neq,2},\mathcal{B}|\na_L|^{\frac12}|\na_{L,h}|^{-1}\big[U_{\neq,2}\cdot\na_L|\na_L|^{-\frac12}|\na_{L,h}|K_{\neq,1}\big]\rangle_{H^s}\big|\\
	&\quad\lesssim \sum_{k,k',l,l'}\int_{\eta,\eta'}\Big[\langle\hat{\na}\rangle^s|\widehat{\mathcal{B}K_{\neq,2}}(k,\eta,l)|\frac{e^{\frac{\lambda}{2}\nu^{\frac13}t}\sqrt{l}}{p_h^{\frac{1}{4}}(k,\eta)}\langle\hat{\na}\rangle^s(|\widehat{U_{\neq,2}^{2}}(k-k',\eta-\eta',l-l')\widehat{|\na_{L}|^{\frac12}|\na_{L,h}|^{\frac{1}{2}}K_{\neq,1}}(k',\eta',l')|) \\
	& \qquad+  \langle\hat{\na}\rangle^s|\widehat{\mathcal{B}K_{\neq,2}}(k,\eta,l)|\frac{e^{\frac{\lambda}{2}\nu^{\frac13}t}\sqrt{l}}{\sqrt{p_h}(k,\eta)}\langle\hat{\na}\rangle^s(|\widehat{|\na_{L,h}|^{\frac{1}{2}}U_{\neq,2}^{2}}(k-k',\eta-\eta',l-l')\widehat{|\na_{L}|^{\frac12}|\na_{L,h}|^{\frac{1}{2}}K_{\neq,1}}(k',\eta',l')|) \\
   & \qquad+   \langle\hat{\na}\rangle^s|\widehat{\mathcal{B}K_{\neq,2}}(k,\eta,l)|\frac{e^{\frac{\lambda}{2}\nu^{\frac13}t}\sqrt{l}}{\sqrt{p_h}(k,\eta)}\langle\hat{\na}\rangle^s(|\widehat{|U_{\neq,2}^{1,3}}(k-k',\eta-\eta',l-l')\widehat{|\na_{L}|^{\frac12}|\na_{L,h}|K_{\neq,1}}(k',\eta',l')|) \\
	&\qquad+  \langle\hat{\na}\rangle^s|\widehat{\mathcal{B}K_{\neq,2}}(k,\eta,l)|\frac{e^{\frac{\lambda}{2}\nu^{\frac13}t}}{\sqrt{|k,l|}}\langle\hat{\na}\rangle^s(|\widehat{U_{\neq,2}^{1,3}}(k-k',\eta-\eta',l-l')\widehat{\na_{x,z}|\na_{L}|^{-\frac12}|\na_{L,h}|K_{\neq,1}}(k',\eta',l')|) \\
	&\qquad+   \langle\hat{\na}\rangle^s|\widehat{\mathcal{B}K_{\neq,2}}(k,\eta,l)|\frac{e^{\frac{\lambda}{2}\nu^{\frac13}t}}{\sqrt{|k,l|}}\langle\hat{\na}\rangle^s(|\widehat{U_{\neq,2}^2}(k-k',\eta-\eta',l-l')\widehat{\p_y^L|\na_{L}|^{-\frac12}|\na_{L,h}|K_{\neq,1}}(k',\eta',l')|)\Big]{\rm d}\eta{\rm d}\eta'\\
	&\quad\lesssim F_{\neq,2}^{\frac{1}{2}}\nu^{-\frac{\kappa}{6(1+\kappa)}}F_{\neq,2}^{\frac{1}{2}}\nu^{-\frac{1}{3}}\langle t\rangle^{\frac{3}{2}} e^{-\frac{\lambda}{2}\nu^{\frac13}t}\delta\nu^{\frac56+\frac{\kappa}{6(1+\kappa)}}+ F^{\frac12}_{\neq,2}(t)F^{\frac12}_{\neq,2}(t)\nu^{-\frac16}\langle t\rangle^2e^{-\frac{\lambda}{2}\nu^{\frac13}t}\delta\nu^{\frac56+\frac{\kappa}{6(1+\kappa)}}   \\
	&\qquad+F^{\frac12}_{\neq,2}(t)\nu^{-\frac16}F^{\frac12}_{\neq,2}(t)\nu^{-\frac16}\langle t\rangle e^{-\frac{\lambda}{2}\nu^{\frac13}t}\delta\nu^{\frac56+\frac{\kappa}{6(1+\kappa)}}+ F^{\frac12}_{\neq,2}(t)\nu^{-\frac16}F^{\frac12}_{\neq,2}(t)\nu^{-\frac{\kappa}{6(1+\kappa)}}\langle t\rangle^2 e^{-\frac{\lambda}{2}\nu^{\frac13}t}\delta\nu^{\frac56+\frac{\kappa}{6(1+\kappa)}}\\
	&\quad\lesssim \delta F_{\neq,2}(t),
\end{align*}
where we have using the following inequality 
\begin{align*}
	\langle l\rangle^{\frac12}p^{\frac14}p_h^{\frac14}=\langle l\rangle^{\frac12}p^{\frac14}p_h^{\frac14}m \frac{p^{\frac{1}{4}}}{|k,l|^{\frac{1}{2}}}\lesssim
	p^{\frac{1}{2}}p_{h}^{\frac{1}{4}}m\lesssim|\hat{\na}|^{\frac{3}{2}}\langle t\rangle^{\frac{3}{2}}m,
	\end{align*}
and $$|k,l|p^{-\frac{1}{4}}p_{h}^{\frac{1}{2}}\lesssim \sqrt{|k,l|}|\nabla|^{\frac{1}{2}} \langle t\rangle^{2}m.$$

Similarly, one has
\begin{align*}
	&\sqrt{\frac{\beta}{\beta-1}}\big|\langle \mathcal{B}K_{\neq,2},\mathcal{B}|\na_L|^{\frac12}|\na_{L,h}|^{-1}\big[U_{0,2}^{1,3}\cdot\na_{x,z}W^3_{\neq,1}\big]\rangle_{H^s}\big|\\
	&=\big|\langle \mathcal{B}K_{\neq,2},\mathcal{B}|\na_L|^{\frac12}|\na_{L,h}|^{-1}\big[U_{0,2}^{1,3}\cdot\na_{x,z}|\na_L|^{-\frac12}|\na_{L,h}|K_{\neq,1}\big]\rangle_{H^s}\big|\\
	&\quad\lesssim F^{\frac12}_{\neq,2}(t)(\delta\nu^a)\langle t\rangle e^{-\frac{\lambda}{2}\nu^{\frac13}t}F^{\frac12}_{\neq,1}(t)\nu^{-\frac16} + F^{\frac12}_{\neq,2}(t)\nu^{-\frac16}(\delta\nu^a)\langle t\rangle e^{-\frac{\lambda}{2}\nu^{\frac13}t}F^{\frac12}_{\neq,1}(t)\nu^{-\frac16}\\& \quad
	\lesssim\delta F_{\neq,2}(t) + \delta\nu^{2a-\frac43}F_{\neq,1}(t).
\end{align*}
The estimate of {$U_{0,2}^2\p_y^L W^3_{\neq,1}$} is similar to the $U_{0,2}^2\p_y^LU_{\neq,1}^3$ in the equation of $Q_{\neq,2}$, it obtains that
{\begin{align*}
	&\sqrt{\frac{\beta}{\beta-1}}\big|\langle \mathcal{B}K_{\neq,2},\mathcal{B}|\na_L|^{\frac12}|\na_{L,h}|^{-1}\big[U_{0,2}^2\p_y^LW^3_{\neq,1}\big]\rangle_{H^s}\big|\\
	&\quad=\big|\langle \mathcal{B}K_{\neq,2},\mathcal{B}|\na_L|^{\frac12}|\na_{L,h}|^{-1}\big[U_{0,2}^2\p_y^L|\na_L|^{-\frac12}|\na_{L,h}|K_{\neq,1}\big]\rangle_{H^s}\big|\\&\quad
	\lesssim \big|\langle \mathcal{B}|\na_L|^{\frac12}|\na_{L,h}|^{-\frac12}K_{\neq,2},e^{\frac{\lambda}{2}\nu^{\frac13}t}\frac{|\na_{x,y}|^{\frac12}}{\langle t\rangle^{\frac12}}\big[|\na_{y,z}|^{\frac12}U_{0,2}^2\p_y^L|\na|^{\frac12}\langle t\rangle^{\frac12}mK_{\neq,1}\big]\rangle_{H^s}\big|\\
&\quad\lesssim \big|\langle \mathcal{B}|\na_L|^{\frac12}|\na_{L,h}|^{-\frac12}K_{\neq,2},e^{\frac{\lambda}{2}\nu^{\frac13}t}\big[|\p_z|^{\frac12}|\na|^{\frac12}U_{0,2}|\na|^{\frac32}\langle t\rangle^{\frac12}m|\na_L|^{\frac12}K_{\neq,1}\big]\rangle_{H^s}\big|\\
&\quad\lesssim F_{\neq,2}^{\frac12}(t)\nu^{-\frac14}F_{0,2}^{\frac14}(t)\nu^{-\frac14}E_{0,2}^{\frac14}(t)F_{\neq,1}^{\frac14}(t)\nu^{-\frac14}E_{\neq,1}^{\frac14}(t)\nu^{-\frac16}\\&\quad\lesssim \delta F_{\neq,2}(t) + \nu^{a-1+\frac{\kappa}{6(1+\kappa)}}F_{0,2}^{\frac12}(t)F_{\neq,1}^{\frac12}(t).
\end{align*}}

Finally, the estimate of {$U_{\neq,1}^2\p_{yy}U^1_{0,2}+U_{\neq,1}^3\p_z\p_yU^1_{0,2}$} is similar to {$\p_y^LU_{\neq,1}\cdot\na U_{0,2}^1+\p_yU_{0,2}\cdot\na_LU^1_{\neq,1}$}. Integrating by part gives 
{\begin{align*}
	&\sqrt{\frac{\beta}{\beta-1}}\big|\langle \mathcal{B}K_{\neq,2},\mathcal{B}|\na_L|^{\frac12}|\na_{L,h}|^{-1}\big[U_{\neq,1}^2\p_{yy}U^1_{0,2}+U_{\neq,1}^3\p_z\p_yU^1_{0,2}\big]\rangle_{H^s}\big|\\&\quad
	\lesssim \sqrt{\frac{\beta}{\beta-1}}\big|\langle \mathcal{B}K_{\neq,2},\mathcal{B}|\na_L|^{\frac12}\big[U_{\neq,1}^2\p_{y}U^1_{0,2}+U_{\neq,1}^3\p_zU^1_{0,2}\big]\rangle_{H^s}\big|\\
	&\qquad+ \sqrt{\frac{\beta}{\beta-1}}\big|\langle \mathcal{B}K_{\neq,2},\mathcal{B}|\na_L|^{\frac12}|\na_{L,h}|^{-1}\big[\p_y^LU_{\neq,1}^2\p_{y}U^1_{0,2}+\p_y^LU_{\neq,1}^3\p_zU^1_{0,2}\big]\rangle_{H^s}\big|\\
	&\quad \lesssim\sqrt{\frac{\beta}{\beta-1}}\big|\langle \mathcal{B}K_{\neq,2},\mathcal{B}|\na_L|\big[|\p_y^L|^{\frac12}U_{\neq,1}^2|\p_y|^{\frac12}U^1_{0,2}\big]\rangle_{H^s}\big|+ \sqrt{\frac{\beta}{\beta-1}}F_{\neq,2}^{\frac12}(t)\nu^{-\frac13}E_{\neq,1}^{\frac12}(t)F_{0,2}^{\frac12}(t)\nu^{-\frac12} \\
	&\qquad+ \delta F_{\neq,2}(t) + \frac{\beta}{\beta-1}[\delta\nu^{2c-\frac{11}{6} }F_{0,2}(t)+ \nu^{\frac13}F_{0,2}(t)] \\&\quad
	\lesssim  \sqrt{\frac{\beta}{\beta-1}}F_{\neq,2}^{\frac12}(t)\nu^{-\frac12}\nu^{-\frac16}F_{\neq,1}^{\frac14}(t)E_{\neq,1}^{\frac14}(t)F_{0,2}^{\frac14}(t)\nu^{-\frac14}E_{0,2}^{\frac14}(t) \\
	&\qquad+ \delta F_{\neq,2}(t) + \frac{\beta}{\beta-1}[\delta\nu^{2c-\frac{11}{6} }F_{0,2}(t) + \nu^{\frac13}F_{0,2}(t) + \nu^{\frac{\kappa}{3(1+\kappa)}}F_{0,2}(t)]\\
	&\quad\lesssim \delta F_{\neq,2}(t) + \frac{\beta}{\beta-1}[\delta\nu^{a- 1 + \frac{\kappa}{6(1+\kappa)}}F_{0,2}^{\frac12}(t)F_{\neq,1}^{\frac12}(t)+\delta\nu^{2c-\frac{11}{6} }F_{0,2}(t) + \nu^{\frac13}F_{0,2}(t) + \nu^{\frac{\kappa}{3(1+\kappa)}}F_{0,2}(t)],
\end{align*}}
where in the third inequality we have used \eqref{ll7} and 
\begin{align*} \widehat{{|\partial_{y}^{L}}|^{\frac{1}{2}}|U_{\neq,1}^{2}|}&\lesssim m  p_{h}^{\frac{1}{4}}\frac{p^{\frac{1}{4}}}{|k,l|^{\frac{1}{2}}}\frac{1}{p_{h}^{\frac{1}{2}}}\left(\frac{|k|}{p^{\frac{1}{4}}}|\widehat{K_{\neq,1}}|+\frac{|l|}{p^{\frac{1}{4}}}|\widehat{Q_{\neq,1}}|\right)\\&\lesssim mp_{h}^{-\frac{1}{4}}\left(|k|^{\frac{1}{2}}|\widehat{K_{\neq,1}}|+|l|^{\frac{1}{2}}|\widehat{Q_{\neq,1}}|\right).\end{align*}

Now we focus on the equation of $H_{\neq,2}$. We use the fact that    to obtain
\begin{align*}
	&\big|\langle\mathcal{B}H_{\neq,2},\mathcal{B}|\na_L|^{\frac12}\big[U_{0,2}^2\p_y^L\Theta_{\neq,1}+U_{\neq,2}^{1,3}\cdot\nabla_{x,z}\Theta_{\neq,1}+U_{\neq,2}^2\p_y^L\Theta_{\neq,1}\big]\rangle_{H^s}\big|\\&\quad
	\lesssim  F_{\neq,2}^{\frac12}(t)\nu^{-\frac13}\Big(F_{0,2}^{\frac14}(t)\nu^{-\frac14}E_{0,2}^{\frac14}F_{\neq,1}^{\frac14}(t)\nu^{-\frac14}E_{\neq,1}^{\frac14}(t) +\nu^{-\frac16} F_{\neq,2}^{\frac12}(t)\nu^{-\frac16}E_{\neq,1}^{\frac12}(t)\\
	&\qquad+\nu^{-\frac16}F_{\neq,2}^{\frac12}(t)\nu^{-\frac{\kappa}{6(1+\kappa)}}E_{\neq,1}^{\frac12}(t)\nu^{-\frac13}\Big)\\
&\quad	\lesssim \delta F_{\neq,2}(t) + \delta\nu^{a-\frac56 + \frac{\kappa}{6(1+\kappa)}}F_{0,2}^{\frac12}(t)F_{\neq,1}^{\frac12}(t) + \nu^{c-\frac23}F_{\neq,2}(t),
\end{align*}
and 
\begin{align*}
	&\big|\langle\mathcal{B}H_{\neq,2},\mathcal{B}|\na_L|^{\frac12}\big[U_{\neq,2}\cdot\nabla\Theta_{0,1} + U_{0,2}^{1,3}\cdot\nabla_{x,z}\Theta_{\neq,1}\big]\rangle_{H^s}\big|\\&\quad
	\lesssim F^{\frac12}_{\neq,2}(t)\nu^{-\frac12}F^{\frac12}_{\neq,2}(t)\nu^{-\frac16}(\delta\nu^c) + F^{\frac12}_{\neq,2}(t)\nu^{-\frac12}(\delta\nu^a)F^{\frac12}_{\neq,1}(t)\nu^{-\frac16}\\&\quad\lesssim \delta\nu^{c-\frac23}F_{\neq,2}(t) + \delta F_{\neq,2}(t) + \delta\nu^{2a-\frac43}F_{\neq,1}(t).
\end{align*}

For the term {$U^2_{\neq,1}\p_y\Theta_{0,2}+U_{\neq,1}^3\p_z\Theta_{0,2}$}, whose estimate is similar to {$U_{\neq,1}^2\p_{yy}U^1_{0,2}+U_{\neq,1}^3\p_z\p_yU^1_{0,2}$} in $K_{\neq,2}$ equation as follows
{\begin{align*}
	&\big|\langle\mathcal{B}H_{\neq,2},\mathcal{B}|\na_L|^{\frac12}\big[U^2_{\neq,1}\p_y\Theta_{0,2}+U_{\neq,1}^3\p_z\Theta_{0,2}\big]\rangle_{H^s}\big|\\
	&\quad\lesssim \big|\langle\mathcal{B}|\na_L|H_{\neq,2},\mathcal{B}\big[|\p_y^L|^{\frac12}U^2_{\neq,1}|\p_y|^{\frac12}\Theta_{0,2}\big]\rangle_{H^s}\big| + F_{\neq,2}^{\frac12}(t)\nu^{-\frac13}E_{\neq,1}^{\frac12}(t)F_{0,2}^{\frac12}(t)\nu^{-\frac12}\\
	&\quad\lesssim F_{\neq,2}^{\frac12}(t)\nu^{-\frac12}\nu^{-\frac16}F_{\neq,1}^{\frac14}(t)E_{\neq,1}^{\frac14}(t)F_{0,2}^{\frac14}(t)\nu^{-\frac14}E_{0,2}^{\frac14}(t) + \delta F_{\neq,2}(t) + \delta\nu^{\frac{\kappa}{3(1+\kappa)}} F_{0,2}(t)\\
	&\quad\lesssim \delta F_{\neq,2}(t) + \delta\nu^{\frac{\kappa}{3(1+\kappa)}} F_{0,2}(t) + \delta\nu^{a- 1 + \frac{\kappa}{6(1+\kappa)}}F_{0,2}^{\frac12}(t)F_{\neq,1}^{\frac12}(t).
\end{align*}}
The estimate of term {$U_{0,1}\cdot\na_L \Theta_{\neq,2}$} is similar to $U_{0,1}^2\p_y^L\Theta_{\neq,1}$, which gives
{\begin{align*}
	\big|\langle\mathcal{B}H_{\neq,2},\mathcal{B}|\na_L|^{\frac12}\big[U_{0,1}\cdot\nabla_L\Theta_{\neq,2}\big]\rangle_{H^s}\big|&\lesssim \big|\langle\mathcal{B}|\na_L|^{\frac12}H_{\neq,2},\big[|\na_{y,z}|^{\frac12}U_{0,1}|\na_L|^{\frac12}mH_{\neq,2}\big]\rangle_{H^s}\big|\\
	&\lesssim F^{\frac12}_{\neq,2}(t)\nu^{-\frac13}\|U_{0,1}\|_{H^{s+\frac12}}F^{\frac12}_{\neq,2}(t)\nu^{-\frac13}\\&\lesssim \delta\nu^{c-\frac23} F_{\neq,2}(t).
\end{align*}}

Due to the fact that $|k,l||\widehat{\Theta_{\neq,2}}|\lesssim \sqrt{|k,l|}|\widehat{mH_{\neq,2}}|,\ \sqrt{|k',l'|}\leq C\sqrt{|k,l|}+ C\sqrt{|k-k',l-l'|}$ and $mp^{\frac14}\lesssim \min\{\nu^{-\frac16}p^{\frac14},p^{\frac12}\}$, we have 
\begin{align*}
	&\big|\langle\mathcal{B}H_{\neq,2},\mathcal{B}|\na_L|^{\frac12}\big[U_{\neq,1}^{1,3}\cdot\nabla_{x,z}\Theta_{\neq,2}+U_{\neq,1}^2\p_y^L\Theta_{\neq,2}\big]\rangle_{H^s}\big|\\
	&\quad\lesssim \big|\langle\mathcal{B}H_{\neq,2},\mathcal{B}|\na_L|^{\frac12}|\nabla_{x,z}|^{\frac12}\big[U_{\neq,1}^{1,3}\cdot mH_{\neq,2}\big]\rangle_{H^s}\big| + \big|\langle\mathcal{B}H_{\neq,2},\mathcal{B}|\na_L|^{\frac12}\big[|\nabla_{x,z}|^{\frac12}U_{\neq,1}^{1,3}\cdot mH_{\neq,2}\big]\rangle_{H^s}\big| \\
	&\qquad+ \big|\langle\mathcal{B}H_{\neq,2},\mathcal{B}|\na_L|^{\frac12}\big[U_{\neq,1}^2\p_y^L\Theta_{\neq,2}\big]\rangle_{H^s}\big|\\
	&\quad\lesssim F^{\frac12}_{\neq,2}(t)\nu^{-\frac12}\nu^{-\frac{1}{6}}(\delta\nu^{\frac56+\frac{\kappa}{6(1+\kappa)}})F^{\frac12}_{\neq,2}(t)\nu^{-\frac16} + \nu^{-\frac16}F^{\frac14}_{\neq,2}(t)\nu^{-\frac14}(\delta\nu^b)^{\frac12}\|U_{\neq,1}^2\|_{H^s}F^{\frac14}_{\neq,2}(t)\nu^{-\frac14}\nu^{-\frac16}(\delta\nu^b)^{\frac12}\\
	& \quad\lesssim\delta\nu^{\frac{\kappa}{6(1+\kappa)}} F_{\neq,2}(t) + \delta\nu^b F^{\frac12}_{\neq,2}(t)\nu^{-\frac56}F^{\frac12}_{\neq,1}(t)\\&\quad\lesssim  \delta\nu^{\frac{\kappa}{6(1+\kappa)}} F_{\neq,2}(t) + \delta F_{\neq,2}(t) + \delta\nu^{2b-\frac53}F_{\neq,1}(t).
\end{align*}

\textit{4.5.3 Contribution without any external force term}
We now consider the contribution of the second system without any force term. We focus on the equation of $Q_{\neq,2}$. Due to 	$mp^{\frac14}\lesssim \min\{\nu^{-\frac16}p^{\frac14},p^{\frac12}\}$ and $\sqrt{\eta-\eta'-(k-k')t}\lesssim \sqrt{|\eta'-k't|}p_h^{\frac14}(k,\eta)$ , one gets
\begin{align*}
	&\big|\langle\mathcal{B}Q_{\neq,2},\mathcal{B}|\na_L|^{\frac32}|\na_{L,h}|^{-1}\big[U_{\neq,2}^{1,3}\cdot\na_{x,z}U^3_{\neq,2} + U_{0,2}\cdot\nabla_LU^3_{\neq,2}+U_{\neq,2}^2\p_y^LU_{\neq,2}^3\big]\rangle_{H^s}\big|\\
	&\quad
	\lesssim F^{\frac12}_{\neq,2}(t)\nu^{-\frac12}F^{\frac12}_{\neq,2}(t)\nu^{-\frac12}[\delta\nu^b+\delta\nu^a] + \big|\langle\mathcal{B}Q_{\neq,2},\mathcal{B}|\na_L|^{\frac32}|\na_{L,h}|^{-\frac12}\big[|\p_y^L|^{\frac12}U_{\neq,2}^2|\p_y^L|^{\frac12}U_{\neq,2}^3\big]\rangle_{H^s}\big|\\
	&\quad \lesssim\delta[\nu^{b-1}+\nu^{a-1}]F_{\neq,2}(t) + F_{\neq,2}^{\frac12}(t)\nu^{-\frac12}(F_{\neq,2}^{\frac14}(t)\nu^{-\frac14})^2(\delta\nu^b)\\
	&\quad \lesssim \delta[\nu^{b-1}+\nu^{a-1}]F_{\neq,2}(t),
\end{align*}
and since $\sqrt{|\eta-\eta'|}\lesssim \sqrt{|\eta'-k't|}p_h^{\frac14}(k,\eta)$,
\begin{align*}
	&\big|\langle\mathcal{B}Q_{\neq,2},\mathcal{B}|\na_L|^{\frac32}|\na_{L,h}|^{-1}\big[U_{\neq,2}^3\p_z U_{0,2}^3+U_{\neq,2}^2\p_y U_{0,2}^3\big]\rangle_{H^s}\big|\\
	&\quad\lesssim F^{\frac12}_{\neq,2}(t)\nu^{-\frac12}(\delta\nu^b)F^{\frac12}_{0,2}(t)\nu^{-\frac12} + \big|\langle\mathcal{B}Q_{\neq,2},\mathcal{B}|\na_L|^{\frac32}|\na_{L,h}|^{-\frac12}\big[|\p_y^L|^{\frac12}U_{\neq,2}^2|\p_y|^{\frac12} U_{0,2}^3\big]\rangle_{H^s}\big|\\
	&\quad\lesssim \delta F_{\neq,2}(t) + \delta^2\nu^{2b-2}F_{0,2}(t) + F_{\neq,2}^{\frac12}(t)\nu^{-\frac12}F_{\neq,2}^{\frac14}(t)\nu^{-\frac14}F_{0,2}^{\frac14}(t)\nu^{-\frac14}\delta\nu^{\frac{a+b}{2}}\\
	&\quad\lesssim \delta F_{\neq,2}(t) + (\delta^2\nu^{2b-2}+\delta^4\nu^{2a+2b-4})F_{0,2}(t).
\end{align*}
For the pressure terms, due to the fact that $\sqrt{|\eta-\eta'-(k-k')t|}\lesssim \sqrt{|\eta'-k't|}p_h^{\frac14}(k,\eta)$ and $\sqrt{|\eta-\eta'|}\lesssim \sqrt{|\eta'-kt|}p_h^{\frac14}(k,\eta)$, we have 
\begin{align*}
	&\big|\langle\mathcal{B}Q_{\neq,2},\mathcal{B}\p_z|\na_L|^{-\frac12}|\na_{L,h}|^{-1}\big[\p_{x,z}U_{\neq,2}\cdot\na_L U_{\neq,2}^{1,3}+\p_y^LU_{\neq,2}^{1,2}\cdot\na_{L,h}U_{\neq,2}^2+\p_y^LU_{\neq,2}^3\p_zU_{\neq,2}^2\big]\rangle_{H^s}\big|\\
	&\quad +\big|\langle\mathcal{B}Q_{\neq,2},\mathcal{B}\p_z|\na_L|^{-\frac12}|\na_{L,h}|^{-1}\big[\p_zU_{0,2}\cdot\na_LU_{\neq,2}^3+\p_yU_{0,2}^{1,2}\cdot\na_{L,h}U_{\neq,2}^2 + \p_yU_{0,2}^3\p_zU_{\neq,2}^2\big]\rangle_{H^s}\big|\\
	&\quad\lesssim (\delta\nu^b)(F^{\frac12}_{\neq,2}(t)\nu^{-\frac12} + F^{\frac12}_{0,2}(t)\nu^{-\frac12})F^{\frac12}_{\neq,2}(t)\nu^{-\frac12} \\
	&\qquad+ \big|\langle\mathcal{B}Q_{\neq,2},\mathcal{B}\p_z|\na_L|^{-\frac12}|\na_{L,h}|^{-\frac12}\big[(|\p_y^L|^{\frac12}U_{\neq,2}^3+|\p_y|^{\frac12}U_{0,2}^3)\p_z|\p_y^L|^{\frac12}U_{\neq,2}^2\big]\rangle_{H^s}\big|\\
	&\quad \lesssim \delta\nu^{b-1}F_{\neq,2}(t) + \delta F_{\neq,2}(t) + \delta^2\nu^{2b-2}F_{0,2}(t) + F_{\neq,2}^{\frac14}(t)\nu^{-\frac14}(\delta\nu^b)^{\frac12}(F_{\neq,2}^{\frac14}\nu^{-\frac14}(\delta\nu^b)^{\frac12}+F_{0,2}^{\frac14}\nu^{-\frac14}(\delta\nu^a)^{\frac12})F_{\neq,2}^{\frac12}(t)\nu^{-\frac12}\\
	&\quad \lesssim \delta\nu^{b-1}F_{\neq,2}(t) + \delta F_{\neq,2}(t) + (\delta^2\nu^{2b-2}+\delta^4\nu^{2a+2b-4})F_{0,2}(t)
\end{align*}

Next we focus on the equation of $K_{\neq,2}$. Since $\sqrt{|k'|}\lesssim \sqrt{|k-k'|}p_h^{\frac14}(k,\eta)$, we have
\begin{align*}
	&\sqrt{\frac{\beta}{\beta-1}}\big|\langle\mathcal{B}K_{\neq,2},\mathcal{B}|\na_L|^{\frac12}|\na_{L,h}|^{-1}\big[\p_xU_{\neq,2}^{1,2}\cdot\na_{L,h}U_{\neq,2}^2-\p_y^LU_{\neq,2}\cdot\na_LU_{\neq,2}^1+\p_xU_{\neq,2}^3\p_z U_{\neq,2}^2\big]\rangle_{H^s}\big|\\
	&\quad\lesssim \sqrt{\frac{\beta}{\beta-1}}(\delta\nu^b)F^{\frac12}_{\neq,2}(t)\nu^{-\frac12}F^{\frac12}_{\neq,2}(t)\nu^{-\frac12} + \sqrt{\frac{\beta}{\beta-1}}\big|\langle\mathcal{B}K_{\neq,2},\mathcal{B}|\na_L|^{\frac12}|\na_{L,h}|^{-\frac12}\big[|\p_x|^{\frac12}U_{\neq,2}^3\p_z|\p_x|^{\frac12} U_{\neq,2}^2\big]\rangle_{H^s}\big|\\
	&\quad\lesssim \sqrt{\frac{\beta}{\beta-1}}\delta\nu^{b-1}F_{\neq,2}(t) + \sqrt{\frac{\beta}{\beta-1}}F_{\neq,2}^{\frac12}(t)\nu^{-\frac14}(\delta\nu^b)F_{\neq,2}^{\frac12}(t)\nu^{-\frac12}\\
	&\quad\lesssim \sqrt{\frac{\beta}{\beta-1}}\delta\nu^{b-1}F_{\neq,2}(t),
	\end{align*}
	and
	\begin{align*}
	&\sqrt{\frac{\beta}{\beta-1}}\big|\langle\mathcal{B}K_{\neq,2},\mathcal{B}|\na_L|^{\frac12}|\na_{L,h}|^{-1}\big[\p_xU_{\neq,2}\cdot\na U_{0,2}^2-\p_y^LU_{\neq,2}\cdot\na U_{0,2}^1-\p_yU_{0,2}\cdot\na_LU_{\neq,2}^1\big]\rangle_{H^s}\big|\\
	&\quad\lesssim \sqrt{\frac{\beta}{\beta-1}}(\delta\nu^b)F^{\frac12}_{\neq,2}(t)\nu^{-\frac12}F^{\frac12}_{0,2}(t)\nu^{-\frac12}\\&\quad \lesssim \delta F_{\neq,2}(t) + \frac{\beta}{\beta-1}\delta\nu^{2b-2}F_{0,2}(t).
\end{align*}
 Integration by part yields that 
\begin{align*}
	&\sqrt{\frac{\beta}{\beta-1}}\big|\langle\mathcal{B}K_{\neq,2},\mathcal{B}|\na_L|^{\frac12}|\na_{L,h}|^{-1}\big[U_{\neq,2}\cdot\na_LW_{\neq,2}^3\big]\rangle_{H^s}\big|\\
	&\quad\lesssim\big|\langle\mathcal{B}K_{\neq,2},\mathcal{B}|\na_L|^{\frac12}\big[U_{\neq,2}\cdot\na_L|\na_L|^{-\frac12}K_{\neq,2}\big]\rangle_{H^s}\big| \\
	&\qquad+ \big|\langle\mathcal{B}K_{\neq,2},\mathcal{B}|\na_L|^{\frac12}|\na_{L,h}|^{-1}\big[|\na_{L,h}|U_{\neq,2}\cdot\na_L|\na_L|^{-\frac12}K_{\neq,2}\big]\rangle_{H^s}\big|\\
	&\quad\lesssim F^{\frac12}_{\neq,2}(t)\nu^{-\frac12}(\delta\nu^b)F^{\frac12}_{\neq,2}(t)\nu^{-\frac12} + (\delta\nu^b)F^{\frac12}_{\neq,2}(t)\nu^{-\frac12}F^{\frac12}_{\neq,2}(t)\nu^{-\frac12}\\
	&\quad\lesssim \delta\nu^{b-1}F_{\neq,2}(t), \end{align*}
	and
	\begin{align*}
	&\sqrt{\frac{\beta}{\beta-1}}\big|\langle\mathcal{B}K_{\neq,2},\mathcal{B}|\na_L|^{\frac12}|\na_{L,h}|^{-1}\big[U_{0,2}\cdot\na_LW_{\neq,2}^3\big]\rangle_{H^s}\big|\\
	&\quad\lesssim\big|\langle\mathcal{B}K_{\neq,2},\mathcal{B}|\na_L|^{\frac12}\big[U_{0,2}\cdot\na_L|\na_L|^{-\frac12}K_{\neq,2}\big]\rangle_{H^s}\big| \\
	&\qquad+ \big|\langle\mathcal{B}K_{\neq,2},\mathcal{B}|\na_L|^{\frac12}|\na_{L,h}|^{-1}\big[\p_yU_{0,2}\cdot\na_L|\na_L|^{-\frac12}K_{\neq,2}\big]\rangle_{H^s}\big|\\
	&\quad\lesssim F^{\frac12}_{\neq,2}(t)\nu^{-\frac12}(\delta\nu^a)F^{\frac12}_{\neq,2}(t)\nu^{-\frac12} + (\delta\nu^b)F^{\frac12}_{0,2}(t)\nu^{-\frac12}F^{\frac12}_{\neq,2}(t)\nu^{-\frac12}\\
	&\quad\lesssim \delta\nu^{a-1}F_{\neq,2}(t) + \delta F_{\neq,2}(t)+\delta\nu^{2b-2}F_{0,2}(t),
\end{align*}
and
\begin{align*}
	&\sqrt{\frac{\beta}{\beta-1}}\big|\langle\mathcal{B}K_{\neq,2},\mathcal{B}|\na_L|^{\frac12}|\na_{L,h}|^{-1}\big[U_{\neq,2}\cdot\na\p_yU_{0,2}^1\big]\rangle_{H^s}\big|\\
	&\quad\lesssim\sqrt{\frac{\beta}{\beta-1}}\left(\big|\langle\mathcal{B}K_{\neq,2},\mathcal{B}|\na_L|^{\frac12}\big[U_{\neq,2}\cdot\na U_{0,2}^1\big]\rangle_{H^s}\big| + \big|\langle\mathcal{B}K_{\neq,2},\mathcal{B}|\na_L|^{\frac12}|\na_{L,h}|^{-1}\big[\p_y^LU_{\neq,2}\cdot\na U_{0,2}^1\big]\rangle_{H^s}\big|\right)\\
	&\quad\lesssim \sqrt{\frac{\beta}{\beta-1}}\left(F^{\frac12}_{\neq,2}(t)\nu^{-\frac12}(\delta\nu^b)F^{\frac12}_{0,2}(t)\nu^{-\frac12} + (\delta\nu^b)F^{\frac12}_{\neq,2}(t)\nu^{-\frac12}F^{\frac12}_{0,2}(t)\nu^{-\frac12}\right)\\&\quad\lesssim \delta F_{\neq,2}(t) + \frac{\beta}{\beta-1}\delta\nu^{2b-2}F_{0,2}(t).
\end{align*}

Finally, for the equation of $H_{\neq,2}$, recall that $H_{\neq}=-|\na_L|^{\frac12}\Theta_{\neq}$ and using the fact that $mp^{\frac14}\lesssim \min\{\nu^{-\frac16}p^{\frac14},p^{\frac12}\}$, we obtain 
\begin{align*}
	&\big|\langle\mathcal{B}H_{\neq,2},\mathcal{B}|\na_L|^{\frac12}\big[U_{\neq,2}\cdot\na_L\Theta_{\neq,2} + U_{0,2}\cdot\na_L\Theta_{\neq,2}\big]\rangle_{H^s}\big|\\
	&\quad\lesssim F^{\frac12}_{\neq,2}(t)\nu^{-\frac12}(\delta\nu^b+\delta\nu^a)F^{\frac12}_{\neq,2}(t)\nu^{-\frac13}\nu^{-\frac16}\\&\quad=\delta(\nu^{b-1}+\nu^{a-1})F_{\neq,2}(t),
\end{align*}
and
	\begin{align*}
	&\big|\langle\mathcal{B}H_{\neq,2},\mathcal{B}|\na_L|^{\frac12}\big[U_{\neq,2}\cdot\na\Theta_{0,2}\big]\rangle_{H^s}\big|\\
	&\quad\lesssim F^{\frac12}_{\neq,2}(t)\nu^{-\frac12}(\delta\nu^b)F^{\frac12}_{0,2}(t)\nu^{-\frac12}\\&\quad\lesssim \delta F_{\neq,2}(t) + \delta\nu^{2b-2}F_{0,2}(t).
\end{align*}
Integrating the above inequality respect to $t$  immediately gives \eqref{L66} and  the proof of Proposition \ref{propEneq2} is finished. \hfill$\square$

Now, we let $c=5/6+\kappa/(6+6\kappa)$. Combing \eqref{L55} with \eqref{L66}  obtain that $a=b=1$, $c=14/15,\ \kappa/(6+6\kappa)=1/10$ and finish the bootstrap argument.

\appendix

\section{Dynamics of nonzero modes when $B_{\beta}>0$} 
\label{secA}
\qquad In fact, when $B_{\beta}>0$,    we can also prove the linear stability of the system of equations (1.6). Let us talk about the differences in the proof process here. At this point, it is sufficient to demonstrate the enhanced dissipation and and the inviscid damping effect at non-zero frequencies. We introduce a new group of good unknowns
\[
	Q^*_{\neq}:=-|\na_L|^2|\na_{L,h}|^{-1}U^3_{\neq},\ K^*_{\neq}:=-i\sqrt{\frac{\beta}{\beta-1}}|\na_L||\na_{L,h}|^{-1}W_{\neq}^3,\ H_{\neq}^*:=-|\na_L|\Theta_{\neq},
\]
and then rewrite \eqref{Non0L} as 
\begin{align}\label{Non0L1}
	\begin{cases}
		\p_t \widehat{Q_{\neq}^*} + \nu p \widehat{Q_{\neq}^*} - \frac12 \frac{\p_t p_{h}}{p_{h}}\widehat{Q_{\neq}^*} - \sqrt{B_{\beta}}\frac{l}{p^{1/2}}\widehat{K_{\neq}^*} + \alpha \frac{p_h^{1/2}}{p^{1/2}}\widehat{H_{\neq}^*} + 2\sqrt{\frac{\beta-1}{\beta}}\frac{k^2}{p_h}\frac{l}{p^{1/2}}\widehat{K_{\neq}^*}  = 0, \\
		\p_t\widehat{K_{\neq}^*} + \nu p\widehat{K_{\neq}^*} +\frac12\frac{l^{2}}{p}\frac{\p_t p_h}{p_h}\widehat{K_{\neq}^*}+ \sqrt{B_{\beta}}\frac{l}{p^{1/2}}\widehat{Q_{\neq}^*} = 0,\\
		\p_t\widehat{H_{\neq}^*} + \nu p\widehat{H_{\neq}^*}  -\frac{1}{2}\frac{\partial_{t}p}{p} \widehat{H_{\neq}^*}+ \alpha \frac{p_h^{1/2}}{p^{1/2}}\widehat{Q_{\neq}^*} = 0.
	\end{cases}
\end{align}
We use a suitably constructed Fourier multiplier $m^*$ with symbol $m^*=m^*(t,k,\eta,l)$, defined by 
\begin{align}\label{A.....4}	
	-\frac{\dot{m^*}}{m^*}=\left\{\begin{array}{ll}
		\frac{| k \left( \eta- kt \right)|}{k^2+\left(\eta- kt\right)^2}, & \text{ if }  |t-\frac{\eta}{k}|\leq 1000\nu^{-\frac13}, \\
		0,  & \text{ if } |t-\frac{\eta}{k}|> 1000\nu^{-\frac13}.
	\end{array}\right. 
\end{align}
with the initial data
\[
	m^*(0,k,\eta,l)=\begin{cases}
		1,\ \text{ if } \ \frac{\eta}{k}\leq 0,\\
		\left(\frac{k^2+\eta^2}{k^2}\right)^{\frac12},\ \text{ if } \ 0<\frac{\eta}{k}<1000\nu^{-\frac13},\\
		\left(\frac{k^2+(1000\nu^{-\frac13}k)^2}{k^2}\right)^{\frac12},\ \text{ if } \ 1000\nu^{-\frac13}\leq\frac{\eta}{k}.
	\end{cases}
\]
The role of $m^{*}$  is to handle the slowly decaying term 
$\frac12  \frac{\p_t p_{h}}{p_{h}}\widehat{Q_{\neq}^*} $, $\frac{1}{2}\frac{l^{2}}{p}\frac{\partial_{t}p_{h}}{p_{h}}\widehat{K_{\neq}^*}$ and $\frac12 \frac{\p_t p}{p}\widehat{H_{\neq}^*}$ in \eqref{Non0L1}, balancing the growth that solutions experience in the transition between the inviscid damping regime $t<\frac{\eta}{k}$, and the dissipative regime $t\gtrsim \frac{\eta}{k}+\nu^{-\frac{1}{3}}$. 

Define the  multiplier
$$M^{*}=M_{1}M_{5},$$
with $M_{1}$ and $M_{5}$ as in  \eqref{4.5} and \eqref{44} respectively.

\begin{lem}
	$(1)$ For $k\neq0$, the multiplier $m^*(t, k, \eta)$ can be given by the following exact formula:
\begin{itemize}
	\item if $ \frac{\eta}{k}<-1000\nu^{-\frac13}: \ m^*(t,k,\eta)\equiv 1;$
	\item if $-1000\nu^{-\frac13}\leq\frac{\eta}{k}\leq 0$:
	\begin{equation}  
		m^*=\begin{cases}
		\left(\frac{k^2+\eta^2}{p_h}\right)^{\frac12},\ \text{ if } \ 0\leq t\leq\frac{\eta}{k}+1000\nu^{-\frac13},\\
		\left(\frac{k^2+\eta^2}{k^2+(1000\nu^{-\frac13}k)^2}\right)^{\frac12},\ \text{ if } \ \frac{\eta}{k}+1000\nu^{-\frac13}\leq t.
	    \end{cases} \nonumber
     \end{equation}
	\item if $0<\frac{\eta}{k}<1000\nu^{-\frac13}$:    \begin{equation}
		m^*=\begin{cases}
		\left(\frac{p_h}{k^2}\right)^{\frac12},\ \text{ if } \ 0\leq t\leq\frac{\eta}{k},\\
		\left(\frac{k^2}{p_h}\right)^{\frac12},\ \text{ if } \ \frac{\eta}{k}\leq t\leq\frac{\eta}{k} + 1000\nu^{-\frac13},\\
		\left(\frac{k^2}{k^2+(1000\nu^{-\frac13}k)^2}\right)^{\frac12},\ \text{ if } \ \frac{\eta}{k} + 1000\nu^{-\frac13}\leq t.
	\end{cases}\nonumber
	\end{equation}
 
	\item if $1000\nu^{-\frac13}\leq\frac{\eta}{k}$: 
	\begin{equation} 
		m^*=\begin{cases}
		\left(\frac{k^2+(1000\nu^{-\frac13}k)^2}{k^2}\right)^{\frac12},\ \text{ if } \ 0\leq t\leq \frac{\eta}{k}-1000\nu^{-\frac13},\\
		\left(\frac{p_h}{k^2}\right)^{\frac12},\ \text{ if } \ \frac{\eta}{k}-1000\nu^{-\frac13}\leq t\leq\frac{\eta}{k},\\
		\left(\frac{k^2}{p_h}\right)^{\frac12},\ \text{ if } \ \frac{\eta}{k}\leq t \leq \frac{\eta}{k}+1000\nu^{-\frac13},\\
		\left(\frac{k^2}{k^2+(1000\nu^{-\frac13}k)^2}\right)^{\frac12},\ \text{ if } \ \frac{\eta}{k}+1000\nu^{-\frac13}\leq t.
	\end{cases}\nonumber
	\end{equation}
\end{itemize}	
$(2)$ In particular,  $m^*(t, k, \eta)$  are bounded above and below, but its bound depends on $\nu$
\begin{equation}\label{A.11-1}
	\nu^{\frac{1}{3}}\lesssim m^*\left(t, k, \eta\right) \leqslant \nu^{-\frac{1}{3}}.
\end{equation}
$(3)$  $m^*(t, k, \eta)$  and the frequency have the following relationship
\begin{equation}\label{A.11-2}
	\frac{|k|}{\sqrt{p_h}} \lesssim m^*\left(t, k, \eta \right) \lesssim\frac{\sqrt{p_h}}{|k|}.
\end{equation}
$(4)$  Moreover,  for all $t, k, \eta, \eta'$,  the following product estimates hold that
\begin{equation}\label{A.11-3}
	m^*(t, k, \eta)\lesssim \langle \eta-\eta' \rangle m^*(t,k, \eta').
\end{equation}
\end{lem}

The proof of \eqref{A.11x}-\eqref{A.33x} as in  Theorem \ref{1.1}  can be directly obtained from the following Lemma, which gives the linear enhanced dissipation and inviscid damping on $U_{\neq}$.  
\begin{lem}\label{lemA.4}
	Assume that $\nu > 0$ and $\beta \in \mathbb{R}$ with $B_\beta>0$, then there holds
	\begin{align}
		&\left\| U^1_{\neq} (t) \right\|_{L^{2}} \lesssim  \, b_{\beta}^* \, e^{-\lambda^*  \nu^{\frac{1}{3}} t} \left(\left\|  u_{\mathrm{in}}\right\|_{H^{3}}+\left\|  \theta_{\mathrm{in}}\right\|_{H^{2}} \right) , \label{A.1x}  \\
		&{\left\langle {t} \right\rangle} \left\|  U^2_{\neq} (t) \right\|_{L^{2}} \lesssim  \, b_{\beta}^* \, e^{-\lambda^* \nu^{\frac{1}{3}} t} \left(\left\|  u_{\mathrm{in}}\right\|_{H^{5}}+\left\|  \theta_{\mathrm{in}}\right\|_{H^{4}} \right) , \label{A.2x}  \\
		&\left\| (U^3_{\neq}, \Theta_{\neq}) (t) \right\|_{L^2} \lesssim   e^{-\lambda^* \nu^{\frac{1}{3}} t} \left(\left\|  u^3_{\mathrm{in}}\right\|_{H^{3}} + \sqrt{\frac{\beta}{\beta-1}} \left\|  \p_xu^2_{\mathrm{in}}-\p_yu^1_{\mathrm{in}}\right\|_{H^{2}}+\left\|  \theta_{\mathrm{in}}\right\|_{H^{2}} \right),\label{A.3x}  
	\end{align}
	where  $\lambda^*=\frac{1}{16}$ and $b_{\beta}^*=\max \left\{\sqrt{\frac{\beta-1}{\beta}}, \sqrt{\frac{\beta}{\beta-1}} \right\}$.
\end{lem}

\pf To control the evolution of the nonzero modes, we introduce the Fourier multiplier
\begin{equation}\label{A88}
	\mathcal{A}^*:=m^*M^{*}e^{\lambda^* \nu^{\frac{1}{3}}t},
\end{equation} 
which incorporates the multiplier $m^*$ already used in the linear analysis,   additional ghost weights ${M^*}$ as well as the time weight $e^{\lambda^*\nu^{\frac{1}{3}}t}$, which $\lambda^*$ can be confirmed later. Recall the system  \eqref{Non0L1}, energy estimates yield that 
\begin{equation}\label{A99}
\begin{aligned}
    \frac12&\frac{d}{dt}|\mathcal{A}(\widehat{Q_{\neq}^*},\widehat{K_{\neq}^*},\widehat{H_{\neq}^*})|^2  + \nu p |\mathcal{A}(\widehat{Q_{\neq}^*},\widehat{K_{\neq}^*},\widehat{H_{\neq}^*})|^2 +\left[(-\frac{\dot{M^{*}}}{M^*})+(-\frac{\dot{m^*}}{m^*})-\lambda^*\nu^{\frac13}\right]|\mathcal{A}(\widehat{Q_{\neq}^*},\widehat{K_{\neq}^*},\widehat{H_{\neq}^*})|^2\\
    &-\frac12\frac{\p_t p_{h}}{p_{h}}|\widehat{\mathcal{A}Q_{\neq}^*}|^2+\frac12\frac{l^{2}}{p}\frac{\p_t p_h}{p_h}|\widehat{\mathcal{A}K_{\neq}^*}|^2-\frac{1}{2}\frac{\partial_{t}{p}}{p}|\widehat{\mathcal{A}H_{\neq}^*}|+2\sqrt{\frac{\beta-1}{\beta}}\frac{k^2}{p_h}\frac{l}{\sqrt{p}}\widehat{\mathcal{A}K_{\neq}^*}\overline{\widehat{\mathcal{A}Q_{\neq}^*}}=0.
\end{aligned}
\end{equation}
Then we have
\begin{align*}
    \frac12&\frac{d}{dt}|\mathcal{A}(\widehat{Q_{\neq}^*},\widehat{K_{\neq}^*},\widehat{H_{\neq}^*})|^2 +\left[ \nu p+\left(-\frac{\dot{M^*}}{M^*}\right)\right]|\mathcal{A}(\widehat{Q_{\neq}^*},\widehat{K_{\neq}^*},\widehat{H_{\neq}^*})|^2  \\
    &=\lambda^*\nu^{\frac13}|\mathcal{A}(\widehat{Q_{\neq}^*},\widehat{K_{\neq}^*},\widehat{H_{\neq}^*})|^2-\left[(-\frac{\dot{m^*}}{m^*})-\frac12 \frac{\p_t p_{h}}{p_{h}}\right]|\widehat{\mathcal{A}Q_{\neq}^*}|^2-\left[(-\frac{\dot{m^*}}{m^*})+\frac12\frac{l^{2}}{p}\frac{\p_t p_h}{p_h}\right]|\widehat{\mathcal{A}K_{\neq}^*}|^2\\
	&\quad-\left[(-\frac{\dot{m^*}}{m^*})-\frac12 \frac{\p_t p}{p}\right]|\widehat{\mathcal{A}H_{\neq}^*}|^2 -2\sqrt{\frac{\beta-1}{\beta}}\frac{k^2}{p_h}\frac{l}{\sqrt{p}}\widehat{\mathcal{A}K_{\neq}^*}\overline{\widehat{\mathcal{A}Q_{\neq}^*}}.
\end{align*}

Donate that $e_{\neq}^*(t):=|\mathcal{A}(\widehat{Q_{\neq}^*},\widehat{K_{\neq}^*},\widehat{H_{\neq}^*})|^2$. Notice that the definition of $m^*$ as in \eqref{A.....4}, we have
 $ 
 -\frac{\dot{m^*}}{m^*}+\frac{\nu p}{2}-\frac12 \big|\frac{\p_t p}{p_{h}}\big|\gtrsim 0$.
A direct calculation yields
\begin{equation}
\begin{aligned}\label{A010}
    &\frac{d}{dt}e_{\neq}^*(t) + \left[\nu p+(-\frac{\dot{M^*}}{M^*})\right]|\mathcal{A}(\widehat{Q_{\neq}^*},\widehat{K_{\neq}^*},\widehat{H_{\neq}^*})|^2\\
    &\quad\leq 2\lambda^*\nu^{\frac13}|\mathcal{A}(\widehat{Q_{\neq}^*},\widehat{K_{\neq}^*},\widehat{H_{\neq}^*})|^2-\left[(-\frac{\dot{M_5}}{M_5})+2\sqrt{\frac{\beta-1}{\beta}}\frac{k^2}{p_h}\frac{l}{\sqrt{p}}\right]2\widehat{\mathcal{A}K_{\neq}^*}\overline{\widehat{\mathcal{A}Q_{\neq}^*}},
\end{aligned}
\end{equation}
In the following, we introduce two new notations
\begin{align*}
    E_{\neq}^*(t):=&\|\mathcal{A}(Q_{\neq}^*,K_{\neq}^*,H_{\neq}^*)\|_{H^r}^2,\\
    F_{\neq}^*(t):=&\nu\|\mathcal{A}\na_L (Q_{\neq}^*,K_{\neq}^*,H_{\neq}^*)\|_{H^r}^2+\left\|\mathcal{A}\sqrt{-\frac{\dot{M^*}}{M^*}} \left(Q_{\neq}^*,K_{\neq}^*,H_{\neq}^*\right)\right\|_{H^r}^2 ,
\end{align*}
for any $r \geq 0$. Thus, recalling \eqref{A010} and going back to the physical space, we obtain that
\begin{align*}
    &\frac{\rm d}{{\rm d}t}E_{\neq}^*(t) + F^*_{\neq}(t) \leq 2\lambda^*\nu^{\frac13}\|\mathcal{A}(Q_{\neq}^*,K_{\neq}^*,H_{\neq}^*)\|_{H^r}^2.
\end{align*}
Taking $\lambda^*=\frac{1}{16}$   and integrating the above inequality in time yields that
\begin{align*}
    &\|\mathcal{A}(Q_{\neq}^*,K_{\neq}^*,H_{\neq}^*)\|_{H^r}^2 + \nu\|\mathcal{A}\na_L(Q_{\neq}^*,K_{\neq}^*,H_{\neq}^*)\|_{L^2_tH^r}^2 + \left\|\mathcal{A}\sqrt{-\frac{\dot{M^*}}{M^*}}(Q_{\neq}^*,K_{\neq}^*,H_{\neq}^*)\right\|_{L^2_tH^r}^2\\
    & \quad\leq \|\mathcal{A}(Q_{\neq}^*,K_{\neq}^*,H_{\neq}^*)(t=0)\|_{H^r}^2 +\frac{\nu^{\frac13}}{8}\|\mathcal{A}(Q_{\neq}^*,K_{\neq}^*,H_{\neq}^*)\|_{L^2_tH^r}^2.
\end{align*}
Due to $\mathcal{A}^*=m^*M^*e^{\lambda^* \nu^{\frac{1}{3}} t }$, \eqref{4.17} and 
 \eqref{A.11-2},  we have
\begin{equation}\label{Ax.0}
\begin{aligned}
	&\|m^* M^*(Q_{\neq}^*,K_{\neq}^*,H_{\neq}^*)\|_{H^r}^2 + \frac{ \nu}{2}\|m^*M^*\na_L(Q_{\neq}^*,K_{\neq}^*,H_{\neq}^*)\|_{L^2 H^r}^2 + \frac{1}{2} \left\|\sqrt{-\dot{M^*}  M^*}m^*(Q_{\neq}^*,K_{\neq}^*,H_{\neq}^*)\right\|_{L^2 H^r}^2  \\
& 	\quad\leq  e^{-2\lambda^* \nu^{\frac{1}{3}} t} \| (Q_{\mathrm{in}}^*,K_{\mathrm{in}}^*,H_{\mathrm{in}}^*)\|_{H^{r+1}}^2.
\end{aligned}
\end{equation}

Next, recalling the definition of $Q_{\neq}^{*}$, $K_{\neq}^{*}$ and  $H_{\neq}^{*}$ and combining with \eqref{A.11-2}, we have the following fact
\begin{align}\label{Ax.1}
	|\widehat{U^3_{\neq}}|= \frac{p_h^{\frac{1}{2}}}{p}|\widehat{Q_{\neq}^*}| \lesssim \frac{p_h}{p|k|}|m^*M^*\widehat{Q_{\neq}^*}|\lesssim \frac{1}{|k|}|m^*M^*\widehat{Q_{\neq}^*}|,
\end{align}
and
\begin{align}\label{Ax.2}
	|\widehat{\Theta_{\neq}}|=\frac{1}{p^{\frac{1}{2}}}|H_{\neq}^*|\lesssim \frac{1}{|k|}|m^*M\widehat{H_{\neq}^*}|.
\end{align}
Hence, together  \eqref{Ax.0} with \eqref{Ax.1}-\eqref{Ax.2} gives 
\begin{equation}
\begin{aligned}\label{Ax.3}
	\left\|{U^3_{\neq}}\right\|_{L^2}^2\lesssim&\left\| m^*M^*{Q_{\neq}^*}   \right\|_{L^2}^2 \\
	\lesssim &  e^{-2\lambda^* \nu^{\frac{1}{3}} t} \| (Q_{\mathrm{in}}^*,K_{\mathrm{in}}^*,H_{\mathrm{in}}^*)\|_{H^{1}}^2 \\
	\lesssim& e^{-2\lambda^* \nu^\frac{1}{3} t} \left(\| u^3_{\mathrm{in}}\|_{H^3}^2+{\frac{\beta}{\beta-1}} \left\|  \p_xu^2_{\mathrm{in}}-\p_yu^1_{\mathrm{in}}\right\|_{H^2}^2+\| \theta_{\mathrm{in}}\|_{H^2}^2 \right),
\end{aligned}
\end{equation}
similarly, we have
\begin{equation}
\begin{aligned}\label{Ax.4}
	\left\|{\Theta_{\neq}}\right\|_{L^2}^2	\lesssim e^{-2\lambda^* \nu^{\frac{1}{3}} t} \left(\| u^3_{\mathrm{in}}\|_{H^3}^2+{\frac{\beta}{\beta-1}} \left\|  \p_xu^2_{\mathrm{in}}-\p_yu^1_{\mathrm{in}}\right\|_{H^2}^2+\| \theta_{\mathrm{in}}\|_{H^2}^2 \right).
\end{aligned}
\end{equation}

Using  \eqref{3.20} and \eqref{A.11-2}, we get $\widehat{U_{\neq}^1}$ and $\widehat{U_{\neq}^2}$ which satisfy
\begin{equation}\label{Al6}
\begin{aligned}
	|\widehat{U_{\neq}^1}|\leq &\frac{1}{p_h}\left( \sqrt{\frac{\beta-1}{\beta}}  \frac{|\eta-kt| p_h^{\frac{1}{2}}}{p^{\frac{1}{2}}} |\widehat{K_{\neq}^*}|+ \frac{|k||l| p_h^{\frac{1}{2}}}{p} |\widehat{Q_{\neq}^*}|\right)\\
	\lesssim & \frac{1}{|k|} \left(\sqrt{\frac{\beta-1}{\beta}}|m^*M^*\widehat{K_{\neq}^*}| + |m^*M^*\widehat{Q_{\neq}^*}|\right),
\end{aligned}
\end{equation}
and 
\begin{equation}\label{Al7}
	\begin{aligned}
	|\widehat{U_{\neq}^2}|\leq &\frac{1}{p_h}\left( \sqrt{\frac{\beta-1}{\beta}}  \frac{|k| p_h^{\frac{1}{2}}}{p^{\frac{1}{2}}} |\widehat{K_{\neq}^*}|+ \frac{|\eta-kt||l| p_h^{\frac{1}{2}}}{p} |\widehat{Q_{\neq}^*}|\right)\\
	\lesssim &  \frac{1}{p^{\frac12}}\left(\sqrt{\frac{\beta-1}{\beta}} |m^*M^*\widehat{K_{\neq}^*}| +\frac{|l|}{|k|} |m^*M^*\widehat{Q_{\neq}^*}|\right).
\end{aligned}
\end{equation}
Thus,  this together \eqref{Ax.0}  leads to $L^2$-norm of $U_{\neq}^1(t)$ and $U_{\neq}^2(t)$
\begin{equation}\label{Ax.5}
\begin{aligned}
	\left\|{U^1_{\neq}}\right\|_{L^2}^2\lesssim & \left({\frac{\beta-1}{\beta}}\left\|m^*M^*\widehat{K_{\neq}^*}\right\|_{L^2}^2 + \left\|m^*M^*\widehat{Q_{\neq}^*}\right\|_{L^2}^2\right) \\
	\lesssim& \max \left\{ {\frac{\beta-1}{\beta}},  {\frac{\beta}{\beta-1}} \right\} e^{-2\lambda^* \nu^{\frac{1}{3}} t} \left(\| u_{\mathrm{in}}\|_{H^3}^2+\| \theta_{\mathrm{in}}\|_{H^2}^2 \right),
\end{aligned}
\end{equation}
and
\begin{equation}\label{Ax.6}
\begin{aligned}
	\left\|{U^2_{\neq}}\right\|_{L^2}^2\lesssim & {\left\langle {t} \right\rangle}^{-2} \left({\frac{\beta-1}{\beta}}\left\|m^*M^*\widehat{K_{\neq}^*}\right\|_{H^{2}}^2 + \left\|m^*M^*\widehat{Q_{\neq}^*}\right\|_{H^{2}}^2\right) \\
	\lesssim& {\left\langle {t} \right\rangle}^{-2} \,\max \left\{ {\frac{\beta-1}{\beta}},  {\frac{\beta}{\beta-1}} \right\} e^{-2\lambda^* \nu^{\frac{1}{3}} t} \left(\| u_{\mathrm{in}}\|_{H^{5}}^2+\| \theta_{\mathrm{in}}\|_{H^{4}}^2 \right),
\end{aligned}
\end{equation}
where we have used  the fact $p^{-1} \lesssim {\left\langle {t} \right\rangle}^{-2} |k, \eta, l|^2$. Hence, together \eqref{Ax.3}, \eqref{Ax.4}, \eqref{Ax.5} with \eqref{Ax.6} gives \eqref{A.1x}-\eqref{A.3x} and  the proof of Lemma \ref{lemA.4} is finished.    \hfill $\square$

With Lemma \ref{lemA.4} at hand, we  directly  give   the enhanced dissipation of the $L^2$-norm of $U_{\neq}$ and $ \Theta_{\neq}$ and the inviscid damping effect of $U_{\neq}^2$ when $s=0$. In fact, when $s>0$, the enhanced dissipation of $U_{\neq}$ and $\Theta_{\neq}$ can also be obtained by using a similar treatment. 

\vspace{4mm} \noindent\textbf{Acknowledgements.} Sun was partially supported by the National Natural Science Foundation
of China (grant 12501296). Xu was partially supported by the National Key R\&D Program of
China (grant 2020YFA0712900) and the National Natural Science Foundation of China (grants
12171040, 11771045 and 12071069).

\end{document}